\documentclass[a4paper,12pt,reqno,natbib]{amsart}

\usepackage{amsmath}
\usepackage{amssymb}
\usepackage{amsthm}
\usepackage{dsfont}
\usepackage{stmaryrd}
\usepackage{float}
\usepackage{color}
\usepackage{tabu}
\usepackage[ruled]{algorithm2e}
\usepackage{amsmath}
\usepackage{subfig}
\usepackage{graphicx}
\usepackage{hyperref}

\SetKwInOut{Input}{Input}
\SetKwInOut{Output}{Output}
\DontPrintSemicolon

\newtheorem{theorem}{Theorem}
\newtheorem{lemma}[theorem]{Lemma}
\newtheorem{proposition}[theorem]{Proposition}

\newtheorem{definition}[theorem]{Definition}

\newtheorem{remark}[theorem]{Remark}

\newcommand{\dx} {\displaystyle\mathrm{d}x}
\newcommand{\dt} {\displaystyle\mathrm{d}t}
\newcommand{\ds} {\displaystyle\mathrm{d}s}

\def\O{\mathcal O}
\def\S{\mathcal S}
\def\K{\mathcal K}
\def\N{\mathcal N}

\def\H{\textbf{H}}
\def\L{\textbf{L}}
\def\div{\text{div}}

\begin{document}

\title{Inverse problem for the Navier-Stokes equations and identification of immersed obstacles in the Mediterranean Sea}

\author{Mourad Hrizi}
\address[Mourad Hrizi]{Department of Mathematics, High Institute of Applied Mathematics and Informatics of Kairouan,
University of Kairouan, Avenue Assad
Iben Fourat, 3100 Kairouan, Tunisia.}
\email{mourad-hrizi@hotmail.fr}

\author{Marwa Ouni}
\address[Marwa Ouni]{ University of Monastir, Faculty of Sciences of Monastir (FSM), UR13ES64, Analysis and Control of PDEs, 5019 Monastir, Tunisia.}
\email{ounimarwa03@gmail.com}

\author{Maatoug Hassine}
\address[Maatoug Hassine]{ University of Monastir, Faculty of Sciences of Monastir (FSM), UR13ES64, Analysis and Control of PDEs, 5019 Monastir, Tunisia.}
\email{maatoug.hassine@enit.rnu.tn}

\begin{abstract}
This paper presents a theoretical and numerical investigation of object detection in a fluid governed by the three-dimensional evolutionary Navier--Stokes equations. To solve this inverse problem, we assume that interior velocity measurements are available only within a localized subregion of the fluid domain. First, we present an identifiability result. We then formulate the problem as a shape optimization task: to identify the obstacle, we minimize a nonlinear least-squares criterion with a regularization term that penalizes the perimeter of the obstacle to be identified. We prove the existence and stability of a minimizer of the least-squares functional. To recover the unknown obstacle, we present a non-iterative identification method based on the topological derivative. The corresponding asymptotic expansion of the least-squares functional is computed in a straightforward manner using a penalty method. Finally, as a realistic application, we demonstrate the robustness and effectiveness of the proposed non-iterative procedure through numerical experiments using the \textsc{INSTMCOTRHD} ocean model, which incorporates realistic Mediterranean bathymetry, stratification, and forcing conditions.
\end{abstract}



\maketitle
\pagestyle{myheadings} \thispagestyle{plain} \markboth{}{}
\tableofcontents

%%%%%%%%%%%%%%%%%%%%%%%%%%%%%%%%%%%%
\section{Introduction}
%%%%%%%%%%%%%%%%%%%%%%%%%%%%%%%%%%%%

The identification of obstacles immersed in fluids is a problem of both theoretical and practical significance, with applications spanning environmental monitoring, medical imaging, autonomous underwater navigation, and the detection of aquatic mines. A timely example is the recent dam collapse in Ukraine, which displaced previously buried landmines into flooded areas, creating an urgent demand for robust detection methods. In large-scale marine environments such as the Mediterranean Sea, direct observation and localization of submerged objects are often infeasible due to poor visibility, significant depth, and limited accessibility. Under such constraints, inverse modeling techniques—particularly those grounded in fluid dynamics—provide a compelling alternative. By analyzing perturbations in the surrounding flow, these methods enable the reconstruction of hidden inclusions from partial or indirect measurements, such as velocity or pressure data obtained in interior or boundary regions of the fluid domain.

Numerous studies have addressed the identification of obstacles in various fluid regimes, employing mathematical models based on the Stokes, Oseen, and Navier–Stokes equations. For example, Alvarez et al.~\cite{AlvarezIPPIP2005} studied the reconstruction of an inaccessible rigid body $\omega^*$ immersed in a viscous fluid within a bounded domain $\Omega$, using velocity and Cauchy force measurements on $\partial\Omega$. Under suitable smoothness assumptions, they established identifiability for both stationary and time-dependent flows governed by the Stokes and Navier–Stokes systems, along with directional stability estimates. In contrast, Badra et al.~\cite{badra2011detecting} demonstrated the instability of such identification in stationary Stokes flow when only Dirichlet or Neumann boundary data are available.  Extending this analysis to the nonlinear setting, Caubet \cite{caubet2013instability} investigates the problem under the stationary Navier--Stokes equations with non-homogeneous Dirichlet boundary conditions. In a related direction, Caubet and Dambrine \cite{caubet2013stability144} study the stability of equilibrium obstacle shapes in the context of energy dissipation minimization governed by the Stokes equations (i.e., the drag minimization problem).
Building on the foundational work of Alessandrini et al. \cite{alessandrini2002detecting}, Beretta et al. \cite{EE11} derive a quantitative estimate for the size of an immersed obstacle in a viscous fluid governed by the  Stokes system, using velocity and Cauchy stress data collected from the outer boundary. In a related contribution, Heck et al. \cite{heck2007reconstruction} reconstruct obstacles in a bounded domain filled with an incompressible fluid by constructing complex geometrical optics (CGO) solutions for the stationary Stokes system with variable viscosity.
Further contributions by Caubet et al. \cite{CaubetIP2012,CaubetIPI2016} focus on the detection of small obstacles immersed in two- and three-dimensional Stokes flows. Their method combines the Kohn--Vogelius formulation with the topological derivative approach, yielding an efficient and theoretically grounded framework for shape reconstruction. More recently, Rabago et al. \cite{rabago2025detecting} revisited the same inverse identification problem addressed in \cite{CaubetIP2012,CaubetIPI2016}, but introduced a novel solution technique based on the coupled complex boundary method (CCBM). This alternative formulation offers improved numerical stability and computational efficiency, particularly for problems involving complex geometries and limited data. In \cite{ikehata2025integrating}, Ikehata develops an integrated theoretical framework combining the probe method and singular source techniques to address an inverse obstacle problem governed by the Stokes system in a bounded domain. In the context of ideal (inviscid) fluids, Conca et al. study the detection of moving obstacles in \cite{conca2008detecting14}, where they show that, for spherical obstacles, both the position and velocity of the center of mass can be recovered from a single boundary measurement. In \cite{conca2010detection14}, they demonstrate using complex analysis that such identifiability does not hold for arbitrary shapes. However, they extend the analysis to moving ellipses, proving that partial detection is possible when the solid exhibits certain symmetry properties. 
On the other hand, the identification of obstacles immersed in an Oseen fluid is studied in \cite{karageorghis2020identification14}, where the authors demonstrate that the shape of the obstacles can be numerically reconstructed from boundary measurements in the stationary regime, using the Method of Fundamental Solutions (MFS). The issue of unique identifiability for this inverse obstacle problem is further examined by Kress and Meyer in \cite{kress2000inverse14}.

In all these works, identification predominantly relies on boundary data, which are relatively easy to obtain but can yield severely ill-posed problems when the obstacle is deeply embedded or its boundary influence is weak. This often limits reconstruction stability and resolution. In contrast, interior measurements—collected by sensors within the fluid—offer more localized information and typically improve accuracy and robustness. However, such measurements are harder to acquire, especially in deep or inaccessible environments.

Motivated by these challenges, this work addresses the problem of detecting objects immersed in a three-dimensional, time-dependent Navier--Stokes fluid using interior velocity measurements. The aim is to recover both the location and the number of submerged obstacles from velocity data collected within a localized subregion of the fluid domain. This inverse problem is well known to be ill-posed \cite{AlvarezIPPIP2005}.  In this paper, we make substantial contributions to both the mathematical analysis and numerical treatment of this challenging problem. On the theoretical side, we first establish the uniqueness of the solution to the inverse problem. To identify the unknown obstacle, we reformulate the problem as an optimization problem that minimizes a least-squares functional measuring the discrepancy between the observed velocity field and the solution of the evolutionary Navier--Stokes equations restricted to the observation region. To further improve stability, we enhance the cost functional with a regularization term penalizing the perimeter of the unknown obstacle.  We then address two fundamental questions: (i) the existence of an optimal solution, and (ii) its stability under small perturbations of the measured data. This analysis requires proving the continuity of the direct problem with respect to variations of the obstacle in the Hausdorff topology, which we establish using Mosco convergence~\cite{BucurBook2005,bucur2001continuity}.

It is well known that a perimeter penalty corresponds to total variation (TV) regularization. However, numerically solving an optimization problem involving the TV term is challenging due to its non-differentiability. Various remedies have been proposed in the literature, including phase-field relaxations~\cite{aspri2022phase,beretta2018detection,bourdin2003design,rondi2011reconstruction} to address both non-convexity and non-differentiability, as well as reconstruction methods based on classical shape derivatives~\cite{afraites2022new,afraites2022coupled,DelfourBook2001,HenrotBook2005,SokolowskiBook1992}. Most of these strategies are iterative and require an initial guess for the obstacle’s location and shape. In contrast, the second part of the present work introduces a self-regularized, non-iterative approach based on topological sensitivity analysis, which operates without any \emph{a priori} knowledge of the obstacle’s position or geometry. The key idea is to compute an asymptotic expansion of the least-squares functional with respect to the insertion of an infinitesimal topological perturbation in the domain. The leading term in this expansion defines the \emph{topological gradient}, an indicator function that reveals the presence of hidden obstacles. Exploiting this quantity, we develop a fast, one-shot detection algorithm that localizes obstacles without iterative refinement. The effectiveness of the proposed method is validated through numerical experiments in a realistic Mediterranean Sea configuration (see Figure \ref{DD14}), using the three-dimensional ocean circulation model INSTMCOTRHD~\cite{alioua,marwa}. This model, built upon the Princeton Ocean Model framework, incorporates realistic bathymetry, stratification, and forcing conditions relevant to oceanographic applications.

\begin{figure}[!htb]
    \centering
    \includegraphics[width=50mm]{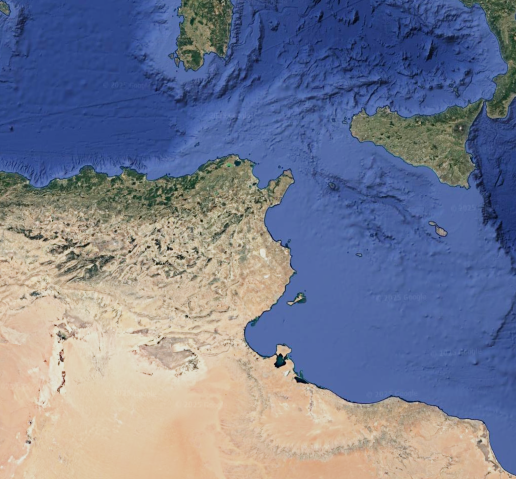}
    \caption{Study domain in the Mediterranean Sea.}
    \label{DD14}
\end{figure}

For completeness, we briefly review the concept of topological sensitivity. Topological sensitivity analysis provides a mathematical framework for quantifying how a shape-dependent cost functional changes in response to small geometric perturbations---such as the introduction of inclusions, cavities, cracks, or localized sources. The concept was first introduced by Schumacher \cite{SchumacherPhD1995} in the context of compliance minimization in linear elasticity. A rigorous mathematical foundation was later established by Sokołowski and Żochowski \cite{sokolowski1999topological}, who analyzed the Laplace operator under circular perturbations. Masmoudi subsequently developed a more general formulation based on the generalized adjoint method and the truncation technique \cite{masmoudi2002topological}. Since then, the framework has been extended to a wide range of partial differential equations \cite{AmstutzJMPA2006,garreau2001topological, bonnet2006topological,masmoudi2005topological, pommier2004topological, samet2003topological,JleliJMPA2015}. 
For a comprehensive overview, see the monograph by Novotny and Sokołowski \cite{novotny2012topological}. This methodology can be viewed as a specific case of asymptotic techniques thoroughly discussed in the books of Ammari and Kang \cite{AmmariBook2004} and Ammari et al. \cite{AmmariBook2013}. Stability and resolution analyses of topological-derivative-based imaging functionals have been carried out by Ammari et al. \cite{AmmariSIAM2012, AmmariSIAM2013}, demonstrating the method’s effectiveness in inverse scattering and elasticity problems; see also related contributions in \cite{GuzinaPRS2015}.\\

\noindent To introduce the main concept of the topological gradient, we consider a shape functional  $\Omega\longmapsto K(\Omega)$ that we aim to minimize, where $\Omega\subset\mathbb{R}^d$ (with $d=2,\,3$) is an open and bounded domain. For a given $\varepsilon>0$, let 
$\Omega\backslash\overline{\mathcal{C}_{z,\varepsilon}}$ represent the perturbed domain formed by removing a small topological perturbation defined as $\mathcal{C}_{z,\varepsilon}=z+\varepsilon\mathcal{C}$  from the original (unperturbed) domain $\Omega$, where $z\in\Omega$ and $\mathcal{C}\subset\mathbb{R}^d$ is a fixed, bounded domain containing the origin. The analysis of topological sensitivity leads to an asymptotic expansion of the shape functional $K$ in the following form :
\begin{align}\label{asy}
K(\Omega\backslash\overline{\mathcal{C}_{z,\varepsilon}})=K(\Omega)+\mu(\varepsilon)D_K(z)+o(\mu(\varepsilon)),
\end{align}
where:
\begin{itemize}
\item $\mu(\varepsilon)$ is a positive function that depends on the size of the geometric perturbation $\varepsilon$ and approaches zero as $\varepsilon$ tends to zero.
\item The function $z\mapsto D_K(z)$ is referred to as the ``topological sensitivity'' or ``topological gradient'' of $K$ at the point $z$. It can be mathematically defined as
\begin{equation}\label{DT-limit}
D_K(z):=\lim_{\varepsilon\rightarrow0}\frac{K(\Omega\backslash\overline{\mathcal{C}_{z,\varepsilon}})-K(\Omega)}{\mu(\varepsilon)}.
\end{equation}
\end{itemize}
In general, the topological gradient $z \mapsto D_K(z)$ provides a spatial map of sensitivity values throughout the domain $\Omega$, thereby serving as an effective indicator for the optimal location of topological changes---such as the insertion or removal of material in shape optimization, or the identification of unknown obstacles, as is the case in this work. In other words, to minimize the cost functional $K$, the optimal location for introducing a small perturbation in $\Omega$ corresponds to the region where $D_K$ attains its most negative values. Specifically, if $D_K(z)<0$, it follows that $$K(\Omega\backslash\overline{\mathcal{C}_{z,\varepsilon}})\leq
K(\Omega)\;\;\; \text{for sufficiently small values of}\;\; \varepsilon.$$ 
A special case occurs when \(\Omega \setminus \overline{\mathcal{C}_{z,\varepsilon}}\) is generated from \(\Omega\) via a family of smooth transformations  
\(
T_\varepsilon = I + \varepsilon \mathcal{V},
\) 
where \(\mathcal{V} : \mathbb{R}^d \to \mathbb{R}^d\) is a Lipschitz vector field and \(\mu(\varepsilon) = \varepsilon\).  
In this case, the limit~\eqref{DT-limit} coincides with the \emph{shape derivative} of \(K\), so the topological gradient can be viewed as a natural generalization of shape derivatives to problems involving topological changes.

The computation of the topological gradient \(D_K\) for the stationary Navier--Stokes equations was addressed in \cite{AmstutzESAIM2005} using a truncation method analogous to that employed in elasticity problems \cite{garreau2001topological}. In contrast, the authors in \cite{hassine2023topological} developed a more elaborate approach---avoiding truncation---to compute \(D_K\) for the non-stationary Navier--Stokes equations. A common difficulty in both works was that \(D_K\) depended on the shape of the inclusion \(\mathcal{C}\), with an explicit formula available only for the specific case \(\mathcal{C} = B(0,1)\). From a numerical perspective, reconstructing the shape of the obstacle is less critical in our setting---unlike determining its location---since we are working within a very large fluid domain (the Mediterranean Sea). Motivated by this, we adopt here a simpler approach based on a penalization technique combined with standard \emph{a priori} estimates of the state variables to compute the topological gradient for our problem. The main idea is to remove the Dirichlet boundary condition on \(\partial \mathcal{C}_{z,\varepsilon}\) by introducing a penalty parameter \(k\) that is large inside \(\mathcal{C}_{z,\varepsilon}\) and vanishes in the background \(\Omega \setminus \overline{\mathcal{C}_{z,\varepsilon}}\). This is reminiscent of the classical penalization technique used in finite element methods to enforce Dirichlet conditions \cite{angot1999penalization}.

This paper is organized as follows. We first introduce the general notation that we
adopt. Then, in Section~\ref{sec:problem-setting}, we describe in detail the considered inverse problem and state its uniqueness result. 
Section~\ref{sec:optimization-problem} reformulates the inverse problem as a topology optimization problem, while Section~\ref{well-posedness-weak-solution-direct-problem} addresses the existence and stability of an optimal solution. In Section~\ref{sec:gradient-topologique}, we compute the topological sensitivity of the problem using a penalization technique. Section~\ref{one-iter-alghorithm} presents a one-shot detection procedure based on the computed topological gradient and evaluates its robustness and effectiveness through numerical experiments using the \textsc{INSTMCOTRHD} ocean model. Concluding remarks are given in Section~\ref{sec:conclusion}. Finally, Section~\ref{sec:jutification} contains the proof of the uniqueness theorem stated in Section~\ref{sec:problem-setting} and the proof of a convergence result for the penalized problem used in Section~\ref{sec:gradient-topologique}.\\

\paragraph{\bf{General notation}}  
In this paper, we adopt the following notations. Let \( \O \subset \mathbb{R}^3 \) be an open set. The Lebesgue measure of \( \O \) is denoted by \( |\O| \).  We denote by \( L^p(\O) \), \( H^1_0(\O) \), and \( H^m(\O) \) the classical Lebesgue and Sobolev spaces, respectively. For vector-valued function spaces, we use bold notation: \( \mathbf{L}^p(\O) \), \( \mathbf{H}^1_0(\O) \), and \( \mathbf{H}^m(\O) \). We define the divergence-free Sobolev space as  
\begin{align*}  
\H^1_{0,\text{div}}(\O) = \big\{ v \in \H^1_0(\O) \mid \text{div } v = 0 \big\},  
\end{align*}  
with its dual space denoted by \( \big(\H^1_{0,\text{div}}(\O)\big)^\prime \).  The duality pairing between this dual space and \( \H^1_{0,\text{div}}(\O) \) is written as \(\big<\cdot,\cdot\big>_\O\).
For two second-order tensors \( \boldsymbol{A} = \{A_{ij}\} \) and \( \boldsymbol{B} = \{B_{ij}\} \) in the three-dimensional Euclidean space \( \mathbb{R}^3 \), we use the standard notation:  
\begin{equation*}
\boldsymbol{A} : \boldsymbol{B} = \sum_{i, j=1}^{3} A_{ij} B_{ij},  
\quad |\boldsymbol{A}| = \left( \sum_{i, j=1}^{3} A_{ij} A_{ij} \right)^{1 / 2}.  
\end{equation*}  
Additionally, for a vector \( \boldsymbol{a} \in \mathbb{R}^3 \), we define the operations  
\begin{equation*}
\boldsymbol{a} \cdot \boldsymbol{A} \quad \text{and} \quad \boldsymbol{A} \cdot \boldsymbol{a},  
\end{equation*}  
where their components are given by  
\begin{equation*}
(\boldsymbol{a} \cdot \boldsymbol{A})_j = \sum_{i=1}^{3} a_i A_{ij},  
\quad \text{and} \quad  
(\boldsymbol{A} \cdot \boldsymbol{a})_i = \sum_{j=1}^{3} a_j A_{ij}.  
\end{equation*}  
Moreover, given a Banach space \( \mathcal{Y} \) with norm \( \|\cdot\|_\mathcal{Y} \), and an interval \( I = (a_0, b_0) \), we denote by \( L^p(I; \mathcal{Y}) \) the space of functions \( h: I \rightarrow \mathcal{Y} \) such that  
\begin{equation*}
\|h\|_{L^p(I; \mathcal{Y})} =  
\begin{cases}  
\displaystyle\left(\int_{a_0}^{b_0} \|h(t)\|_\mathcal{Y}^p \,\dt \right)^{1/p}, & 1 \leq p < \infty, \\  
\displaystyle\operatorname{ess\,sup}_{t \in I} \|h(t)\|_\mathcal{Y}, & p = \infty,  
\end{cases}  
\end{equation*}  
is finite. Finally, to introduce the upcoming definition of domain regularity, we adopt the following notation. For any \( x \in \mathbb{R}^d \) (with $d=2,\,3$), we write
\[
x = (x^\prime, x_d), \quad \text{where } x^\prime \in \mathbb{R}^{d-1},\; x_d \in \mathbb{R}.
\]
Given \( r > 0 \), we denote by \( \mathcal{B}_r(x) \subset \mathbb{R}^d \) the set
\(
\mathcal{B}_r(x) := \left\{ (x^\prime, x_d) \in \mathbb{R}^d \; ; \; |x^\prime|^2 + x_d^2 < r^2 \right\},
\)
and we denote by \( \mathcal{B}_r^\prime(x^\prime) \subset \mathbb{R}^{d-1} \) the set
\(
\mathcal{B}_r^\prime(x^\prime) := \left\{ x^\prime \in \mathbb{R}^{d-1} \; ; \; |x^\prime|^2 < r^2 \right\}.
\)

\begin{definition}[Definition 2.1 \cite{alessandrini2002detecting}]\label{def-0}
Let $\O$ be a bounded domain in $\mathbb{R}^d$. Given $m,$ $\beta$ with $m\in\mathbb{N}$, $0<\beta\leq1$, we say that the boundary $\partial \O$ is of class $\mathcal{C}^{m,\beta}$ with constants $r_0$, $N_0,$ if for any point $s$ in the boundary $\partial \O$, there exists a rigid transformation of coordinates under which we have that $s$ is mapped to the origin (i.e., $s=0$) and
    \begin{equation*}
        \mathcal{B}_{r_0}(0)\cap\Omega=\big\{x\in \mathcal{B}_{r_0}(0)\; ; \; x_d> \Theta(x^\prime)  \big\},
    \end{equation*}
    where $\Theta$ is a $\mathcal{C}^{m,\beta}$ function on $\mathcal{B}^\prime_{r_0}(0)$, such that 
\begin{align*}
    &\Theta(0)=0,\\
    &\nabla\Theta(0)=0,\quad\text{when}\;m\geq1,\\
    &\big\| \Theta\big\|_{\mathcal{C}^{m,\beta}(\mathcal{B}^\prime_{r_0}(0),\mathbb{R}^d)}\leq r_0N_0.
\end{align*}
When $m=0$ and $\beta=1$ (i.e. $\partial \O\in \mathcal{C}^{0,1}$ with constants $r_0$, $N_0$), we also say that $\partial \O$ is of Lipschitz class with constants $r_0$, $N_0$.
\end{definition}

%%%%%%%%%%%%%%%%%%%%%%%%%%%%%%%%%
\section{The problem setting} \label{sec:problem-setting}
%%%%%%%%%%%%%%%%%%%%%%%%%%%%%%%%%
 Let $T>0$ and $\Omega\subset\mathbb{R}^3$ be a bounded open domain of class $\mathcal{C}^{2,1}$ with constants $r_0$ and $N_0$, containing an incompressible Newtonian fluid. Within this fluid flow domain $\Omega$, we assume the existence of an obstacle (rigid body) $\omega^*$ immersed in it (i.e. $\omega^*\subset\subset\Omega$). For simplicity, and without any loss of generality, we assume that the fluid density is equal to one. The fluid motion in the spatial-temporal domain $(\Omega\backslash\overline{\omega^*})\times (0,\,T)$ is governed by the non-stationary Navier-Stokes equations. For a given source term \(\mathcal{G}\), typically representing external forces such as gravity, and boundary data \(\phi \in \mathcal{C}^1([0,T]; \H^{\frac{3}{2}}(\partial \Omega))\) satisfying the flux compatibility condition
 \begin{equation}\label{condition-normalization}
    \displaystyle\int_{\partial\Omega}\phi\cdot\textbf{n}\,\ds=0,
  \end{equation}  
    the velocity field $u$ and pressure field $\pi$ satisfy the following system:
\begin{equation}\label{direct-problem}
\left\{
\begin{array}{rll}
\displaystyle \frac {\partial u}{\partial t}-\nu
\Delta u+ \N(u)+ \nabla \pi
 &= \mathcal{G} &\mbox{ in }  \,\,(\Omega\backslash\overline{\omega^*})\times (0,\,T) ,\\
\mbox{div}\, u &=  0 & \mbox{ in } \,\, (\Omega\backslash\overline{\omega^*}) \times (0,\,T) ,\\
u &= \phi & \mbox{ on }\,\,  \partial\Omega \times (0,\,T) , \\
u &=  0 & \mbox{ on } \,\, \partial\omega^* \times (0,\,T) ,\\
u(.,0)&= 0&\mbox{ in } \,\,  \Omega\backslash\overline{\omega^*}.
\end{array}
 \right.
\end{equation}
In this context, \( \mathbf{n} \) denotes the unit outward normal vector along the boundary \( \partial \Omega \). The parameter \( \nu > 0 \) represents the kinematic viscosity coefficient of the fluid, which can be interpreted as \( 1/Re \), where \( Re \) is the Reynolds number. Moreover, the convective term \( \mathcal{N}(u) \) is defined as
$$
\N(u):=(u\cdot\nabla)u=\big(\sum_{j=1}^{3}u_j {\partial}/{\partial x_j}\big)u,
$$
with $u_j$ is the $j^{th}$ component of the velocity field $u$ and ${\partial}/{\partial x_j}$ is the partial derivative with respect the $j^{th}$ coordinate $x_j$.

%%%%%%%%%%%%%%%%%%%%%%%%%%%%

\vspace{0.3cm}

The forward problem consists of determining the velocity field \( u \) and the pressure field \( \pi \) in the fluid domain \( (\Omega \setminus \overline{\omega^*}) \times (0, T) \), given the Dirichlet boundary data \( \phi \), the source term \( \mathcal{G} \), and the obstacle \( \omega^* \). To analyze the well-posedness of this problem, we adopt the classical decomposition approach introduced by Leray \cite{leray1933etude} and further developed by Hopf \cite{hopf1940allgemeiner,hopf1955nonlinear}. Specifically, we write
\[
u = v + V,
\]
where \( V \) is a sufficiently smooth, divergence-free (solenoidal) vector field in \( \Omega \setminus \overline{\omega^*} \), satisfying the boundary conditions \( V = \phi \) on \( \partial\Omega \) and \( V = 0 \) on \( \partial\omega^* \). For a more comprehensive treatment of the existence of such a vector field $V$, the reader is referred to Lemma IV.2.3 in \cite{girault2012finite} and Lemma IX.4.2 in \cite{galdi2011introduction}. Consequently, the auxiliary field \( v \) satisfies a Navier--Stokes system with homogeneous Dirichlet boundary conditions in the perforated domain \( \Omega \setminus \overline{\omega^*} \). Therefore, proving the existence and uniqueness of a solution to the original system \eqref{direct-problem} reduces to establishing the well-posedness of the problem for \( v \). The existence and uniqueness of solutions to this homogeneous Navier--Stokes system have been extensively studied in the literature. We refer, for example, to the foundational works of Leray \cite{leray1933etude,leray1934essai, leray1934mouvement}, among many others. In her monograph \cite{ladyzhenskaya1969mathematical}, Ladyzhenskaya further observed that a Leray--Hopf weak solution of \eqref{direct-problem} is unique in the class \( L^8(0, T ; \L^4(\Omega)) \), which corresponds to a particular case of the Prodi--Serrin condition \cite{prodi1959teorema, Serrin1963TheIV} for the uniqueness of Leray--Hopf solutions:
\begin{align}\label{c-L}
v \in L^p(0, T ; \L^q(\Omega)), \quad \frac{2}{p} + \frac{3}{q} \leq 1, \quad q > 3.
\end{align}
J. L. Lions (1960) later generalized this result, proving that uniqueness holds in any spatial dimension \( d \), provided that \( v \in L^s(0,T; \L^r(\Omega)) \) with
\begin{equation*}
    \frac{2}{s} + \frac{d}{r} \leq 1, \quad \text{if } \Omega \text{ is bounded,}
\end{equation*}
and
\begin{equation*}
    \frac{2}{s} + \frac{d}{r} = 1, \quad \text{if } \Omega \text{ is unbounded.}
\end{equation*}
Moreover, Ladyzhenskaya \cite{ladyzhenskaya1969uniqueness} proved that any Leray--Hopf solution satisfying condition \eqref{c-L} is in fact smooth. This condition, now known as the Ladyzhenskaya--Prodi--Serrin regularity criterion, plays a fundamental role in the theory of incompressible flows. The endpoint case \( L^{\infty}(0, T ; \L^3(\Omega)) \) was later established by Escauriaza, Seregin, and {\v{S}}ver\'ak \cite{escauriaza2003l3} in 2003.  Furthermore, Lions and Masmoudi \cite{lions2001uniqueness} proved the uniqueness of mild and very weak solutions of the Navier--Stokes equations in the space \( \mathcal{C}([0,T); \L^3(\Omega)) \). Notably, Ladyzhenskaya and Kiselev (1957) \cite{olga1963} proved the existence of a weak solution to the Navier--Stokes problem \eqref{direct-problem} under the assumption that the Dirichlet boundary data $\phi$ belongs to the class \( L^{\infty}(0, T ; \L^4(\Omega)) \cap H^1([0,T];\H^1(\Omega)) \), and the source term lies in \( L^2(0,T; (\H^1_{0,\text{div}}(\Omega))') \). For completeness, we also mention the asymptotic regime where the Reynolds number \( Re \) is very large (i.e., viscosity \( \nu \) is small). In this case, Masmoudi \cite{masmoudi2007remarks} proved that the Navier--Stokes system \eqref{direct-problem} with vanishing boundary data (\( \phi = 0 \)) and nonzero initial condition behaves asymptotically like the incompressible Euler system. Related results can be found in the works of Constantin \cite{Constantin1986}, Swann \cite{swann1971convergence}, and Kato \cite{kato1972nonstationary}. In another direction, Masmoudi \cite{masmoudi1998euler} showed that weak solutions of the Navier--Stokes equations with a large Coriolis term converge to the Euler system with damping, as the Rossby number and both horizontal and vertical viscosities tend to zero. This convergence analysis makes use of boundary layer theory, particularly the Ekman layer \cite{ekman1905influence}, under appropriate initial data assumptions.

\medskip

The inverse problem addressed in the present paper involves the identification of an unknown obstacle \( \omega^* \) from partial domain measurements of the velocity field. To formalize this, let \( D^* \subset\subset \Omega \) be a fixed nonempty open set, and consider the following class of admissible obstacles:
\[
\mathcal{D} = \left\{ \omega \subset\subset \Omega \; ; \; \omega \text{ is a Lipschitz open set},\; \Omega \setminus \overline{\omega} \text{ is connected},\; \text{and } \omega \subset\subset D^* \right\}.
\]

\medskip

\noindent\textbf{Inverse Problem.}  
Let \( \Omega_0 \subset \Omega \setminus \overline{D^*} \) be a nonempty open set. The inverse problem studied in this work is to reconstruct an unknown obstacle \( \omega^* \in \mathcal{D} \) from internal measurements of the velocity field \( u \) in \( \Omega_0 \times (0, T) \). More precisely, given measurements \( u_{\text{meas}} \in L^2(0,T; \L^2(\Omega_0)) \), determine \( \omega^* \in \mathcal{D} \) such that
\[
u = u_{\text{meas}} \quad \text{in } \Omega_0 \times (0, T).
\]
%where \( u \) is the solution to the problem \eqref{direct-problem} corresponding to \( \omega^* \).\\

\begin{remark}
We say that the measured velocity data \( u_{\text{meas}} \) is compatible if there exists an obstacle \( \omega^* \in \mathcal{D} \) such that the solution \( u \) to the forward problem \eqref{direct-problem}, corresponding to this obstacle, satisfies \( u = u_{\text{meas}} \) in \( \Omega_0 \times (0, T) \). Throughout this work, we assume that the data \( u_{\text{meas}} \) is compatible, i.e., the associated inverse problem admits at least one solution \( \omega^* \).
\end{remark}

In the study of the geometric inverse problem under consideration, three fundamental aspects typically arise: uniqueness (identifiability), stability, and identification (reconstruction). As a first step, we establish the following identifiability result.

\begin{theorem}[Uniqueness]\label{uniqueness} Let $\Omega_0$ be a nonempty open subset of $\Omega\backslash\overline{D^*}$.
Let $\omega_1^*$ and $\omega_2^*$ be two sets in $\mathcal{D}$, \( \mathcal{G} \in L^2(0,T; (\H^1_{0,\mathrm{div}}(\Omega))') \), and \( \phi \in  \mathcal{C}^1([0,T]; \H^{\frac{3}{2}}(\partial \Omega)) \) with $\phi\not=0$, satisfying the flux condition \eqref{condition-normalization}. Let \( (u_\ell, \pi_\ell) \) for \( \ell = 1,2 \) be a solution of 
    \begin{equation}\label{u-ell}
    \left\{
    \begin{array}{rll}
        \displaystyle \frac{\partial u_\ell}{\partial t} - \nu \Delta u_\ell + \N(u_\ell) + \nabla \pi_\ell
        &= \mathcal{G}, &\quad \text{in } (\Omega \setminus \overline{\omega^*_\ell}) \times (0, T), \\[8pt]
        \text{div}\; u_\ell &= 0, &\quad \text{in } (\Omega \setminus \overline{\omega^*_\ell}) \times (0, T), \\[8pt]
        u_\ell &= \phi, &\quad \text{on } \partial\Omega \times (0, T), \\[8pt]
        u_\ell &= 0, &\quad \text{on } \partial\omega^*_\ell \times (0, T), \\[8pt]
        u_\ell(\cdot,0) &= 0, &\quad \text{in } \Omega \setminus \overline{\omega^*_\ell}.
    \end{array}
    \right.
    \end{equation}
    Assume that \( (u_\ell, \pi_\ell) \) are such that
    \begin{equation}\label{obsevtion-condition-1}
      u_1 = u_2 \quad \text{in } \Omega_0 \times (0, T).   
    \end{equation}
    Then, it follows that \( \omega^*_1 = \omega^*_2 \).
\end{theorem}

\begin{proof}
    For the proof of the uniqueness result stated in Theorem \ref{uniqueness}, we refer the reader to Section \ref{sec:uniqueness}.
\end{proof}

\begin{remark}
The assumption \( \Omega_0 \subset \Omega \setminus \overline{D^*} \) ensures that \( \Omega_0 \cap \omega^*_\ell = \emptyset \) for \( \ell = 1, 2 \), which is a key requirement for establishing the uniqueness result. While a more ideal assumption would be \( \Omega_0 \subset\subset \Omega \setminus \overline{\omega^*_1 \cup \omega^*_2} \), this is not practical, as the obstacles \( \omega^*_1 \) and \( \omega^*_2 \) are unknown. This practical limitation motivates the introduction of the intermediate subdomain \( D^* \subset\subset \Omega \) in the problem formulation.
\end{remark}

To streamline the mathematical framework, we adopt the following assumption throughout the remainder of the paper:

\paragraph{\textbf{Assumption (A1).}}  
The source term \( \mathcal{G} \) is a nontrivial function and is assumed to be sufficiently small.

\begin{remark} The assumption $(\textbf{A}1)$ plays a crucial role in ensuring the existence and uniqueness of solutions to \eqref{direct-problem}; see, for example, \cite[Section 3.5.2]{temam1973theory}. 
\end{remark}

Next, we reformulate the geometric inverse problem of detecting the obstacle \( \omega^* \) as an optimization problem.

%%%%%%%%%%%%%%%%%%%%%%%%%%%%%%%%%%%%%%%%%%%%%%%%%%%%%%%%%%%%%%%%%%%%%%%%%%%%%%%%%
\section{Optimization problem}\label{sec:optimization-problem}
%%%%%%%%%%%%%%%%%%%%%%%%%%%%%%%%%%%%%%%%%%%%%%%%%%%%%%%%%%%%%%%%%%%%%%%%%%%%%%%%%

The inverse problem of determining $\omega^*$ in \eqref{direct-problem} is known to be ill-posed \cite{AlvarezIPPIP2005,doubova2007identification}. To address this challenge, we reformulate it as a topology optimization problem. To present this optimization framework, we introduce the following class of admissible geometries:
\begin{equation}\label{D-ad}
    \mathcal{D}_{ad}:=\big\{ \omega\subset\subset\Omega\,;\; \omega\text{ is an open set },\; \partial\omega\in\mathcal{C}^{0,1}\;\;\text{with constants } r_0,\; N_0,\text{ and } \omega\subset\subset D^* \big\}.
\end{equation}
Unlike the class $\mathcal{D}$, the admissible class $\mathcal{D}_{ad}$ does not require $\Omega \setminus \overline{\omega}$ to be connected. The mathematical justification for this choice is provided in Remark~\ref{rem-jus}.

In this context, the unknown obstacle $\omega^*$ is determined as the solution to the following optimization problem:
\begin{equation}\label{topopt-problem}
\underset{\omega \in \mathcal{D}_{ad}}{\operatorname{Minimize}}\; \mathcal{K}(\omega)
\quad \text{subject to} \quad \eqref{problem-initial-guess},
\end{equation}
where $\mathcal{K}$ denotes the least-squares cost functional, defined for each admissible obstacle $\omega \in \mathcal{D}_{ad}$ by
\begin{equation}\label{L-S}
\mathcal{K}(\omega) := \int_0^T \int_{\Omega_0} 
\big| u_\omega - u_{\mathrm{meas}} \big|^2 \, \mathrm{d}x \, \mathrm{d}t.
\end{equation}
Here, the velocity field \(u_\omega\) and the associated pressure \(\pi_\omega\) denote 
the solution of the Navier–Stokes system in \(\Omega \setminus \overline{\omega}\):
\begin{equation}\label{problem-initial-guess}
\left\{
\begin{array}{rll}
\displaystyle \frac {\partial u_\omega}{\partial t}-\nu
\Delta u_\omega+ \N(u_\omega)+ \nabla \pi_\omega
 &= \mathcal{G} &\mbox{ in }  \,\,(\Omega\backslash\overline{\omega})\times (0,\,T) ,\\
\mbox{div}\, u_\omega &=  0 & \mbox{ in } \,\, (\Omega\backslash\overline{\omega}) \times (0,\,T) ,\\
u_\omega &= \phi & \mbox{ on }\,\,  \partial\Omega \times (0,\,T) , \\
u_\omega &=  0 & \mbox{ on } \,\, \partial\omega \times (0,\,T) ,\\
u_\omega(.,0)&= 0&\mbox{ in } \,\,  \Omega\backslash\overline{\omega}.
\end{array}
 \right.
\end{equation}

Since the interior observation data \(u_{\mathrm{meas}}\) is compatible, there exists 
\(\omega \in \mathcal{D}_{ad}\) solving the inverse problem. For this \(\omega\), we have 
\(u_\omega = u_{meas}\) in \(\Omega_0 \times (0,T)\), which implies \(\mathcal{K}(\omega) = 0\); 
hence, \(\omega\) is a minimizer of \(\mathcal{K}\). 
Let \(\omega \in \mathcal{D}_{ad}\) be a solution of \eqref{topopt-problem} and 
\(\widetilde{\omega} \in \mathcal{D}_{ad}\) be another solution such that 
\(\mathcal{K}(\widetilde{\omega}) = 0\). Then 
\(u_{\widetilde{\omega}} = u_{\mathrm{meas}} = u_{\omega}\) in 
\(\Omega_0 \times (0,T)\). By the identifiability result 
(Theorem~\ref{uniqueness}), it follows that \(\omega = \widetilde{\omega}\).  In summary, this discussion implies that the solution of \eqref{topopt-problem} is ``equivalent" to the solution of the considered inverse problem.

From a numerical perspective, the stability of the reconstruction remains a significant challenge, particularly in the presence of noise \cite{hadamard1923lectures}, as small perturbations in the measurements can lead to large deviations in the recovered obstacle. To address this numerical instability, regularization strategies are commonly employed to stabilize the inverse operator. Popular approaches include Tikhonov regularization and total variation techniques. In the present work, we adopt a regularization framework that augments the quadratic misfit functional \( \mathcal{K} \) with a perimeter-based penalty term designed to promote geometrically meaningful reconstructions. This leads to the formulation of a regularized optimization problem aimed at improving both the stability and robustness of the solution. Accordingly, we consider the following minimization problem:
\begin{equation}\label{regularization-topopt-problem}
    \displaystyle\operatorname*{Minimize}_{\omega\in\mathcal{D}_{ad}}\
\mathcal{K}_\gamma(\omega) = \mathcal{K}(\omega)+\gamma\text{Per}_\Omega(\omega),
\end{equation}
where $\gamma>0$ is a regularization parameter and $\text{Per}_\Omega(\omega)$ denotes the relative perimeter of $\omega$ in $\Omega.$ For completeness, we provide the definition of the relative perimeter.

\begin{definition}[See \cite{HenrotBook2005}]
    The relative perimeter of a set $\omega$ in $\Omega$ is defined according to the De Giorgi formula as:
    \begin{align}
\text{Per}_\Omega(\omega)=\sup\left\{\int_{\omega}\text{div}\
\Psi\,\dx\; ; \; \Psi\in \mathcal{C}^1_c(\Omega,\mathbb{R}^3),\;\|\Psi\|_{\infty}\leq1\right\},
\end{align}
where $\|\cdot\|_{\infty}$ is the essential supremum norm and $\mathcal{C}^1_c(\Omega,\mathbb{R}^3)$ is the space of continuously differentiable functions with compact support in $\Omega$. If $\text{Per}_\Omega(\omega)<\infty$, we say that $\omega$  has a finite perimeter in
$\Omega$. In this case, the perimeter $\text{Per}_\Omega(\omega)$ coincides with the Total Variation of the distributional gradient of $\chi_{\omega}$ (the characteristic function of the set $\omega$), namely
\begin{equation}
\text{Per}_\Omega(\omega)=|D\chi_{\omega}|(\Omega).
\end{equation}
\end{definition}

Next, we discuss the existence and stability of an optimal solution to the optimization problem \eqref{regularization-topopt-problem}.

%%%%%%%%%%%%%%%%%%%%%%%%%%%%%%%%%%%%%%%%%%%%%%%%%%%%%%%
\section{Well-posedness of the optimization problem}\label{well-posedness-weak-solution-direct-problem}
%%%%%%%%%%%%%%%%%%%%%%%%%%%%%%%%%%%%%%%%%%%%%%%%%%%%%%%

In this section, we address the question related to the existence and stability of a solution to the optimization problem described in \eqref{regularization-topopt-problem}. We start by examining essential foundational properties that underpin our analysis. These include the convergence in the Hausdorff sense, compactness of the set $\mathcal{D}_{ad}$, and Mosco convergence in Sobolev spaces.

\vspace{0.4cm}
\paragraph{\bf Hausdorff distance} Given two non-empty closed sets $F$ and $\Tilde{F}$ of $\Omega$, the Hausdorff distance between these sets is defined as follows:
\begin{equation}
    d_H(F,\,\Tilde{F})=\max\bigg\{\sup_{x\in F}\,\inf_{y\in \Tilde{F}} \text{dist}(x,y),\; \sup_{x\in \Tilde{F}}\,\inf_{y\in {F}} \text{dist}(x,y)\bigg\},
\end{equation}
where $\text{dist}(x,y)$ represents the Euclidean distance between the points $x$ and $y.$ In the case of open sets, where $F$ and $\Tilde{F}$ are open subsets of $\Omega$, the Hausdorff distance is defined in terms of their complements:
\begin{equation}
    d_{H^c}(F,\,\Tilde{F}):=d_H(\Omega\backslash F,\,\Omega\backslash\Tilde{F}).
\end{equation}
For further details and properties of the Hausdorff distance, we refer to \cite[Chapter 2]{HenrotBook2005}.

\vspace{0.4cm}
\paragraph{\bf Compactness of the class $\mathcal{D}_{ad}$} Let $\left\{\omega_n\right\}_{n \in \mathbb{N}}\subset\mathcal{D}_{ad}$  and $\omega\in\mathcal{D}_{ad}$. We say that $\omega_n$ converges to $\omega$ in the Hausdorff sense if
$$d_{H^c}(\omega_n,\omega)\to0\;\; \text{as}\;\; n\to\infty.$$
The following compactness result holds for the class of admissible solutions $\mathcal{D}_{ad}$:
\begin{lemma}\label{compacteness-D-ad}Let $\omega_n$ be a sequence of open sets in the class $\mathcal{D}_{ad}$. Then, there exists an open set 
$\omega \in\mathcal{D}_{ad}$ and a subsequence $\omega_{n_k}$ that converges to $\omega$ in the Hausdorff sense, and in the sense of characteristic functions. Furthermore,  $\overline{\omega}_{n_k}$ and $\partial \omega_{n_k}$ converge in the Hausdorff sense to $\overline{\omega}$ and $\partial \omega$, respectively.
\end{lemma}

\begin{proof}
   This result follows directly from Theorem 2.4.10 and Theorem 2.4.7 in \cite{HenrotBook2006}.
\end{proof}

For completeness, we clarify that $\omega_n$ converges to $\omega$ in the sense of characteristic functions as $n \to \infty$ if (see, for instance, \cite[Definition 2.2.3]{HenrotBook2005}):
$$
\chi_{\omega_n}\longrightarrow\chi_{\omega}\;\;\hbox{in}\; L^1(\Omega).
$$

\begin{remark} \label{rem-jus}
The compactness of the class $\mathcal{D}_{ad}$ plays a crucial role in establishing the existence of an optimal solution to the minimization problem \eqref{regularization-topopt-problem}. Moreover, in $\mathcal{D}_{ad}$, we do not impose the connectedness of the set $\Omega \setminus \overline{\omega}$, as connectedness is not preserved under Hausdorff convergence of open sets. For a counterexample, see \cite[p. 33]{HenrotBook2005}.
\end{remark}

\vspace{0.4cm}
\paragraph{\bf Mosco convergence for Sobolev spaces} Following \cite{BucurBook2005,bucur2001continuity,giacomini2002stability}, we recall the definition of Mosco convergence and its associated properties. Let $X$ be a reflexive Banach space, and let 
$\left\{A_n\right\}_{n \in \mathbb{N}}$ be a sequence of closed subspaces of $X.$ We define two sets associated with this sequence:
\begin{itemize}
    \item $A^{\prime}$: the set of elements in $X$ that are weak limits of sequences taken from a subsequence $\left\{A_{n_\ell}\right\}_\ell$ of $\left\{A_n\right\}_n$, formally:
    $$A^{\prime}=\left\{x \in X \mid x=w-\lim _{\ell \rightarrow \infty} x_{n_\ell},\; x_{n_\ell} \in A_{n_\ell}\right\}.$$
    \item  $A^{\prime \prime}$: the set of elements in $X$ that are strong limits of sequences in $\left\{A_n\right\}$, formally:
    $$
    A^{\prime \prime}=\left\{x \in X \mid x=s-\lim _{n \rightarrow \infty} x_n,\; x_n \in A_n\right\}.
    $$
\end{itemize}

\begin{definition}
 Let $X$ be a reflexive Banach space, $\left\{A_n\right\}_{n \in \mathbb{N}}$ be a sequence of closed subspaces of $X$. We say that $A_n$ converges in the sense of Mosco as $n \rightarrow \infty$ if there exists $A\subset X$ such that  $A^\prime=A^{\prime \prime}=A.$ The subspace $A$ is called the Mosco limit of $A_n.$
\end{definition}

From the definition and the properties of $A^{\prime}$ and $A^{\prime \prime},$ it follows that the sequence $A_n$ converges to $A$ in the sense of Mosco if the following two conditions are satisfied:
\begin{align}
    \label{condition1}&\text{If}\; v_{n_\ell}\in A_{n_\ell}\;\text{is such that}\; v_{n_\ell}\rightharpoonup v\;\text{in}\; X\; \text{as}\; \ell \rightarrow \infty,\; \text{then}\; v\in A;\\
    \label{condition2}&\text{For any}\; v\in A,\;\text{there exists a sequence}\; v_{n}\in A_n\;\text{such that}\; v_n\rightarrow v\; \text{in}\; X\; \text{as}\; n \rightarrow \infty.
\end{align}

For any $\omega \in \mathcal{D}_{ad}$, we define an isometric immersion of $\H^1(\Omega \backslash \overline{\omega})$ into $L^2\left(\Omega, \mathbb{R}^{3+3^2}\right)$ as follows: To each $u \in \H^1(\Omega \backslash \overline{\omega}),$ we associate the vector $(u, \nabla u)$ with the convention that $u$ and $\nabla u$ are extended to zero in $\omega$. More precisely, the Sobolev space $\H^1(\Omega \backslash \omega)$ can be identified with a closed subspace of $L^2\left(\Omega, \mathbb{R}^{3+3^2}\right)$ via the map
\begin{equation}\label{identified}
\begin{aligned}
\H^1(\Omega \backslash \overline{\omega}) & \hookrightarrow L^2\left(\Omega, \mathbb{R}^{3+3^2}\right) \\
u & \rightarrow\left(u,\; \partial_i u_j\right), \quad\quad \forall i, j=1,\,2,\, 3
\end{aligned}
\end{equation}
with $u$ and its partial derivatives $\partial_i u_j$ are extended by zero within $\omega.$ Similarly, for $\Omega \backslash \overline{\omega_n},$ the vector-valued functions $u_n$ and $\nabla u_n$ are also extended by zero in $\omega_n$, ensuring a consistent identification.

Since we are dealing with uniform Lipschitz domains, the following result holds as an adaptation of Theorem 7.2.7 in \cite{BucurBook2005}. For further details, see also \cite{HenrotBook2005}.

\begin{theorem}\label{Mosco-convergence-Sobolev}
 Let us assume that $\omega_n,\; \omega \subset \Omega$ belong to the class of admissible solutions $\mathcal{D}_{ad}$. If $\overline{\omega_n}$ converges to $\overline{\omega}$ in the Hausdorff sense as $n\to\infty$, then $\H^1_{0,\text{div}}(\Omega\backslash\overline{\omega_n})$ converges to $\H^1_{0,\text{div}}(\Omega\backslash\overline{\omega})$ in the sense of Mosco  as $n\to\infty$.
\end{theorem}

%%%%%%%%%%%%%%%%%%%%%%%%%%%%%%%%%%%%%%%%%%%%%%%%%%%%%%%
\subsection{Existence of minimizer for the functional}
%%%%%%%%%%%%%%%%%%%%%%%%%%%%%%%%%%%%%%%%%%%%%%%%%%%%%%%

The existence of an optimal solution to the optimization problem \eqref{regularization-topopt-problem} is established in the following theorem. Prior to presenting the theorem, we introduce the following supporting lemmas, which lays the groundwork for the subsequent proof.

\begin{lemma}(see \cite[Lemma 3.5, p. 237]{temam1973theory})\label{estimat-L-4} Let $\mathcal{O}$ be a bounded domain in \( \mathbb{R}^3 \) with Lipschitz boundary $\partial\O$. For all $w\in\H^1_0(\mathcal{O})$, the following inequality holds:
    \begin{equation*}
        \big\|w \big\|_{L^4(\mathcal{O})}^2\leq 2 \big\|w \big\|_{L^2(\mathcal{O})}^{1/2}\big\|\nabla w \big\|_{L^2(\mathcal{O})}^{3/2}.
    \end{equation*}
\end{lemma}

\begin{lemma} (see  \cite[ Lemma IX.2.1, p. 591]{galdi2011introduction}).\label{lem-NS-prelim20} Let \( \O \) be a bounded domain with locally Lipschitz boundary in \( \mathbb{R}^3 \). Let \( \varphi \in \H^1(\O) \) with \( \text{div} \, \varphi = 0 \) in \( \O \). Then 
    \begin{equation}
        \int_{\O} (\varphi\cdot \nabla)v \cdot w\, \dx = -\int_{\O} (\varphi\cdot \nabla)w \cdot v\, \dx\;\;\;\;\text{   for all  }\; v,w\in\H^1_0(\Omega).
    \end{equation}
\end{lemma}

Then, we have the following existence result.
\begin{theorem}\label{existence-optimal-solution}
For any $\gamma > 0$, there exists at least one minimizer for the optimization problem \eqref{regularization-topopt-problem}. \end{theorem}

\begin{proof}To simplify the mathematical analysis involved in the proof of Theorem~\ref{existence-optimal-solution}, we assume homogeneous Dirichlet boundary conditions on $\partial\Omega$; that is, $\phi = 0$ on $\partial\Omega \times (0,T)$. Nonetheless, the proof remains valid for non-homogeneous boundary data, with appropriate technical modifications.

From \eqref{regularization-topopt-problem}, it can be easily observed that the functional $\K_\gamma$ is non-negative on the set of admissible solutions $\mathcal{D}_{ad}$. Therefore, there exists a minimizing sequence $\{\omega_n\}_n\subset\mathcal{D}_{ad}$ such that
\begin{align*}
\lim_{n\rightarrow\infty}\K_\gamma(\omega_n)=\inf_{\omega\in\mathcal{D}_{ad}}\K_\gamma(\omega).
\end{align*}
Since $\omega_n \in \mathcal{D}_{ad}$ and leveraging the compactness result established in Lemma \ref{compacteness-D-ad}, it follows that there exists an open set $\omega^0 \in \mathcal{D}_{ad}$ and a subsequence of $\{\omega_n\}_{n}$, still denoted by $\{\omega_n\}_{n}$, such that $\omega_n$ converges to $\omega^0$ in the Hausdorff sense and in the sense of
characteristic functions.

Next, we demonstrate that $\omega^0$ is a minimizer of the optimization problem \eqref{regularization-topopt-problem}. To achieve this, let $u_n$ denote the solution of the following problem:
\begin{equation}\label{problem-u-n}
\left\{
\begin{array}{rll}
\displaystyle \frac {\partial u_n}{\partial t}-\nu
\Delta u_n+ \N(u_n)+ \nabla \pi_n
 &= \mathcal{G} &\mbox{ in }  \,\,(\Omega\backslash\overline{\omega_n})\times (0,\,T) ,\\
\mbox{div}\, u_n &=  0 & \mbox{ in } \,\, (\Omega\backslash\overline{\omega_n}) \times (0,\,T) ,\\
u_n &= 0 & \mbox{ on }\,\,  \partial\Omega \times (0,\,T) , \\
u_n &=  0 & \mbox{ on } \,\, \partial\omega_n \times (0,\,T) ,\\
u_n(\cdot,0)&= 0&\mbox{ in } \,\,  \Omega\backslash\overline{\omega_n}.
\end{array}
 \right.
\end{equation}
We aim to show that $u_n$ converges to $u_{\omega^0}$ as $n\to\infty$, where \( u_{\omega^0} \) denotes the solution to the problem \eqref{problem-initial-guess} with \( \omega = \omega^0 \) and \( \phi = 0 \).

To establish this convergence result, we begin by recalling the notion of weak solutions to the Navier--Stokes system \eqref{problem-initial-guess} under homogeneous Dirichlet boundary conditions on \( \partial\Omega \). Various definitions of weak solutions can be found in the literature. Here, we adopt the classical approach introduced by Leray \cite{leray1933etude} and further developed by Temam \cite[Chapter III, Section 3]{temam1973theory}. Specifically, for a given source term $\mathcal{G} \in L^2\big(0,T;(\H^1_{0,\text{div}}(\Omega \setminus \overline{\omega_0}))^\prime\big)$, we say that $u_{\omega^0} \in L^2\big(0,T;\H^1_{0,\text{div}}(\Omega \setminus \overline{\omega_0})\big)$ is a weak solution to \eqref{problem-initial-guess} if it satisfies the following weak formulation:
\begin{align*}
\begin{split}
&\displaystyle \frac{\text{d}}{\dt} \int_{\Omega\backslash\overline{\omega_0}}u_{\omega^0}\cdot \zeta \dx+\nu \int_{\Omega\backslash\overline{\omega_0}}\hspace{-0.0cm}\nabla u_{\omega^0} :\nabla \zeta \dx + \int_{\Omega\backslash\overline{\omega_0}} (u_{\omega^0}\cdot \nabla) u_{\omega^0} \cdot \zeta\, \dx=\int_{\Omega\backslash\overline{\omega_0}}\hspace{-0.0cm} \mathcal{G}\cdot \zeta \dx,\\ 
&\text{for all }  \zeta \in \H^1_{0,\text{div}}(\Omega\backslash\overline{\omega_0})\;\;\text{ and }\; u_{\omega^0}(\cdot,0)= 0 \quad \mbox{in the}\quad \L^2-\text{sense},\;\text{i.e.},\\
&\qquad\qquad\qquad\qquad\qquad\qquad\qquad\big\| u_{\omega^0}(\cdot,t)\big\|_{\L^2(\Omega\backslash\overline{\omega_0})}\longrightarrow0\;\;\text{as}\;t\longrightarrow0^+.
\end{split}
\end{align*} 
According to \cite[Theorem 3.1, p. 226]{temam1973theory}, and under \textbf{Assumption (A1)}, the problem \eqref{problem-initial-guess} with \( \omega = \omega^0 \) and \( \phi = 0 \) admits a unique weak solution.\\

\noindent Similarly, one can show that problem \eqref{problem-u-n} admits a unique weak solution. Furthermore, using estimates (3.6) and (3.8) from Chapter 3 of \cite{pokorny2022navier}, combined with the Poincaré inequality, we obtain the following a priori estimate:
\begin{align}\label{estimation-u-omega}
\begin{split}
\big\|\frac{\partial u_n}{\partial t}\big\|_{L^{4/3}(0,T;(\H^1_{0,\text{div}}(\Omega\backslash\overline{\omega_n}))^\prime)}&+\big\|u_n\big\|_{L^\infty(0,T;\L^2(\Omega\backslash\overline{\omega_n}))}\\
&\qquad+\big\|u_n\big\|_{ L^2(0,T;\H^1(\Omega\backslash\overline{\omega_n}))}\leq
\overline{C} \big\|\mathcal{G}\big\|_{L^2(0,T;(\H^1_{0,\text{div}}(\Omega\backslash\overline{\omega_n}))^\prime)},
\end{split}
\end{align}
where \( \overline{C} = C(\Omega \setminus \overline{\omega_n}, \nu) \) is a positive constant depending only on the domain and the viscosity coefficient \( \nu \).\\

\noindent In the next step, we will show that the constant \( \overline{C} \) remains bounded independently of \( n \). Thanks to the uniform Lipschitz regularity of the boundaries $\partial\left(\Omega \backslash \overline{\omega_n}\right)=\partial\Omega\cup\partial\omega_n$ and $\partial(\Omega \backslash \overline{\omega})=\partial\Omega\cup\partial\omega$, the Poincaré inequality is uniform with respect to $n$ in $\H^1_0\left(\Omega \backslash \overline{\omega_n}\right)$. This uniformity arises because the Poincaré constants depend solely on the Lipschitz parameters $r_0$ and $N_0$ (as defined in Definition \ref{def-0}) of the domain $\partial\left(\Omega \backslash \overline{\omega_n}\right)$; see references \cite{HenrotBook2006}. Consequently, from the estimation \eqref{estimation-u-omega}, one can deduce that there exists a constant $C>0$, independent of $n$, such that

\begin{align}\label{estimation-u-omega-1}
\begin{split}
\big\|\frac{\partial u_n}{\partial t}\big\|_{L^{4/3}(0,T;(\H^1_{0,\text{div}}(\Omega\backslash\overline{\omega_n}))^\prime)}+\big\|u_n\big\|_{L^\infty(0,T;\L^2(\Omega\backslash\overline{\omega_n}))}+\big\|u_n\big\|_{ L^2(0,T;\H^1(\Omega\backslash\overline{\omega_n}))}\leq
C.
\end{split}
\end{align}
As a result, from the identification \eqref{identified}, we deduce that $\left\|u_n\right\|_{L^2(0,T;L^2(\Omega,\, \mathbb{R}^{3+3^2}))}$ is uniformly bounded. Up to subsequences, there exists $u^* \in L^2(0,T;L^2(\Omega, \mathbb{R}^{3+3^2}))$ such that
\begin{equation}\label{cv-identified}
u_n \rightharpoonup u^* \text { in } L^2(0,T;L^2(\Omega, \,\mathbb{R}^{3+3^2}))\;\;\text{as}\;n\rightarrow\infty,
\end{equation}
which implies that
\begin{equation}\label{cv-identified-new}
u_n(\cdot,t) \rightharpoonup u^*(\cdot,t) \text { in } L^2(\Omega, \,\mathbb{R}^{3+3^2})\;\;\text{a.a.}\; t\in(0,\,T).
\end{equation}
By applying Theorem \ref{Mosco-convergence-Sobolev} and utilizing the first condition of Mosco convergence for the spaces $A_n=\H^1_{0,\text{div}}(\Omega\backslash\overline{\omega_n}),\;\, A=\H^1_{0,\text{div}}(\Omega\backslash\overline{\omega^0})$, and $X=L^2(\Omega,\, \mathbb{R}^{3+3^2})$ (see \eqref{condition1}), we conclude that $u^*(\cdot,t) \in \H^1_{0,\text{div}}(\Omega\backslash\overline{\omega^0})$ for almost every  $t\in(0,T).$ Furthermore, for any $\psi \in \H^1_{0,\text{div}}(\Omega\backslash\overline{\omega^0})$, there exists a sequence $\psi_n\in \H^1_{0,\text{div}}(\Omega\backslash\overline{\omega_n})$ as indicated by the second condition \eqref{condition2}, such that
\begin{equation}
    \label{cv-s0}\psi_n \rightarrow \psi \text { in } L^2\left(\Omega,\, \mathbb{R}^{3+3^2}\right)\;\;\text{as}\;n\rightarrow\infty.
\end{equation}

\noindent Considering the weak formulation for the problem \eqref{problem-u-n} and we take $\psi_n$ as test function, we obtain
\begin{align}\label{weak-2-2}
\begin{split}
\displaystyle \int_0^T\big<\frac{\partial u_n}{\partial t},\, \psi_n \big>_{\Omega\backslash\overline{\omega_n}}\dt&+\nu \int_0^T\int_{\Omega\backslash\overline{\omega_n}}\nabla u_n :\nabla \psi_n \dx\dt\\
&\quad+ \int_0^T\int_{\Omega\backslash\overline{\omega_n}} (u_n\cdot \nabla) u_n \cdot \psi_n\, \dx\dt=\int_0^T\big< \mathcal{G},\, \psi_n \big>_{\Omega\backslash\overline{\omega_n}}\dt.
\end{split}
\end{align}

\noindent We now analyze each term on the right-hand side of \eqref{weak-2-2}. First, the integral on the right-hand side of \eqref{weak-2-2} can be split as follows:
\begin{equation}\label{eq-21}
    \int_0^T\int_{\Omega\backslash\overline{\omega_n}} \mathcal{G}\cdot \psi_n \dx\dt=\int_0^T\int_{\Omega\backslash\overline{\omega_n}}\mathcal{G}\cdot\big(\psi_n-\psi\big)\dx\dt+\int_0^T\int_{\Omega\backslash\overline{\omega_n}} \mathcal{G}\cdot \psi \dx\dt
\end{equation}
By the Cauchy-Schwarz inequality and \eqref{cv-s0}, we have
\begin{equation}
  \big|\int_0^T\int_{\Omega\backslash\overline{\omega_n}} \mathcal{G}\cdot \big(\psi_n-\psi\big) \dx\dt\big|\leq T\big\|\mathcal{G} \big\|_{L^2(0,T;(\H^1_{0,\text{div}}(\Omega\backslash\overline{\omega_n}))^\prime)}\big\|\psi_n-\psi \big\|_{\H^1_{0,\text{div}}(\Omega\backslash\overline{\omega_n})}\longrightarrow0\;\;\text{as}\;n\to\infty.
\end{equation}
Thus, due to the convergence of $\omega_n$ to $\omega^0$ in the Hausdorff or characteristic sense as $n \to \infty$, we deduce:
\begin{equation}\label{eq-21-cv}
    \int_0^T\int_{\Omega\backslash\overline{\omega_n}} \mathcal{G}\cdot \psi_n \dx\dt\longrightarrow\int_0^T\int_{\Omega\backslash\overline{\omega^0}} \mathcal{G}\cdot \psi \dx\dt\;\;\text{as}\;\; n\longrightarrow\infty.
\end{equation}
Similarly, for the first term on the right-hand side of \eqref{weak-2-2}, we have
\begin{equation}\label{kj-1}
    \int_0^T\int_{\Omega\backslash\overline{\omega_n}}\frac{\partial u_n}{\partial t}\cdot \psi_n \dx\dt=\int_0^T\int_{\Omega\backslash\overline{\omega_n}}\frac{\partial u_n}{\partial t}\cdot \big(\psi_n -\psi\big)\dx\dt+\int_0^T\int_{\Omega\backslash\overline{\omega_n}}\frac{\partial u_n}{\partial t}\cdot \psi \dx\dt.
\end{equation}
Using H\"older's inequality (for $p=4/3$ and $q=4$) along with \eqref{estimation-u-omega-1}, we get
\begin{align*}
    \big|\int_0^T\int_{\Omega\backslash\overline{\omega_n}}\frac{\partial u_n}{\partial t}\cdot \big(\psi_n -\psi\big)\dx\dt\big|& \leq \big(\int_0^T\big\|\frac{\partial u_n}{\partial t} \big\|_{(\H^1_{0,\text{div}}(\Omega\backslash\overline{\omega_n}))^\prime}\dt\big)\big\|\psi_n-\psi \big\|_{\H^1_{0,\text{div}}(\Omega\backslash\overline{\omega_n})}\\
    &\leq T^{1/4}\big\|\frac{\partial u_n}{\partial t} \big\|_{L^{4/3}(0,T;(\H^1_{0,\text{div}}(\Omega\backslash\overline{\omega_n}))^\prime)}\big\|\psi_n-\psi \big\|_{\H^1_{0,\text{div}}(\Omega\backslash\overline{\omega_n})}\\
    &\leq C T^{1/4}\big\|\psi_n-\psi \big\|_{\H^1_{0,\text{div}}(\Omega\backslash\overline{\omega_n})}.
\end{align*}
Consequently, from the convergence \eqref{cv-s0}, we conclude
\begin{equation}\label{kj-11}
    \int_0^T\int_{\Omega\backslash\overline{\omega_n}}\frac{\partial u_n}{\partial t}\cdot \big(\psi_n -\psi\big)\dx\dt\longrightarrow0\;\;\text{as}\;n\to\infty.
\end{equation}
Taking the limit as $n\to\infty$ in \eqref{kj-1}, and applying the results from \eqref{cv-identified} and \eqref{kj-11}, we derive
\begin{equation}\label{eq-21-cv21}
    \int_0^T\int_{\Omega\backslash\overline{\omega_n}}\frac{\partial u_n}{\partial t}\cdot \psi_n \dx\dt\longrightarrow\int_0^T\int_{\Omega\backslash\overline{\omega^0}}\frac{\partial u^*}{\partial t}\cdot \psi \dx\dt\;\;\text{as}\;\; n\longrightarrow\infty.
\end{equation}
Similarly, one can prove that 
\begin{equation}\label{cv-g-r}
    \int_0^T\int_{\Omega\backslash\overline{\omega_n}}\nabla u_n :\nabla \psi_n \dx\dt\longrightarrow \int_0^T\int_{\Omega\backslash\overline{\omega^0}}\nabla u^* :\nabla \psi \dx\dt\;\;\text{as}\;\; n\longrightarrow\infty.
\end{equation}
Let us now examine the convergence of the convective term, i.e., we prove that 
  \begin{equation}\label{convect-limit}
    \int_0^T\int_{\Omega\backslash\overline{\omega_n}} (u_n\cdot \nabla) u_n \cdot \psi_n\, \dx\dt\longrightarrow  \int_0^T\int_{\Omega\backslash\overline{\omega^0}} (u^*\cdot \nabla) u^* \cdot \psi\, \dx\dt\;\;\text{as}\;\; n\longrightarrow\infty.
\end{equation}
To prove this convergence result, let us split the convective term as follows
\begin{align}\label{covect-t-cv}
\begin{split}
\int_0^T\int_{\Omega\backslash\overline{\omega_n}} (u_n\cdot \nabla) u_n \cdot \psi_n\, \dx\dt&=  \int_0^T\int_{\Omega\backslash\overline{\omega_n}} (u_n\cdot \nabla) u_n \cdot \big(\psi_n-\psi\big)\, \dx\dt\\
&\qquad\quad+ \int_0^T\int_{\Omega\backslash\overline{\omega_n}} (u_n\cdot \nabla) u_n \cdot \psi\, \dx\dt.
    \end{split}
\end{align} 
We will now estimate the first term on the right-hand side of \eqref{covect-t-cv}.  By applying H\"older's inequality in space along with Lemma \ref{estimat-L-4}, we have
\begin{align*}
&\big|\int_0^T\int_{\Omega\backslash\overline{\omega_n}} (u_n\cdot \nabla) u_n \cdot \big(\psi_n-\psi\big)\, \dx\dt\big|\\
&\leq\int_0^T\int_{\Omega\backslash\overline{\omega_n}} \big|u_n\cdot\big( u_n\cdot \nabla\big(\psi_n-\psi\big)\big)\big|\, \dx\dt\\
&\leq \int_0^T\big\| \nabla\big(\psi_n-\psi\big) \big\|_{\L^2(\Omega\backslash\overline{\omega_n})} \big\| u_n \big\|_{\L^4(\Omega\backslash\overline{\omega_n})}^2\dt\\
&\leq 2 \big\| \nabla\big(\psi_n-\psi\big) \big\|_{\L^2(\Omega\backslash\overline{\omega_n})} \int_0^T\big\|\nabla u_n \big\|_{\L^2(\Omega\backslash\overline{\omega_n})}^{3/2} \big\| u_n \big\|_{\L^2(\Omega\backslash\overline{\omega_n})}^{1/2}\dt\\
&\leq 2 \big\| \nabla\big(\psi_n-\psi\big) \big\|_{\L^2(\Omega\backslash\overline{\omega_n})} \big\|\nabla u_n \big\|_{L^{3/2}(0,T;\L^{2}(\Omega\backslash\overline{\omega_n}))}^{2/3} \big\| u_n \big\|_{L^\infty(0,T;\L^2(\Omega\backslash\overline{\omega_n}))}^{1/2}.
\end{align*} 
Since $L^{2}(0,T;\L^{2}(\Omega\backslash\overline{\omega_n}))\hookrightarrow L^{3/2}(0,T;\L^{2}(\Omega\backslash\overline{\omega_n})),$  there exists a constant $C>0,$ independent of $n,$ such that
\begin{align*}
&\big|\int_0^T\int_{\Omega\backslash\overline{\omega_n}} (u_n\cdot \nabla) u_n \cdot \big(\psi_n-\psi\big)\, \dx\dt\big|\\
&\leq C \big\| \nabla\big(\psi_n-\psi\big) \big\|_{\L^2(\Omega\backslash\overline{\omega_n})} \big\|\nabla u_n \big\|_{L^{2}(0,T;\L^{2}(\Omega\backslash\overline{\omega_n}))}^{2/3} \big\| u_n \big\|_{L^\infty(0,T;\L^2(\Omega\backslash\overline{\omega_n}))}^{1/2}.
\end{align*} 
Therefore, from \eqref{cv-s0} and \eqref{estimation-u-omega-1}, we get
\begin{equation}
    \int_0^T\int_{\Omega\backslash\overline{\omega_n}} (u_n\cdot \nabla) u_n \cdot \big(\psi_n-\psi\big)\, \dx\dt\longrightarrow  0\;\;\text{as}\;\; n\longrightarrow\infty.
\end{equation}

\noindent To estimate the second term on the right-hand side of \eqref{covect-t-cv}, we note that the functions $u_n=\big(u_n^i\big)_{1\leq i\leq 3}$ and $\nabla u_n=\big( \partial_i u_n^j\big)_{1\leq i,j\leq 3}$ are extended by zero in $\omega_n$, while $u^*=\big(u^*_i\big)_{1\leq i\leq 3}$ and $\nabla u^*=\big( \partial_i u^*_j\big)_{1\leq i,j\leq 3}$ are extended by zero in $\omega^0$. Furthermore, applying Lemma \ref{lem-NS-prelim20} (since $u_n=u^*=\psi=0$ on $\partial\Omega$), we have
\begin{align*}
&\big|\int_0^T\int_{\Omega\backslash\overline{\omega_n}} (u_n\cdot \nabla) u_n \cdot \psi\, \dx\dt  -\int_0^T\int_{\Omega\backslash\overline{\omega^0}} (u^*\cdot \nabla) u^* \cdot \psi\, \dx\dt\big|\\
&=\big|\int_0^T\int_{\Omega} (u_n\cdot \nabla )u_n \cdot \psi\, \dx\dt  -\int_0^T\int_{\Omega} (u^*\cdot \nabla) u^* \cdot \psi\, \dx\dt\big|\\
&=\big|\int_0^T\int_{\Omega} \big[(u^*\cdot \nabla) \psi \cdot u^*\, \dx\dt  - (u_n\cdot \nabla) \psi \cdot u_n\big]\, \dx\dt\big|\\
&=\big| \sum_{k=1}^3\sum_{i=1}^3\int_0^T\int_\Omega\big( \big(u_i^*-u_n^i\big)\frac{\partial \psi_k}{\partial x_i}u^*_k\big)\dx\dt   -   \sum_{k=1}^3\sum_{i=1}^3\int_0^T\int_\Omega\big( u_n^i\frac{\partial \psi_k}{\partial x_i}\big(u_k^*-u_n^k\big)\big)\dx\dt  \big|.
\end{align*} 
Using the triangle inequality, we have
\begin{align*}
&\big|\int_0^T\int_{\Omega\backslash\overline{\omega_n}} (u_n\cdot \nabla )u_n \cdot \psi\, \dx\dt  -\int_0^T\int_{\Omega\backslash\overline{\omega^0}} (u^*\cdot \nabla )u^* \cdot \psi\, \dx\dt\big|\\
&\leq  \sum_{k=1}^3\sum_{i=1}^3\int_0^T\int_\Omega\big| \big(u_i^*-u_n^i\big)\frac{\partial \psi_k}{\partial x_i}u^*_k\big|\dx\dt  
+     \sum_{k=1}^3\sum_{i=1}^3\int_0^T\int_\Omega\big| u_n^i\frac{\partial \psi_k}{\partial x_i}\big(u_k^*-u_n^k\big)\big|\dx\dt .
\end{align*} 
By H\"older's inequality in space (with $p=6,$ $q=3$ and $r=2$), along with the Cauchy-Schwarz inequality (in time), we obtain
\begin{align*}
&\big|\int_0^T\int_{\Omega\backslash\overline{\omega_n}} (u_n\cdot \nabla) u_n \cdot \psi\, \dx\dt  -\int_0^T\int_{\Omega\backslash\overline{\omega^0}} (u^*\cdot \nabla) u^* \cdot \psi\, \dx\dt\big|\\
&\leq\big(\int_0^T\big\|u^*-u_n  \big\|_{L^3(\Omega)}   \big\| u^* \big\|_{L^6(\Omega)}\dt \big)  \big\| \psi \big\|_{L^2(\Omega)}+\big(\int_0^T\big\|u^*-u_n  \big\|_{L^3(\Omega)}   \big\| u_n \big\|_{L^6(\Omega)}\dt \big)  \big\| \psi \big\|_{L^2(\Omega)}\\
&\leq  \big\| \psi \big\|_{L^2(\Omega)} \big\|u^*-u_n  \big\|_{L^2(0,T;L^3(\Omega))}\big(  \big\|u^*  \big\|_{L^2(0,T;L^6(\Omega))}+\big\|u_n  \big\|_{L^2(0,T;L^6(\Omega))} \big).
\end{align*} \\

\noindent Finally, leveraging the embedding  $H^1(\Omega)\subset L^s(\Omega)$ for all $1\leq s\leq 6,$ along with the uniform estimate from \eqref{estimation-u-omega-1} and the convergence results from \eqref{cv-identified} and \eqref{cv-s0}, we conclude that the convergence result stated in \eqref{convect-limit} has been established.\\

\noindent Now, we pass to the limit as $n\longrightarrow\infty$ in \eqref{weak-2-2} and using the convergence results \eqref{eq-21-cv}, \eqref{eq-21-cv21}, \eqref{cv-g-r}, and \eqref{convect-limit} to obtain
\begin{align}\label{weak-2-2-2}
\begin{split}
\displaystyle \int_0^T\int_{\Omega\backslash\overline{\omega^0}}\frac{\partial u^*}{\partial t}\cdot \psi \,\dx\dt&+\nu \int_0^T\int_{\Omega\backslash\overline{\omega^0}}\nabla u^* :\nabla \psi\dx\dt\\
&\quad+ \int_0^T\int_{\Omega\backslash\overline{\omega^0}} (u^*\cdot \nabla) u^* \cdot \psi\, \dx\dt=\int_0^T\int_{\Omega\backslash\overline{\omega^0}} \mathcal{G}\cdot \psi \,\dx\dt,
\end{split}
\end{align}
for all $\psi\in \H^1_{0,\text{div}}(\Omega\backslash\overline{\omega^0}).$ \\

\noindent To conclude $u^*=u_{\omega^0}$, it remains to verify that $u^*(\cdot,0)=0$. To this end, let $\Psi\in C^1[0,\, T]$ with $\Psi(T) = 0$, and let $\theta\in \L^2(\Omega)$ be arbitrary. Integrating by parts over the time interval $(0,T)$ and using the fact that $u_n(\cdot,0)=0$ in $\Omega\backslash\overline{\omega_n}$, we obtain
\begin{align}
   \int_{0}^{T} \int_{\Omega\backslash\overline{\omega_n}} \big(\frac{\partial u_n}{\partial t}\cdot \theta\big)\Psi(t)\, \dx\dt=-\int_{0}^{T} \int_{\Omega\backslash\overline{\omega_n}}  \big(u_n \cdot\theta\big)\Psi'(t)\, \dx\dt.
\end{align}
Letting $n \rightarrow \infty$ in the above equality and applying the same arguments used earlier for convergence, we deduce
\begin{align}
   \int_{0}^{T} \int_{\Omega\backslash\overline{\omega^0}} \big(\frac{\partial u^*}{\partial t}\cdot\theta\big) \Psi(t)\, \dx\dt=-\int_{0}^{T} \int_{\Omega\backslash\overline{\omega^0}}  \big(u^* \cdot\theta\big)\Psi'(t)\, \dx\dt.
\end{align}
By performing integration by parts with respect to $t$ again, we also obtain:\begin{align*}
   \int_{0}^{T} \int_{\Omega\backslash\overline{\omega^0}} \big(\frac{\partial u^*}{\partial t}\cdot\theta\big) \Psi(t)\, \dx\dt=-\int_{0}^{T} \int_{\Omega\backslash\overline{\omega^0}}  \big(u^* \cdot\theta\big)\Psi'(t)\, \dx\dt-\big(\int_{\Omega\backslash\overline{\omega^0}} u^*(\cdot,0)\cdot\theta(\cdot)\, \dx\big)\Psi(0).
\end{align*}
Combining these two expressions yields:
\begin{equation*}
    \big(\int_{\Omega\backslash\overline{\omega^0}} u^*(x,0)\cdot\theta(x)\, \dx\big)\Psi(0)=0.
\end{equation*}
Choosing $\Psi(0)\not=0$ implies
\begin{equation}\label{initial-u0}
    \int_{\Omega\backslash\overline{\omega^0}} u^*(x,0)\cdot\theta(x)\, \dx=0,\;\;\;\text{for all}\;\theta\in\L^2(\Omega),
\end{equation}
which shows that $u^*(\cdot,0)=0$ in $\Omega\backslash\overline{\omega^0}$. Finally, from this result, the weak formulation \eqref{weak-2-2-2}, and the uniqueness of the weak solution to problem \eqref{problem-initial-guess}  (with $\omega=\omega^0$), we conclude that
\begin{equation}\label{uw-0}
    u^*=u_{\omega^0}.
\end{equation}

The concluding step of this proof is based on the lower semi-continuity of the cost functional  $\K_\gamma$. It is well-known that the 
$L^2(\Omega_0)$-norm exhibits lower semi-continuity. Additionally, the lower semi-continuity of the relative perimeter $\hbox{Per}_\Omega(\omega)$ (the second term in  $\K_\gamma$) follows directly from results presented in \cite{HenrotBook2005}. By integrating the lower semi-continuity of $\K_\gamma$ with the convergence results 
$d_{H^c}\big(\omega_n,\omega^0\big)\to0$ as $n\to\infty$ and the findings in \eqref{cv-identified} and \eqref{uw-0}, we derive:
\begin{equation}\mathcal{K}_\gamma\left(\omega^0\right)  \leq \liminf _{n \rightarrow \infty}\int_0^T \left(\int_{\Omega_0}\big|
u_{\omega_n}-u_{meas}\big|^2 \dx\right)\dt+\gamma \liminf _{n \rightarrow \infty}\;\hbox{Per}_\Omega(\omega_n).\end{equation}
This leads to the conclusion that
$$\mathcal{K}_\gamma\left(\omega^0\right) \leq \liminf _{n \rightarrow \infty} \mathcal{K}_\gamma\left(\omega_n\right)=\inf _{\omega\in \mathcal{D}_{ad}} \mathcal{K}_\gamma(\omega),$$
which demonstrates that $\omega^0$ is indeed a minimizer of the optimization problem described in \eqref{regularization-topopt-problem}.
\end{proof}

%%%%%%%%%%%%%%%%%%%%%%%%%%%%%%%%%%%%%%%%%%%%%%%%%%%%%%%
\subsection{Stability of the optimization problem}
%%%%%%%%%%%%%%%%%%%%%%%%%%%%%%%%%%%%%%%%%%%%%%%%%%%%%%%

In this section, we establish the stability of \eqref{regularization-topopt-problem}, showing that this minimization system effectively stabilizes the geometric inverse problem under consideration in relation to perturbations in the observed data within $\Omega_0$. To this end, let $u_{meas}^n$ represent a sequence of measurements of $u_{meas}$
  in the space $L^2(0,T;\L^2(\Omega_0)$. For each $n\in\mathbb{N}$, we define $\omega_n\in\mathcal{D}_{ad}$ as the solution to the following minimization problem:
\begin{equation}\label{reg-pertubed-optimi}
\operatorname*{Minimize}_{\omega\in \mathcal{D}_{ad}} \mathcal{K}^n_\gamma(\omega):=\int_0^T \big( \int_{\Omega_0} \big| u_\omega - u^n_{meas} \big|^2 \dx \big) \dt + \gamma \text{Per}_\Omega(\omega).
\end{equation}
In the subsequent theorem, we analyze the convergence of the sequence $\omega_n$ under the condition that the measured data satisfies $u_{meas}^n \longrightarrow u_{meas}^n$ in $L^2(0,T;\L^2(\Omega_0))$ as $n\longrightarrow \infty$.

\begin{theorem}\label{theorem-stability} If
$u_{meas}^n \longrightarrow u_{meas}$ in $L^2(0,T;\L^2(\Omega_0))$ as $n\longrightarrow \infty$, then there exists a subsequence of $\{\omega_n\}_n\subset\mathcal{D}_{ad}$, such that
$$d_{H^c}(\omega_n,\omega^\star)\longrightarrow0\;\; \text{as}\;\; n\longrightarrow\infty.$$
where $\omega^\star\in\mathcal{D}_{ad}$ is a minimizer of the optimization problem defined in \eqref{regularization-topopt-problem}, corresponding to the exact data $u_{meas}$.
\end{theorem}

\begin{proof}
The existence of each $\omega_n$ is guaranteed by Theorem \ref{existence-optimal-solution}. Utilizing the compactness result from Lemma \ref{compacteness-D-ad}, there exists a set $\omega^\star\in\mathcal{D}_{ad}$ and a subsequence, still denoted by $\omega_n$, such that
$$d_{H^c}(\omega_n,\omega^\star)\longrightarrow0\;\; \text{as}\;\; n\longrightarrow\infty.$$
It remains to demonstrate that the set $\omega^\star$ is indeed a minimizer of the optimization problem defined in \eqref{regularization-topopt-problem}. By employing the same reasoning as in the proof of Theorem \ref{existence-optimal-solution}, we obtain the following convergence, up to a further subsequence still denoted by $\omega_n$:
\begin{align}\label{es44} u_{\omega_n} \rightharpoonup  u_{\omega^\star} \quad \text{in} \quad L^2(0,T; L^2(\Omega,\,\mathbb{R}^{3+3^2})) \quad \text{as} \quad n \longrightarrow \infty. \end{align}
 By leveraging the strong convergence of $u_{meas}^n$ to $u_{meas}$ in $L^2(0,T;\L^2(\Omega_0))$ as $n\longrightarrow \infty$ and using \eqref{es44}, we obtain
 \begin{align*}
u_{\omega_n}-u_{meas}^n \rightharpoonup  u_{\omega^\star}-u_{meas} \quad \text{in} \quad L^2(0,T; L^2(\Omega_0)) \quad \text{as} \quad n \longrightarrow \infty.
\end{align*}

\noindent Therefore, for any $\omega\in\mathcal{D}_{ad},$ we can utilize the lower semi-continuity of the $L^2$-norm and the relative perimeter to establish that
\begin{align*}
\mathcal{K}_\gamma(\omega^\star)&=\int_0^T\int_{\Omega_0}\big|u_{\omega^\star}-
u_{meas}\big|^2\dx\dt+\gamma\hbox{Per}_\Omega(\omega^\star)\\
&\leq\lim_{n\rightarrow\infty}\inf\int_0^T\int_{\Omega_0}\big|u_{\omega_n}-
u_{meas}^n\big|^2\dx\dt+\gamma\lim_{n\rightarrow\infty}\inf\hbox{Per}_\Omega(\omega^n)\\
&\leq\lim_{n\rightarrow\infty}\inf\big(\int_0^T\int_{\Omega_0}\big|u_{\omega_n}-
u_{meas}^n\big|^2\dx\dt+\gamma\hbox{Per}_\Omega(\omega^n)\big).
\end{align*}
On the other hand, since $\omega_n$ is a solution to the minimization problem \eqref{reg-pertubed-optimi}, we can conclude that
\begin{align*}
\mathcal{K}_\gamma(\omega^\star)&\leq\lim_{n\rightarrow\infty}\big(\int_0^T\int_{\Omega_0}\big|u_{\omega}-
u_{meas}^n\big|^2\dx\dt+\gamma\hbox{Per}_\Omega(\omega)\big)\\
&=\int_0^T\int_{\Omega_0}\big|u_{\omega}-
u_{meas}\big|^2\dx\dt+\gamma\hbox{Per}_\Omega(\omega)\\
&=\mathcal{K}_\gamma(\omega), \qquad \;\text{for all }
\omega\in\mathcal{D}_{ad}.
\end{align*}
This confirms that $\omega^\star\in\mathcal{D}_{ad}$ is a minimizer of the optimization problem defined in \eqref{regularization-topopt-problem}.
\end{proof}

In the following section, we numerically address the optimization problem \eqref{regularization-topopt-problem} and develop a non-iterative identification algorithm to reconstruct the location of the unknown obstacle $\omega^*$. Specifically, obtaining successive approximations for the solution to the considered geometric inverse problem by solving the optimization problem \eqref{regularization-topopt-problem} with a fixed parameter $\gamma>0$ presents several technical challenges. These include the non-differentiability of the cost functional $\mathcal{K}_\gamma$ arising from the non-differentiability of the relative perimeter function $\omega\longmapsto\text{Per}_\Omega(\omega)$, as well as the non-convexity of the admissible solutions set $\mathcal{D}_{ad}$. To address these challenges, we employ a self-regularized reconstruction approach based on the topological derivative method, as described below.

\begin{remark}
Importantly, the concept of the topological gradient is unaffected by the specific choice of Dirichlet boundary conditions on \( \partial\Omega \), provided the forward problem \eqref{direct-problem} remains well-posed in the sense of existence, uniqueness, and regularity. This invariance stems from the fact that the derivation of the topological gradient relies solely on the continuous dependence of the solution on the input data. As a result, whether the Dirichlet condition is homogeneous or not does not impact the subsequent analysis of the geometric inverse problem. Therefore, for simplicity and without loss of generality, we assume homogeneous Dirichlet boundary conditions for the computation of the topological gradient, namely,
\[
\phi = 0 \quad \text{on } \partial\Omega \times (0, T).
\]
\end{remark}

%%%%%%%%%%%%%%%%%%%%%%%%%%%%%%%%%%%%%%%%%%%%%%%%%%%%%%%
\section{Topological gradient-based approach}\label{sec:gradient-topologique}
%%%%%%%%%%%%%%%%%%%%%%%%%%%%%%%%%%%%%%%%%%%%%%%%%%%%%%%

We recall that the topological sensitivity analysis consists in the study of the variations of the functional $\mathcal{K}_\gamma$ with respect to the insertion of a small obstacle in $\Omega$. To illustrate the main idea of this method, we consider a small geometric perturbation defined as $\mathcal{C}_{z,\varepsilon}=z+\varepsilon\mathcal{C}$, where $\varepsilon>0$ represents the size of the perturbation and $z\in\Omega$ is its center. Here, \( \mathcal{C} \subset \mathbb{R}^3 \) denotes a fixed, bounded domain containing the origin, with a smooth boundary-for instance, of class \( \mathcal{C}^{2,1} \).

To compute the asymptotic expansion of $\K_\gamma$ without relying on the truncation method commonly used in the literature, we examine the following penalized Navier-Stokes problem:
\begin{equation}\label{problem-penalized-0}
\left\{
\begin{array}{rll}
\displaystyle \frac {\partial u_\varepsilon}{\partial t}-\nu
\Delta u_\varepsilon+ \N(u_\varepsilon)+ \nabla \pi_\varepsilon
 +k\chi_{\mathcal{C}_{z,\varepsilon}}\, u_\varepsilon&= \mathcal{G} &\mbox{ in }  \,\,\Omega\times (0,\,T) ,\\
\mbox{div}\, u_\varepsilon &=  0 & \mbox{ in } \,\, \Omega \times (0,\,T) ,\\
u_\varepsilon &= 0 & \mbox{ on }\,\,  \partial\Omega \times (0,\,T) , \\
u_\varepsilon(\cdot,0)&= 0&\mbox{ in } \,\,  \Omega.
\end{array}
 \right.
\end{equation}
In this context, $k$ is a positive constant and $\chi_{\mathcal{C}_{z,\varepsilon}}$ is the characteristic function of $\mathcal{C}_{z,\varepsilon}.$ Following the approach outlined in \cite{TemamBook1977}, the weak formulation for problem \eqref{problem-penalized} is stated as follows: Find $u_\varepsilon\in  L^2(0,T; \H^1_{0,\text{div}}(\Omega))$ such that

\begin{align}\label{weak-f-penalized}
\left\{
  \begin{array}{ll}
    \displaystyle\big<\frac{\partial u_\varepsilon}{\partial t},\, \varphi\big>_\Omega+ a_\varepsilon(u_\varepsilon,\varphi)=\big<\mathcal{G},\, \varphi\big>_\Omega \qquad \forall \varphi\in \H^1_{0,\text{div}}(\Omega)  \\
    \\
    \text{and } a.a.\,t\in(0,T)\;\text{ and } u_\varepsilon(\cdot,0)=0\;\text{in } \;\Omega. 
  \end{array}
\right.
\end{align}
In this formulation, the bilinear form $a_\varepsilon$ is defined by
\begin{align}
    \label{bilinear}a_\varepsilon(u,\varphi)&=\nu\int_{\Omega}\nabla u :\nabla \varphi \,\dx+ \int_{\Omega} (u\cdot \nabla) u \cdot \varphi\, \dx+\int_{\Omega} k\chi_{\mathcal{C}_{z,\varepsilon}} u\cdot \varphi\, \dx.
\end{align}
Using the Galerkin approximation method outlined in \cite{TemamBook1977}, we can demonstrate the existence of a weak solution $u_\varepsilon \in L^2(0,T; \H^1_{0,\text{div}}(\Omega))$ for the penalized problem \eqref{problem-penalized-0}.

The penalization method is a well-established technique often employed in finite element approximations to enforce Dirichlet boundary conditions. Through this approach, we can show that as $k$ approaches infinity within $\mathcal{C}_{z,\varepsilon}$ for a given
$\varepsilon,$ the corresponding solution $u_\varepsilon$
converges to the solution of the following perturbed Navier-Stokes system 
\begin{equation}\label{problem-perturbed}
\left\{
\begin{array}{rll}
\displaystyle \frac {\partial u_\varepsilon}{\partial t}-\nu
\Delta u_\varepsilon+ \N(u_\varepsilon)+ \nabla \pi_\varepsilon
 &= \mathcal{G} &\mbox{ in }  \,\,(\Omega\backslash\overline{\mathcal{C}_{z,\varepsilon}})\times (0,\,T) ,\\
\mbox{div}\, u_\varepsilon &=  0 & \mbox{ in } \,\, (\Omega\backslash\overline{\mathcal{C}_{z,\varepsilon}}) \times (0,\,T) ,\\
u_\varepsilon &= 0 & \mbox{ on }\,\,  \partial\Omega \times (0,\,T) , \\
u_\varepsilon &=  0 & \mbox{ on } \,\, \partial\mathcal{C}_{z,\varepsilon} \times (0,\,T) ,\\
u_\varepsilon(.,0)&= 0&\mbox{ in } \,\,  \Omega\backslash\overline{\mathcal{C}_{z,\varepsilon}}.
\end{array}
 \right.
\end{equation}
The rigorous mathematical justification for this convergence result is established in Theorem \ref{cv-penalized} of Section \ref{Justification-penalizing}.

\begin{remark}In the remainder of this section, we assume that the parameter $k$ is sufficiently large to guarantee that the solution of the penalized problem \eqref{problem-penalized-0} converges to the solution of the perturbed problem \eqref{problem-perturbed}. In addition, for clarity and to avoid confusion, we denote by $u_\varepsilon$ the solution of the problem \eqref{problem-penalized-0}.
\end{remark}

From these elements, we define the following shape function :
$$
K(\Omega\backslash\overline{\mathcal{C}_{z,\varepsilon}}):=\mathcal{K}_\gamma(\mathcal{C}_{z,\varepsilon}),
$$
where the perturbed functional $ \mathcal{K}_\gamma(\mathcal{C}_{z,\varepsilon})$ is defined by
\begin{align*}
\mathcal{K}_\gamma(\mathcal{C}_{z,\varepsilon})=\int_0^T \left(\int_{\Omega_0}\big|
u_\varepsilon-u_{meas}\big|^2 \dx\right)\dt+\gamma\text{Per}_\Omega(\mathcal{C}_{z,\varepsilon}).
\end{align*}
Since $\text{Per}_\Omega(\mathcal{C}_{z,\varepsilon})=\varepsilon^2\text{Per}_\Omega(\mathcal{C})$, it follows that
\begin{align}\label{perturbed-shape-function}
K(\Omega\backslash\overline{\mathcal{C}_{z,\varepsilon}})=\int_0^T \left(\int_{\Omega_0}\big|
u_\varepsilon-u_{meas}\big|^2 \dx\right)\dt+\gamma\varepsilon^2\text{Per}_\Omega(\mathcal{C}).
\end{align}
As stated in the introduction, we will derive an asymptotic expansion for
$K$ on the form
\begin{align}\label{asy}
K(\Omega\backslash\overline{\mathcal{C}_{z,\varepsilon}})=K(\Omega)+\mu(\varepsilon)D_K(z)+o(\mu(\varepsilon)),
\end{align}
where the unperturbed shape function $K(\Omega)$ is defined by $L^2$-norm without
the regularization term:
\begin{align}\label{unperturbed-shape-function}
K(\Omega)=\mathcal{K}(\emptyset)=\int_0^T \left(\int_{\Omega_0}\big|
u_0-u_{meas}\big|^2 \dx\right)\dt.
\end{align}
In this setting, the pair $(u_0,\pi_0)\in L^2(0,T; \H^1_{0,\text{div}}(\Omega))\times L^2(0,T;L^2(\Omega))$ is solution of
\begin{equation}\label{problem-perturbed-00}
\left\{
\begin{array}{rll}
\displaystyle \frac {\partial u_0}{\partial t}-\nu
\Delta u_0+ \N(u_0)+ \nabla \pi_0
 &= \mathcal{G} &\mbox{ in }  \,\,\Omega\times (0,\,T) ,\\
\mbox{div}\, u_0 &=  0 & \mbox{ in } \,\, \Omega \times (0,\,T) ,\\
u_0&= 0 & \mbox{ on }\,\,  \partial\Omega \times (0,\,T) , \\
u_0(.,0)&= 0&\mbox{ in } \,\,  \Omega.
\end{array}
 \right.
\end{equation}
From assumption (\textbf{A}$1$) and \cite[Theorems 3.7 and 3.8]{temam1973theory}, Temam established that the problem \eqref{problem-perturbed-00} admits a unique weak solution \( u_0\in L^\infty(0,T; \H^2(\Omega))\). By the Sobolev embedding theorem, this further implies that 
\begin{eqnarray}\label{regularity1}
    u_0 \in L^\infty(0,T; \mathcal{C}(\overline{\Omega})).
\end{eqnarray}

\begin{remark}\label{regularity}
As noted in \cite[Remark 3.8, p. 247]{TemamBook1977}, it is important to emphasize that a more regular solution to the evolutionary Navier-Stokes equations \eqref{problem-perturbed-00} can be obtained if the source term is sufficiently smooth. Specifically, if \( \mathcal{G} \in \mathcal{C}^\infty(\overline{\Omega} \times (0, T)) \) and \( \Omega \in\mathcal{C}^\infty \), then the solution \( u_0 \) to \eqref{problem-perturbed-00} belongs to \( \mathcal{C}^\infty(\overline{\Omega} \times (0, T)) \).
\end{remark}

Next, we determine the scalar function \( \varepsilon \mapsto \mu(\varepsilon) \) and the topological gradient \( D_K(z) \). Our approach relies on a preliminary estimate that captures the leading term of the variation in the velocity field, which significantly simplifies the mathematical analysis. Furthermore, to address the difficulties caused by the limited regularity of the data, we impose an additional assumption that alleviates these issues and facilitates a more straightforward derivation.

\vspace{0.3cm}
\paragraph{\bf Assumption (A2)} There exists \( \delta > 0 \) such that 
    \begin{equation}\label{condition-u0}
    \left\|\nabla u_0\right\|_{L^\infty(0,T;\,\L^2(\Omega))} \leq \delta < \frac{\nu}{\rho(\Omega)},
    \end{equation}
    where \begin{equation}\label{uni-con11}
        \rho(\Omega) = \displaystyle \frac{2\sqrt{2}}{3} \operatorname{meas}(\Omega)^{1/6} .
        \end{equation}

\begin{remark}\label{remark-smoothness}
The above assumption has been used by Galdi in \cite{galdi2011introduction} as a sufficient condition for ensuring the uniqueness of solutions to the Navier–Stokes problem.
\end{remark}

%%%%%%%%%%%%%%%%%%%%%%%%%%%%%%%%%%%%%%%%%%%%%%%%%%%%%%%%%%%%%
\subsection{Estimation of the perturbed velocity field}
%%%%%%%%%%%%%%%%%%%%%%%%%%%%%%%%%%%%%%%%%%%%%%%%%%%%%%%%%%%%%

The first step of our analysis is to examine the effect of the small obstacle 
$\mathcal{C}_{z,\varepsilon}$ on the fluid flow within the domain $\Omega$. 
We derive an estimate for the resulting perturbation in the velocity field. 
To this end, we first present two preliminary results.

\begin{lemma} \label{lem-NS-prelim1}(see \cite[Lemma IX.1.1, p. 588]{galdi2011introduction}).  Let \( \O \) be a bounded domain with locally Lipschitz boundary in \( \mathbb{R}^3 \). For all $(\varphi,v,w)\in\H^1(\O)\times \H^1(\O)\times \H^1(\O)$, there exists a constant $c(\O)>0$ such that
\begin{eqnarray}
\displaystyle \left|\int_{\O}  (\varphi\cdot \nabla)v\cdot w\, \dx\right |\leq c(\O)\,\left\|\varphi\right\|_{\H^1(\O)}\left\|\nabla v\right\|_{\L^2(\O)}\left\| w\right\|_{\H^1(\O)}.
\end{eqnarray}
Moreover, if $\varphi=w=0\,\mbox{ on }\,\partial \Omega$, we have
\begin{eqnarray}
\displaystyle  \left|\int_{\O}  (\varphi\cdot \nabla)v\cdot w\, \dx\right |\leq \rho(\O)\,\left\|\nabla \varphi\right\|_{\L^2(\O)}\left\|\nabla v\right\|_{\L^2(\O)}\left\| \nabla w\right\|_{\L^2(\O)},
\end{eqnarray}
where $\rho(\O)$ is defined analogously to \eqref{uni-con11}.
\end{lemma}

We recall the following integral identity. 

\begin{lemma} (see  \cite[ Lemma IX.2.1, p. 591]{galdi2011introduction}).\label{lem-NS-prelim2} Let \( \O \) be a bounded domain with a locally Lipschitz boundary in \( \mathbb{R}^3 \). If \( \varphi \in \H^1(\O) \) satisfies \( \text{div} \, \varphi = 0 \) in \( \O \), then  
    \begin{equation}
        \int_{\O} (\varphi\cdot \nabla)v \cdot v\, \dx = 0\;\;\; \text{  for all  }\;v\in\H^1_0(\O).
    \end{equation}
\end{lemma}

We now establish an estimate that characterizes the behavior of the perturbed velocity field $u_\varepsilon$ with respect to the obstacle size $\varepsilon.$

\begin{lemma}\label{lemma-es1}
 Let $u_\varepsilon$ and $u_0$ be the solutions to \eqref{problem-penalized-0} and \eqref{problem-perturbed-00}, respectively. Then, there exists a constant $C>0$, independent of $\varepsilon$, such that
\begin{equation}\label{estimation-perturbed-velocity}
\big\|u_\varepsilon-u_0\big\|_{L^\infty(0,T;\L^2(\Omega))}+\big\|u_\varepsilon-u_0\big\|_{L^2(0,T;\H^1(\Omega))}\leq C\, \varepsilon^{\frac{5}{2}}.
\end{equation}
\end{lemma}
\begin{proof}
Let \( w_\varepsilon \) be the difference between \( u_\varepsilon \) and \( u_0 \), i.e., \( w_\varepsilon = u_\varepsilon - u_0 \). Then, from \eqref{problem-penalized-0} and \eqref{problem-perturbed-00}, one can easily verify that \( w_\varepsilon \) satisfies the following boundary value problem
\begin{equation}\label{problem-diff}
\left\{
\begin{array}{rll}
\displaystyle \frac {\partial w_\varepsilon}{\partial t}-\nu
\Delta w_\varepsilon+ \N(w_\varepsilon)+\big< D \N(w_\varepsilon),\,u_0 \big> +\nabla \Pi_\varepsilon
 +k\chi_{\mathcal{C}_{z,\varepsilon}}\, w_\varepsilon&= -k\chi_{\mathcal{C}_{z,\varepsilon}}\, u_0 &\mbox{ in }  \,\,\Omega\times (0,\,T) ,\\
\mbox{div}\, w_\varepsilon &=  0 & \mbox{ in } \,\, \Omega \times (0,\,T) ,\\
w_\varepsilon &= 0 & \mbox{ on }\,\,  \partial\Omega \times (0,\,T) , \\
w_\varepsilon(\cdot,0)&= 0&\mbox{ in } \,\,  \Omega.
\end{array}
 \right.
\end{equation}
Here, \( \Pi_\varepsilon = \pi_\varepsilon - \pi_0 \) represents the pressure variation, and \( D\mathcal{N}(w_\varepsilon) \) denotes the differential of the map \( w \mapsto \mathcal{N}(w) \) at \( w_\varepsilon \), which is applied to the field \( u_0 \) as follows
\begin{equation}\label{DN}
   \big< D \N(w_\varepsilon),\,u_0 \big>=(u_0\cdot\nabla )w_\varepsilon + (w_\varepsilon \cdot \nabla)u_0. 
\end{equation}
By utilizing the variational formulation of \eqref{problem-diff} and choosing $w_\varepsilon$ as a test function, we obtain
\begin{align}
\begin{split}\label{estim-R3-1}
    \int_0^{\tau}\int_{\Omega}\frac {\partial w_\varepsilon}{\partial t}\cdot w_\varepsilon\, \dx\dt&+\nu\int_0^{\tau}\int_\Omega\big|\nabla w_\varepsilon\big|^2\dx\dt+\int_0^{\tau}\int_\Omega\N(w_\varepsilon)\cdot w_\varepsilon\dx\dt\\
&+\int_0^\tau\int_\Omega\big[ (u_0\cdot\nabla )w_\varepsilon + (w_\varepsilon \cdot \nabla)u_0\big]\cdot w_\varepsilon\,\dx\dt+k\int_0^{\tau}\int_{\mathcal{C}_{z,\varepsilon}}\big|w_\varepsilon\big|^2\dx\dt\\
&\qquad\quad=-k\int_0^{\tau}\int_{\mathcal{C}_{z,\varepsilon}}u_0\cdot w_\varepsilon\,\dx\dt,
\end{split}
\end{align} 
for almost all $\tau\in [0,\,T].$ Since $w_\varepsilon=0$ on $\partial \Omega$, applying Lemma \ref{lem-NS-prelim20} and Lemma \ref{lem-NS-prelim2}, we get
\begin{eqnarray}\label{rel2}
\int_0^{\tau}\hspace{-0.2cm}\int_{\Omega}\N(w_\varepsilon) \cdot w_\varepsilon\, \dx\dt=0\,\text{ and }
\,\int_0^{\tau}\hspace{-0.2cm}\int_{\Omega} \hspace{-0.2cm}(u_0\cdot \nabla) w_\varepsilon \cdot w_\varepsilon\, \dx\dt=0.
\end{eqnarray}
Integrating in time and using the initial condition $w_\varepsilon(\cdot,0)=0$ in $\Omega$, it follows
\begin{eqnarray}\label{rel1}
\displaystyle \int_0^{\tau}\int_{\Omega}\frac {\partial w_\varepsilon}{\partial t}\cdot w_\varepsilon\, \dx\dt=\frac{1}{2} \int_{\Omega}\big|w_\varepsilon(\cdot,\tau)\big|^2\dx.
\end{eqnarray}
Inserting the relations \eqref{rel2} and \eqref{rel1} into \eqref{estim-R3-1}, we obtain
\begin{align}
\begin{split}\label{estim-R3-1-new}
   \frac{1}{2} \int_{\Omega}\big|w_\varepsilon(\cdot,\tau)\big|^2\dx&+\nu\int_0^{\tau}\int_\Omega\big|\nabla w_\varepsilon\big|^2\dx\dt+k\int_0^{\tau}\int_{\mathcal{C}_{z,\varepsilon}}\big|w_\varepsilon\big|^2\dx\dt\\
   &=-k\int_0^{\tau}\int_{\mathcal{C}_{z,\varepsilon}}u_0\cdot w_\varepsilon\,\dx\dt-\int_0^\tau\int_\Omega(w_\varepsilon \cdot \nabla)u_0\cdot w_\varepsilon\,\dx\dt,
\end{split}
\end{align} 
for almost all $\tau\in[0,\,T].$ \\

\noindent To derive the desired estimate for \(w_\varepsilon\), we analyze each term on the right-hand side of the equality \eqref{estim-R3-1-new}. 

\begin{itemize}
    \item[$-$] \emph{Estimate of the first term:}  
    Using the Cauchy-Schwarz inequality and the boundedness of \( u_0 \), we obtain  
    \begin{align}\label{gh12}
        \big| k \int_0^{\tau} \int_{\mathcal{C}_{z,\varepsilon}} u_0 \cdot w_\varepsilon \, \dx\dt \big|  
        \leq C_1\,\varepsilon^{\frac{3}{2}} \int_0^T \big(\int_{\mathcal{C}_{z,\varepsilon}} \big|w_\varepsilon\big|^2 \dx \big)^{\frac{1}{2}} \dt.
    \end{align}
    Applying the Sobolev embedding \( H^1(\Omega) \subset L^6(\Omega) \) and H\"older's inequality (with \( p=3 \) and \( q=3/2 \)), we deduce  
    \begin{align*}
        \int_0^T \big(\int_{\mathcal{C}_{z,\varepsilon}} \big|w_\varepsilon\big|^2 \dx \big)^{\frac{1}{2}} \dt  
        &\leq C_2 \varepsilon \int_0^T \big\| w_\varepsilon \big\|_{\L^6(\mathcal{C}_{z,\varepsilon})} \dt \\  
        &\leq C_3 \varepsilon \int_0^T \big\| w_\varepsilon \big\|_{\H^1(\Omega)} \dt.
    \end{align*}
    Substituting this bound into \eqref{gh12} and using the Cauchy-Schwarz inequality again, we conclude  
    \begin{align}\label{gh12f1}
        \big| k \int_0^{\tau} \int_{\mathcal{C}_{z,\varepsilon}} u_0 \cdot w_\varepsilon \, \dx\dt \big|  
        \leq C_4\,\varepsilon^{\frac{5}{2}} \big\|w_\varepsilon\big\|_{L^2(0,T;\H^1(\Omega))}.
    \end{align}

    \item[$-$] \emph{Estimate of the second term:}  
    Since \( w_\varepsilon = 0 \) on \( \partial \Omega \), we can apply Lemma \ref{lem-NS-prelim1} to obtain  
    \begin{align}
        \big| \int_0^\tau \int_\Omega (w_\varepsilon \cdot \nabla) u_0 \cdot w_\varepsilon \, \dx\dt \big|  
        \leq \rho(\Omega) \big\|\nabla w_\varepsilon\big\|_{L^2(0,T;\L^2(\Omega))}^2 \big\|\nabla u_0\big\|_{L^\infty(0,T;\L^2(\Omega))}.
    \end{align}
    Using assumption (A2), i.e.,  
    \begin{align*}
        \left\|\nabla u_0\right\|_{L^\infty(0,T;\,\L^2(\Omega))} \leq \delta < \frac{\nu}{\rho(\Omega)},
    \end{align*}
    As a result, the trilinear term satisfies  
    \begin{align}\label{estim-R3-3}
        \big| \int_0^\tau \int_\Omega (w_\varepsilon \cdot \nabla) u_0 \cdot w_\varepsilon \, \dx\dt \big|  
        \leq {\delta} \, \rho(\Omega)  \left\|\nabla w_\varepsilon\right\|^2_{L^2(0,T;\L^2(\Omega))}.
    \end{align}
\end{itemize}
Now, returning to the identity  \eqref{estim-R3-1-new}, we have
\begin{align}
\begin{split}\label{estim-R3-1-new1}
   \frac{1}{2} \int_{\Omega}\big|w_\varepsilon(\cdot,\tau)\big|^2\dx
   + \nu\int_0^{\tau}\int_\Omega\big|\nabla w_\varepsilon\big|^2\dx\dt
   &\leq\big|k\int_0^{\tau}\int_{\mathcal{C}_{z,\varepsilon}}u_0\cdot w_\varepsilon\,\dx\dt\big| \\
   &\quad+\big|\int_0^\tau\int_\Omega(w_\varepsilon \cdot \nabla)u_0\cdot w_\varepsilon\,\dx\dt\big|,
\end{split}
\end{align}
for almost every \( \tau \in [0, T] \). By combining \eqref{estim-R3-1-new1} with estimates \eqref{gh12f1} and \eqref{estim-R3-3}, we obtain
\begin{align*}
   \frac{1}{2} \int_{\Omega}\big|w_\varepsilon(\cdot,\tau)\big|^2\dx
   + \nu\int_0^{\tau}\int_\Omega\big|\nabla w_\varepsilon\big|^2\dx\dt
   \leq {\delta} \, \rho(\Omega)  \left\|\nabla w_\varepsilon\right\|^2_{L^2(0,T;\L^2(\Omega))}
   + C_4\,\varepsilon^{\frac{5}{2}} \big\|w_\varepsilon\big\|_{L^2(0,T;\H^1(\Omega))}.
\end{align*}
for almost every \( \tau \in [0, T] \). Taking the supremum over time, we deduce
\begin{align*}
    \big\|w_\varepsilon\big\|^2_{L^\infty(0,T;\L^2(\Omega))}
    + 2\big(\nu - {\delta} \, \rho(\Omega)\big)
    \big\|\nabla w_\varepsilon\big\|^2_{L^2(0,T;\L^2(\Omega))}
    \leq 2\,C_4\,\varepsilon^{\frac{5}{2}} \big\|w_\varepsilon\big\|_{L^2(0,T;\H^1(\Omega))}.
\end{align*}
Since \( \nu - {\delta}\,\rho(\Omega) > 0 \), we can apply the Poincaré inequality to conclude that
\begin{align*}
    \big\|w_\varepsilon\big\|^2_{L^\infty(0,T;\L^2(\Omega))}
    + \big\| w_\varepsilon\big\|^2_{L^2(0,T;\H^1(\Omega))}
    \leq \frac{2\,C_4}{\min\{1, 2\,c(\Omega)\,(\nu - {\delta} \, \rho(\Omega))\}}\,\varepsilon^{\frac{5}{2}} 
    \big\|w_\varepsilon\big\|_{L^2(0,T;\H^1(\Omega))}.
\end{align*}
Finally, using Young's inequality, we obtain 
\begin{align*}
    \big\|w_\varepsilon\big\|_{L^\infty(0,T;\L^2(\Omega))}
    + \big\| w_\varepsilon\big\|_{L^2(0,T;\H^1(\Omega))}
    \leq C\,\varepsilon^{\frac{5}{2}},
\end{align*}
where \( C \) is given by
\[
C := \frac{2\,C_4}{\min\{1, 2\, c(\Omega)\,(\nu - {\delta} \, \rho(\Omega))\}},
\]
which completes the proof.
\end{proof}

The derived estimate \eqref{estimation-perturbed-velocity} quantifies the impact of a small obstacle on variations in the velocity field and plays a key role in establishing the topological asymptotic expansion.

%%%%%%%%%%%%%%%%%%%%%%%%%%%%%%%%%%%%%%%%%%%%%%%%%%%%%%%%%%%%%
\subsection{Asymptotic analysis}
%%%%%%%%%%%%%%%%%%%%%%%%%%%%%%%%%%%%%%%%%%%%%%%%%%%%%%%%%%%%%

We are now prepared to present the main results of this section, namely the asymptotic expansion \eqref{asy}. Building upon the previous estimate, we develop a topological sensitivity analysis for the non-stationary Navier-Stokes problem. To this end, we introduce the Lagrangian \(\mathcal{L}_\varepsilon\) associated with the cost function \(K\), defined as follows:
\begin{align*}
\mathcal{L}_\varepsilon(\psi,\varphi)=K(\Omega\backslash\overline{\mathcal{C}_{z,\varepsilon}})+\int_0^T\displaystyle\big<\frac{\partial \psi}{\partial t},\, \varphi\big>_\Omega\dt+ \int_0^T a_\varepsilon(\psi,\varphi)\dt-\int_0^T \big<\mathcal{G},\, \varphi\big>_\Omega\dt,
\end{align*}
for all $(\psi,\varphi)\in L^2(0,T; \H^1_{0,\text{div}}(\Omega))\times \H^1_{0,\text{div}}(\Omega).$ As a direct consequence, since \(u_\varepsilon\) is the solution to \eqref{problem-penalized-0}, one can deduce that 
\begin{align*}
\mathcal{L}_\varepsilon(u_\varepsilon,\varphi)=K(\Omega\backslash\overline{\mathcal{C}_{z,\varepsilon}}),\qquad\forall\varphi\in \H^1_{0,\text{div}}(\Omega).
\end{align*}
Thus, from \eqref{perturbed-shape-function} and \eqref{unperturbed-shape-function}, we obtain  
\begin{align*}
\nonumber
K(\Omega\backslash\overline{\mathcal{C}_{z,\varepsilon}})-K(\Omega)&=\mathcal{L}_\varepsilon(u_\varepsilon,\varphi)-\mathcal{L}_0(u_0,\varphi)\\
\nonumber&=\K_\gamma(\mathcal{C}_{z,\varepsilon})-\K(\emptyset)+\int_0^T\displaystyle\big<\frac{\partial (u_\varepsilon-u_0)}{\partial t},\, \varphi\big>_\Omega\dt+ \int_0^T \big[a_\varepsilon(u_\varepsilon,\varphi)-a_0(u_0,\varphi)\big]\dt\\
\nonumber&=\int_0^T \int_{\Omega_0}\big|
(u_\varepsilon-u_0)+(u_0-u_{meas})\big|^2 \dx\dt-\int_0^T \int_{\Omega_0}\big|
u_0-u_{meas}\big|^2 \dx\dt\\
&+\int_0^T\big<\frac{\partial (u_\varepsilon-u_0)}{\partial t},\, \varphi\big>_\Omega\dt+ \int_0^T \big[a_\varepsilon(u_\varepsilon,\varphi)-a_0(u_0,\varphi)\big]\dt+\gamma\varepsilon^{2}\text{Per}_\Omega(\mathcal{C})\\
\nonumber&=2\int_0^T \int_{\Omega_0}\big(u_\varepsilon-u_0\big)\cdot\big(u_0-u_{meas}\big) \dx\dt+\big\|
u_\varepsilon-u_0\big\|^2_{L^2(0,T;\L^2(\Omega_0))}\\
&+\int_0^T\big<\frac{\partial (u_\varepsilon-u_0)}{\partial t},\, \varphi\big>_\Omega\dt+\int_0^T \big[a_\varepsilon(u_\varepsilon,\varphi)-a_0(u_0,\varphi)\big]\dt +\gamma\varepsilon^{2}\text{Per}_\Omega(\mathcal{C}),
\end{align*}
for all $\varphi\in \H^1_{0,\text{div}}(\Omega).$ Using the estimation \eqref{estimation-perturbed-velocity}, we deduce the existence of a constant $C>0$ independent of $\varepsilon,$ such that
\begin{align*}
    \big\|
u_\varepsilon-u_0\big\|^2_{L^2(0,T;\L^2(\Omega_0))}\leq C\varepsilon^{5}.
\end{align*}
Consequently, we have
\begin{equation}
     \big\|
u_\varepsilon-u_0\big\|^2_{L^2(0,T;\L^2(\Omega_0))}=o(\varepsilon^3).
\end{equation}
By inserting this estimation into the variation $ K(\Omega\backslash\overline{\mathcal{C}_{z,\varepsilon}})-K(\Omega),$ we obtain 
\begin{align}
\begin{split}\label{as-cl}
 K(\Omega\backslash\overline{\mathcal{C}_{z,\varepsilon}})-K(\Omega)&=2\int_0^T \int_{\Omega_0}\big(u_\varepsilon-u_0\big)\cdot\big(u_0-u_{meas}\big) \dx\dt\\
 &\qquad\quad+\int_0^T\big<\frac{\partial (u_\varepsilon-u_0)}{\partial t},\, \varphi\big>_\Omega\dt\\
&\quad\qquad\qquad +\int_0^T \big[a_\varepsilon(u_\varepsilon,\varphi)-a_0(u_0,\varphi)\big]\dt\\
&\qquad\qquad\qquad +\gamma\varepsilon^{2}\text{Per}_\Omega(\mathcal{C})+o(\varepsilon^3),\;\;\text{ for all } \varphi\in \H^1_{0,\text{div}}(\Omega) .
\end{split}
\end{align}
We will now estimate the term $$\displaystyle2\int_0^T \int_{\Omega_0}\big(u_\varepsilon-u_0\big)\cdot\big(u_0-u_{meas}\big) \dx\dt.$$ 
To perform this estimation, we first introduce an adjoint state $v_0$, which is defined as the solution to the following auxiliary boundary value problem: find $(v_0,p_0)$ such that
\begin{equation}\label{problem-adjoint}
\left\{
\begin{array}{rll}
\displaystyle -\frac {\partial v_0}{\partial t}-\nu
\Delta v_0+{}^t\nabla u_0\,v_0- (u_0\cdot\nabla)v_0+ \nabla p_0
&= -2\big(u_0-u_{meas}\big)\chi_{\Omega_0} &\mbox{ in }  \,\,\Omega\times (0,\,T) ,\\
\mbox{div}\, v_0 &=  0 & \mbox{ in } \,\, \Omega \times (0,\,T) ,\\
v_0 &= 0 & \mbox{ on }\,\,  \partial\Omega \times (0,\,T) , \\
v_0(\cdot,T)&= 0&\mbox{ in } \,\,  \Omega.
\end{array}
 \right.
\end{equation}
Since \( u_0 \) satisfies \eqref{condition-u0}, the existence and uniqueness of the solution to \eqref{problem-adjoint} follow directly by the change of variables \( t \longleftarrow T - t \) and the application of \cite[Lemma 2.1]{imanuvilov2001remarks} (see also \cite{TemamBook1977}). Moreover, this solution satisfies the regularity property $v_0\in L^2(0,T;\H^2(\Omega))$. In particular, by the Sobolev embedding theorem, we deduce that
\begin{equation}\label{regularity2}
    v_0\in L^2(0,T;\mathcal{C}(\overline{\Omega})).
\end{equation}

\vspace{0.4cm}
From the weak formulation of the adjoint problem \eqref{problem-adjoint}, choosing \( w_\varepsilon = u_\varepsilon - u_0 \) as a test function yields
\begin{align*}
    2\int_0^T\int_{\Omega_0}\big(u_0-u_{meas}\big)\cdot\big(u_\varepsilon-u_0\big)\dx\dt&=\int_0^T\int_{\Omega}\frac {\partial v_0}{\partial t} \cdot w_\varepsilon\,\dx\dt-\nu\int_0^T\int_\Omega\nabla v_0:\nabla w_\varepsilon\,\dx\dt\\
    &\qquad\quad-\int_0^T\int_\Omega\big({}^t\nabla u_0\,v_0- (u_0\cdot\nabla)v_0\big)\cdot w_\varepsilon\,\dx\dt.
\end{align*}
Recall that the velocity field variation \( w_\varepsilon = u_\varepsilon - u_0 \) is a solution of problem \eqref{problem-diff}.  Given that \( w_\varepsilon(\cdot,0) =v_0(\cdot,T) = 0 \) in \( \Omega \), integrating by parts in time gives
\begin{equation*}
    \int_0^T\int_{\Omega}\frac {\partial v_0}{\partial t} \cdot w_\varepsilon\,\dx\dt=-\int_0^T\int_{\Omega}\frac {\partial w_\varepsilon}{\partial t} \cdot v_0\,\dx\dt.
\end{equation*}
Moreover, using that \( v_0 = w_\varepsilon = 0 \) on \( \partial\Omega \times (0,T) \) and applying Lemma \ref{lem-NS-prelim20}, we obtain
\begin{align*}
       \int_0^T\int_\Omega\big( {}^t\nabla u_0\,v_0-(u_0\cdot\nabla)v_0\big)\cdot w_\varepsilon\,\dx\dt&=\int_0^T\int_\Omega\big((w_\varepsilon\cdot\nabla)u_0+ (u_0\cdot\nabla)w_\varepsilon\big)\cdot v_0\,\dx\dt\\
       &=\int_0^T\int_\Omega\big<D\N(w_\varepsilon),\,u_0\big>\cdot v_0\,\dx\dt,
\end{align*}
where $D\N$ denotes the differential of the nonlinear operator $\N,$ as defined in \eqref{DN}. Consequently, we deduce that
\begin{align}\label{cl-1}
  \nonumber  2\int_0^T\int_{\Omega_0}\big(u_0-u_{meas}\big)\cdot\big(u_\varepsilon-u_0\big)\dx\dt&=-\int_0^T\int_{\Omega}\frac {\partial w_\varepsilon}{\partial t} \cdot v_0\,\dx\dt-\nu\int_0^T\int_\Omega\nabla w_\varepsilon:\nabla v_0\,\dx\dt\\
    &\qquad\quad-\int_0^T\int_\Omega\big<D\N(w_\varepsilon),\,u_0\big>\cdot v_0\,\dx\dt.
\end{align}
On the other hand, by subtracting \eqref{weak-f-penalized} for \( \varepsilon \neq 0 \) with \( \varphi = v_0 \) from \eqref{weak-f-penalized} for \( \varepsilon = 0 \) with \( \varphi = v_0 \), and applying Lemma \ref{lem-NS-prelim20} (noting that \( v_0 = u_\varepsilon = 0 \) on \( \partial\Omega \)), we obtain
\begin{align}
\begin{split}\label{cl-2}
    \int_0^T\big<\frac{\partial (u_\varepsilon-u_0)}{\partial t},\, v_0\big>_\Omega&\dt+\int_0^T \big[a_\varepsilon(u_\varepsilon,v_0)-a_0(u_0,v_0)\big]\dt\\
    &=\int_0^{T}\int_{\Omega}\frac {\partial w_\varepsilon}{\partial t}\cdot v_0\, \dx\dt+\nu\int_0^{T}\int_\Omega\nabla w_\varepsilon:\nabla v_0\dx\dt\\
    &\quad+\int_0^T\int_\Omega\big<D\N(w_\varepsilon),\,u_0\big>\cdot v_0\,\dx\dt\\
&\qquad\quad+k\int_0^{T}\int_{\mathcal{C}_{z,\varepsilon}}u_\varepsilon\cdot v_0\,\dx\dt.
\end{split}
\end{align} 
Taking \( \varphi = v_0 \) in \eqref{as-cl} and using \eqref{cl-1} and \eqref{cl-2}, we arrive at
\begin{align}\label{as-cl-1}
 K(\Omega\backslash\overline{\mathcal{C}_{z,\varepsilon}})-K(\Omega)=k\int_0^{T}\int_{\mathcal{C}_{z,\varepsilon}}u_\varepsilon\cdot v_0\,\dx\dt+\gamma\varepsilon^{2}\text{Per}_\Omega(\mathcal{C})+o(\varepsilon^3).
\end{align}
In the final part of this paragraph, we focus on estimating the first term on the right-hand side of \eqref{as-cl-1}. We have
\begin{align*}
   k\int_0^{T}\int_{\mathcal{C}_{z,\varepsilon}} u_\varepsilon\cdot v_0 \,\dx\dt
   &= k\int_0^{T}\int_{\mathcal{C}_{z,\varepsilon}} (u_\varepsilon - u_0) \cdot v_0 \,\dx\dt 
   + k\int_0^{T}\int_{\mathcal{C}_{z,\varepsilon}} u_0 \cdot v_0 \,\dx\dt.
\end{align*}
Applying the change of variables \( x = z + \varepsilon y \), we deduce
\begin{align*}
    k\int_0^{T}\int_{\mathcal{C}_{z,\varepsilon}} u_0 \cdot v_0 \,\dx\dt
    &= k |\mathcal{C}| \varepsilon^3 \int_0^T u_0(z,t) \cdot v_0(z,t) \,\dt \\
    &\quad + k  \int_0^T \int_{\mathcal{C}_{z,\varepsilon}} 
    \big( u_0(x,t) \cdot v_0(x,t) - u_0(z,t) \cdot v_0(z,t) \big) \,\dx\dt.
\end{align*}
Using the continuity properties of $u_0$ and $v_0$, as established in \eqref{regularity1} and \eqref{regularity2}, we have
\[u_0(x,t) \cdot v_0(x,t) - u_0(z,t) \cdot v_0(z,t) \longrightarrow0\; \text{  as  } x\to z.\]
Consequently, 
\[
  \int_0^T \int_{\mathcal{C}_{z,\varepsilon}} 
    \big( u_0(x,t) \cdot v_0(x,t) - u_0(z,t) \cdot v_0(z,t) \big) \,\dx\dt = o(\varepsilon^3).
\]
By applying the same analysis as in the estimation of \eqref{gh12f1}, there exists a constant \( C > 0 \) (independent of \( \varepsilon \)) such that
\begin{align*}
    \big| k\int_0^{T}\int_{\mathcal{C}_{z,\varepsilon}} (u_\varepsilon - u_0) \cdot v_0 \,\dx\dt \big|  
    \leq C\,\varepsilon^{\frac{5}{2}} \big\|u_\varepsilon - u_0\big\|_{L^2(0,T;\H^1(\Omega))}.
\end{align*}
Using the estimate from Lemma \ref{lemma-es1}, we obtain
\begin{align*}
    \big| k\int_0^{T}\int_{\mathcal{C}_{z,\varepsilon}} (u_\varepsilon - u_0) \cdot v_0 \,\dx\dt \big| = o(\varepsilon^3).
\end{align*}
Thus, we obtain
\begin{align*}
   k\int_0^{T}\int_{\mathcal{C}_{z,\varepsilon}} u_\varepsilon \cdot v_0 \,\dx\dt
   =  k |\mathcal{C}| \varepsilon^3 \int_0^T u_0(z,t) \cdot v_0(z,t) \,\dt + o(\varepsilon^3).
\end{align*}

Finally, we conclude that the shape function \( K \) admits the following asymptotic expansion:
\begin{equation*}
    K(\Omega\backslash\overline{\mathcal{C}_{z,\varepsilon}}) - K(\Omega) 
    = k |\mathcal{C}| \varepsilon^3 \int_0^T u_0(z,t) \cdot v_0(z,t) \,\dt + \gamma\varepsilon^{2}\text{Per}_\Omega(\mathcal{C})+o(\varepsilon^3).
\end{equation*}

One of the key advantages of the topological derivative method is that it does not require an initial guess, in the sense that the initial domain can be chosen to be empty, i.e., \( \mathcal{C} = \varnothing \). This property makes the method highly efficient for identifying obstacles without the need for a carefully selected starting domain. Moreover, the topological derivative method—particularly when applied to misfit functions involving the \( L^2 \)-norm-has been observed to exhibit a self-regularization property. This means that additional regularization techniques, which are typically required to stabilize the inverse problem, become unnecessary. This self-regularizing behavior has been reported in various contexts, including the Stokes problem, Laplace equation, linear elasticity problem, and fractional diffusion problem \cite{AmstutzCC2005,CanelasIP2015,CanelasJCP2014,CaubetIP2012,CaubetIPI2016,hrizi2019new, FerreiraIP2017,FernandezAA2018,RochaSMO2017,menoret2020kohn,prakash2021noniterative}. However, despite strong numerical evidence supporting this phenomenon, a rigorous mathematical proof is still lacking. From this discussion, it follows that the regularization term, specifically the relative perimeter \( \text{Per}_\Omega(\mathcal{C}) \), has little to no impact on the reconstruction process. Consequently, the parameter \( \gamma \) can be chosen arbitrarily without significantly impacting the reconstruction process. In light of this and for the sake of simplicity, we set \( \gamma = \varepsilon \) for the remainder of this paper. Based on this assumption, we now summarize the topological asymptotic expansion of \( K \) in the following theorem.

\begin{theorem}\label{TD-G}
    The shape function \( K \)  (see \eqref{perturbed-shape-function}) admits the following topological asymptotic expansion:
    \begin{equation*}
        K(\Omega\backslash\overline{\mathcal{C}_{z,\varepsilon}}) - K(\Omega) 
        = k |\mathcal{C}| \varepsilon^3 D_K(z) + o(\varepsilon^3),
    \end{equation*}
    where \( D_K \) is the topological gradient, defined for any point \( x \in \Omega \) as
    $$
        D_K(x) := \int_0^T u_0(x,t) \cdot v_0(x,t) \,\dt + \text{Per}_\Omega(\mathcal{C}).
    $$
Here, \( u_0 \) and \( v_0 \) are the solutions to the state and adjoint problems \eqref{problem-perturbed-00} and \eqref{problem-adjoint}, respectively.
\end{theorem}

%%%%%%%%%%%%%%%%%%%%%%%%%%%%%%%%%%%%%%%%%%%%%%%%%%%%%%%%%%%%%%%%%%%%%%%%%%%%%%%
\section{Implementation details and numerical experiments}\label{one-iter-alghorithm} 
%%%%%%%%%%%%%%%%%%%%%%%%%%%%%%%%%%%%%%%%%%%%%%%%%%%%%%%%%%%%%%%%%%%%%%%%%%%%%%%

This section presents numerical experiments designed to illustrate the effectiveness and robustness of the proposed reconstruction algorithm for identifying an unknown obstacle \( \omega^* \) embedded in a fluid flow domain, using interior velocity measurements. The reconstruction strategy is based on topological sensitivity analysis.

According to Theorem~\ref{TD-G}, the topological gradient of the cost functional \( K \) is given by
\[
D_K(x) := \int_0^T u_0(x,t) \cdot v_0(x,t) \, \dt + \text{Per}_\Omega(\mathcal{C}),
\]
where \( u_0 \) and \( v_0 \) denote the solutions to the state and adjoint problems \eqref{problem-perturbed-00} and \eqref{problem-adjoint}, respectively.

Recall that the proposed topological reconstruction algorithm does not require an initial guess. In particular, we initialize the algorithm with \( \mathcal{C} = \emptyset \), in which case the regularization term vanishes:
\[
\text{Per}_\Omega(\mathcal{C}) = 0.
\]
As a result, the topological gradient simplifies to
\begin{equation}\label{gradient-final}
    D_K(x) := \int_0^T u_0(x,t) \cdot v_0(x,t) \, \dt.
\end{equation}
Based on this simplified form of \( D_K \), the identification procedure is implemented as a one-shot algorithm composed of the following steps:

\vspace{0.3cm}

\begin{algorithm}[H]
\caption{One-Iteration Topological Identification Algorithm}
\label{alg:one-shot}
\begin{itemize}
    \item Solve the direct problem \eqref{problem-perturbed-00}.
    \item Solve the adjoint problem \eqref{problem-adjoint}.
    \item Compute the topological gradient \( D_K(x)=\displaystyle\int_0^T u_0(x,t) \cdot v_0(x,t) \, \mathrm{d}t \) at each point \( x \in \Omega \).
    \item Identify the negative local minima of \( D_K(x) \).
\end{itemize}
\end{algorithm}

\vspace{0.2cm}

\noindent In this non-iterative framework:
\begin{itemize}
    \item The location of the obstacle \( \omega^* \) is approximated by the point \( z^* \in \Omega \) where the topological gradient is most negative, i.e.,
    \[
    z^* = \operatorname*{arg\,min}_{x \in \Omega} D_K(x).
    \]
    \item The optimal size of the reconstructed obstacle \( \omega^* \) is approximated by a level-set of the topological gradient \( D_K \).
\end{itemize}

First-order topological gradient-based algorithms have been successfully applied in a variety of inverse problems, including crack detection from overdetermined boundary data \cite{AmstutzCC2005}, reconstruction of contact regions in semiconductor transistors \cite{hrizi2019reconstruction}, fluorescence optical tomography \cite{LaurainIP2013}, cardiac electrophysiology \cite{beretta2017reconstruction,BerettaIP2017}, analysis of 2D and 3D Fresnel experimental data \cite{carpio2021processing}, damage detection in thin plates \cite{pena2019detecting}, identification of multiple scatterers in 3D electromagnetism \cite{le2017topological1,le2018topological1}, and localization of small gas bubbles or obstacles in Stokes flow \cite{BenAbdaSIAM2009,CaubetIP2012}, among others.

In the current work, we extend this approach to a practical application. More precisely, we apply our one-shot numerical procedure to the identification of unknown submerged obstacles in the Mediterranean Sea. We begin by defining the computational domain and outlining the implementation details for computing the state \( u_0 \) and the adjoint state \( v_0 \).

%%%%%%%%%%%%%%%%%%%%%%%%%%%%%%%%%%%
\subsection{Implementation details}
%%%%%%%%%%%%%%%%%%%%%%%%%%%%%%%%%%%

The computational domain \( \Omega = L \times L \times H \) represents a sub-region of the central Mediterranean Sea along the Tunisian coastline, as illustrated in Figure \ref{domain}. Here, \( L \) denotes the horizontal extent and \( H \) the vertical depth of the domain. More precisely, \( \Omega \) corresponds to the intersection of the Mediterranean Sea and the rectangular area defined by the red square \( ABCD \). The corners of this square are geographically positioned as follows:
\begin{itemize}
    \item Point \( A \): south of Sardinia (Italy),
    \item Point \( B \): in the Tyrrhenian Sea, near Sicily (Italy),
    \item Point \( C \): southern Tunisia, near the Tunisia–Algeria border,
    \item Point \( D \): near the Tunisia–Libya border.
\end{itemize}
The rectangle \( ABCD \) spans approximately \( 700\,\mathrm{km} \) in length and \( 600\,\mathrm{km} \) in width.

\vspace{0.3cm}

The numerical simulations of the direct problem \eqref{problem-perturbed-00} and the adjoint problem \eqref{problem-adjoint} are performed using the three-dimensional ocean circulation model INSTMCOTRHD \cite{alioua, marwa}, which is built upon the well-established Princeton Ocean Model (POM). POM is a primitive-equation model designed to simulate large-scale and regional ocean dynamics under realistic atmospheric and hydrodynamic forcing conditions. The governing equations are formulated in a Cartesian coordinate system \( (O, x, y, z) \), where the \( x \)-axis points southward, the \( y \)-axis eastward, and the \( z \)-axis is oriented vertically upward. However, this Cartesian framework offers limited resolution for accurately representing complex bathymetry and capturing fine-scale processes near the surface layer.  To overcome these limitations, Blumberg and Mellor (1987) introduced the $\sigma$-coordinate transformation, a vertical coordinate system that follows the contours of the ocean bottom. This approach enhances the ability to resolve topographic features and boundary-layer dynamics with higher fidelity. The configuration is illustrated in Figure \ref{sigma}(a), and is particularly effective in representing the physical processes of interest within the oceanographic domain. The mathematical formulation and numerical implementation of POM are comprehensively detailed in the official manuals \cite{guide, guide2002}, the foundational works by Blumberg and Mellor \cite{mellor80, mellor87}, and the extended theoretical analyses provided by Mellor \cite{mellor85, mellor98}. The model computes the evolution of the following key variables:
\begin{itemize}
    \item The three components of the velocity field \( u = (u_1, u_2, u_3) \), representing zonal, meridional, and vertical velocities;
    \item Temperature $ \mathbf{T}$ and salinity $ \mathbf{S}$ fields, which influence density-driven flows and stratification;
    \item The free surface elevation \( \eta \), representing sea level height variations.
\end{itemize}

\vspace{0.2cm}
\paragraph{\bf{Spatial discretization using the Arakawa C-grid}} Spatial discretization is performed using a finite difference scheme on a rectilinear Arakawa C-grid \cite{arakawa}, which offers improved accuracy for geophysical flows by staggering scalar and vector quantities (see Figure \ref{arakawa}-(a)). The computational grid consists of \( 350 \times 300 \) horizontal points, corresponding to a spatial resolution of approximately 2 km (see Figure \ref{arakawa}-(b)). In the vertical direction, the model employs 18 terrain-following \(\sigma\)-levels to adapt to the bathymetry and to better capture boundary layer dynamics.
\begin{figure}[!htbp]
\centering
\includegraphics[width=50mm]{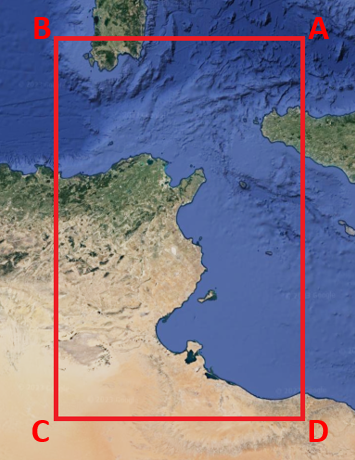}  
\caption{Sub-region of the Mediterranean Sea defining the computational domain \( \Omega \).}
\label{domain}
\end{figure}

\begin{figure} [!htbp]
     \centering
    \begin{tabular}{cc}
    \includegraphics[width=60mm]{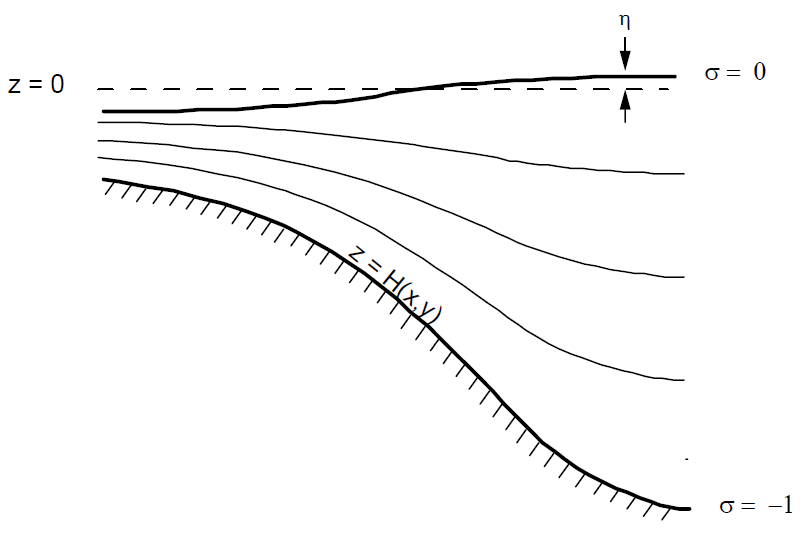}& \includegraphics[width=90mm]{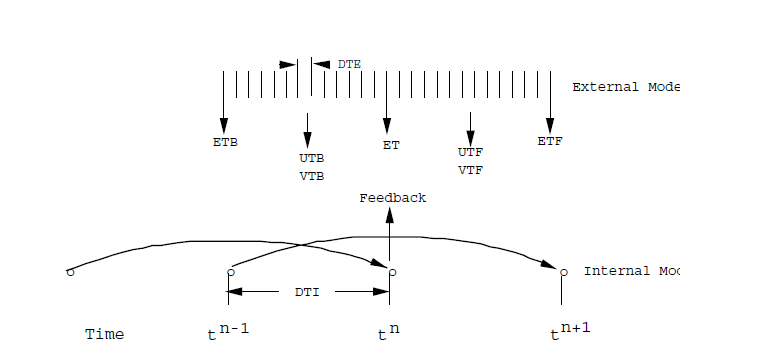}\\
(a)& (b)
    \end{tabular}
    \caption{ (a) The sigma coordinate system \cite{guide}. (b) A simplified illustration of the interaction of the External Mode and the Internal Mode. The former uses a short time step, $\mathrm{DTE}$, whereas the latter uses a long time step, $\mathrm{DTI}$. The external mode primarily provides the surface elevation to the internal mode whereas, as symbolized by ``Feedback'', the internal mode provides intergrals of momentum advection, density integrals and bottom stress to the external mode \cite{guide}.}\label{sigma}
    \end{figure}

\begin{figure} [!htbp]
     \centering
    \begin{tabular}{cc}
    \includegraphics[width=70mm]{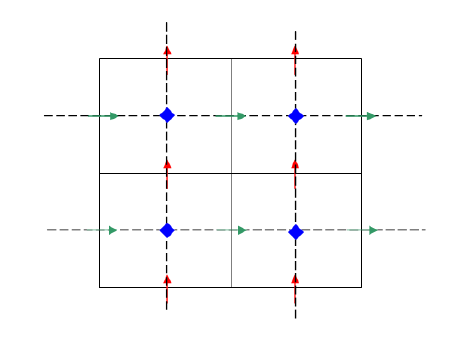}& \includegraphics[width=70mm]{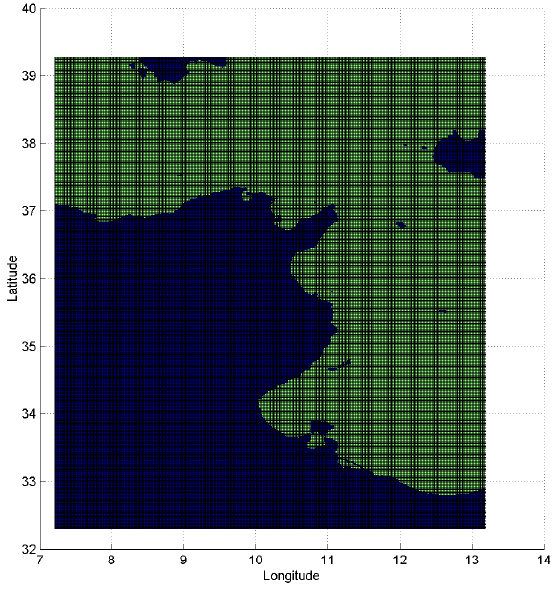}\\
(a)& (b)
    \end{tabular}
    \caption{ (a) Typical Arakawa C-grid layout: the green arrows indicate the location of the \( u_1 \)-velocity, and the red arrows indicate the \( u_2 \)-velocity components. Scalar quantities such as pressure, temperature, and salinity are located at the center (blue dot) of each cell. (b) Computational domain and horizontal discretization grid used in INSTMCOTRHD.}\label{arakawa}
    \end{figure}

\vspace{0.2cm}   
\paragraph{\bf{Temporal discretization}}
A \textit{leapfrog} (centered) differencing scheme is employed for the temporal discretization of the equations. The horizontal time differencing is treated explicitly, while the vertical differencing---used for vertical diffusion---is handled implicitly. The implicit formulation enables the use of fine vertical resolution, essential for resolving surface and bottom boundary layers, without requiring a reduction in the time step. To reduce computational cost, a \textit{time-splitting} technique---commonly referred to as \textit{mode splitting} \cite{simons,madala}---is implemented. This approach decouples the \textit{barotropic mode}, which governs fast, two-dimensional free surface variations (external mode), from the \textit{baroclinic mode}, which describes slower, three-dimensional internal dynamics associated with density variations (internal mode). External gravity waves propagate rapidly and thus require small time steps for numerical stability, whereas internal gravity waves propagate more slowly but demand high vertical resolution. Consequently, a much smaller time step is used for the barotropic mode. In the model, the external (barotropic) mode operates on a short time step (\texttt{DTE}$= 3$ seconds), while the internal (baroclinic) mode evolves on a longer time step (\texttt{DTI}$= 3$ minutes). A schematic representation of the internal and external time stepping is provided in Figure \ref{sigma}(b). Both time steps are determined based on the classical CFL condition.

\vspace{0.3cm}   
\paragraph{\bf{Time step constraints}} The choice of spatial and temporal resolution is governed by the Courant-Friedrichs-Lewy (CFL) stability condition \cite{guide}, which ensures that the numerical propagation of waves and advection processes remains stable. This criterion is expressed as:
\[
C \Delta t < \Delta x,
\]
where \( \Delta x \) denotes the horizontal grid spacing, \( \Delta t \) is the time step---either internal (\texttt{DTI}, see Figure \ref{sigma}(b)) or external (\texttt{DTE}, see Figure \ref{sigma}(b))---and \( C \) represents the maximum wave propagation speed, typically associated with gravity waves.

In practice, the horizontal grid resolution is fixed first, and the time step is chosen to satisfy the CFL condition. This ensures accurate resolution of both fast barotropic motions and slower baroclinic processes.

\vspace{0.3cm}
\paragraph{\bf{Lateral boundary conditions}} The lateral boundary conditions are imposed to ensure a realistic interaction between the computational domain and the surrounding ocean environment. For ``open boundaries'', a Dirichlet condition is applied:
\[
u = \phi \quad \text{on } \partial\Omega \times (0,\,T),
\]
where \( \phi \) represents the prescribed three-dimensional ocean velocity field. This data is obtained from the Mediterranean Sea Physics Analysis and Forecast dataset (product code: MEDSEA\_ANALYSISFORECAST\_PHY\_006\_013), which is produced using the Nucleus for European Modelling of the Ocean (NEMO) ocean model, version 3.6 \cite{NEMO}.

\noindent For ``closed boundaries'', the velocity field is set to zero, effectively imposing a no-flow condition:
\[
u = 0 \quad \text{on the impermeable parts of } \partial\Omega.
\]
The boundary velocity data \( \phi \) are available through the Copernicus Marine Environment Monitoring Service (CMEMS), accessible via their official portal at \url{https://marine.copernicus.eu}. These data provide temporally and spatially resolved forecasts and reanalyses of ocean circulation, making them highly suitable for high-fidelity boundary forcing in regional ocean models.

\vspace{0.3cm}
\paragraph{\bf Surface and bottom boundary conditions}

At the \emph{free surface} (i.e., the ocean--atmosphere interface), Neumann-type boundary conditions are imposed to ensure mass conservation and consistency with physical processes such as surface elevation dynamics and air--sea exchanges. Specifically, for the \emph{vertical velocity component} \( u_3 \), the kinematic boundary condition takes the form (see, e.g., \cite{mellor87}):
\[
u_3 = \frac{\partial \eta}{\partial t} + \mathbf{u} \cdot \nabla \eta \quad \text{on } z = \eta(x, y, t),
\]
where \( \eta(x, y, t) \) denotes the free surface elevation, and \( \mathbf{u} = (u_1, u_2) \) is the horizontal velocity vector. This condition ensures that fluid particles remain on the moving surface, preserving the impermeability of the air–sea interface.

%%Mourad voir ça: OK
\medskip

\noindent For \emph{horizontal velocity components} \( \mathbf{u} = (u_1, u_2) \), the vertical shear is directly influenced by the surface wind stress. The governing equations are given by \cite{mellor87}:
\begin{equation*}
\left\{
\begin{aligned}
& \frac{\partial u_1}{\partial z} = \frac{\tau_{0x}}{\rho_0 \, {K}_M}, \\
& \frac{\partial u_2}{\partial z} = \frac{\tau_{0y}}{\rho_0 \,  {K}_M},
\end{aligned}
\right.
\end{equation*}
where \( (\tau_{0x}, \tau_{0y}) \) denotes the components of the surface wind stress vector, \(  {K}_M \) is the vertical eddy viscosity (or vertical diffusivity) coefficient, and \( \rho_0 \) is a reference density. 

\begin{remark}
    In our numerical model, the kinematic viscosity term \( \nu \) corresponds to the coefficient \( K_M \), which quantifies the vertical turbulent viscosity. Notably, \( K_M \) is not treated as a constant but is instead computed dynamically using the Mellor–Yamada turbulence closure scheme \cite{mellor82}.
\end{remark}

\vspace{0.3cm}

\noindent\textbf{Surface forcing:} Wind stress at the ocean surface is a key driver of ocean circulation and is incorporated into the model as an external forcing term. To compute this wind stress, we utilize the $10\,\text{m}$ wind vector components (eastward and northward) provided by the European Centre for Medium-Range Weather Forecasts (ECMWF) \cite{owens2018ecmwf}. These data are essential for capturing the momentum exchange at the surface and are updated at regular intervals throughout the simulation to reflect evolving atmospheric conditions.

\medskip

\noindent At the \emph{ocean bottom} (i.e., the seafloor), located at \( z = -H(x, y) \), where \( H(x, y) \) is the bathymetry, a no-penetration (impermeability) condition is enforced:
\[
u_3 = 0 \quad \text{on } z = -H(x, y).
\]
Figure~\ref{bathymetry} displays the bathymetric profile of the computational domain. Scalar fields may also be subjected to Neumann or Robin-type conditions at the bottom, depending on modeled vertical mixing or specified boundary fluxes.

\begin{figure}[!htb]
\centering
\includegraphics[width=60mm]{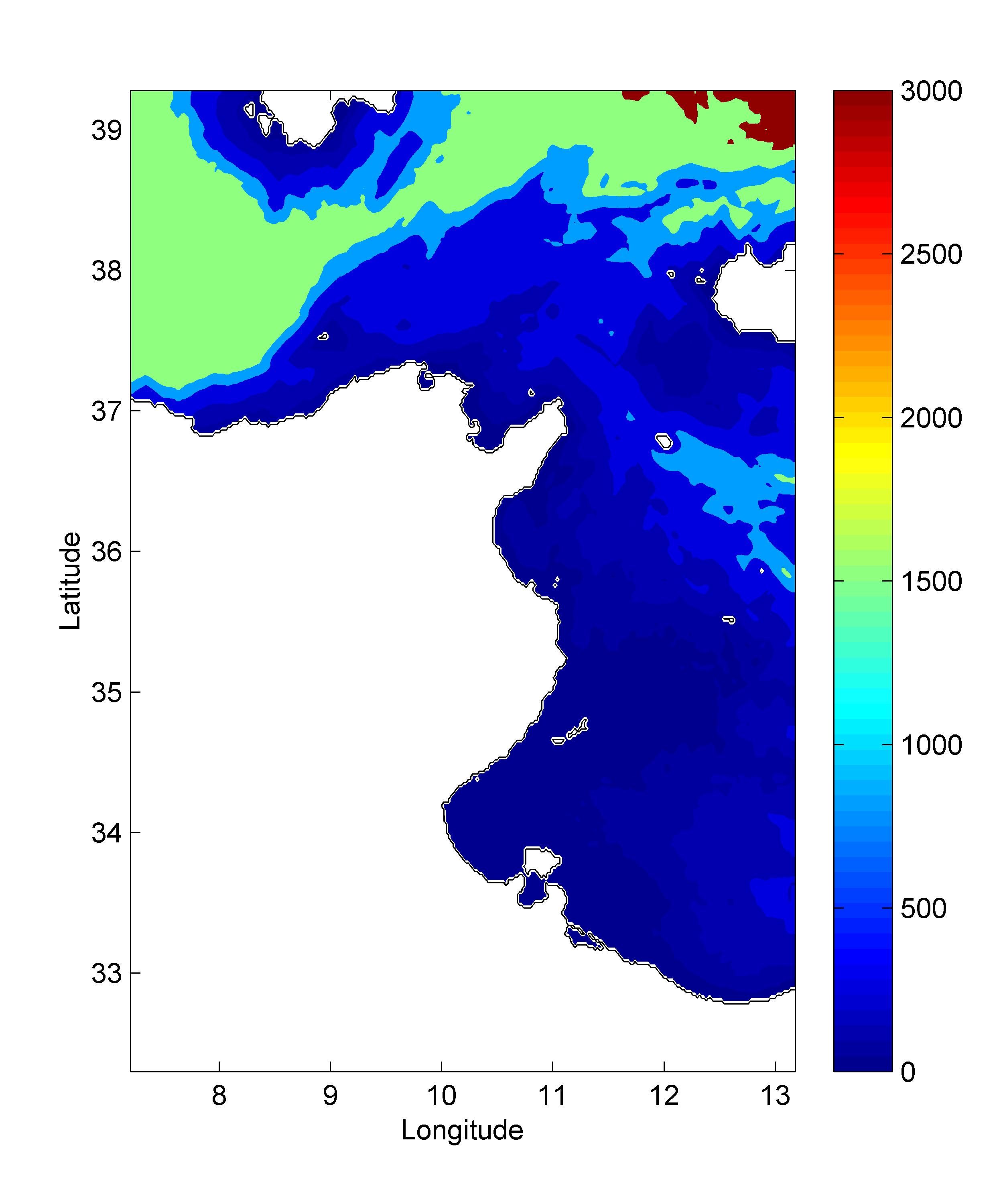}  
\caption{The model bathymetry.}
\label{bathymetry}
\end{figure}

\medskip

\noindent\textbf{External force \( \mathcal{G} \):} In the momentum equations governing the system, the source term \( \mathcal{G} \) represents the combined effect of all external forces influencing ocean dynamics, in addition to the standard advection and diffusion components. Specifically, \( \mathcal{G} \) accounts for the Coriolis force, atmospheric pressure gradients, surface wind stress, and bottom friction.

%%%%%%%%%%%%%%%%%%%%%%%%%%%%%%%%%%%%%%%%%%%%%%%%%%%%%%%%%%%%
\subsection{Numerical experiments}
%%%%%%%%%%%%%%%%%%%%%%%%%%%%%%%%%%%%%%%%%%%%%%%%%%%%%%%%%%%%

In this section, we assess the performance of our non-iterative detection procedure (see Algorithm 1) using seven numerical experiments. The POM model was executed on June 1, 2019, over a 24-hour simulation period. Numerical simulations were carried out using the \texttt{Fortran} programming language, with \texttt{MATLAB} employed for post-processing and visualization. In Examples 1 through 6, the observation domain \( \Omega_0 \) corresponds to the first vertical layer of the sigma-coordinate system, with a depth that varies spatially from \( 0.003\,\text{m} \) to \( 8\,\text{m} \), depending on the local bathymetry. In contrast, Example~7 considers a different observation domain \( \Omega_0 \), testing the robustness of the proposed method under varying spatial measurement configurations.

%%%%%%%%%%%%%%%%%%%%%%%%%%%%%%%%%%%%%%%%%%%%%%%%%%%%%%%%%%%%%%%%%%%%%%%%%%%%%%%%%%%%%%
\subsubsection{Example 1: Obstacle identification in different sub-regions of the Sea}\label{example1}
%%%%%%%%%%%%%%%%%%%%%%%%%%%%%%%%%%%%%%%%%%%%%%%%%%%%%%%%%%%%%%%%%%%%%%%%%%%%%%%%%%%%%%

The objective of the first numerical experiment is to evaluate the performance of the proposed topological sensitivity-based algorithm for detecting submerged obstacles located in various sub-regions of the Mediterranean Sea. Specifically, we aim to identify four distinct obstacles, denoted by \( \omega_i^* \) for \( i = 1, 2, 3, 4 \), each situated at a different location within the computational domain. The characteristics of these obstacles are defined as follows:
%%%%%%%%%%%%%%%%%%%%%%%
\begin{itemize}
    \item \textbf{Obstacle \( \omega_1^* \):} Defined over the grid region
    \[
    (201{:}205) \times (271{:}275) \times (2{:}6),
    \]
    which spans 4 grid points in each horizontal direction and 4 vertical layers. These vertical indices correspond to sigma levels near the surface, placing the obstacle at an average depth of approximately \( 6\,\mathrm{m} \). Its horizontal size is \( 8\,\text{km} \times 8\,\text{km} \), and vertically is in average between \( 68\,\mathrm{m} \) and \( 130\,\mathrm{m} \). 
    
    \item \textbf{Obstacle \( \omega_2^* \):} Located at
    \[
    (81{:}85) \times (291{:}295) \times (2{:}6),
    \]
    and lies in a shallower bathymetric region, resulting in an average depth of about \( 1\,\mathrm{m} \). Its horizontal size is \( 8\,\text{km} \times 8\,\text{km} \), its vertical size ranges on average between \( 22\,\mathrm{m} \) and \( 28\,\mathrm{m} \).

    \item \textbf{Obstacle \( \omega_3^* \):} Occupying the region
    \[
    (251{:}255) \times (171{:}175) \times (2{:}6),
    \]
    with characteristics similar to \( \omega_2^* \), and an average depth of approximately \( 1\,\mathrm{m} \). The horizontal size of the obstacle is \( 8\,\text{km} \times 8\,\text{km} \), its vertical size ranges on average between \( 18\,\mathrm{m} \) and \( 26\,\mathrm{m} \).

    \item \textbf{Obstacle \( \omega_4^* \):} Located in the southeastern part of the domain:
    \[
    (181{:}185) \times (91{:}95) \times (2{:}6),
    \]
    and positioned in an extremely shallow region, with an average depth of only \( 0.01\,\mathrm{m} \). The obstacle extends horizontally over \( 8\,\text{km} \times 8\,\text{km} \), with a vertical dimension equal to \( 1\,\mathrm{m} \).
\end{itemize}
Each obstacle \( \omega_i^* \) is marked as a small black square in Figure \ref{locobsS}, indicating its location within the Mediterranean basin. Particularly, the results corresponding to the detection of obstacle \( \omega_1^* \) are shown in Figure \ref{obs}. More specifically:
\begin{itemize}
    \item Figure \ref{obs}(a) displays the iso-values of the topological gradient \( D_K \), defined in \eqref{gradient-final}, over the entire computational domain. The red regions highlight the negative values of \( D_K \), which are indicative of the potential obstacle locations.
    
    \item A zoomed-in view of the iso-values near the actual position of \( \omega_1^* \) (small black square) is presented in Figure \ref{obs}(b), offering a clearer visualization of the detection precision.

    \item Figure \ref{obs}(c) illustrates a 3D plot of the topological gradient \( D_K \) in the vicinity of the true obstacle location, providing additional insight into the gradient behavior in three dimensions.
\end{itemize}

From the results shown in Figure \ref{obs}, we observe that the proposed algorithm successfully detects the obstacle \( \omega_1^* \) at the location where the topological gradient \( D_K \) reaches its most negative values (see the red region in Figure \ref{obs}(a)). This confirms the effectiveness of our numerical method in identifying \( \omega_1^* \).

\medskip

The detection results for the remaining obstacles \( \omega_2^*, \omega_3^* \), and \( \omega_4^* \) are presented in Figure \ref{Horizz}. As observed, each obstacle \( \omega_i^* \), for \( i = 2, 3, 4 \), is accurately detected in the region where the topological gradient \( D_K \) attains its most negative values. These results, along with those in Figure \ref{obs}, demonstrate that the proposed one-iteration algorithm reliably localizes the unknown obstacles with high precision, irrespective of their spatial location, depth, or the particular sub-region of the Mediterranean Sea in which they are embedded.

\begin{figure}[!htb]
\centering
\includegraphics[width=56mm]{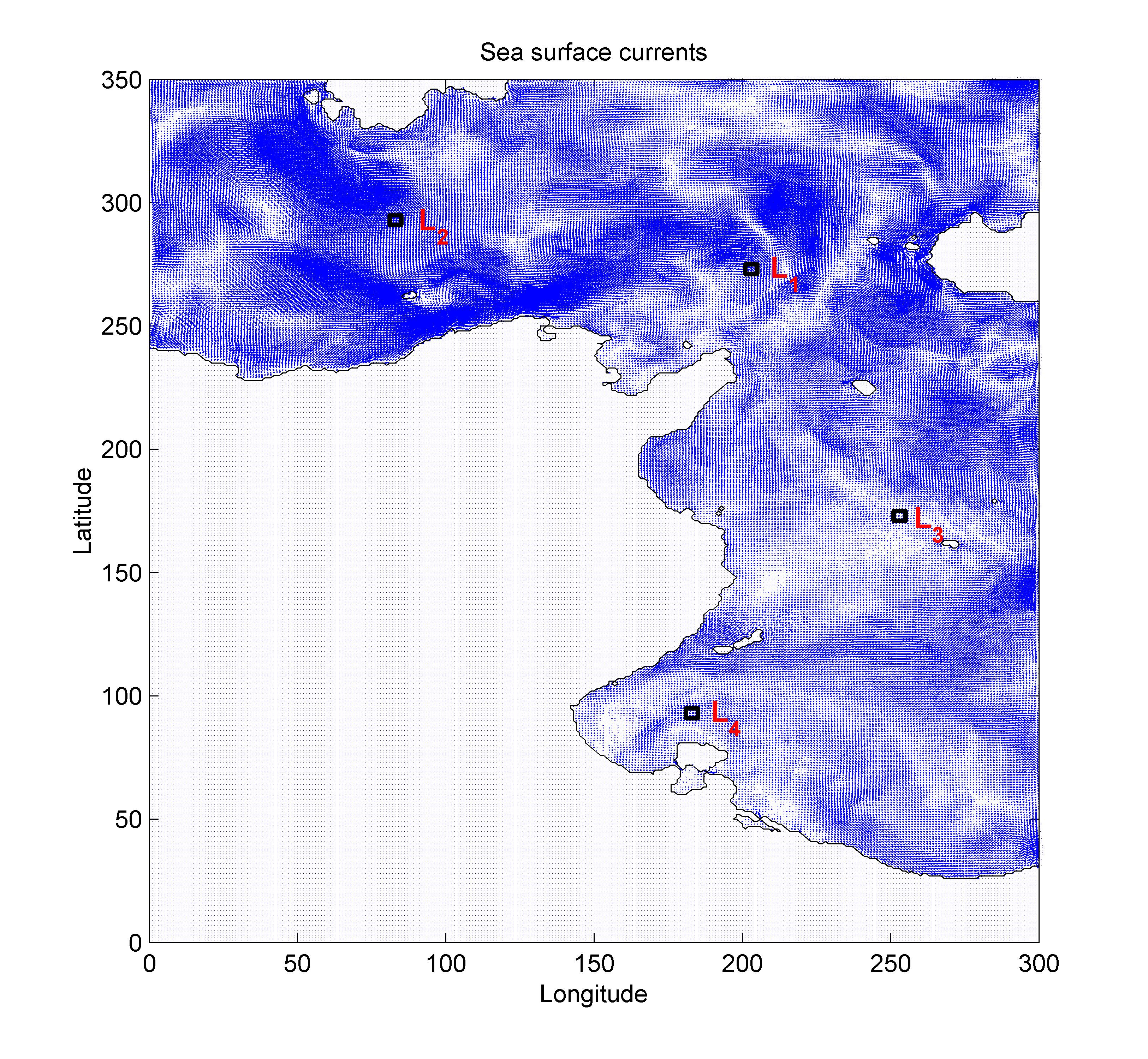}  
\caption{Locations \( \mathrm{L}_i \) of the true obstacles \( \omega_i^* \) (for \( i = 1, 2, 3, 4 \)), represented by four small black squares, to be identified in the presence of velocity flow within the computational domain.}
\label{locobsS}
\end{figure}

\begin{figure}[!htb]
    \centering
    \begin{tabular}{ccc}
        \includegraphics[width=50mm]{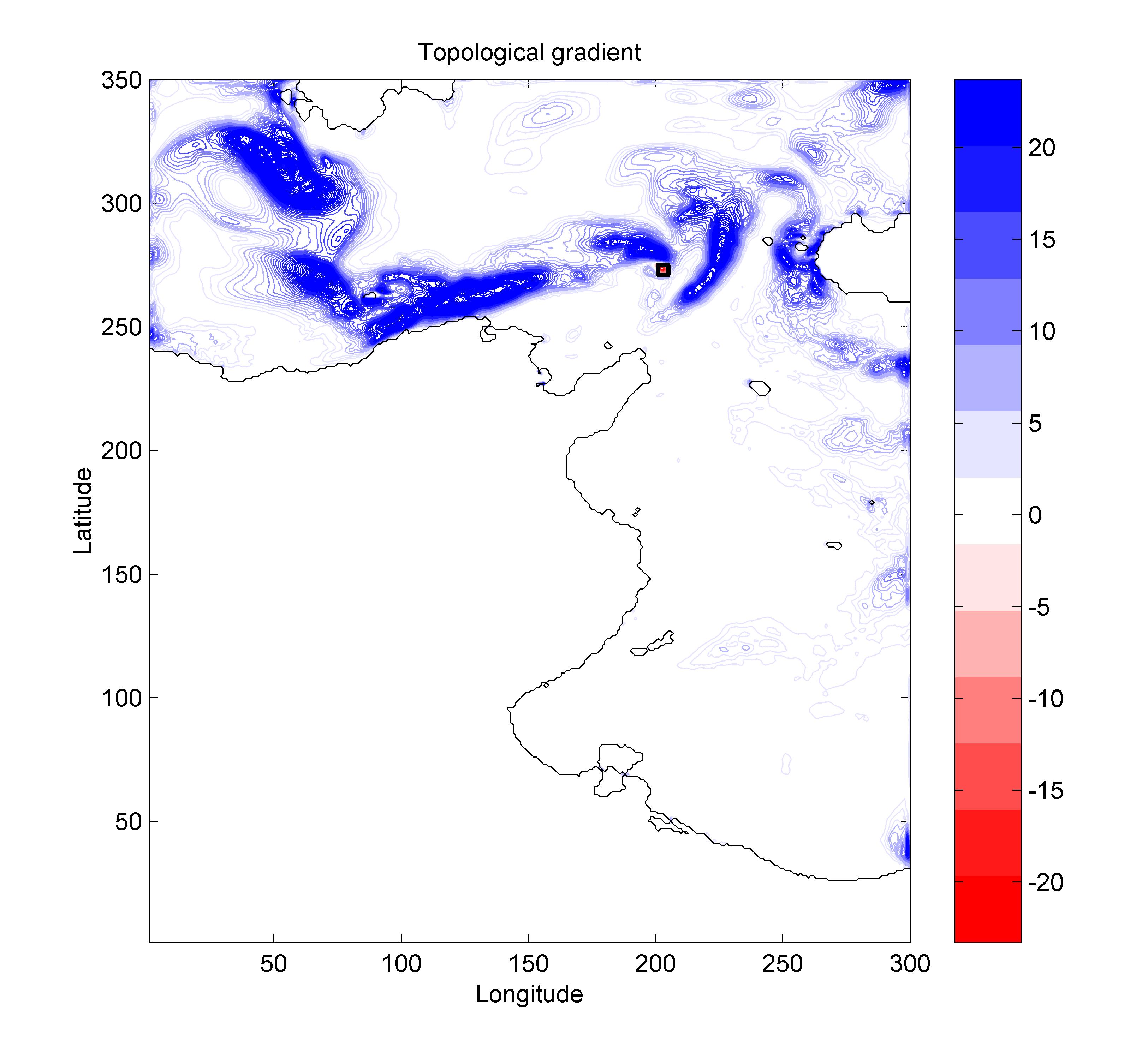} &
        \includegraphics[width=50mm]{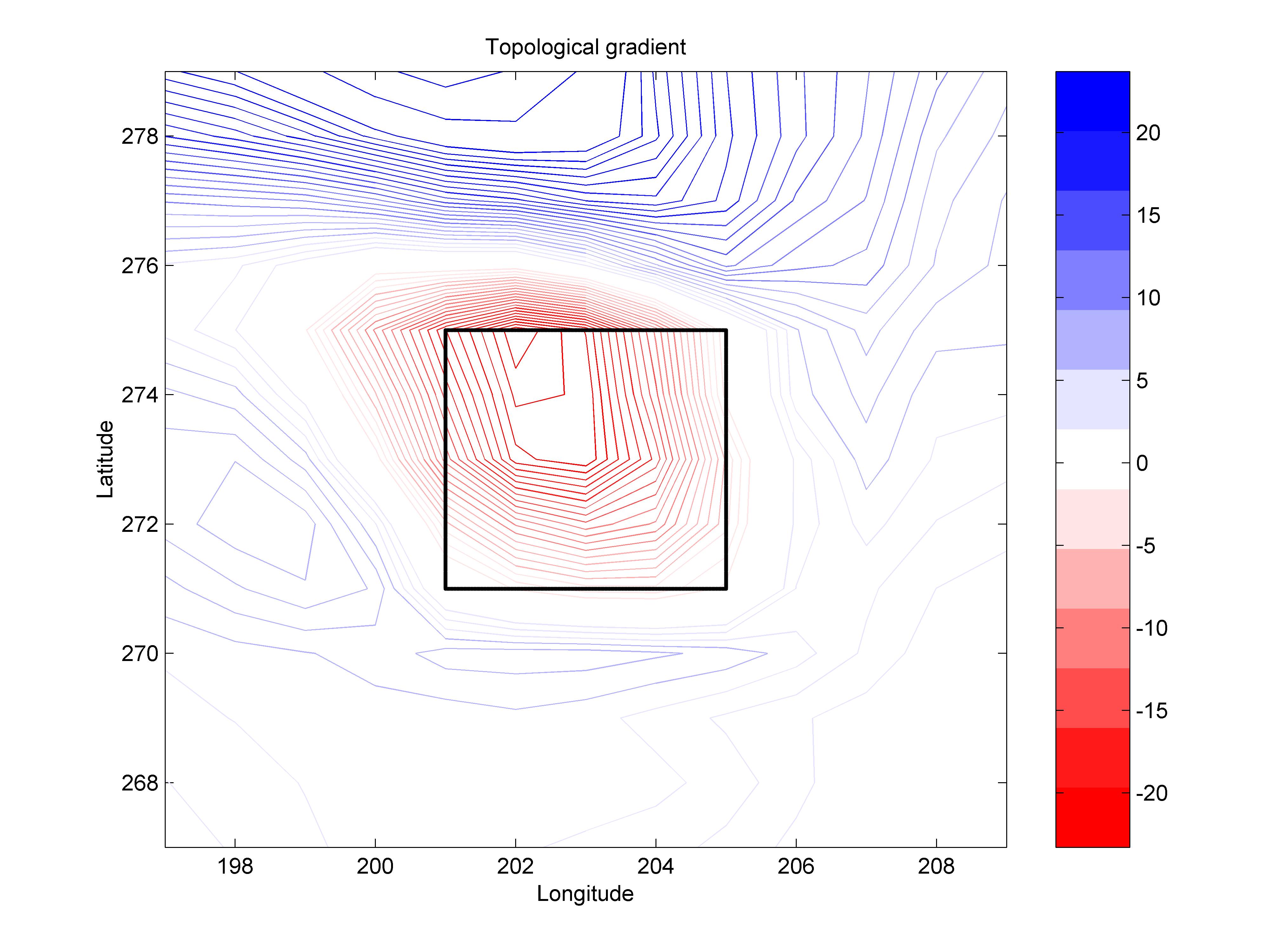} & 
        \includegraphics[width=50mm]{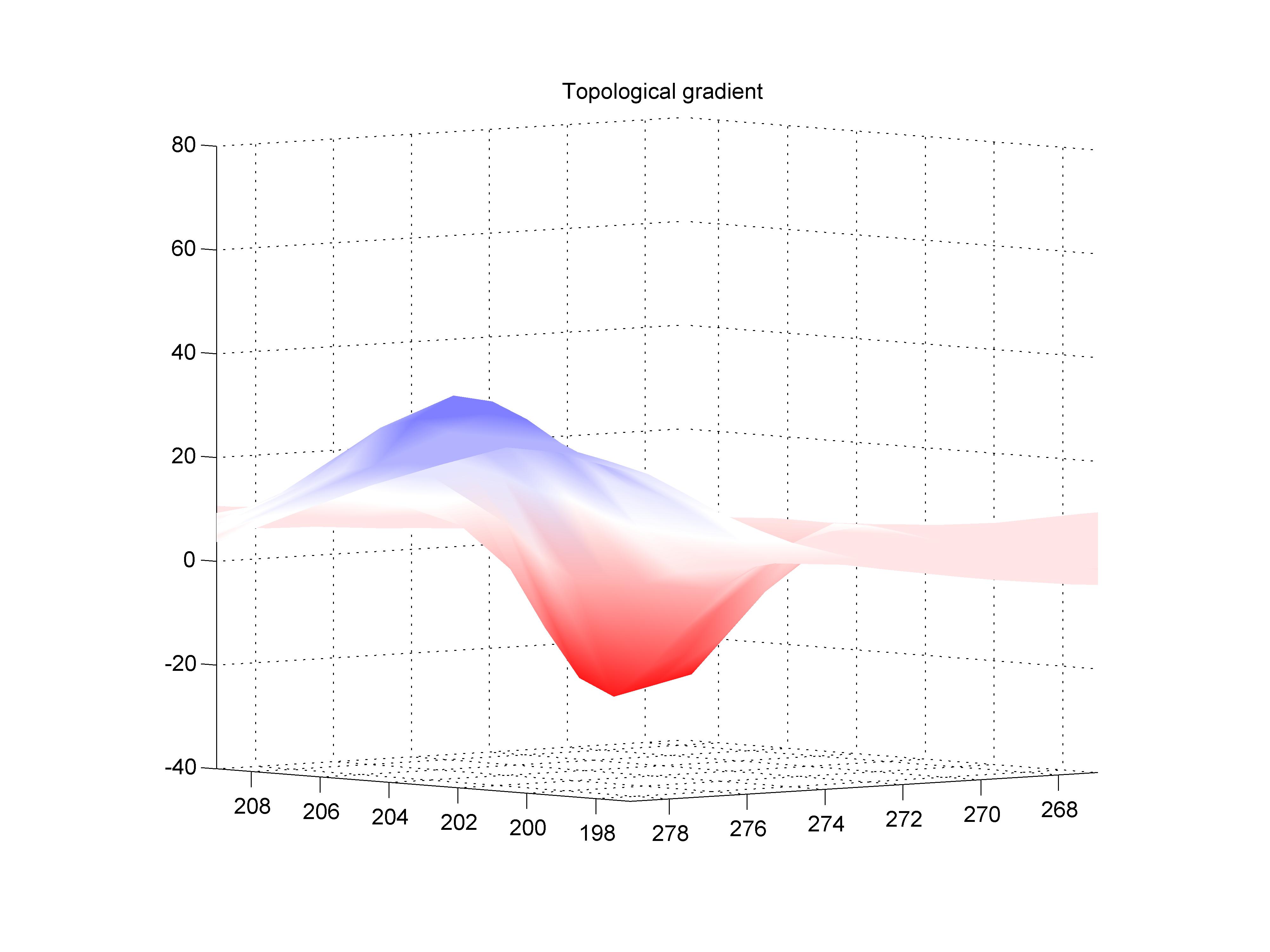} \\
        (a) &
        (b) &
        (c) 
    \end{tabular}
    \caption{Topological gradient-based identification of the obstacle \( \omega_1^* \) located at position \( \mathrm{L}_1 \). 
    (a) Iso-values of the topological gradient \( D_K \) over the domain, 
    (b) Zoom near the exact location of \( \omega_1^* \), highlighting the region of negative values of \( D_K \), and 
    (c) 3D visualization of \( D_K \) illustrating the local minimum corresponding to the obstacle.}
    \label{obs}
\end{figure}

\begin{figure}[!htb]
    \centering
    \begin{tabular}{ccc}
        \includegraphics[width=50mm]{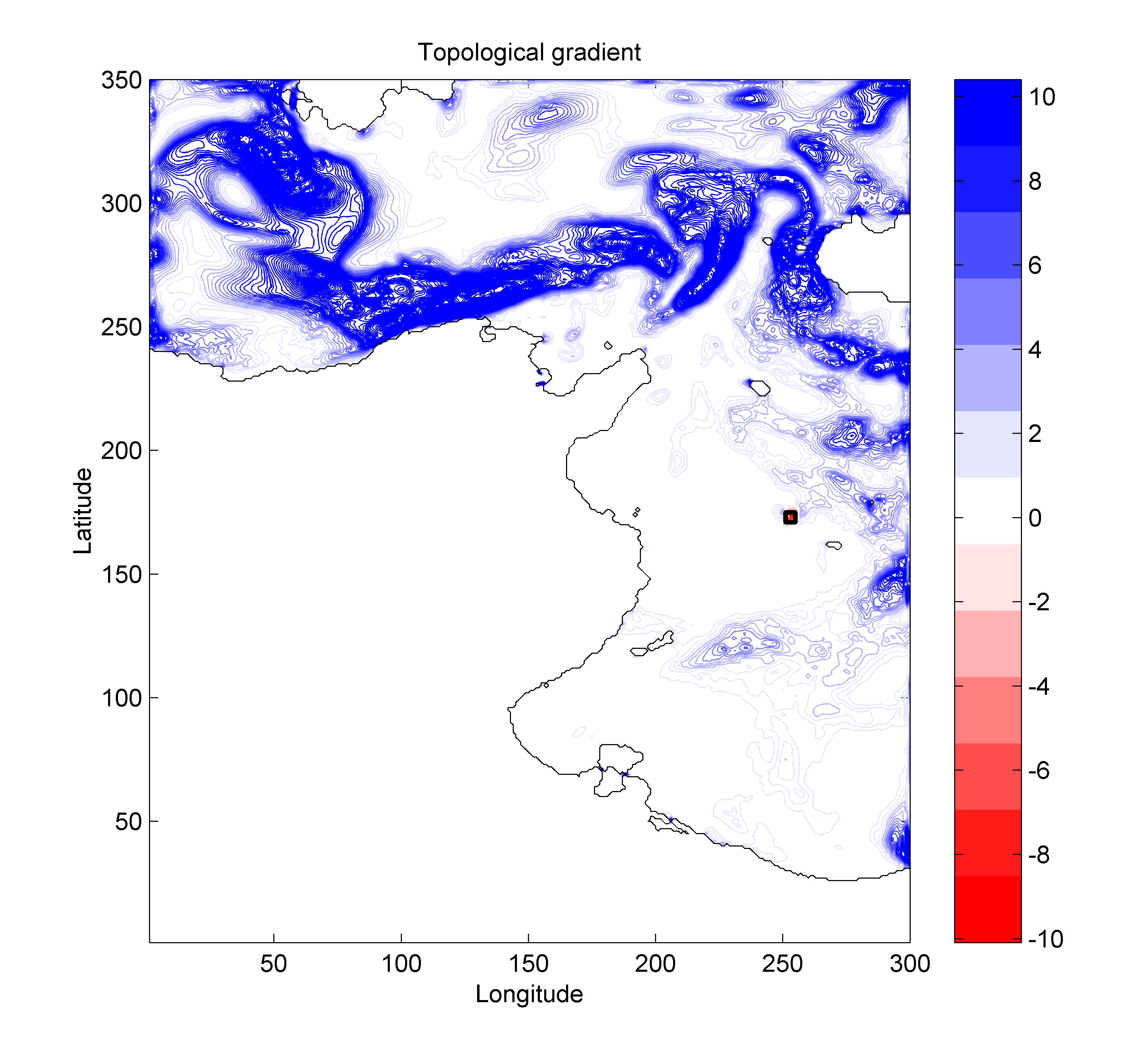} & \includegraphics[width=50mm]{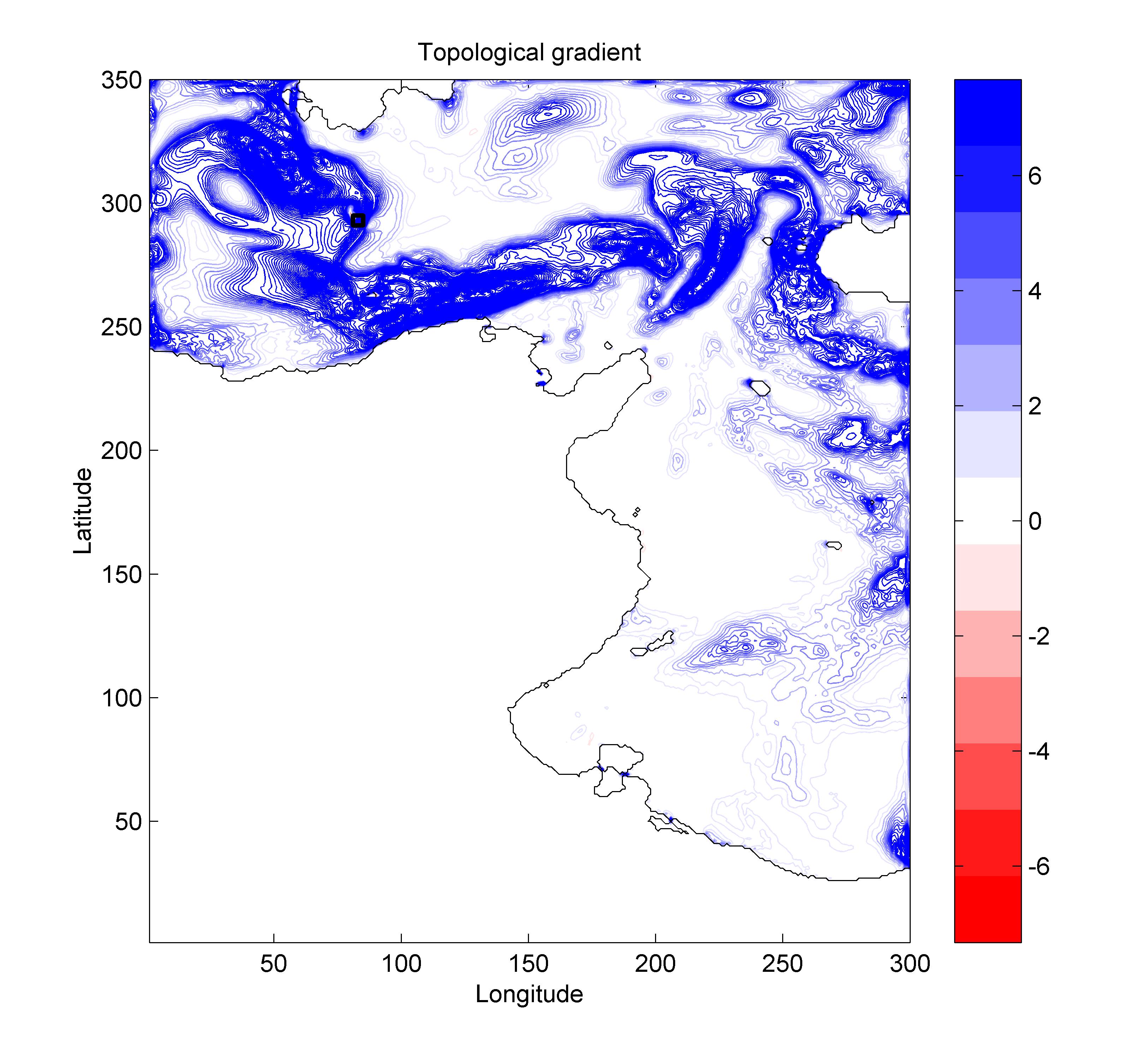} &  \includegraphics[width=50mm]{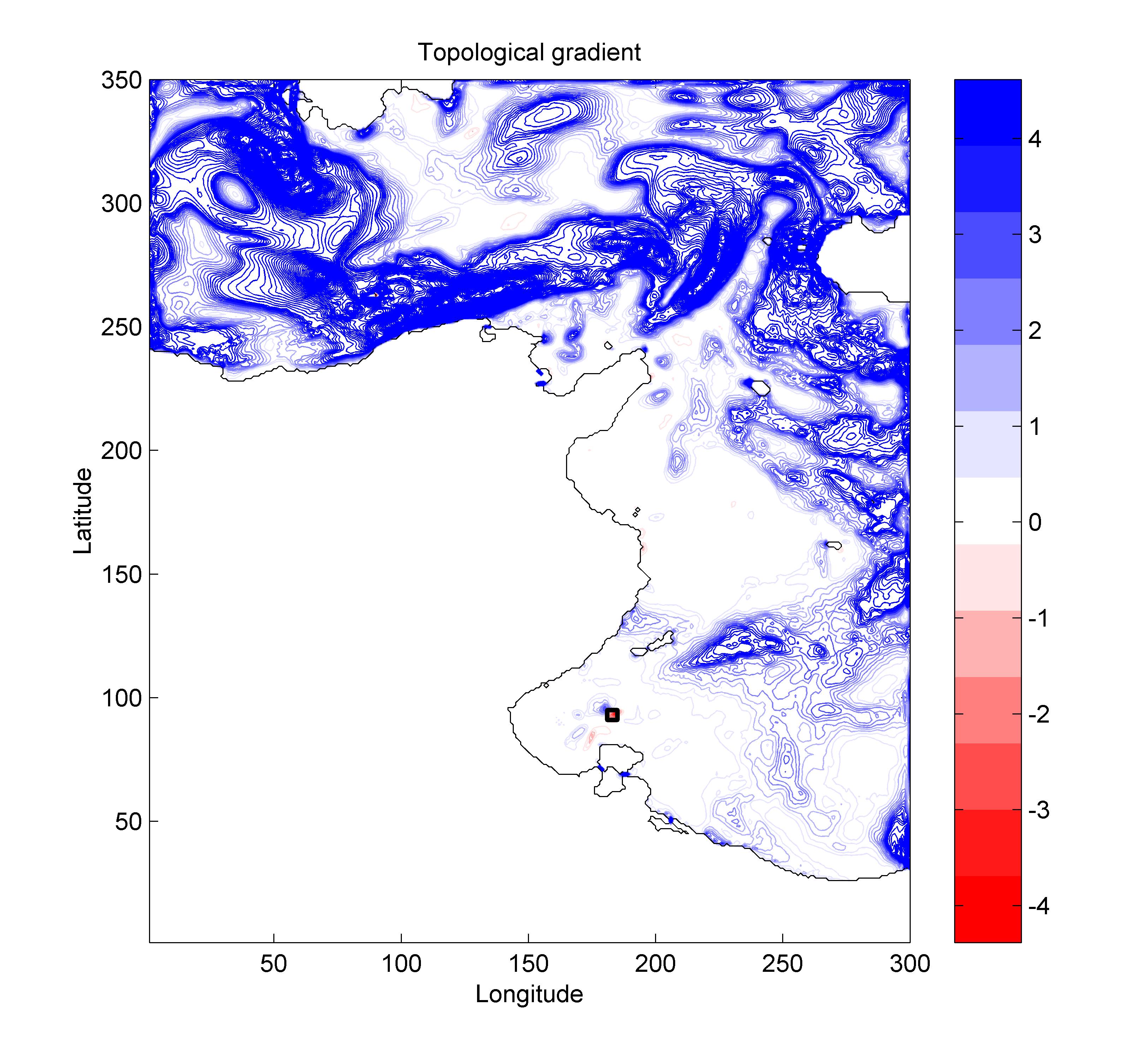}\\
(a) Location: \( \mathrm{L}_2 \)  & (b) Location: \( \mathrm{L}_3 \) & (c) Location: \( \mathrm{L}_4 \)\\
        \includegraphics[width=50mm]{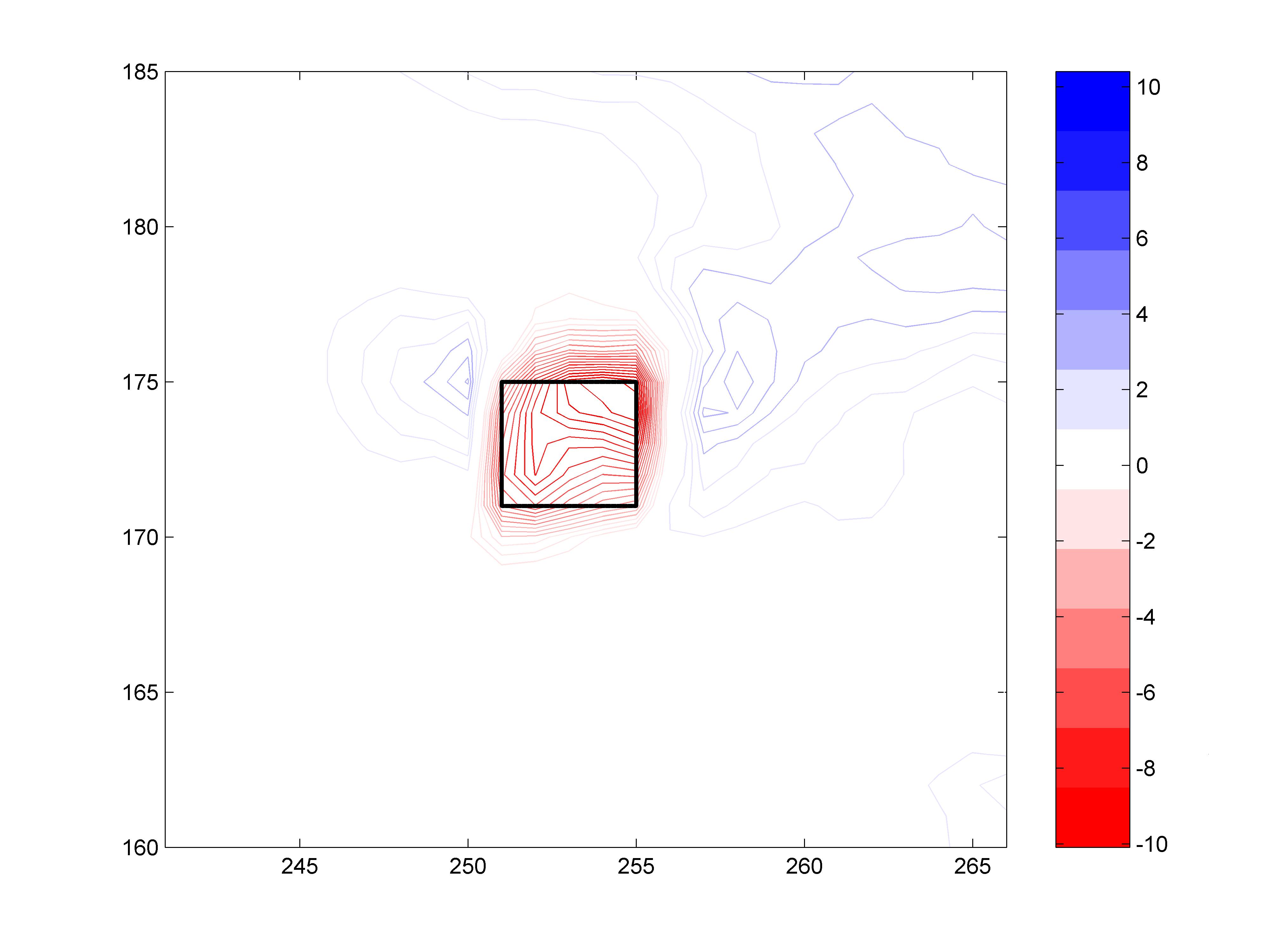} &
        \includegraphics[width=50mm]{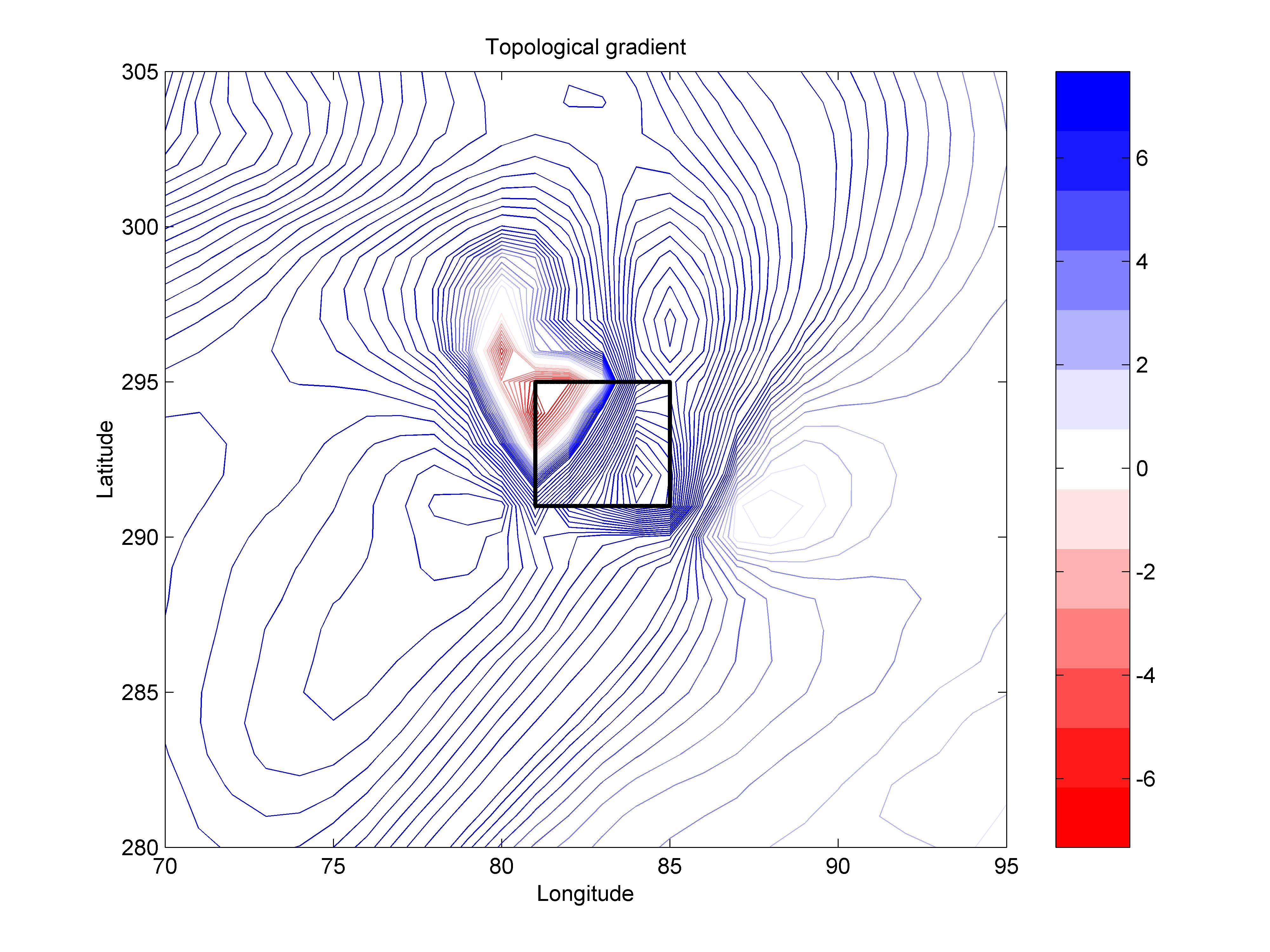} 
        & 
        \includegraphics[width=50mm]{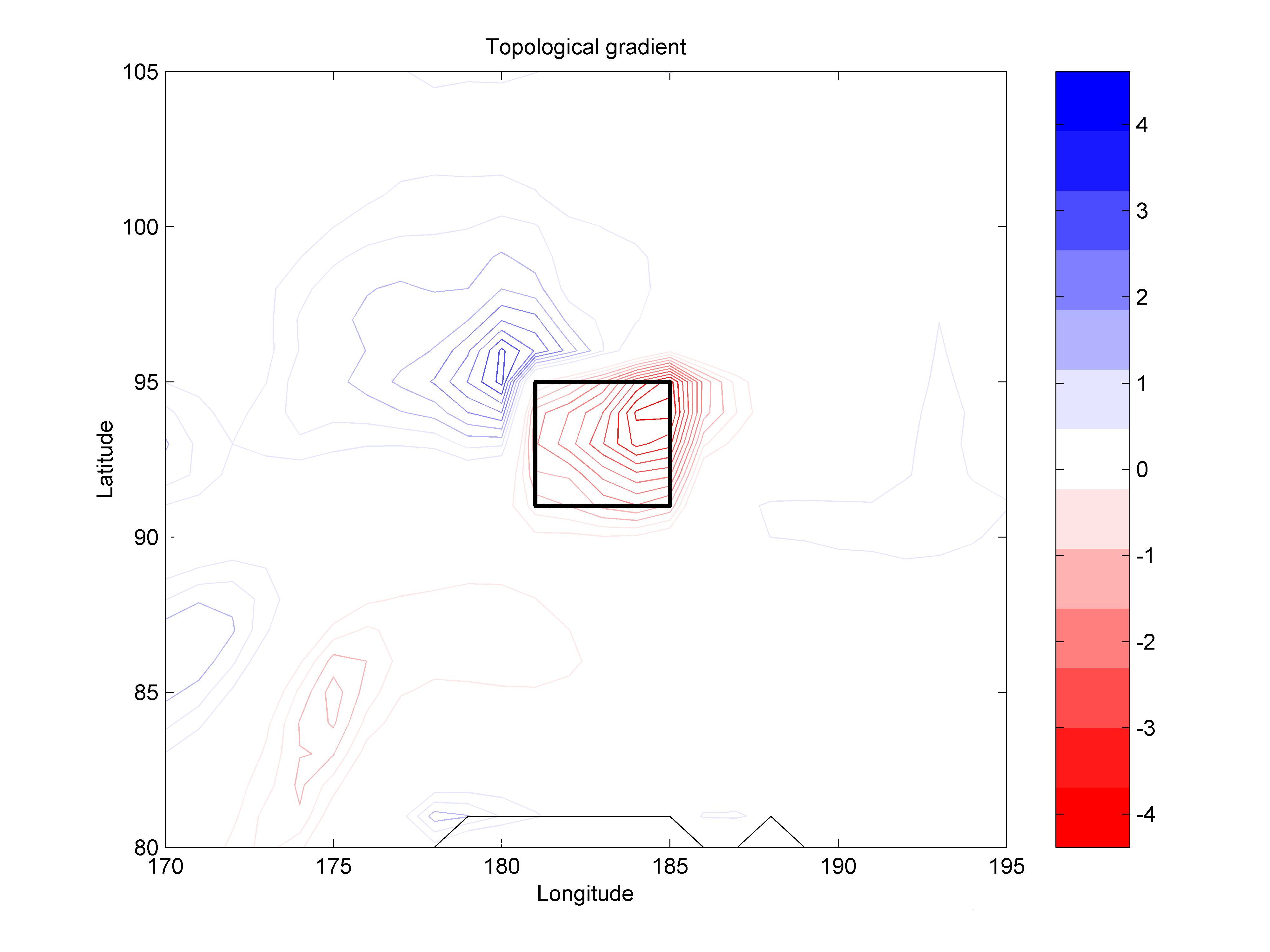} \\
        (d) Zoom of (a) & (e) Zoom of (b) & (f) Zoom of (c)
    \end{tabular}
    \caption{Topological gradient-based identification of obstacles \( \omega_i^* \) for \( i = 2, 3, 4 \) located in distinct sub-regions of the Mediterranean Sea. 
Subfigures (a), (b), and (c) display the iso-values of the topological gradient \( D_K \) at the corresponding locations \( \mathrm{L}_2 \), \( \mathrm{L}_3 \), and \( \mathrm{L}_4 \), respectively. 
Subfigures (d), (e), and (f) provide zoomed-in views around the true positions of the obstacles, illustrating the regions where \( D_K \) attains its most negative values, which indicate potential obstacle locations.}
    \label{Horizz}
\end{figure}

\begin{remark}
In the following examples, the horizontal dimensions of the obstacle are consistently fixed at \( 8\,\text{km} \times 8\,\text{km} \), unless stated otherwise. However, the vertical extent is not uniform, as it is influenced by the sigma-coordinate system, which adapts to the underlying bathymetry. For this reason, the vertical size will be computed and reported only for selected cases.
\end{remark}

%%%%%%%%%%%%%%%%%%%%%%%%%%%%%%%%%%%%%%%%%%%%%%%%%%%%%%%%%%%%%%%%%%%%
\subsubsection{Example 2: Detection at varying depth levels}
%%%%%%%%%%%%%%%%%%%%%%%%%%%%%%%%%%%%%%%%%%%%%%%%%%%%%%%%%%%%%%%%%%%%

In the first example, the obstacle was located relatively close to the sea surface. In the present example, we investigate how the performance of the proposed identification method is affected when the obstacle is positioned progressively deeper in the ocean, moving away from the surface toward the seabed.

This experiment aims to evaluate the effectiveness and robustness of the topological derivative-based algorithm in detecting a single submerged obstacle \( \omega^* \) situated at various depths within the water column. The goal is to assess the method's accuracy and sensitivity across a broad vertical range--from near-surface conditions to deep-sea environments. To this end, we consider three configurations in which the obstacle \( \omega^* \) is placed at average depths of approximately \( 6\,\mathrm{m} \), \( 260\,\mathrm{m} \), and \( 930\,\mathrm{m} \) below the sea surface (see Figure~\ref{locobsVert}).
\begin{figure}[!htb]
    \centering
    \includegraphics[width=60mm]{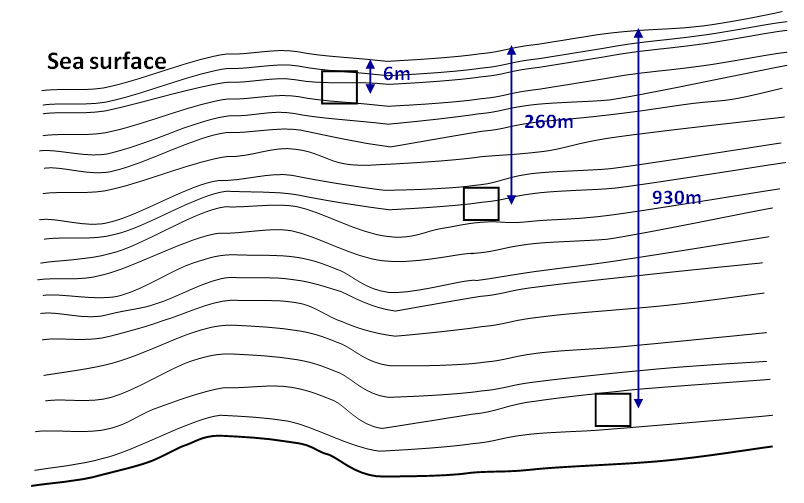}
    \caption{Depth Levels of the obstacle $\omega^*$ below the sea surface.}
    \label{locobsVert}
\end{figure}

\begin{figure}[!htb]
    \centering
    \begin{tabular}{ccc}
        \includegraphics[width=50mm]{Figures/6_obsS.jpg} & 
        \includegraphics[width=50mm]{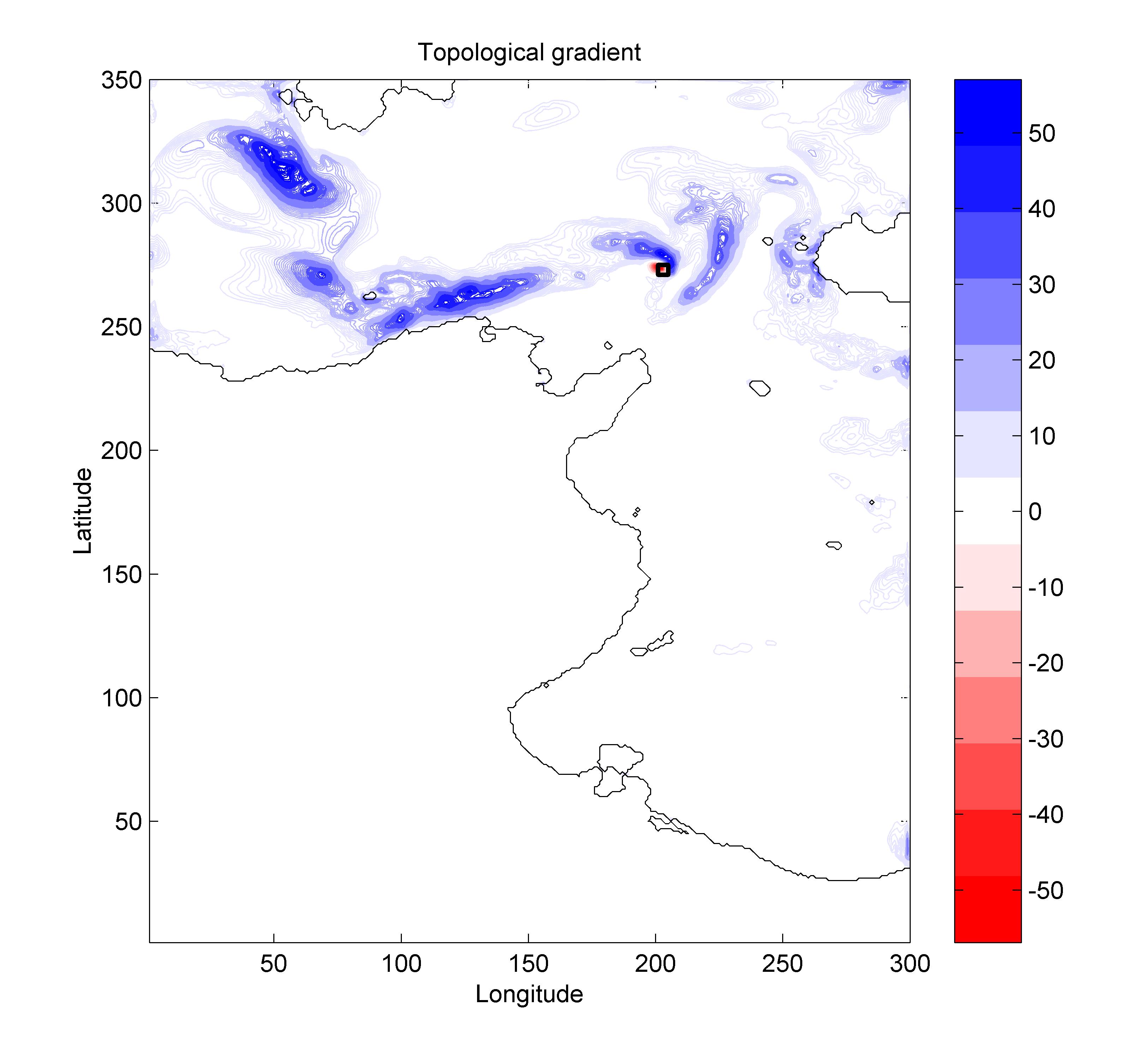} &  
        \includegraphics[width=50mm]{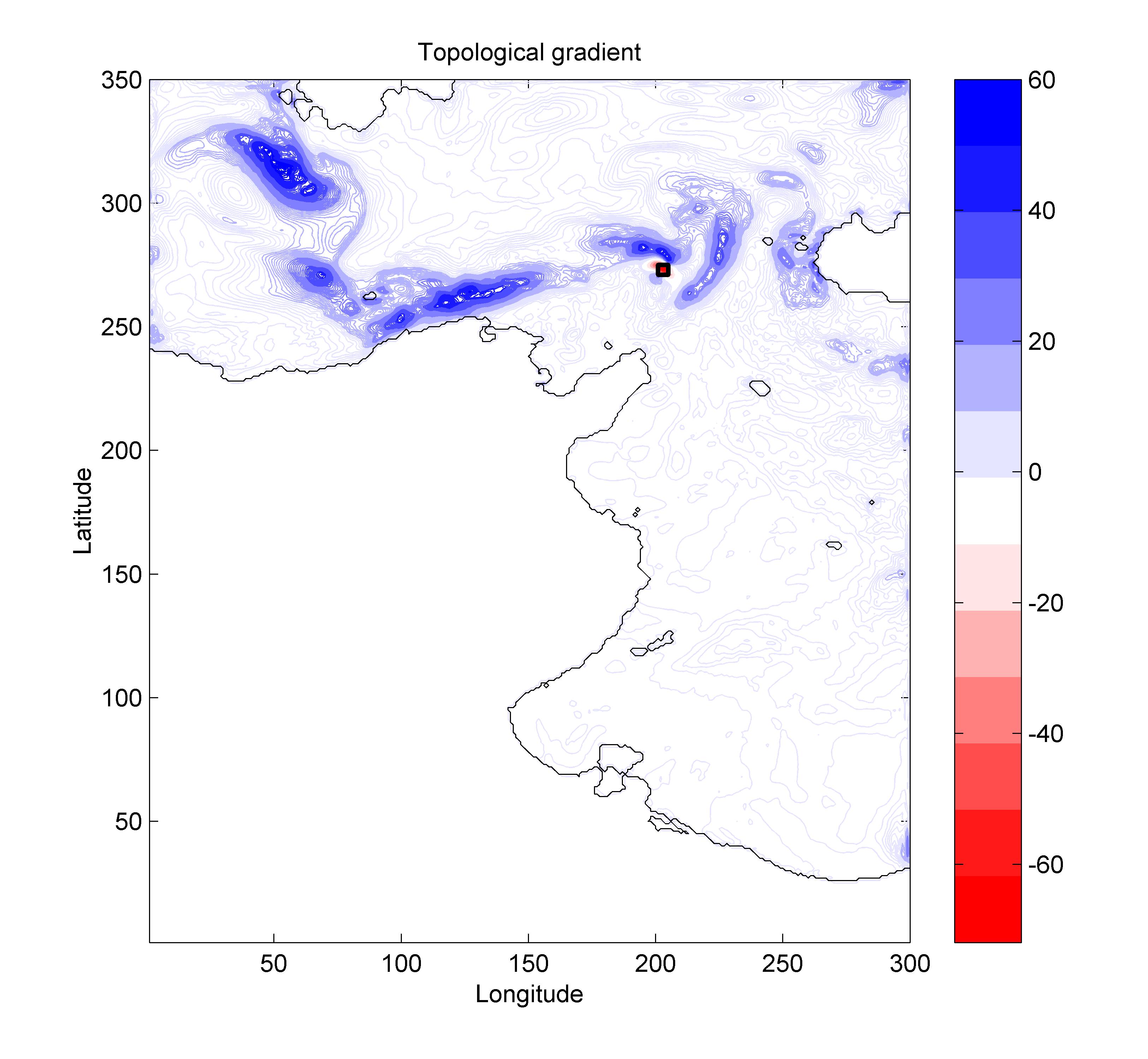} \\
        (a) Depth: \(6\,\mathrm{m}\) & (b) Depth: \(260\,\mathrm{m}\) & (c) Depth: \(930\,\mathrm{m}\) \\
        \includegraphics[width=50mm]{Figures/6_obsSz.jpg} &
        \includegraphics[width=50mm]{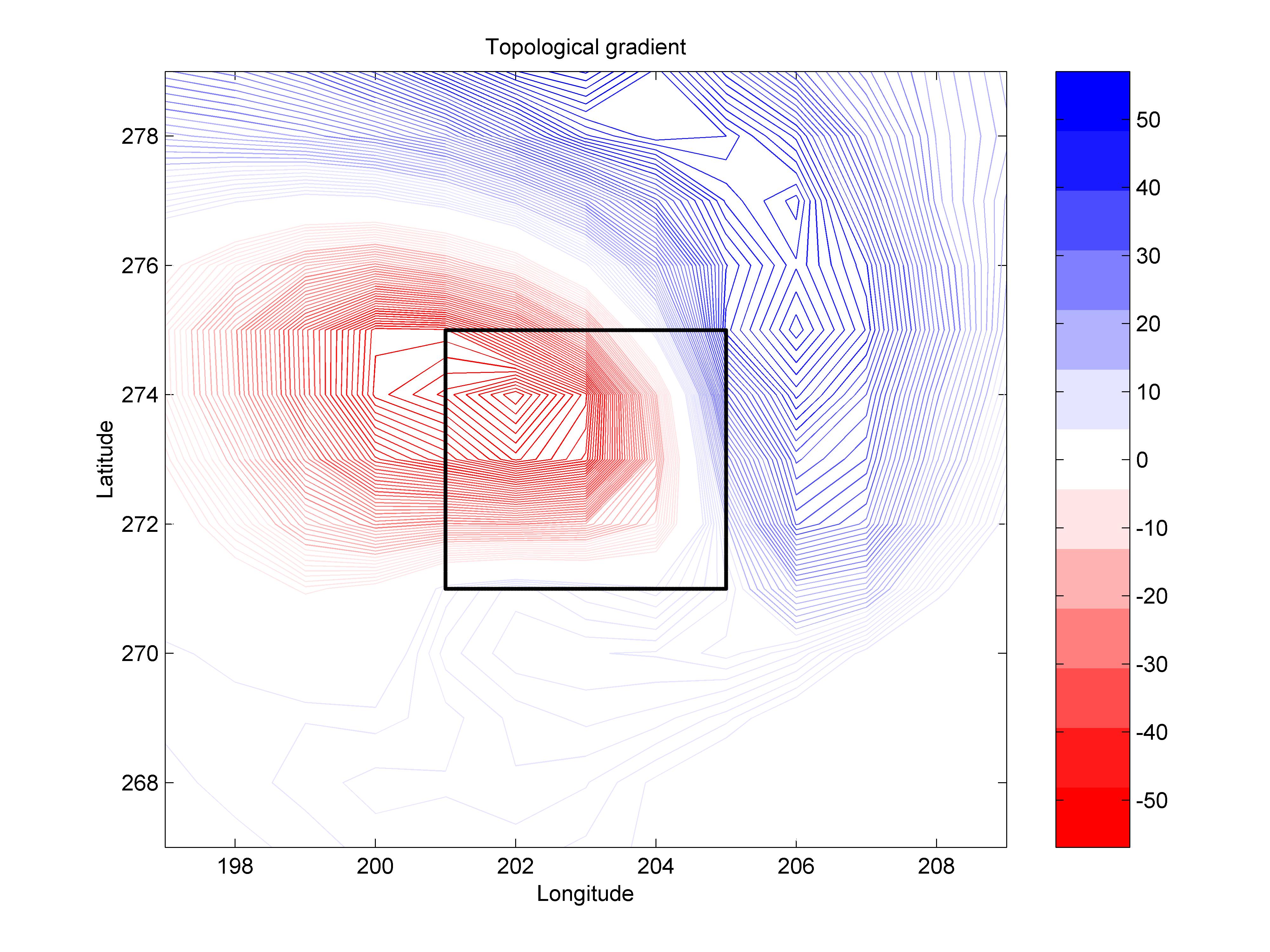} & 
        \includegraphics[width=50mm]{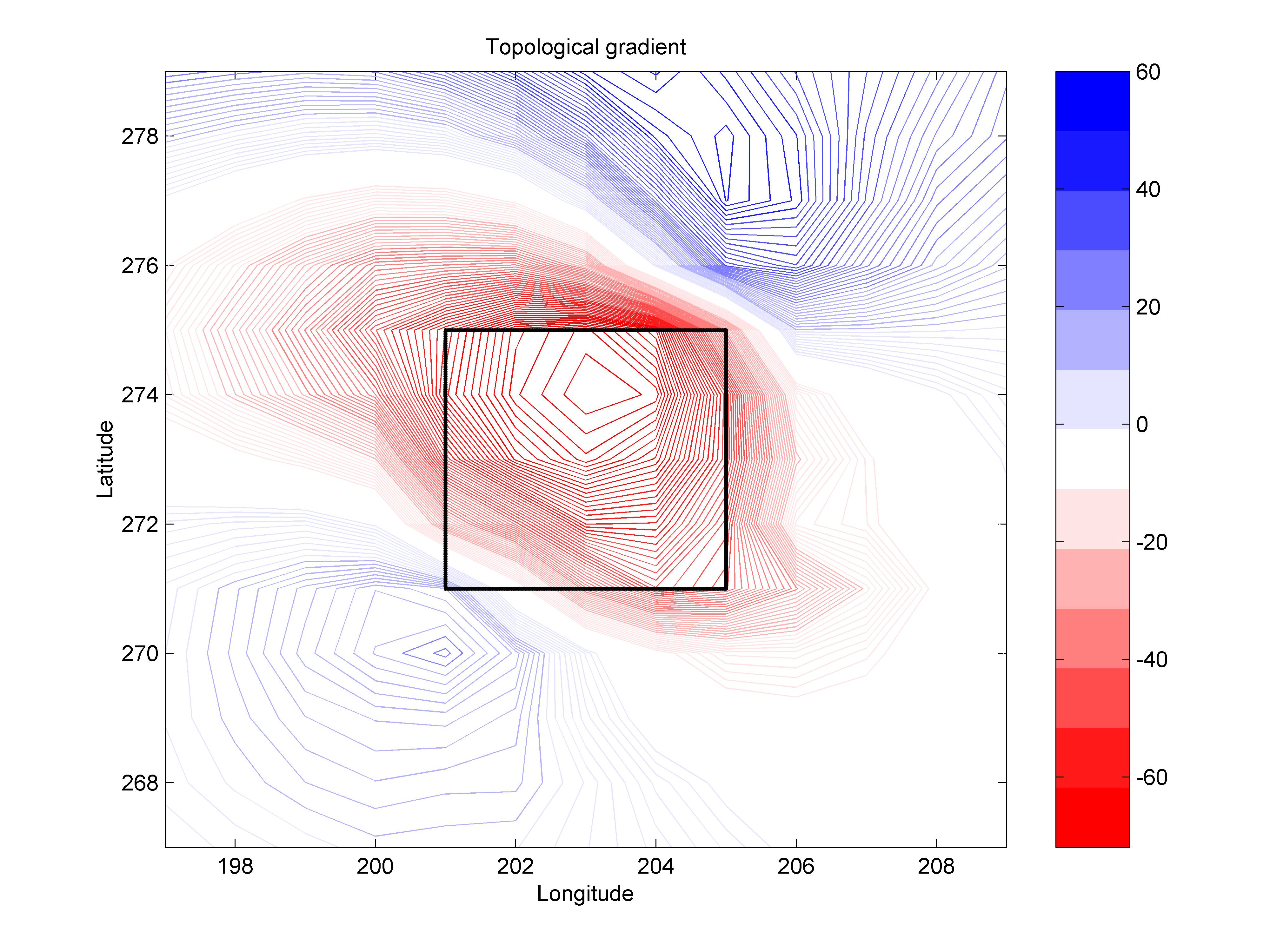} \\
        (d) Zoom of (a) & (e) Zoom of (b) & (f) Zoom of (c)
    \end{tabular}
    \caption{
        Topological gradient-based identification of a single obstacle \( \omega^* \) at three different depths: near-surface (6 m), mid-depth (260 m), and deep-water (930 m). Subfigures (a), (b), and (c) display the iso-values of \( D_K \) for each depth configuration. Subfigures (d), (e), and (f) present corresponding zoomed-in views near the true obstacle locations, highlighting the regions where \( D_K \) is most negative, indicating likely positions of the inclusion.
    }
    \label{SMF-1}
\end{figure}

The corresponding detection results are displayed in Figure \ref{SMF-1}. These numerical experiments demonstrate that the proposed method accurately identifies the location of the obstacle in all three configurations. In each case, the obstacle is successfully localized within the region where the topological gradient \( D_K \) attains its most negative values—highlighting the sensitivity of the method to the presence of inclusions.

The reconstruction results confirm that the topological sensitivity-based approach remains both effective and stable, even as the obstacle is positioned at increasing depths. This result underscores the robustness and versatility of the algorithm across different vertical layers of the ocean, from near-surface to deep-sea conditions.

%%%%%%%%%%%%%%%%%%%%%%%%%%%%%%%%%%%%%%%%%%%%%%%%%%%%%%%%%%%%%%%%%%%%%%%%%%%%%
\subsubsection{Example 3 : Sensitivity to the length of the obstacle}\label{size}
%%%%%%%%%%%%%%%%%%%%%%%%%%%%%%%%%%%%%%%%%%%%%%%%%%%%%%%%%%%%%%%%%%%%%%%%%%%%%

In this numerical experiment, we evaluate the sensitivity and robustness of the proposed one-shot algorithm with respect to variations in the \textbf{horizontal extent} of the submerged obstacle. Specifically, we consider three configurations of an inclusion with progressively decreasing horizontal dimensions, as illustrated in Figure~\ref{locobs-lengths1}, while maintaining an approximately constant vertical location.

\begin{figure}[!htb]
    \centering
    
    \begin{tabular}{ccc}
        \includegraphics[width=15mm]{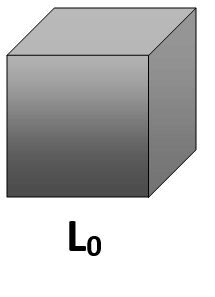} & 
        \includegraphics[width=15mm]{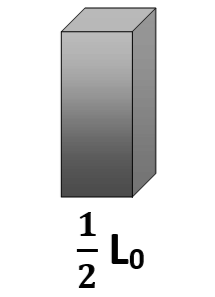} & 
        \includegraphics[width=15mm]{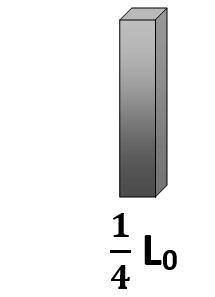}
    \end{tabular}
    \caption{The considered obstacles and their relative lengths.}
    \label{locobs-lengths1}
\end{figure}

\begin{figure}[!htb]
    \centering
    \begin{tabular}{ccc}
        \includegraphics[width=50mm]{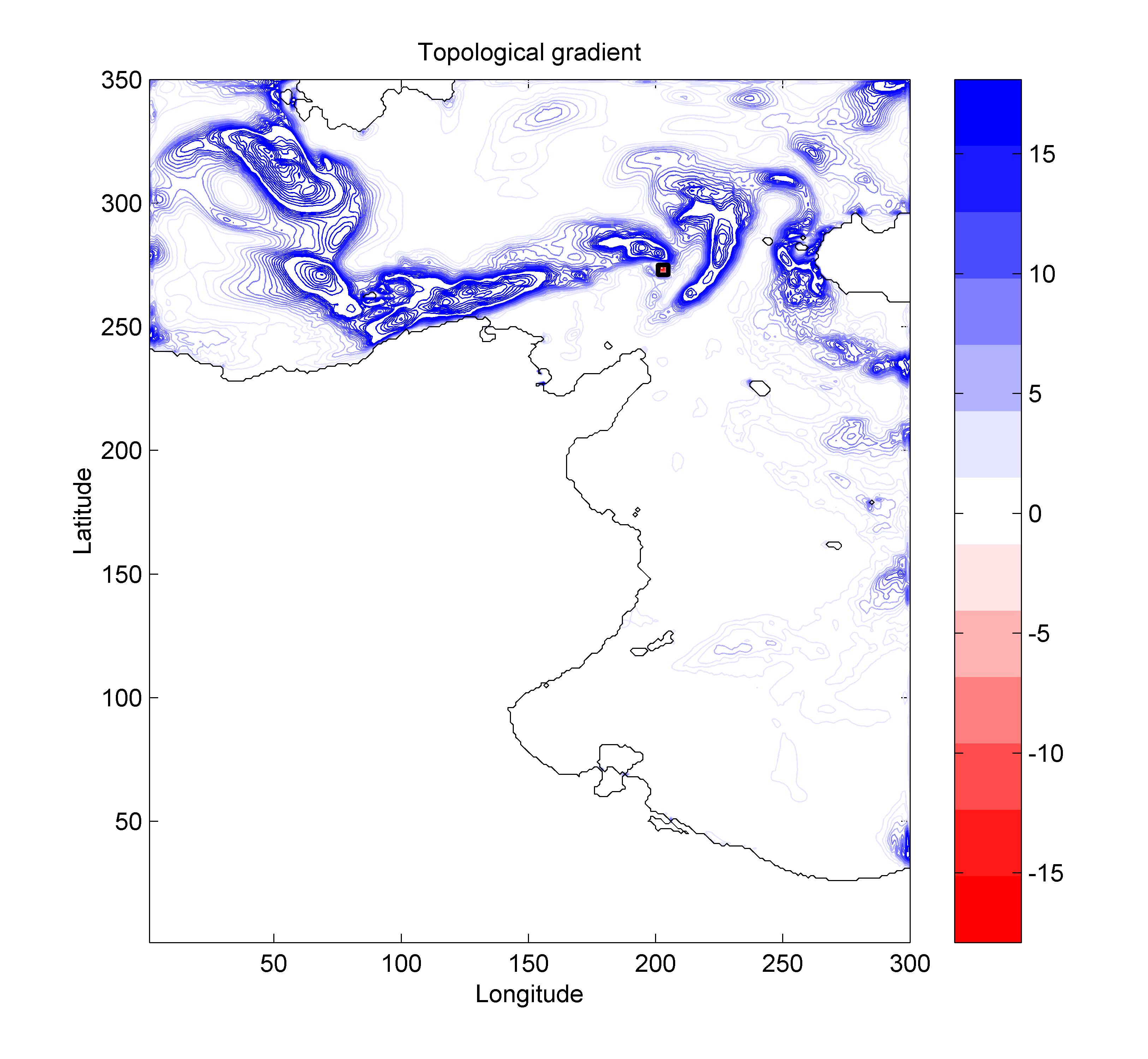} & 
        \includegraphics[width=50mm]{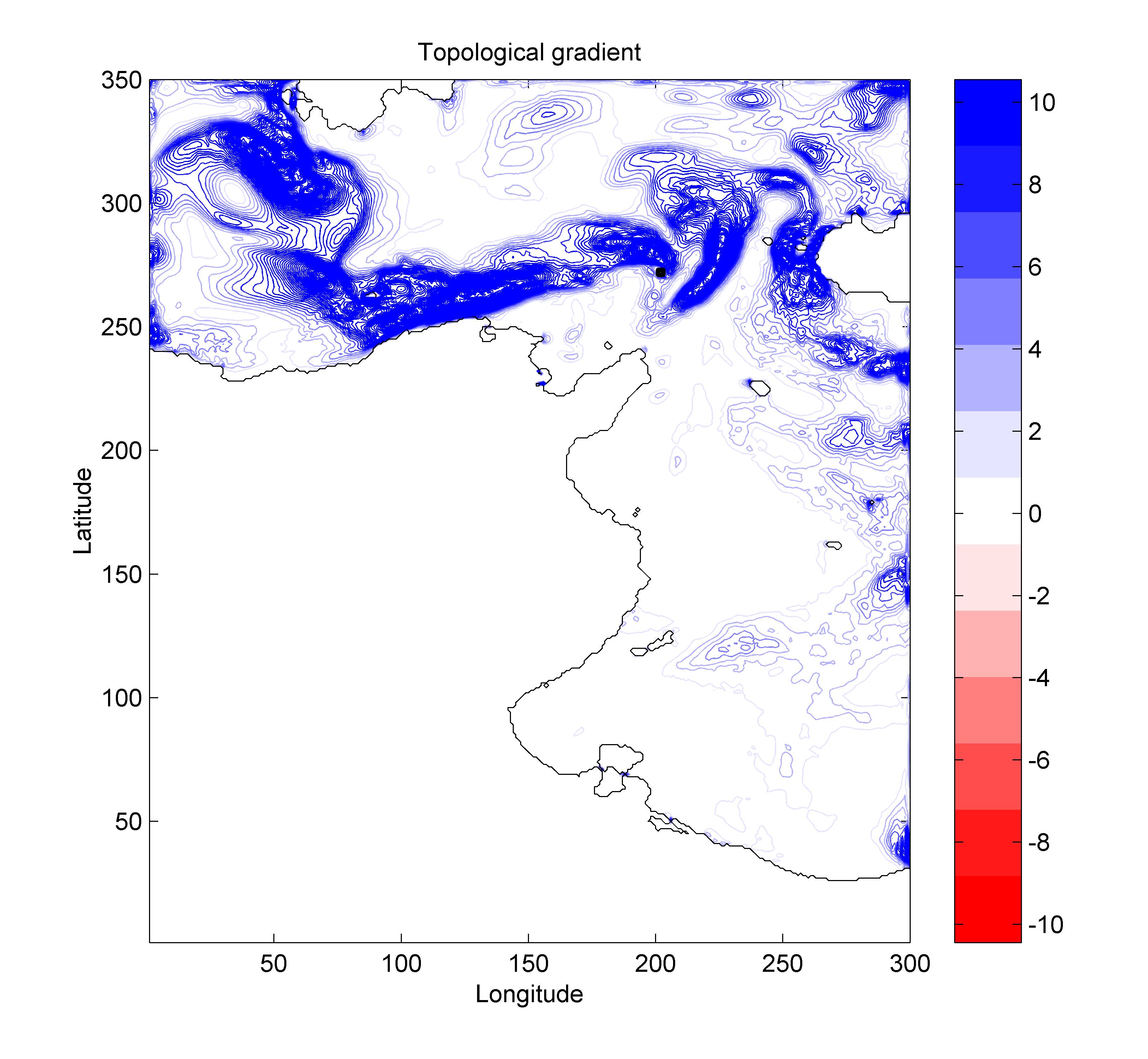} &  
        \includegraphics[width=50mm]{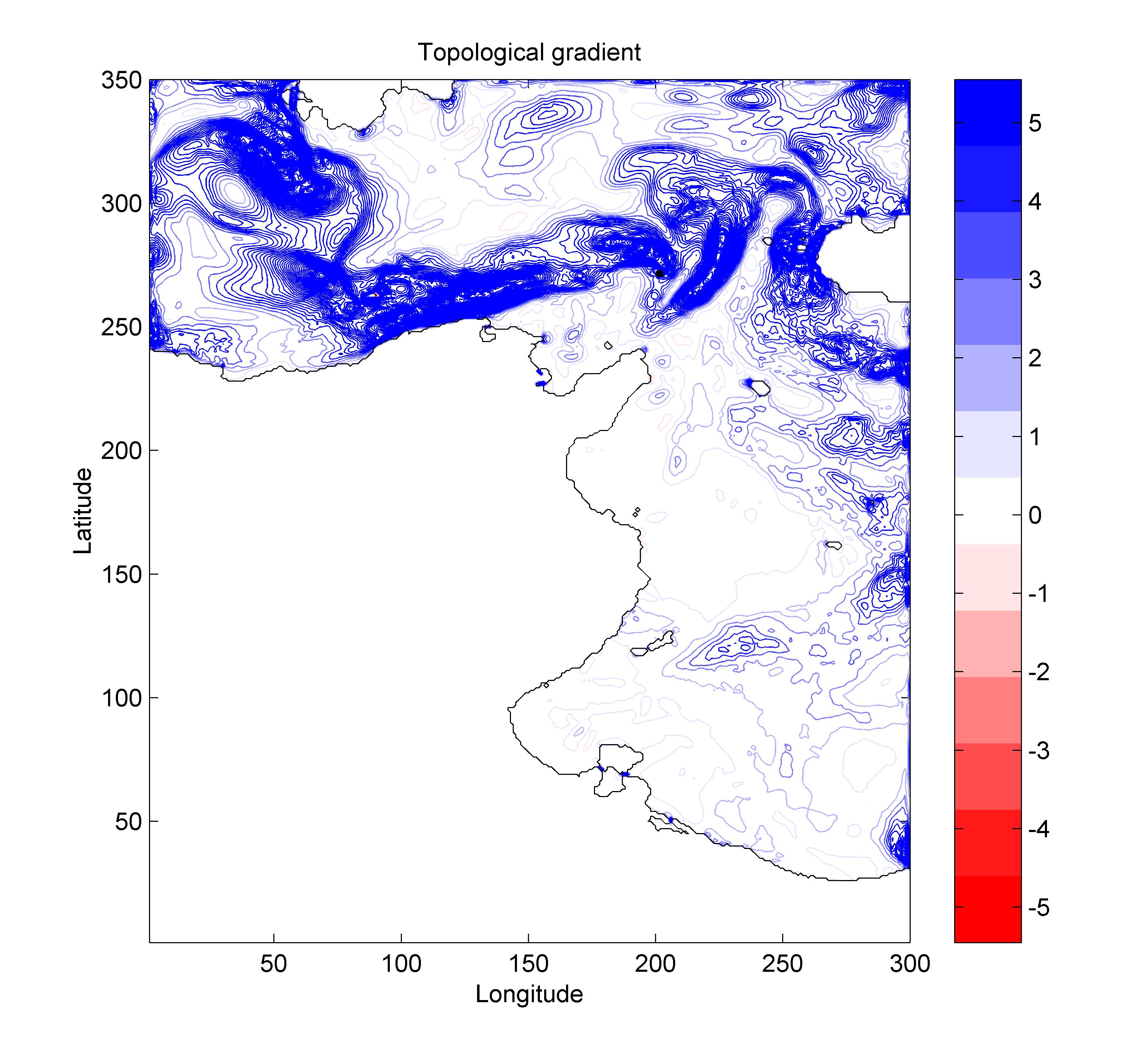} \\
        (a) Length: $\mathrm{L}_0$ & (b) Length: $\displaystyle\frac{1}{2}\mathrm{L}_0$ & (c) Length: $\displaystyle\frac{1}{4}\mathrm{L}_0$\\
        \includegraphics[width=50mm]{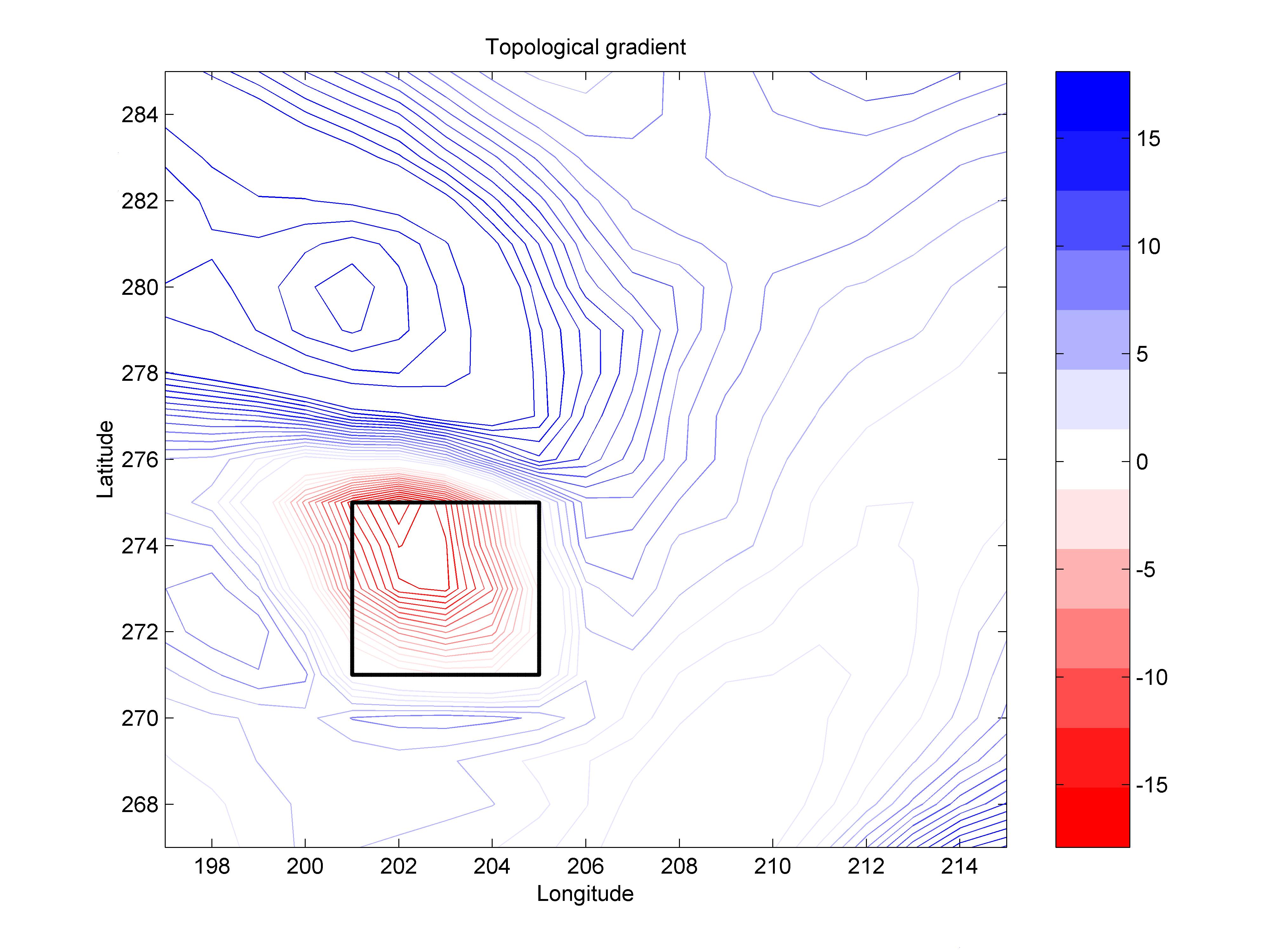} &
        \includegraphics[width=50mm]{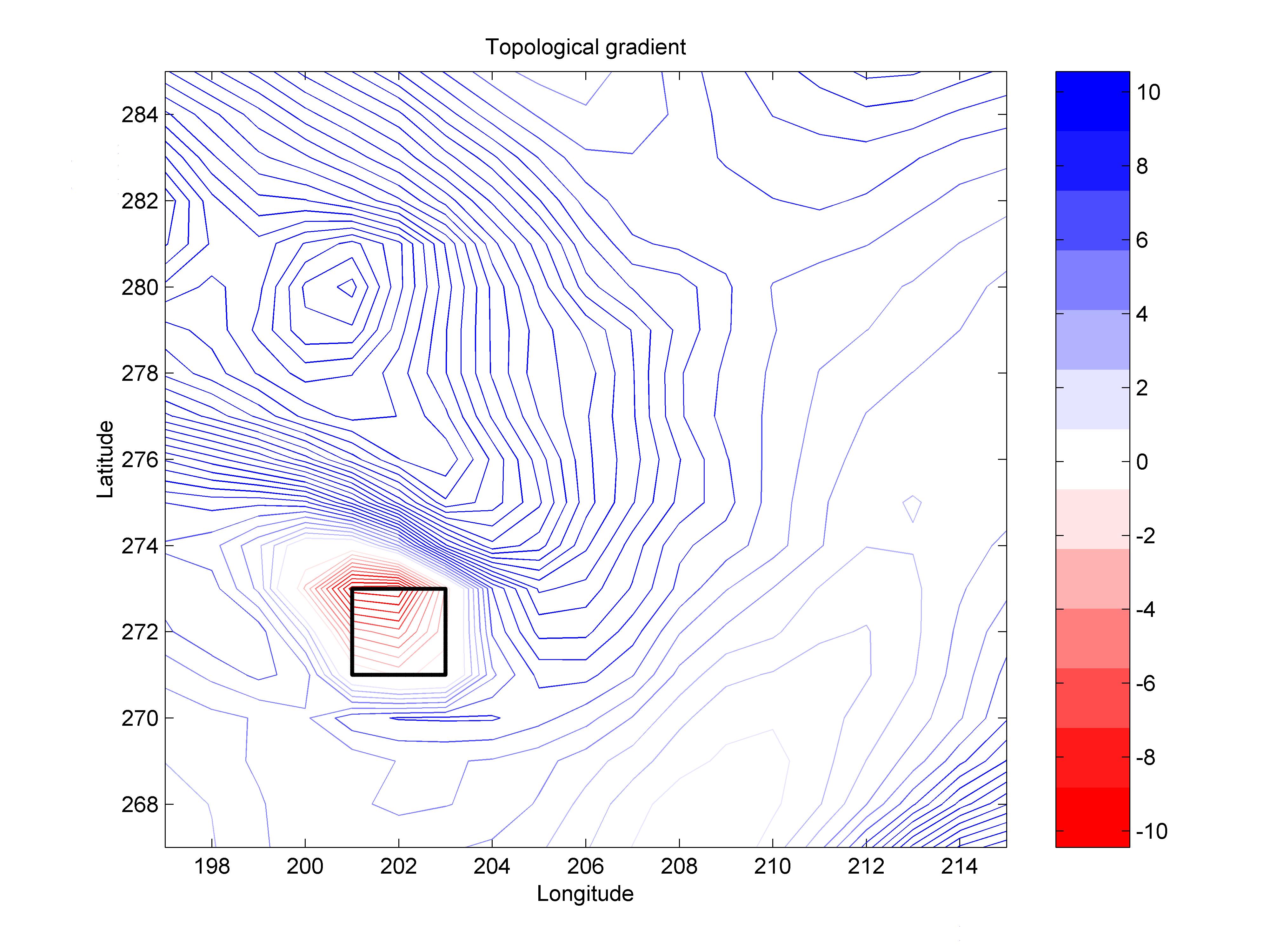} & 
        \includegraphics[width=50mm]{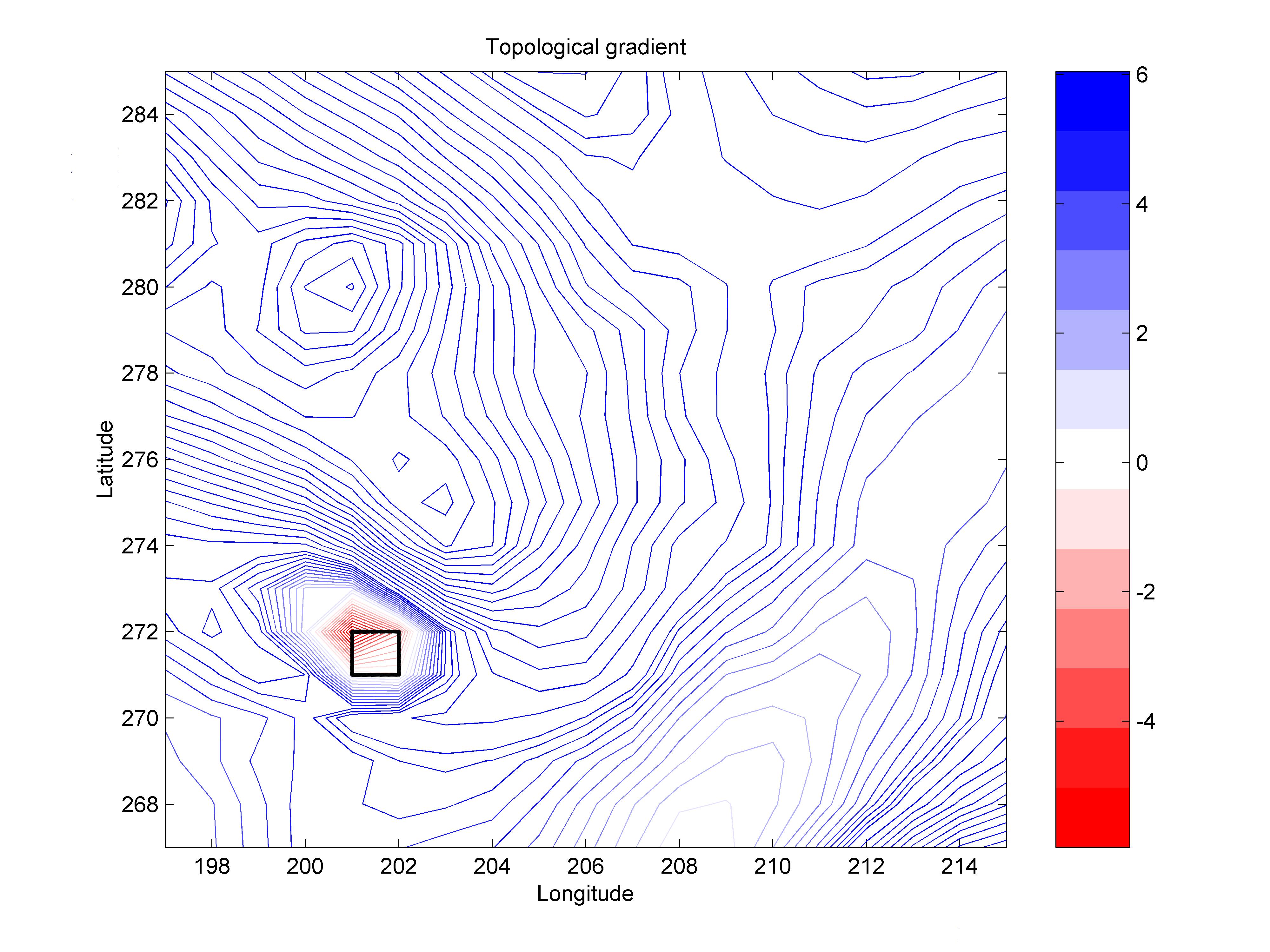} \\
        (d) Zoom of (a) & (e) Zoom of (b) & (f) Zoom of (c)
    \end{tabular}
    \caption{
        Topological gradient-based identification of obstacles with varying horizontal lengths. Subfigures (a), (b), and (c) illustrate the iso-value contours of $D_K$ for obstacles with lengths $\mathrm{L}_0=8\,\text{km}\times 8\,\text{km}$, $\frac{1}{2}\mathrm{L}_0=4\,\text{km}\times 4\,\text{km}$, and $\frac{1}{4}\mathrm{L}_0=2\,\text{km}\times 2\,\text{km}$, respectively. Subfigures (d), (e), and (f) provide magnified views near the actual obstacle locations, emphasizing regions where $D_K$ attains its most negative values, which correspond to the most probable positions of the obstacles.
    }
\label{sizeV-22}
\end{figure}

The number of vertical sigma layers is held fixed across all configurations; however, their physical thickness in meters may slightly vary due to bathymetric fluctuations. The vertical height of the first obstacle ranges from approximately $4.3\,\text{m}$ to $8\,\text{m}$, the second from $4.3\,\text{m}$ to $6\,\text{m}$, and the third from $4.3\,\text{m}$ to $5\,\text{m}$.

The aim is to examine how changes in the geometric scale of the inclusion affect the accuracy, sharpness, and stability of the reconstruction. This analysis provides insight into the algorithm’s capability to detect smaller-scale features and its potential limitations when applied to obstacles of diminishing size.

Figure \ref{sizeV-22} presents the corresponding detection results for each case. From these results, we observe that the proposed algorithm successfully localizes the obstacles in all configurations, with the topological gradient correctly highlighting the regions of negative sensitivity associated with the true obstacle locations. The reconstructions remain stable and precise across varying obstacle lengths, confirming that the algorithm is robust to geometric scaling and capable of accurately identifying inclusions of different horizontal extents.

%%%%%%%%%%%%%%%%%%%%%%%%%%%%%%%%%%%%%%%%%%%%%%%%%%%%%%%%%%%%%%%%%%%%%%%%%%%%%
\subsubsection{Example 4: Sensitivity to the height of the obstacle} \label{size}
%%%%%%%%%%%%%%%%%%%%%%%%%%%%%%%%%%%%%%%%%%%%%%%%%%%%%%%%%%%%%%%%%%%%%%%%%%%%%

Similar to Example 3, this experiment aims to evaluate the sensitivity of the proposed detection algorithm with respect to the ``vertical extent'' (i.e., height) of the submerged obstacle. In this test, we consider three obstacles that share approximately the same average depth, but have different vertical heights—ranging from full height \( \mathrm{H}_0 \), to half height \(\displaystyle \frac{1}{2}\,\mathrm{H}_0 \), and quarter height \( \displaystyle\frac{1}{4}\,\mathrm{H}_0 \). These configurations are depicted in Figure \ref{locobs-lengths}. The first obstacle has 4 sigma layers, the second 2 sigma layers, and the third a single sigma layer. All three obstacles are positioned at approximately the same depth level of 6\,m from the sea surface. The purpose of this test is to examine whether variations in obstacle height affect the accuracy and stability of the reconstruction.

The corresponding detection results are shown in Figure \ref{sizeV-2}. As illustrated, the proposed one-iteration algorithm remains robust and accurately localizes each obstacle regardless of its vertical scale. In all cases, the obstacle is clearly identified in the region where the topological gradient \( D_K \) attains its most negative values, confirming the method’s effectiveness even when the obstacle height becomes small.

%%%%%%%%%%%%%%%%%%%%%%%%%%%%%%%%%%%%%%%%%%%%%%%%%%%%%%%%%%%%%%%%%%%%%%%%
\subsubsection{Example 5: Identification of multiple obstacles}
%%%%%%%%%%%%%%%%%%%%%%%%%%%%%%%%%%%%%%%%%%%%%%%%%%%%%%%%%%%%%%%%%%%%%%%%

In Examples 1--4, we focused on the detection of a single submerged obstacle. However, it is well established in the literature that the numerical computation of the topological derivative is inherently independent of the number of obstacles or cavities present within the fluid domain (see, for instance, \cite{BenAbdaSIAM2009,CaubetIPI2016,CaubetIP2012}). This robustness stems from the fact that the topological gradient is computed pointwise and reflects the local sensitivity of the cost functional to the introduction of an infinitesimal inclusion at each point in the domain. As a result, the gradient field is naturally capable of capturing multiple anomalies simultaneously, without requiring any a priori knowledge of their number or positions.

To illustrate this key property, we now present a numerical experiment involving the simultaneous presence of multiple submerged obstacles. The aim is to evaluate the capability of the proposed one-shot reconstruction algorithm to detect all obstacles with high accuracy, regardless of their number, depth, or spatial configuration.

In this setting, we consider three obstacles, denoted by \( \omega_1^* \), \( \omega_2^* \), and \( \omega_3^* \), whose exact locations are illustrated in Figure \ref{exact-locations-3obs33}. Each obstacle has a horizontal extent of approximately \( 8\,\mathrm{km} \times 8\,\mathrm{km} \), and all are positioned at similar vertical levels (between layers 8 and 10). Specifically:
\begin{itemize}
    \item \( \omega_1^* \): located at grid indices \((201{:}205) \times (271{:}275) \times (8{:}10)\), corresponding to an average depth of approximately \(260\,\mathrm{m}\);
    \item \( \omega_2^* \): located at \((216{:}220) \times (271{:}275) \times (8{:}10)\), at an average depth of about \(255\,\mathrm{m}\);
    \item \( \omega_3^* \): located at \((191{:}195) \times (281{:}285) \times (8{:}10)\), with an average depth near \(207\,\mathrm{m}\).
\end{itemize}

The corresponding identification results are shown in Figure \ref{obs45}. As observed, the topological gradient clearly highlights three distinct negative zones corresponding to the true obstacle locations. These results confirm the robustness and efficiency of the proposed topological sensitivity-based method, which successfully detects multiple obstacles simultaneously with no prior information about their quantity or placement.

\begin{figure}[!htb]
    \centering
    \begin{tabular}{ccc}
        \includegraphics[width=15mm]{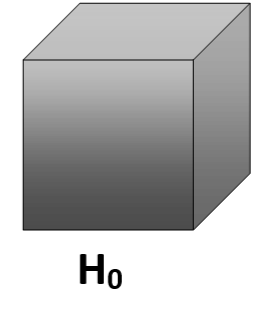} & 
        \includegraphics[width=15mm]{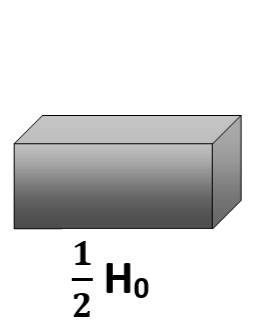} & 
        \includegraphics[width=15mm]{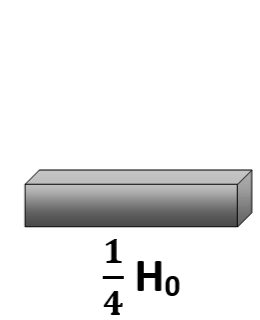}
    \end{tabular}
    \caption{The considered obstacles and their relative vertical heights: \( \mathrm{H}_0\), \( \frac{1}{2}\,\mathrm{H}_0 \), and \( \frac{1}{4}\,\mathrm{H}_0 \).}
    \label{locobs-lengths}
\end{figure}

\begin{figure}[!htb]
    \centering
    \begin{tabular}{ccc}
        \includegraphics[width=50mm]{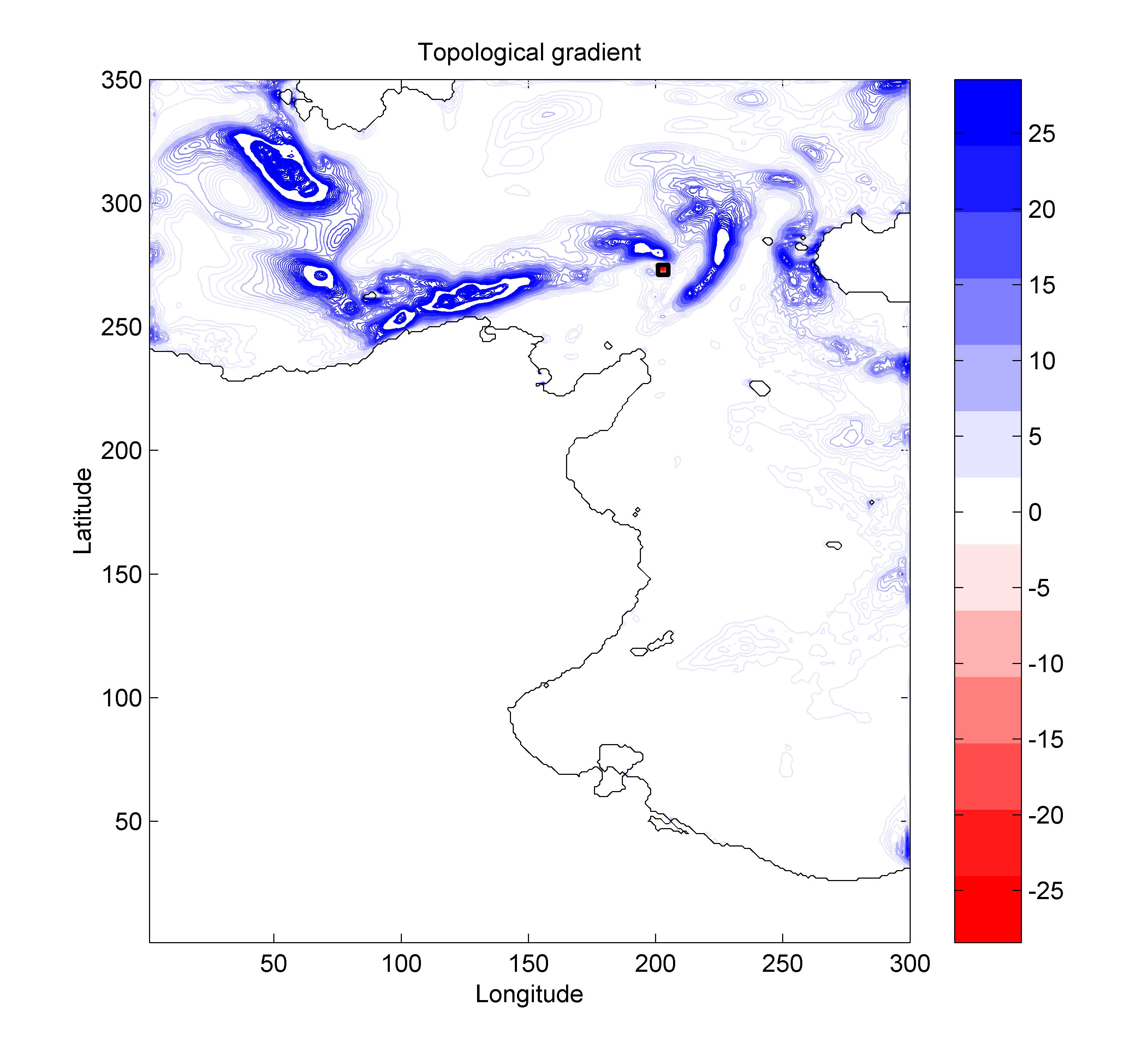} & 
        \includegraphics[width=50mm]{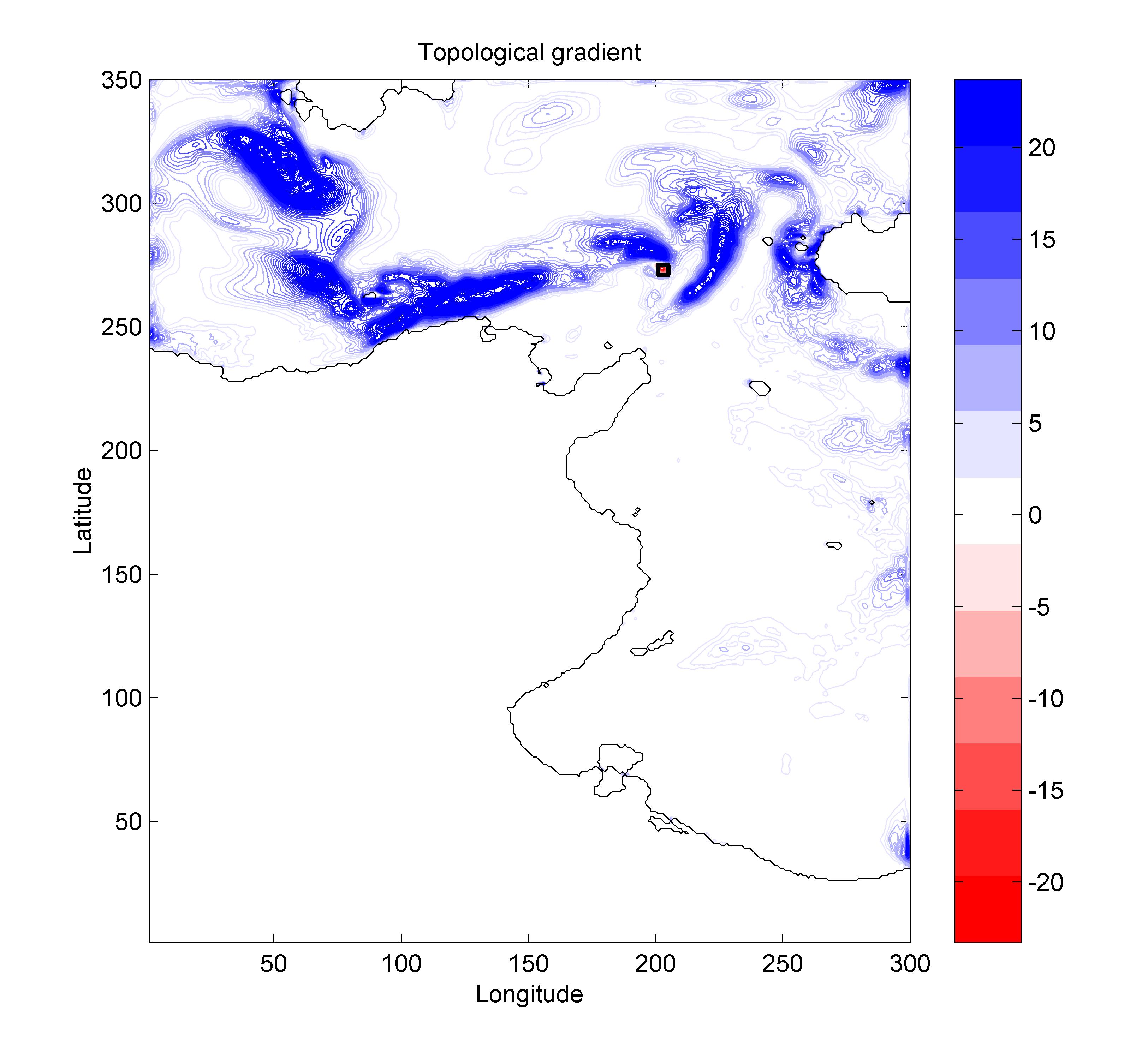} &  
        \includegraphics[width=50mm]{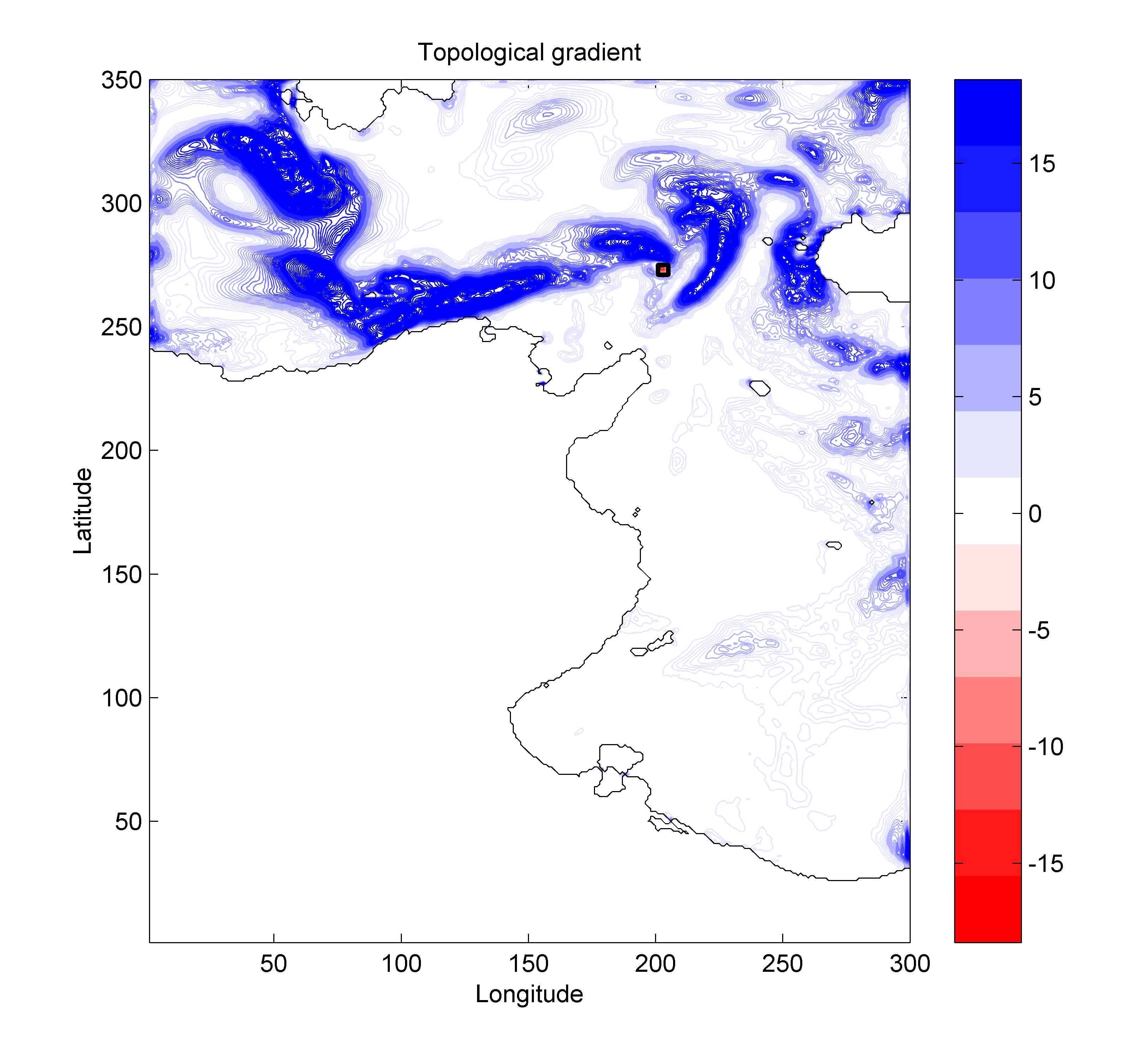} \\
        (a) Height: $\mathrm{H}_0$ & (b) Height: $\displaystyle\frac{1}{2}\mathrm{H}_0$ & (c) Height: $\displaystyle\frac{1}{4}\mathrm{H}_0$\\
        \includegraphics[width=50mm]{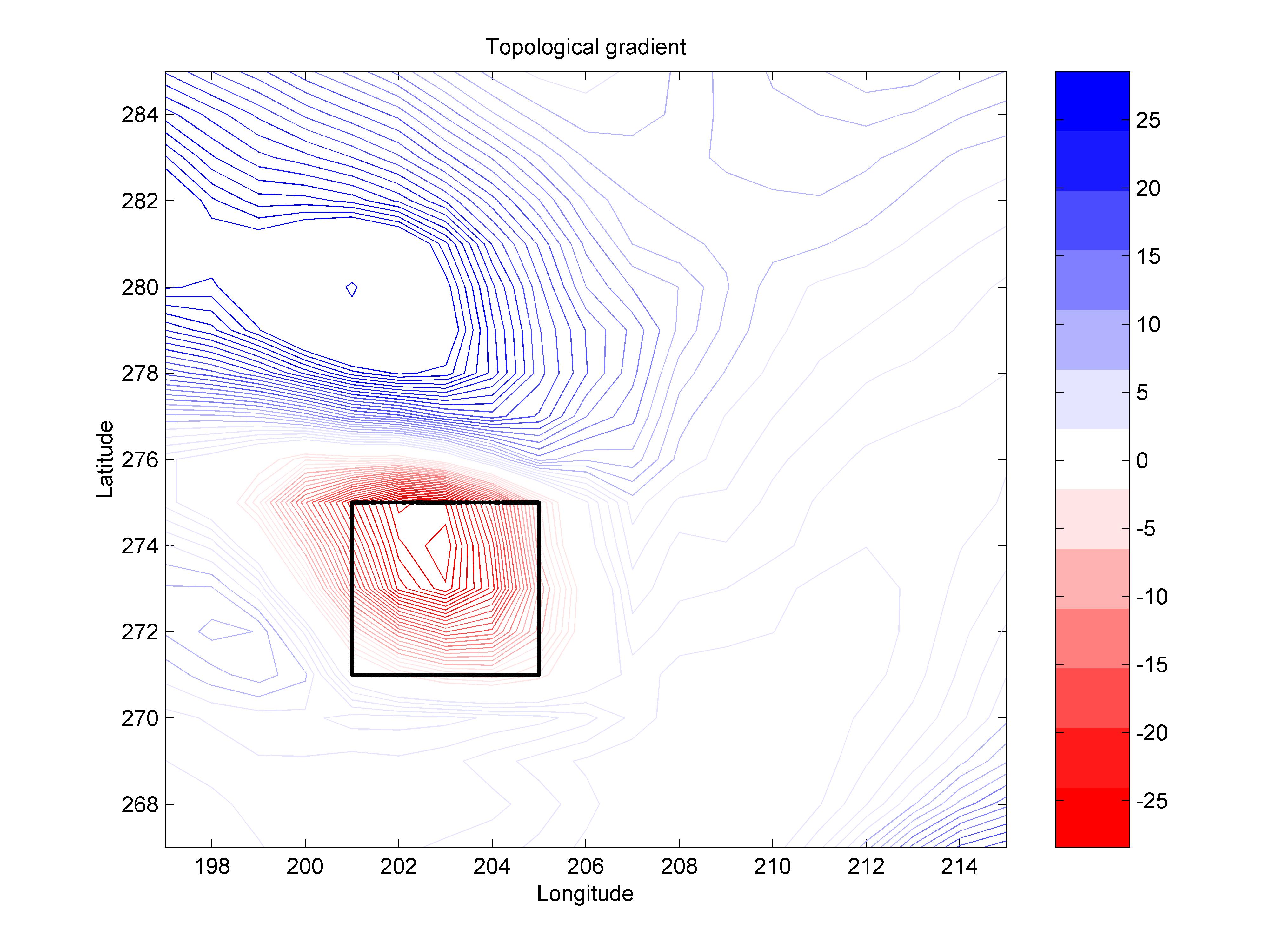} &
        \includegraphics[width=50mm]{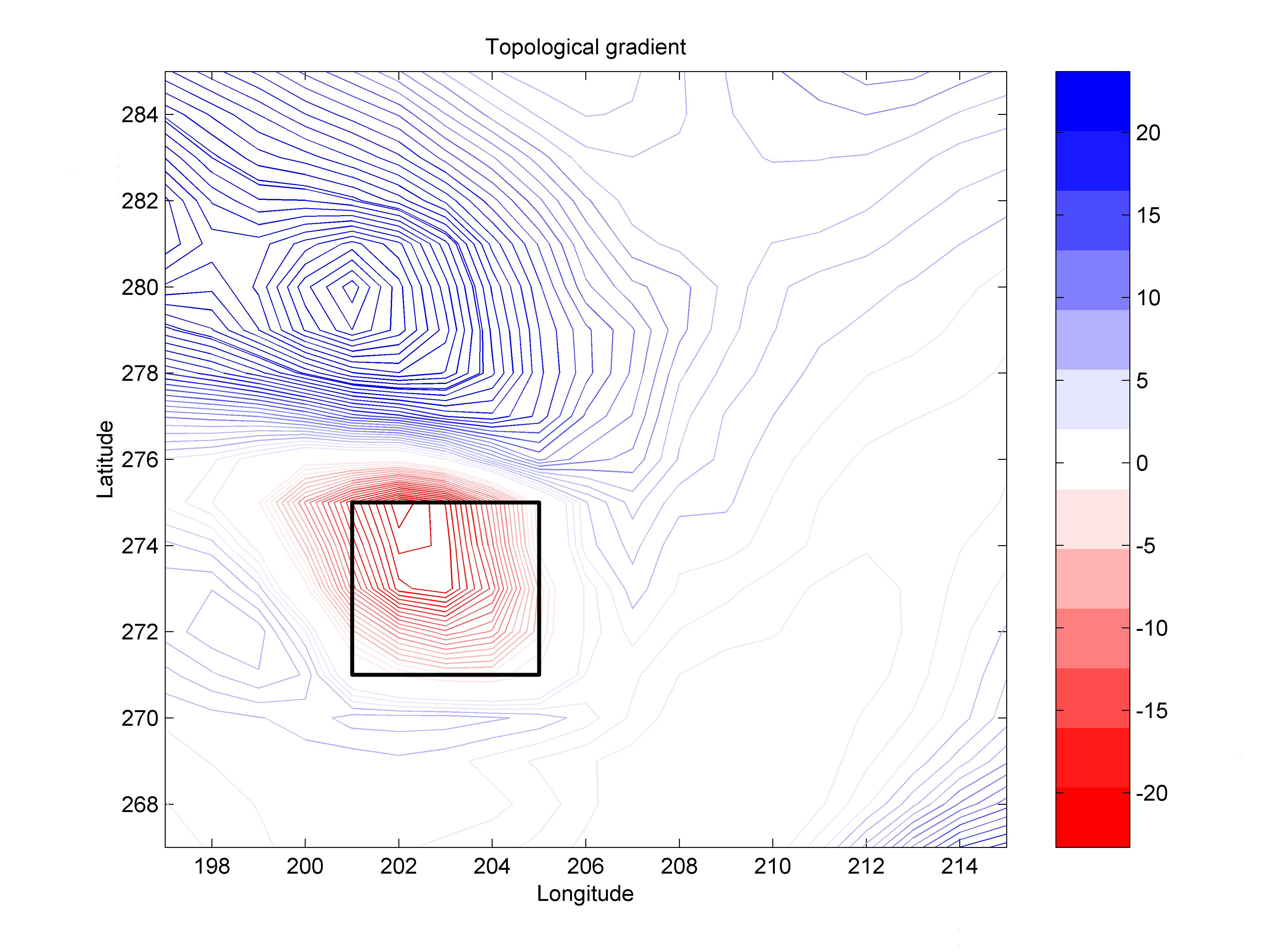} & 
        \includegraphics[width=50mm]{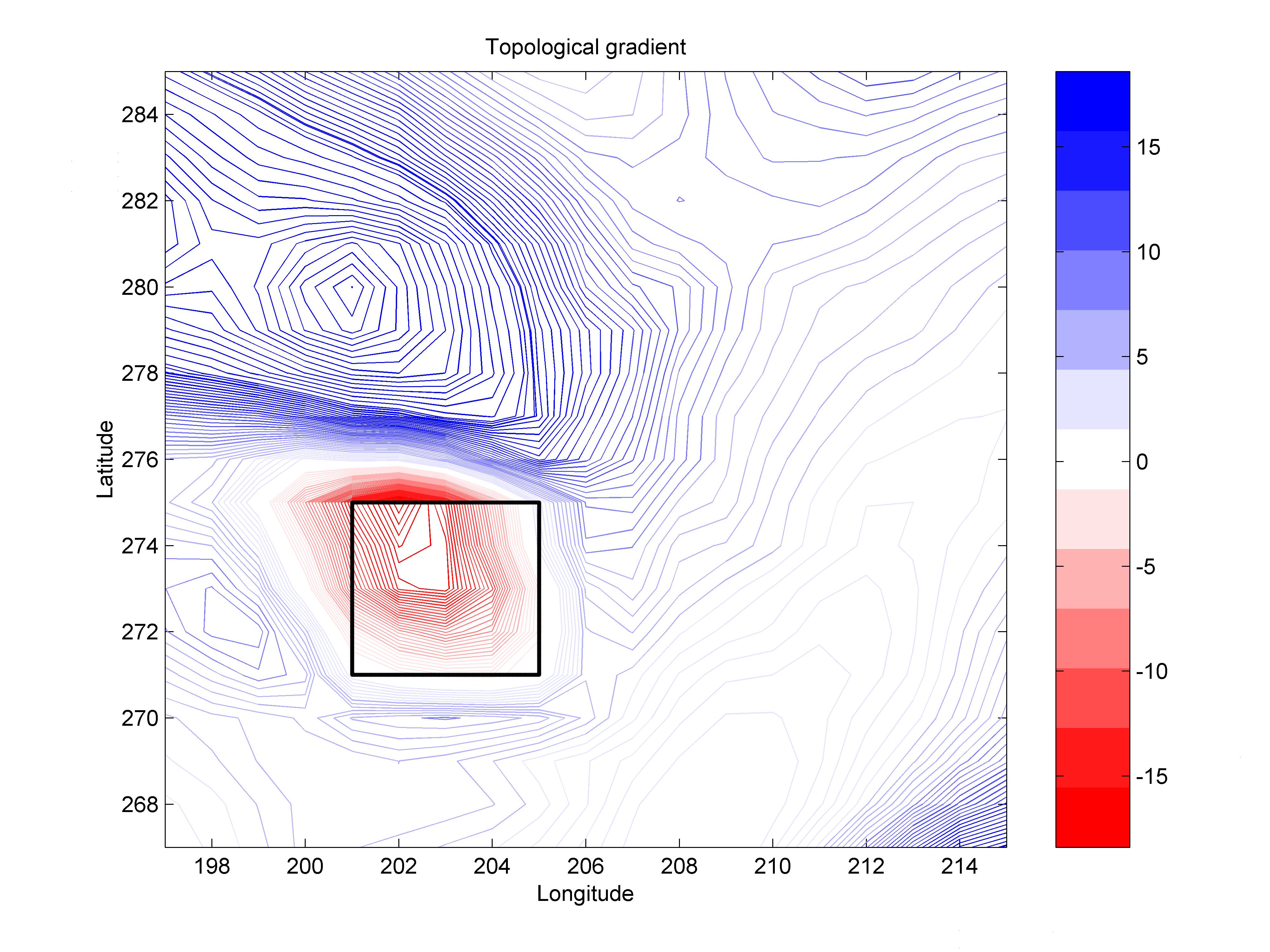} \\
        (d) Zoom of (a) & (e) Zoom of (b) & (f) Zoom of (c)
    \end{tabular}
    \caption{
    Topological gradient-based identification results for obstacles with decreasing vertical height. Subfigures (a), (b), and (c) display the iso-contour plots of \( D_K \) for obstacles with heights \( \mathrm{H}_0 \), \( \frac{1}{2}\mathrm{H}_0 \), and \( \frac{1}{4}\mathrm{H}_0 \), respectively.
    Subfigures (d), (e), and (f) present corresponding zoomed-in views around the true obstacle locations, highlighting the regions where \( D_K \) reaches its most negative values, indicating the likely positions of the inclusions.
}
\label{sizeV-2}
\end{figure}

\begin{figure}[!htb]
    \centering
    \begin{tabular}{cc}
        \includegraphics[width=50mm]{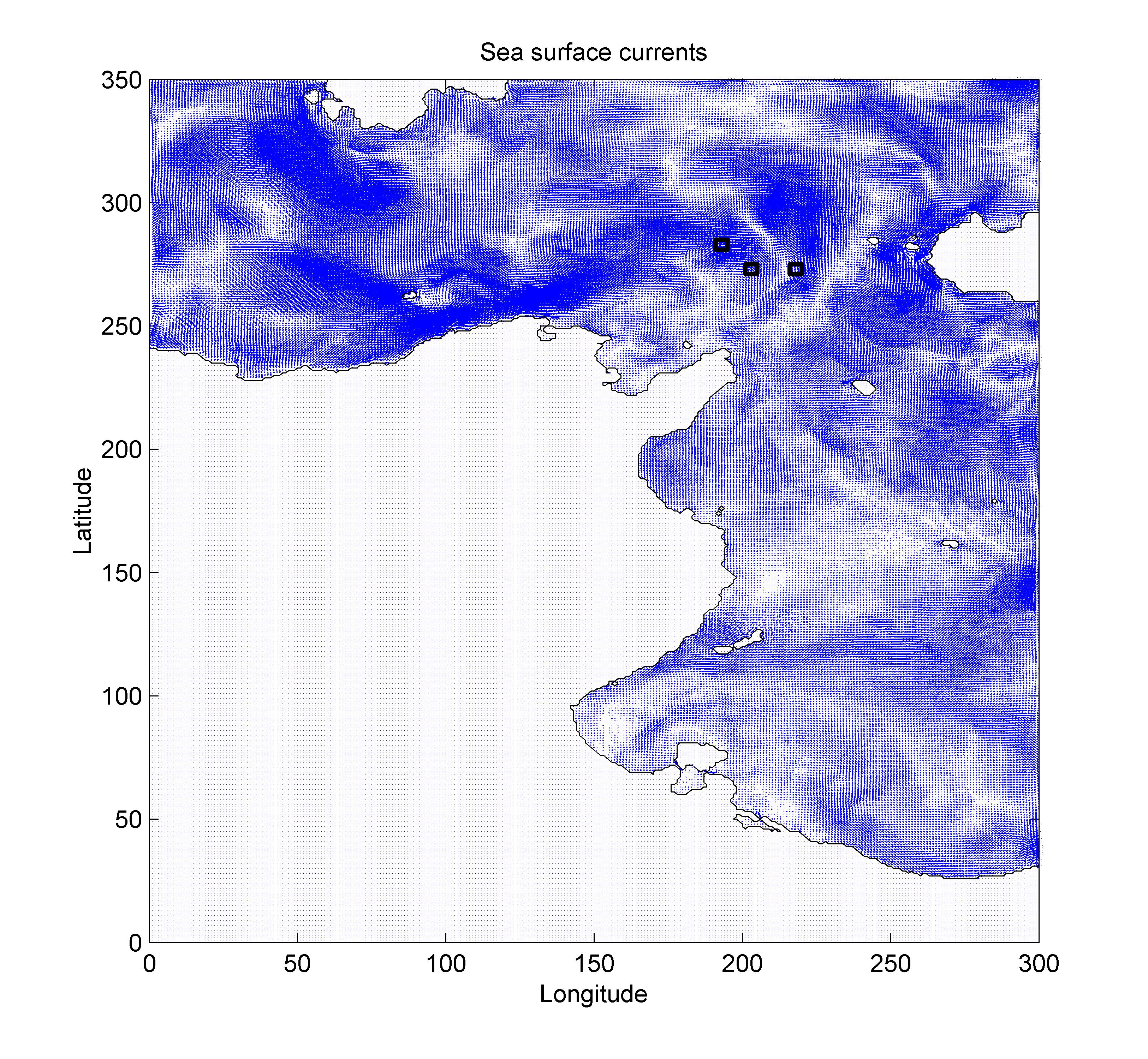} &
        \includegraphics[width=50mm]{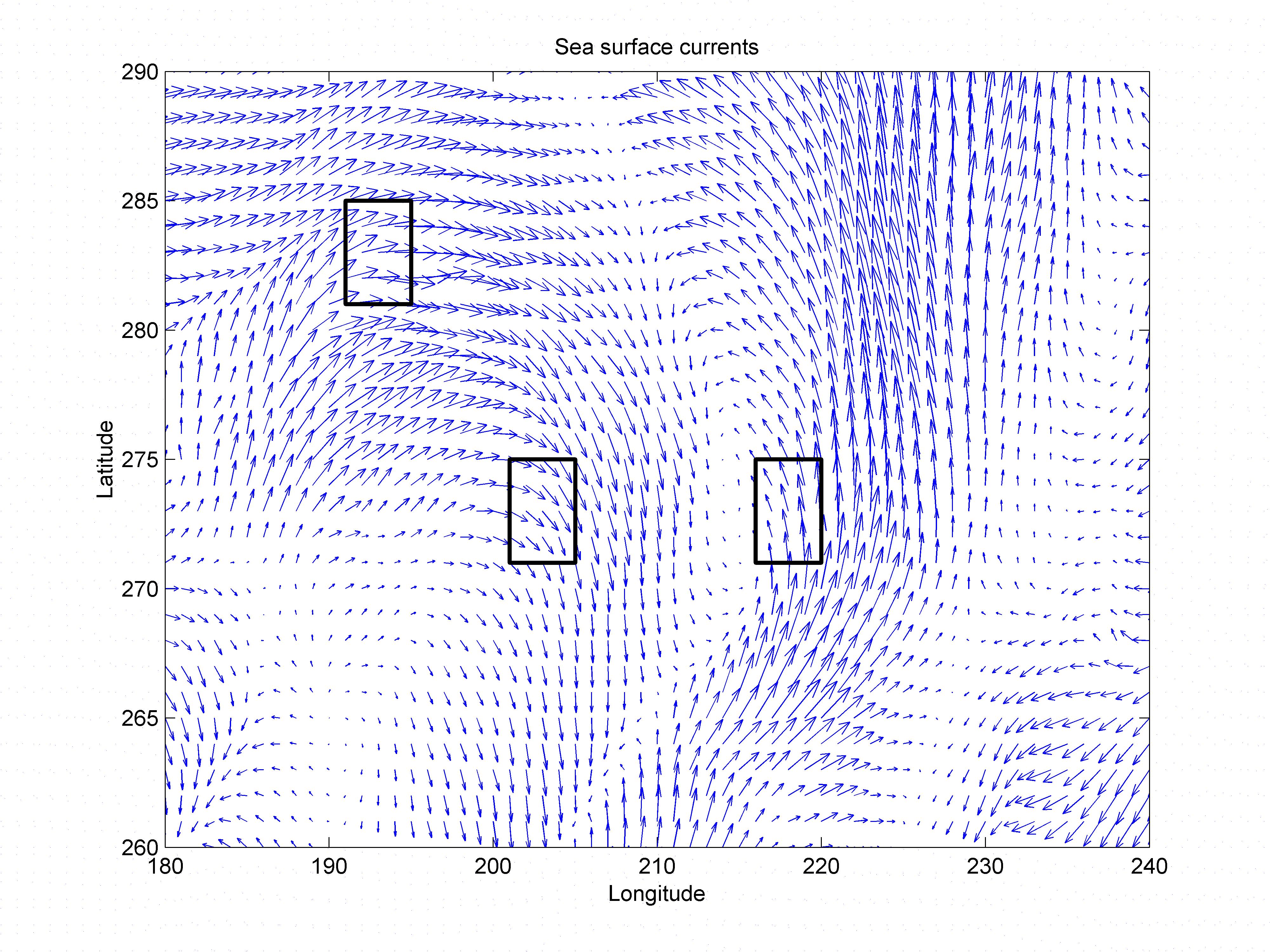} \\
        (a) &
        (b) 
    \end{tabular}
    \caption{
        (a) Locations of the three obstacles, shown as small black rectangles, to be identified within the computational domain under the given velocity field. 
        (b) Enlarged view of the region containing these obstacles.
    }
   \label{exact-locations-3obs33}
\end{figure}

\begin{figure}[!htb]
    \centering
    \begin{tabular}{ccc}
        \includegraphics[width=50mm]{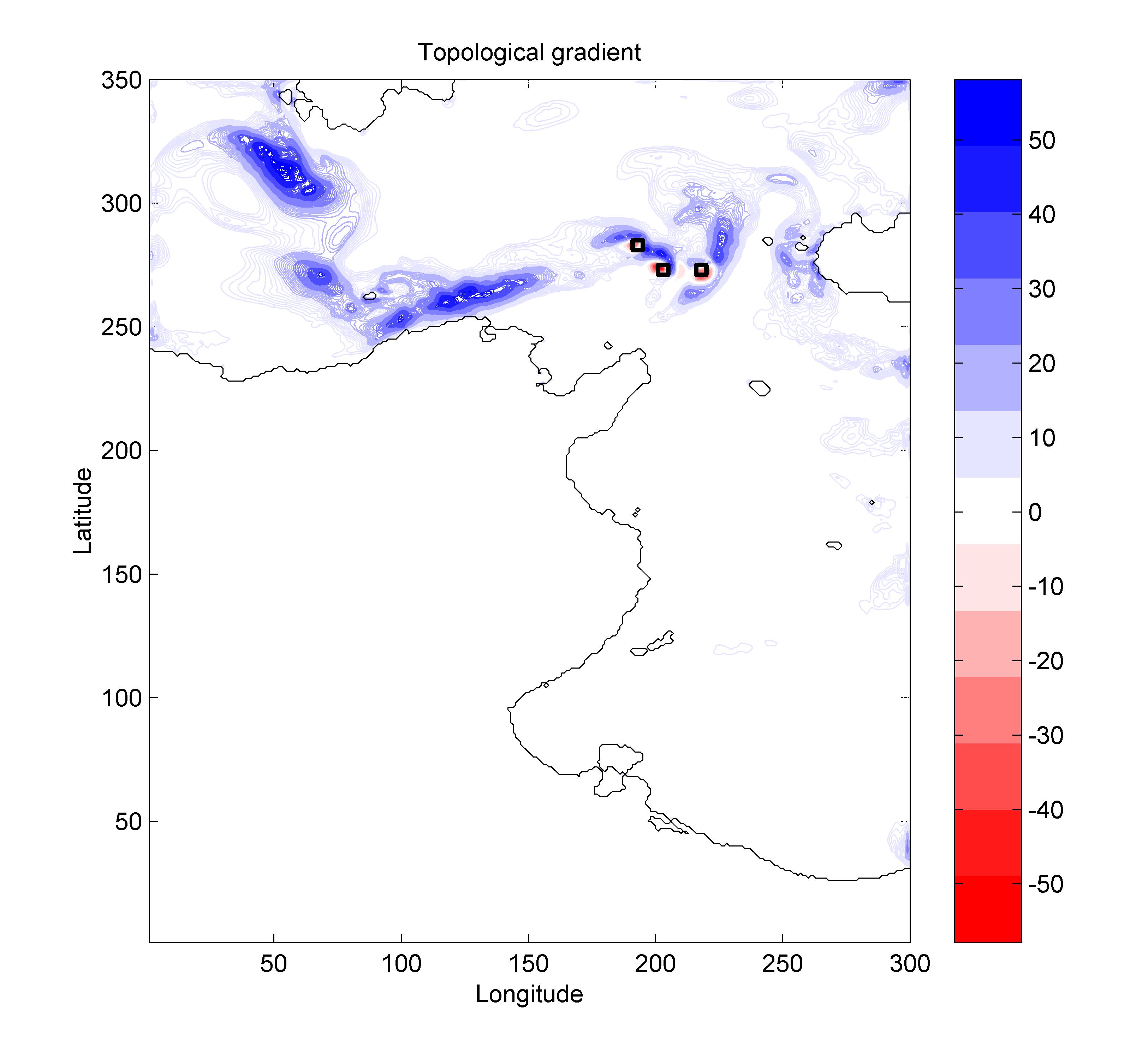} &
        \includegraphics[width=50mm]{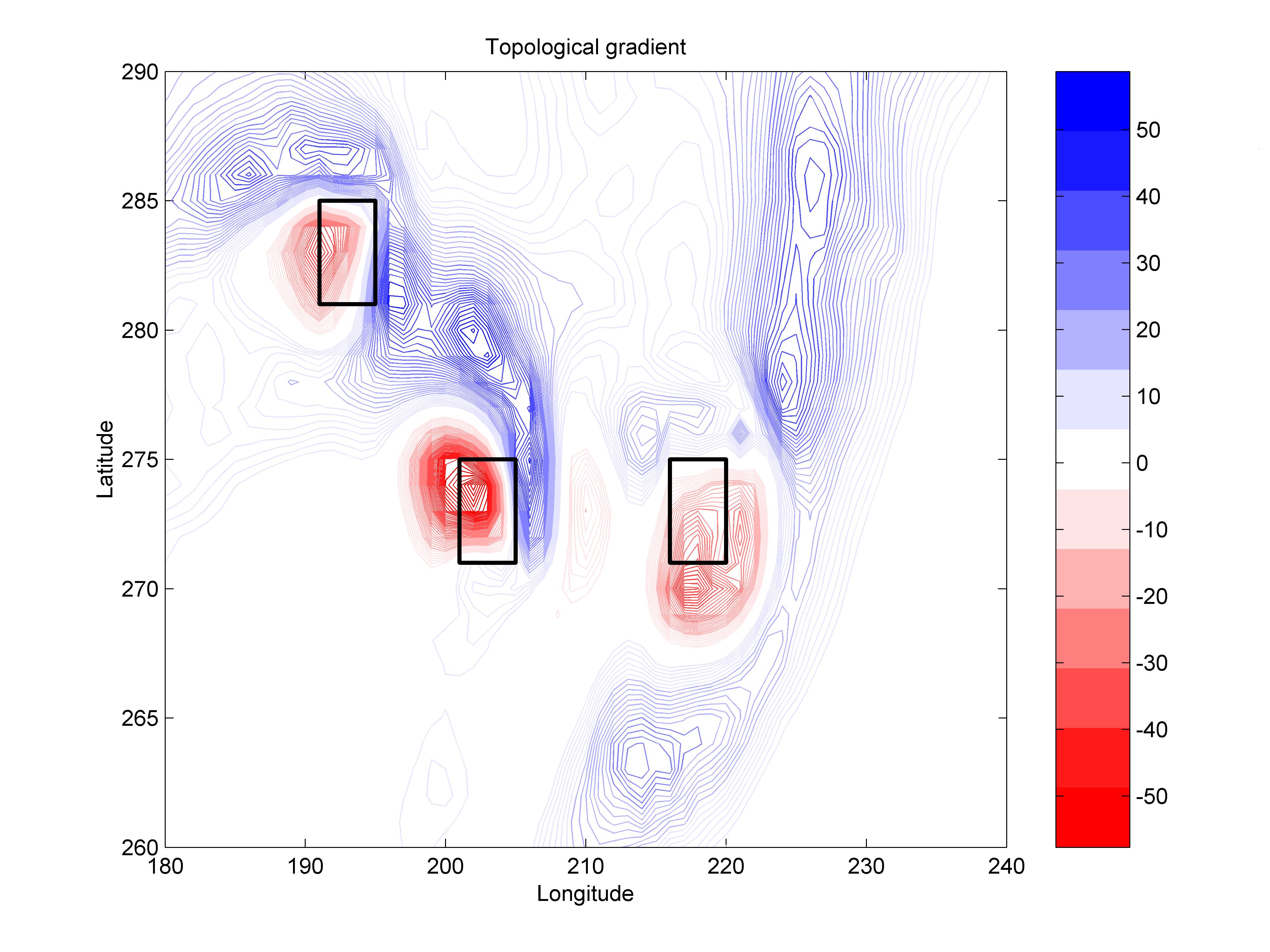} & 
        \includegraphics[width=50mm]{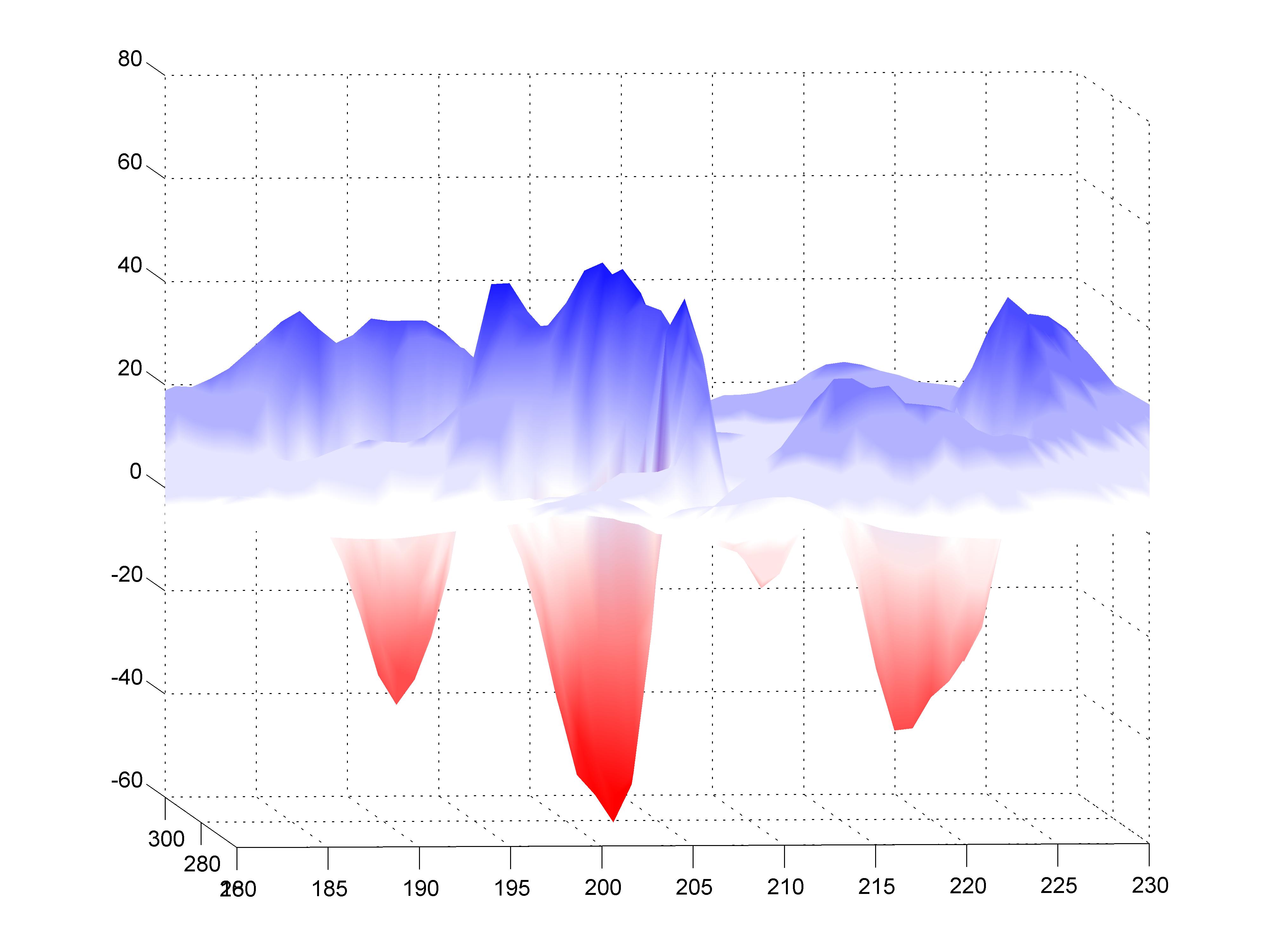} \\
        (a) &
        (b) &
        (c) 
    \end{tabular}
    \caption{
        Simultaneous identification of three submerged obstacles.
        (a) Iso-values of the topological gradient \( D_K \) across the full computational domain. (b) Zoomed-in view near the exact locations of the three obstacles (indicated by small black rectangles), highlighting the negative peaks of \( D_K \).
        (c) 3D surface plot of \( D_K \), showing distinct local minima corresponding to the positions of the three obstacles.
    }
   \label{obs45}
\end{figure}

%%%%%%%%%%%%%%%%%%%%%%%%%%%%%%%%%%%%%%%%%%%%%%%%%%%%%%%%%%%%%%%%%%%%%%%%%%%%%%%%%%
\subsubsection{Example 6 : Influence of the distance between two obstacles}\label{example6}
%%%%%%%%%%%%%%%%%%%%%%%%%%%%%%%%%%%%%%%%%%%%%%%%%%%%%%%%%%%%%%%%%%%%%%%%%%%%%%%%%%

In this section, we investigate the impact of the separation distance \( d > 0 \) between two submerged obstacles \( \omega^*_1 \) and \( \omega^*_2 \) on the accuracy of the identification process. Both obstacles are assumed to have the same size, and their exact locations are indicated by black squares in the domain.

We distinguish between four test cases, characterized by different values of the separation distance \( d \in \{1, 2, 11, 126\} \), measured in horizontal grid points. The reference obstacle \( \omega^*_1 \) is fixed at position:
\[
(201{:}205) \times (271{:}275) \times (8{:}10),
\]
corresponding to an averag
e depth of approximately \( 260\,\mathrm{m} \). The second obstacle \( \omega^*_2 \) is placed at different horizontal positions across the tests, while maintaining a similar vertical extent. The configurations are as follows:

\begin{itemize}
    \item \textbf{Test 1: Well-separated obstacles.}
\( \omega^*_2 \) located at: \( (71{:}75) \times (271{:}275) \times (8{:}10) \), with an average depth of \( 48\,\mathrm{m} \). Horizontal separation: \( d = 126 \) grid points.

    \item \textbf{Test 2: Moderately close obstacles.} \( \omega^*_2 \) located at: \( (216{:}220) \times (271{:}275) \times (8{:}10) \), with an average depth of \( 255\,\mathrm{m} \).
    Horizontal separation: \( d = 11 \) grid points.

    \item \textbf{Test 3: Very close obstacles.}
    \( \omega^*_2 \) located at: \( (207{:}211) \times (271{:}275) \times (8{:}10) \), with an average depth of \( 315\,\mathrm{m} \).
    Horizontal separation: \( d = 2 \) grid points.

    \item \textbf{Test 4: Nearly adjacent obstacles.}
    \( \omega^*_2 \) located at: \( (206{:}210) \times (271{:}275) \times (8{:}10) \), with an average depth of \( 222\,\mathrm{m} \).
    Horizontal separation: \( d = 1 \) grid point.
\end{itemize}

The corresponding identification results are presented in Figure \ref{iso-gt-2obs-d1}. When the two obstacles are well separated, the topological gradient clearly exhibits two distinct local minima, indicating that both obstacles are accurately detected (see Figures \ref{iso-gt-2obs-d1}(b)-(e)). However, as the distance \( d \) between the obstacles decreases, the ability to resolve them individually deteriorates. Specifically, the two minima gradually converge and eventually merge into a single dominant minimum, as illustrated in Figure \ref{iso-gt-2obs-d1}(i)-(l). This phenomenon suggests that when obstacles are too close, the algorithm perceives them as a single ``equivalent'' inclusion, highlighting a resolution limitation in the detection process.

\newpage

\begin{figure}[!htb]
    \centering
    \begin{tabular}{ccc}
        \includegraphics[width=50mm]{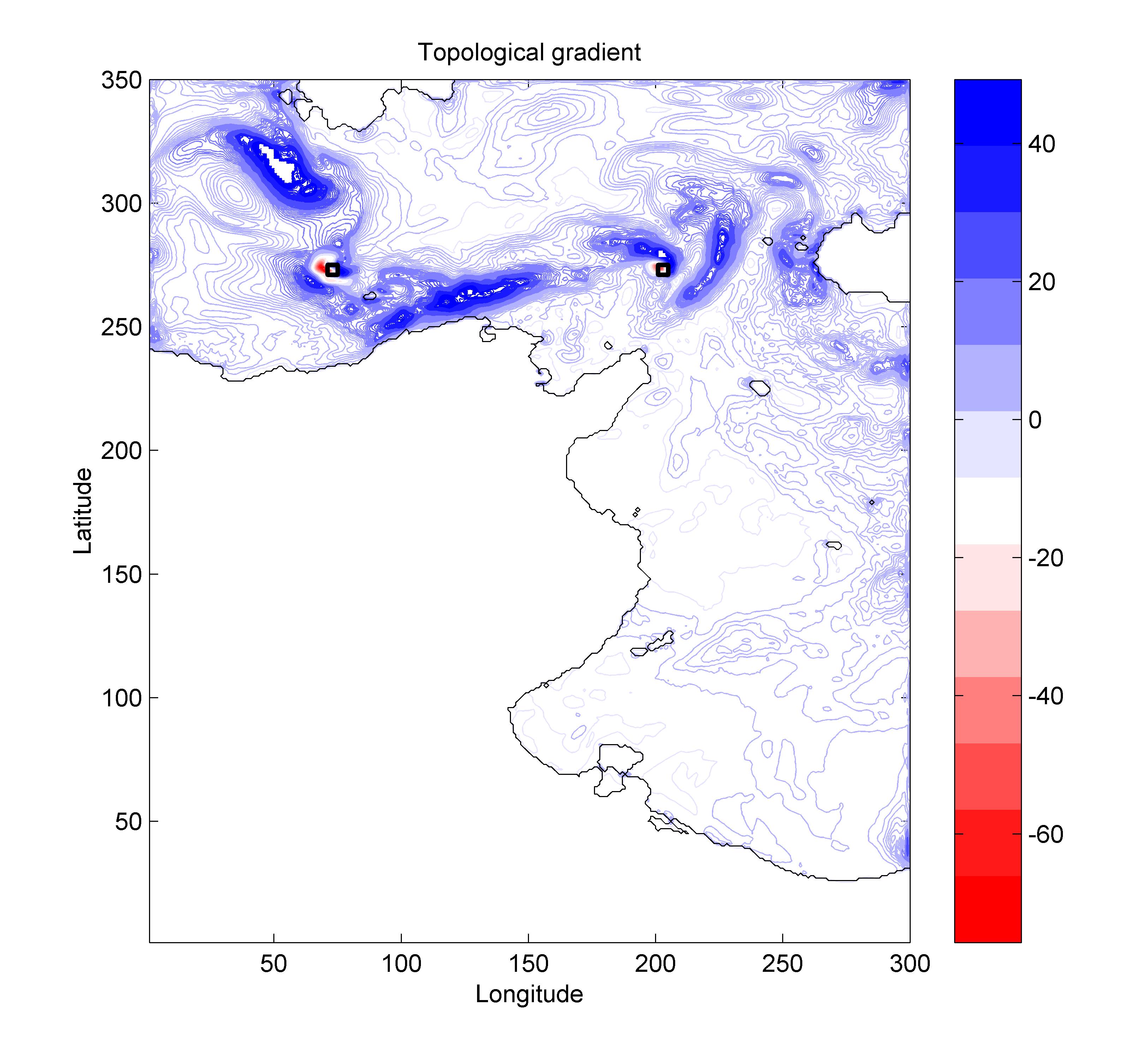} & 
        \includegraphics[width=50mm]{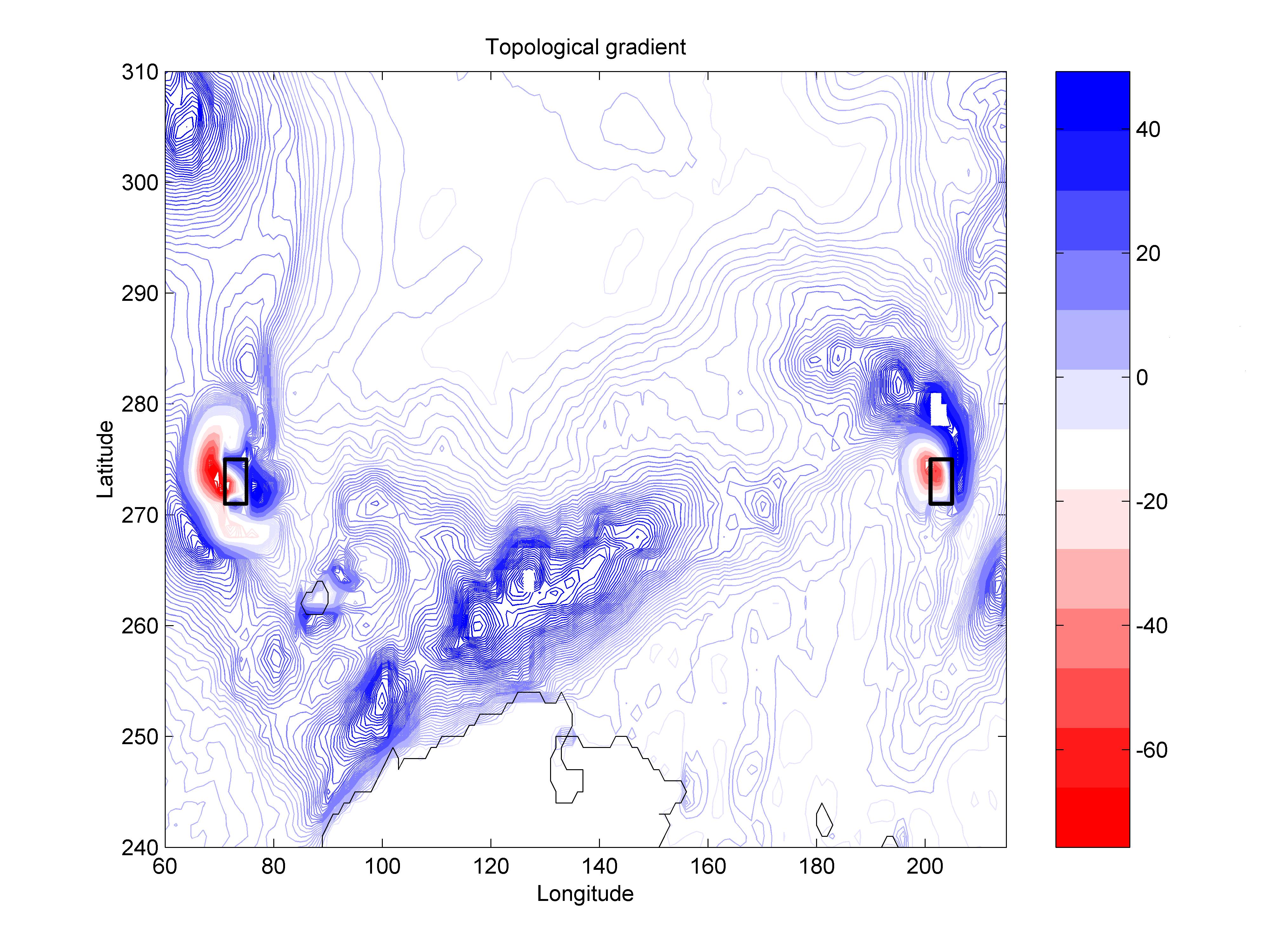}&  \includegraphics[width=50mm]{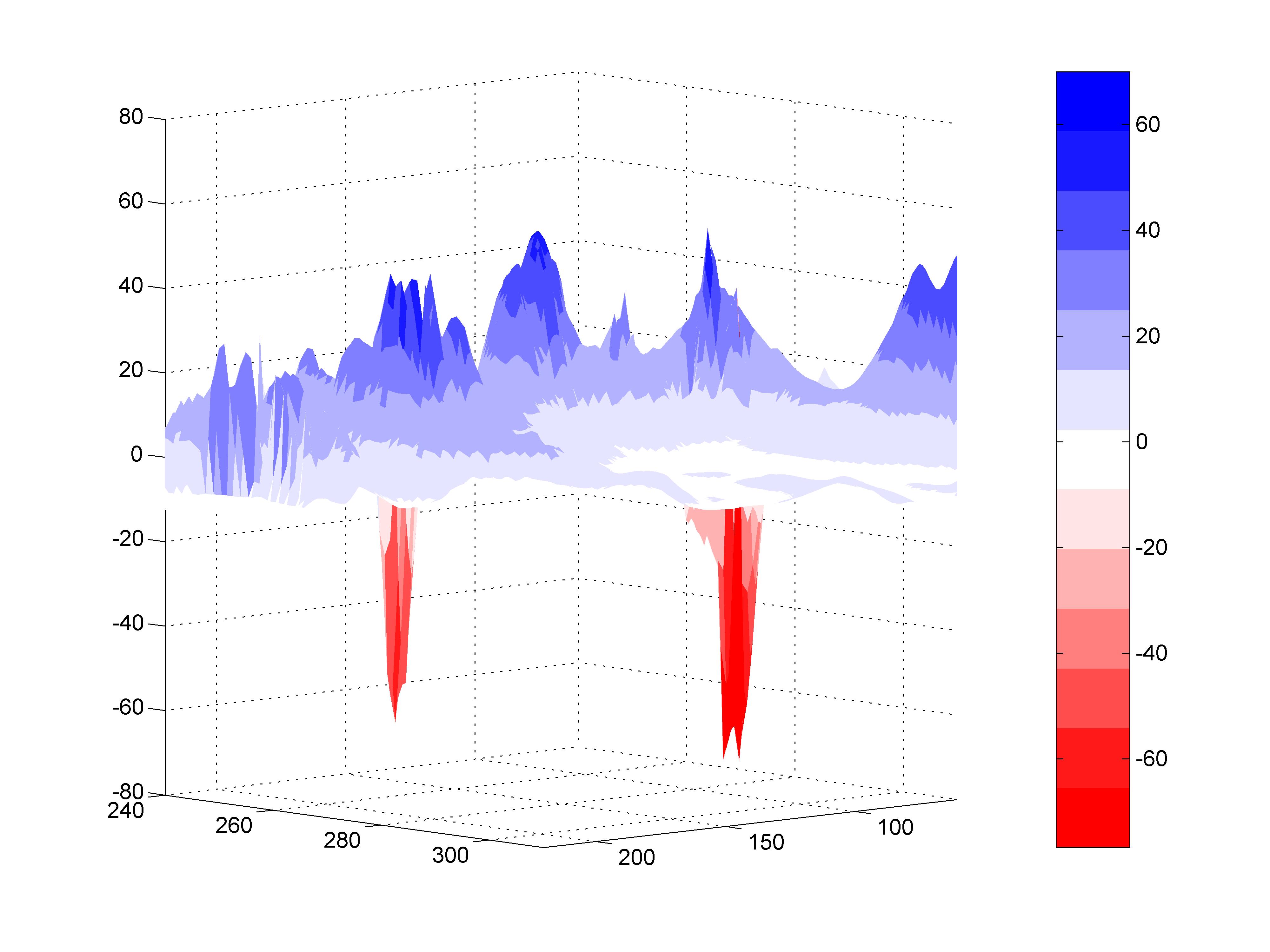}\\
        (a) \( d = 126 \) grid points & 
        (b) Zoom of (a)& (c) 3D plot of (b) \\
        \includegraphics[width=50mm]{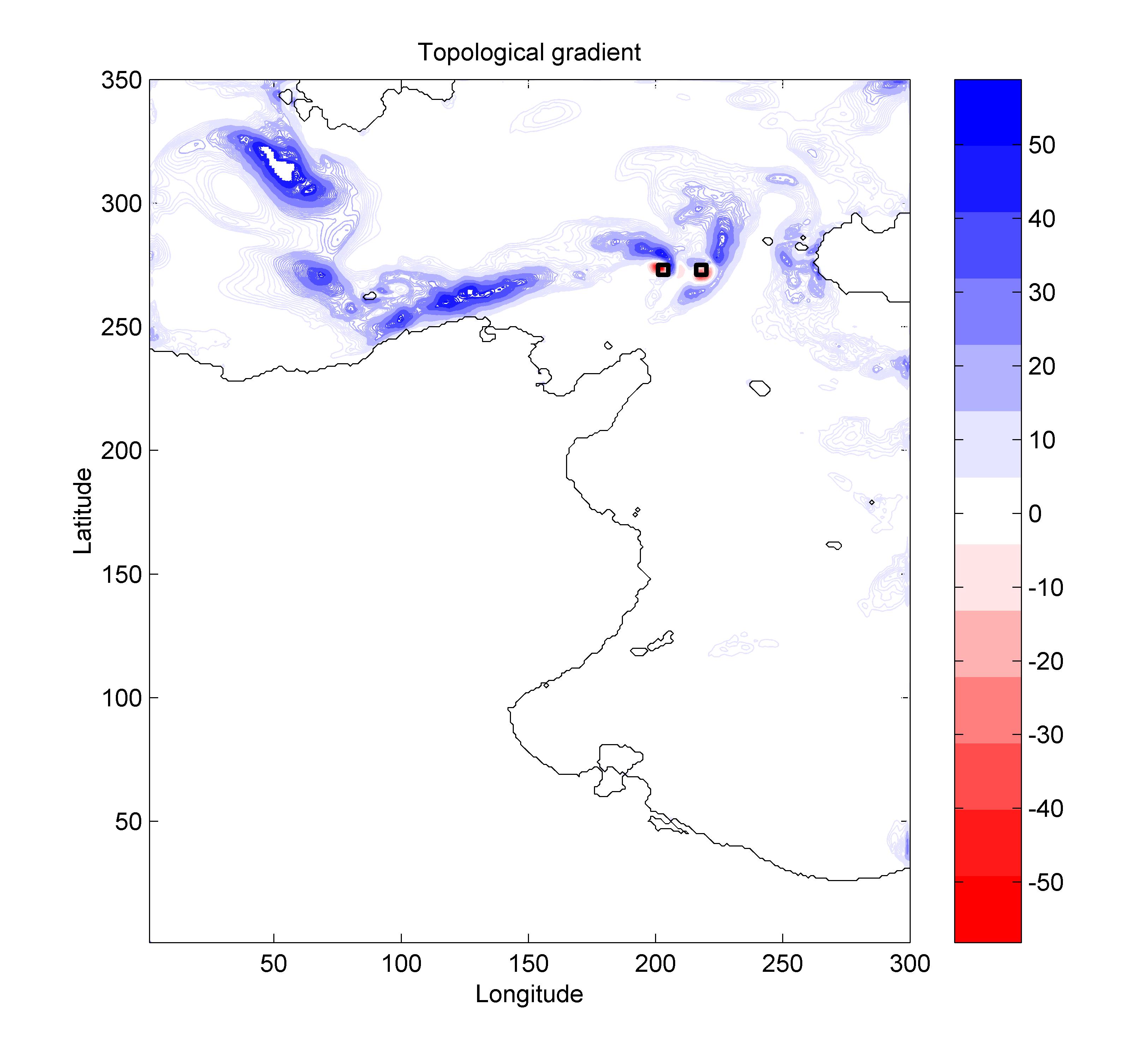} & 
        \includegraphics[width=50mm]{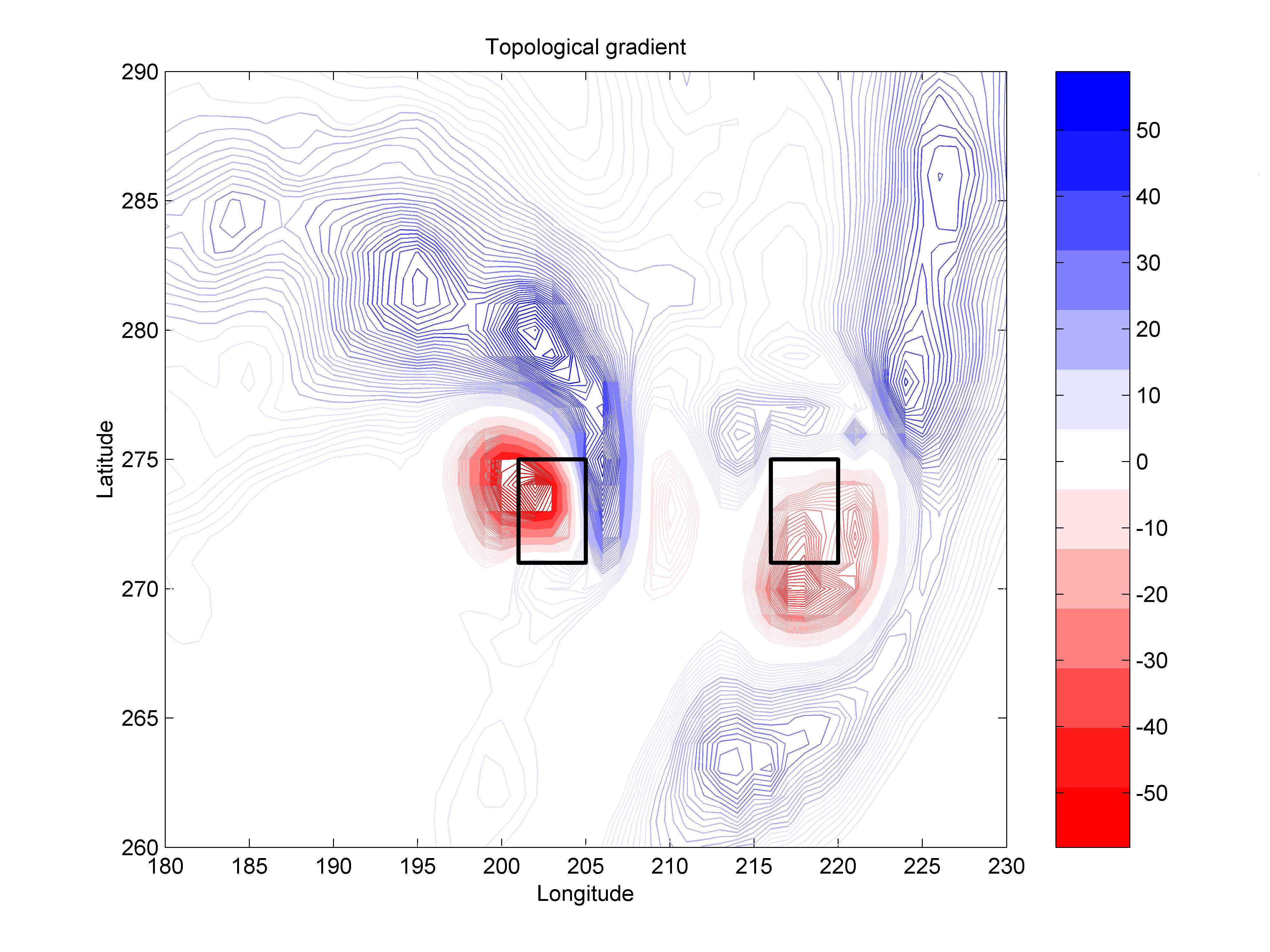}&\includegraphics[width=50mm]{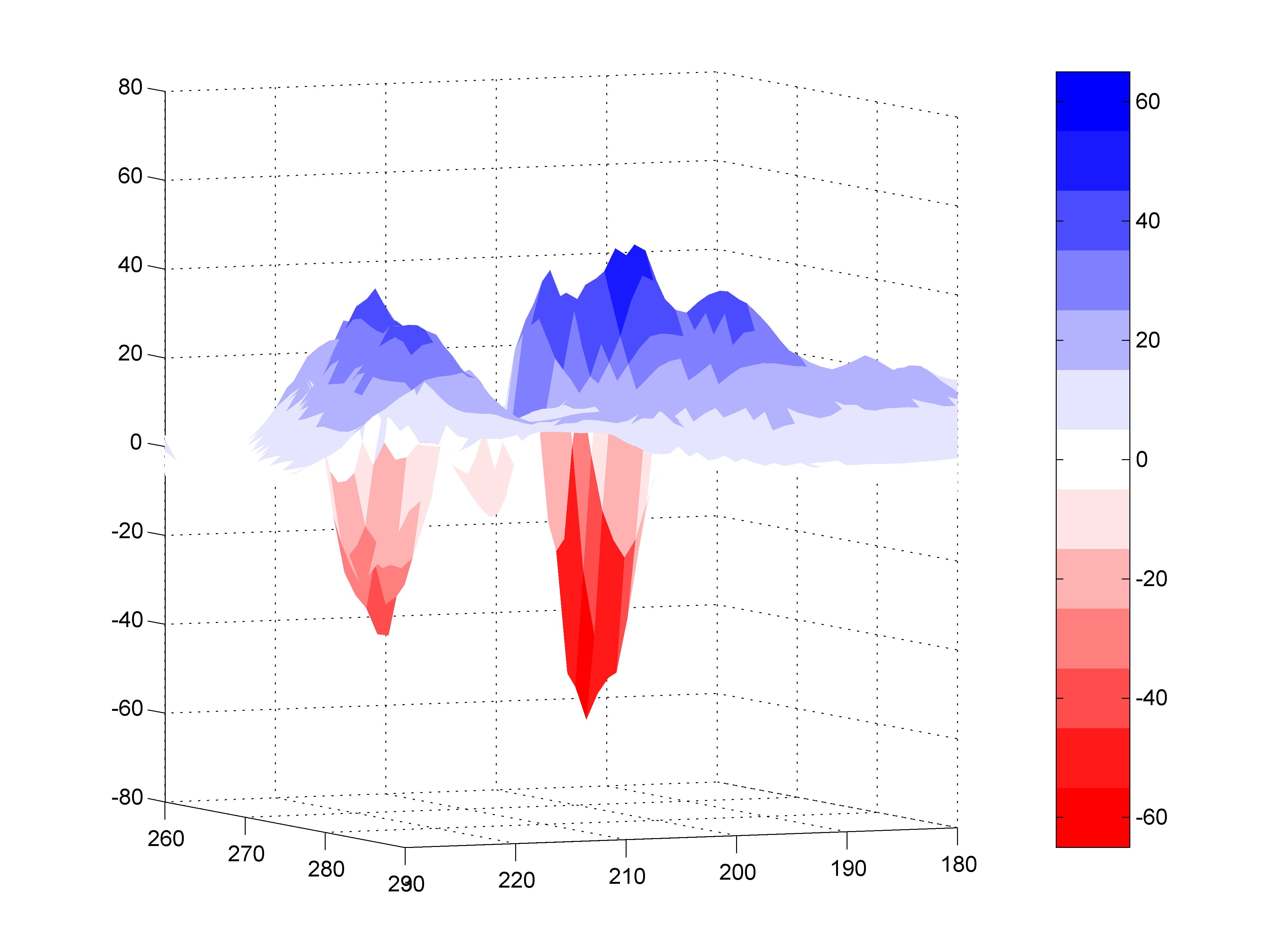} \\
        (d) \( d = 11 \) grid points   & 
        (e) Zoom of (d) & (f) 3D plot of (e) \\
        \includegraphics[width=50mm]{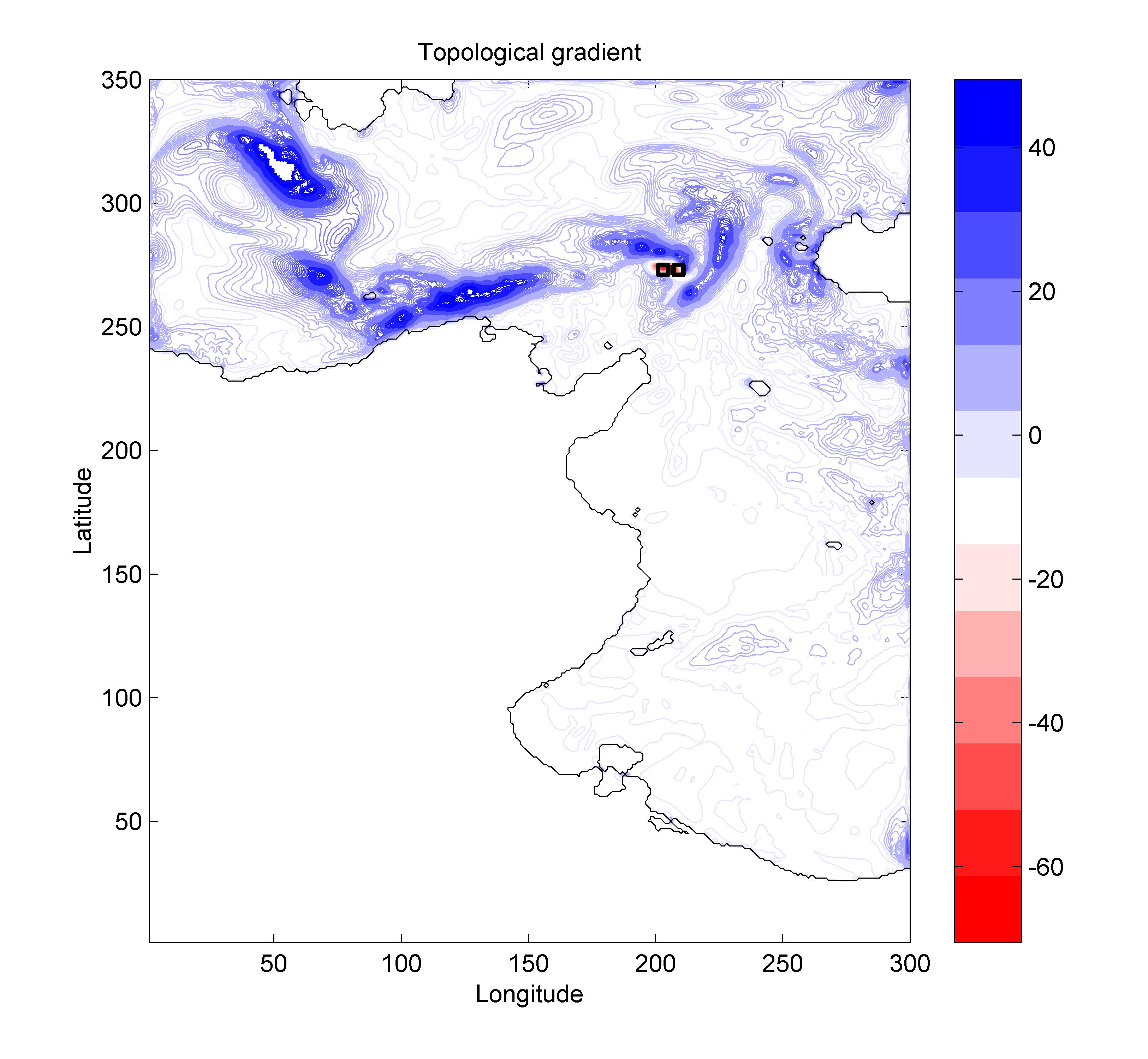} & 
        \includegraphics[width=50mm]{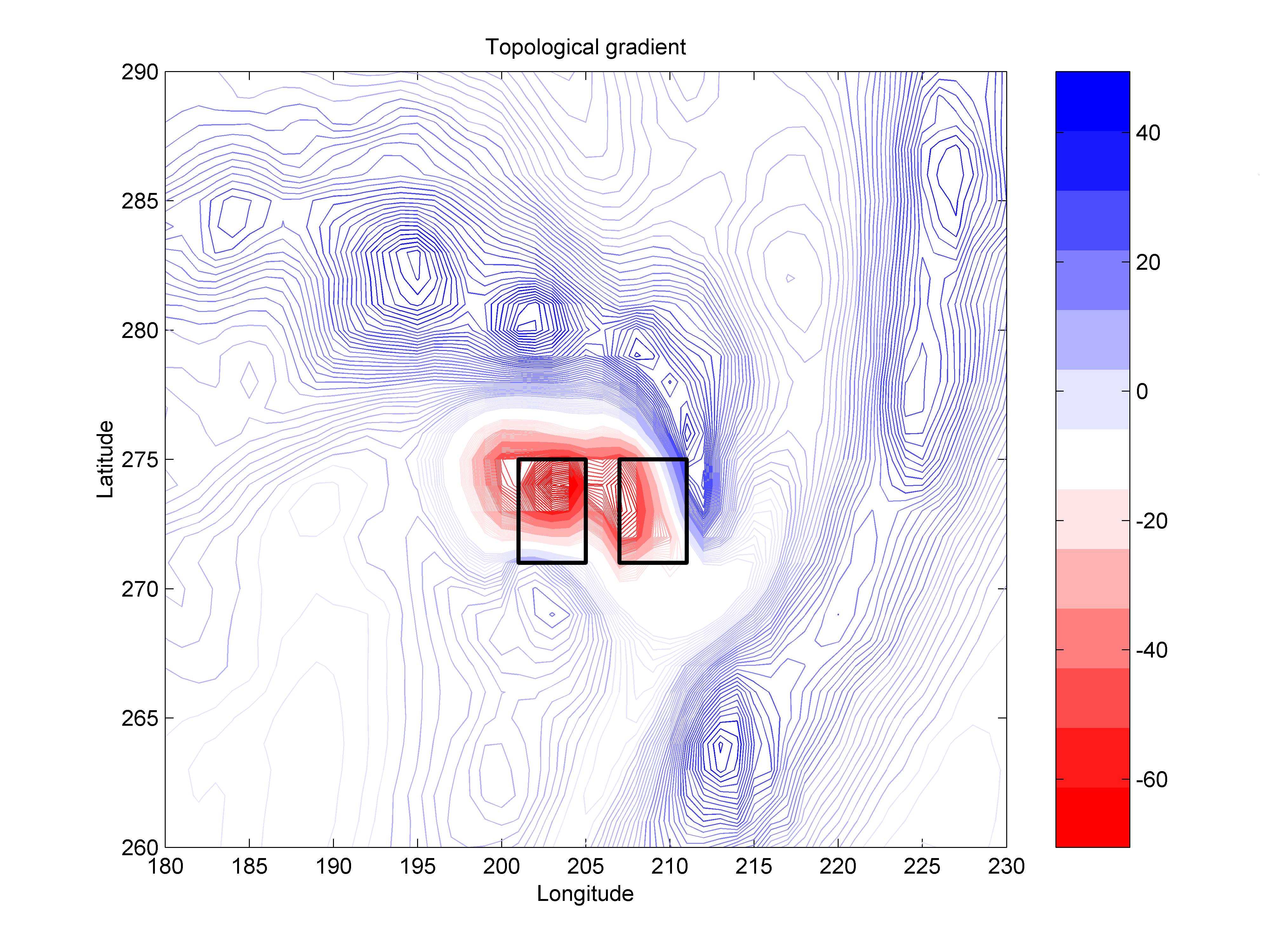}& \includegraphics[width=50mm]{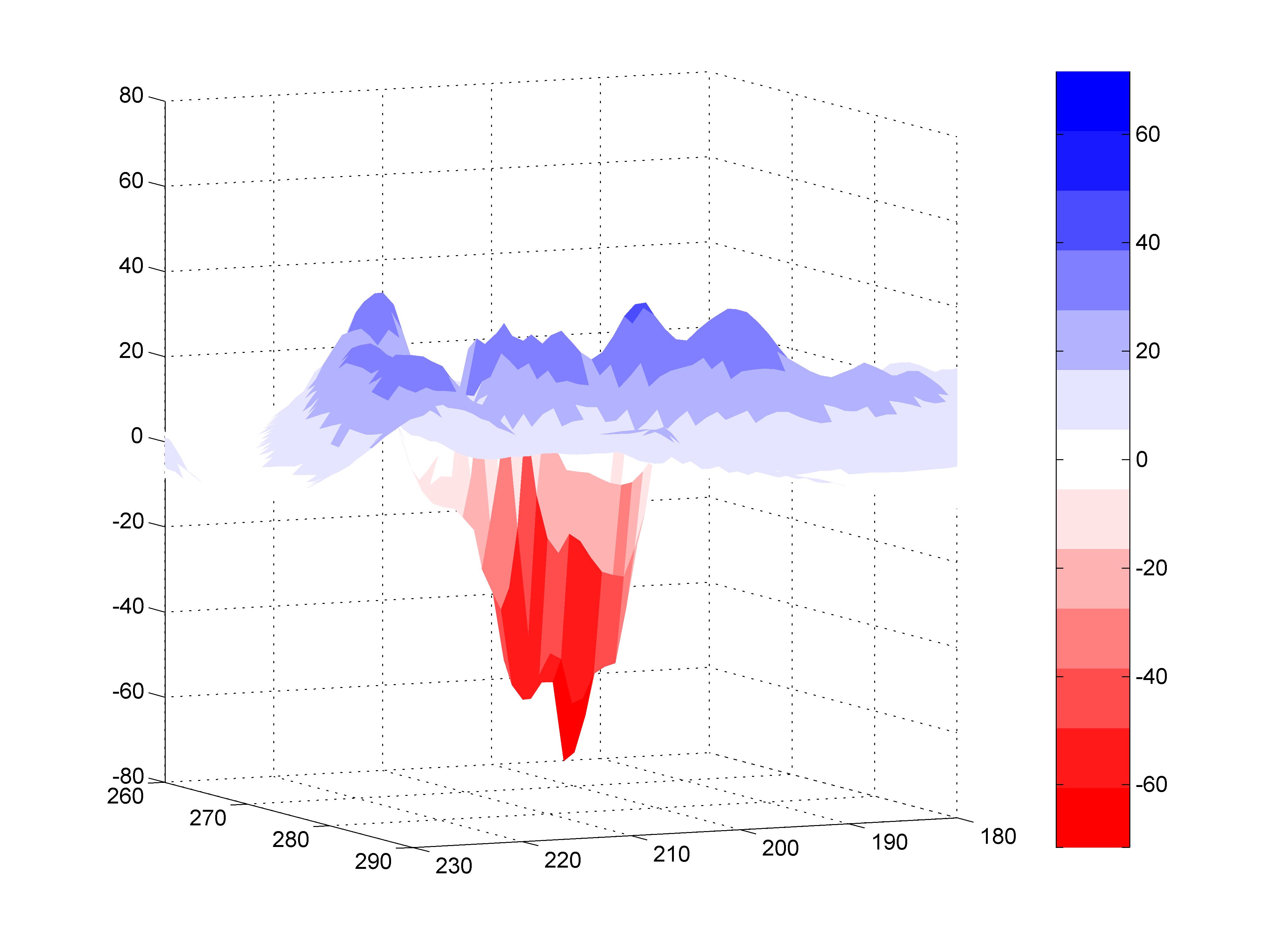} \\
        (g) \( d = 2 \) grid points & 
        (h) Zoom of (g) & (i) 3D plot of (h)\\
         \includegraphics[width=50mm]{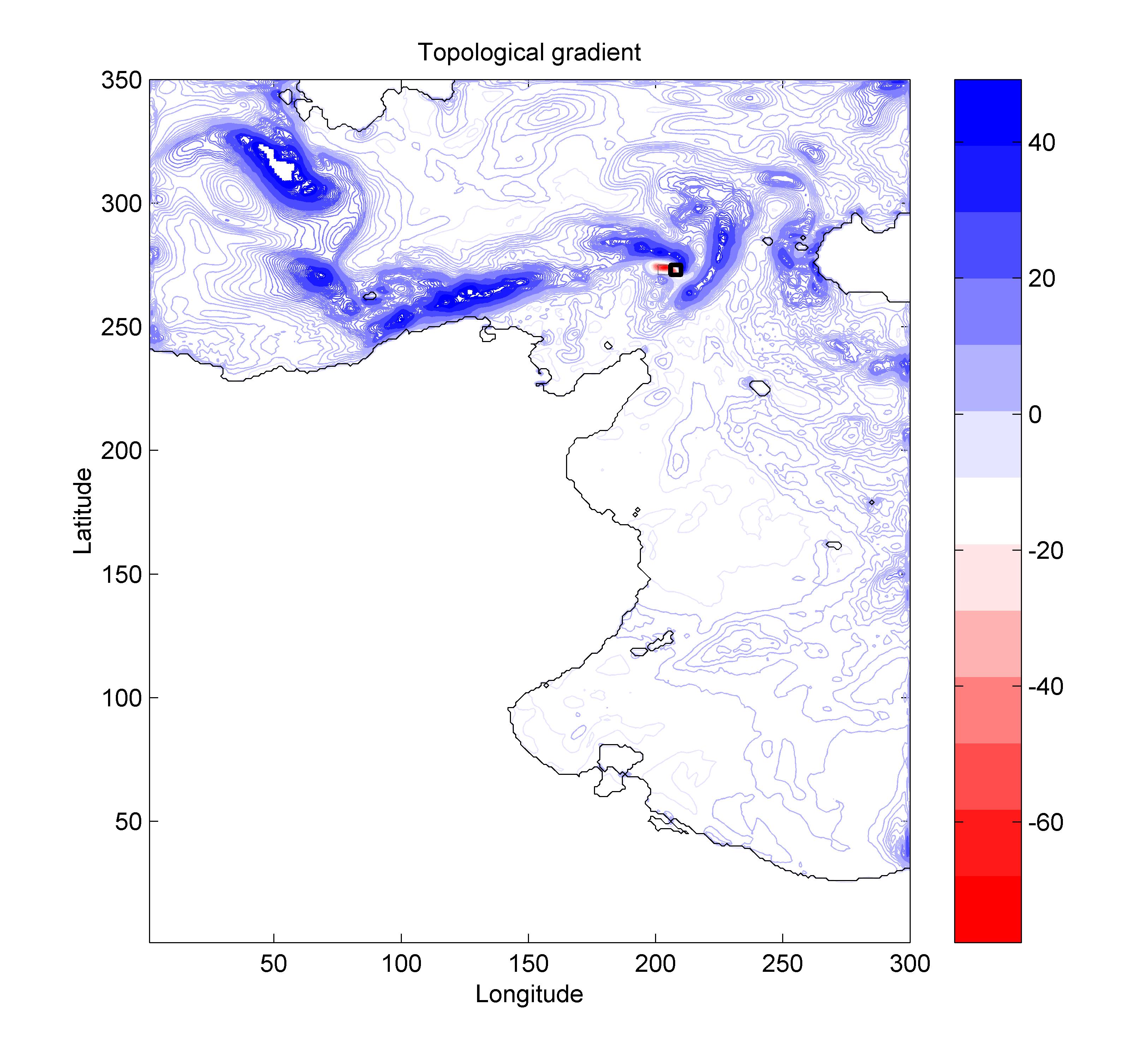} & 
        \includegraphics[width=50mm]{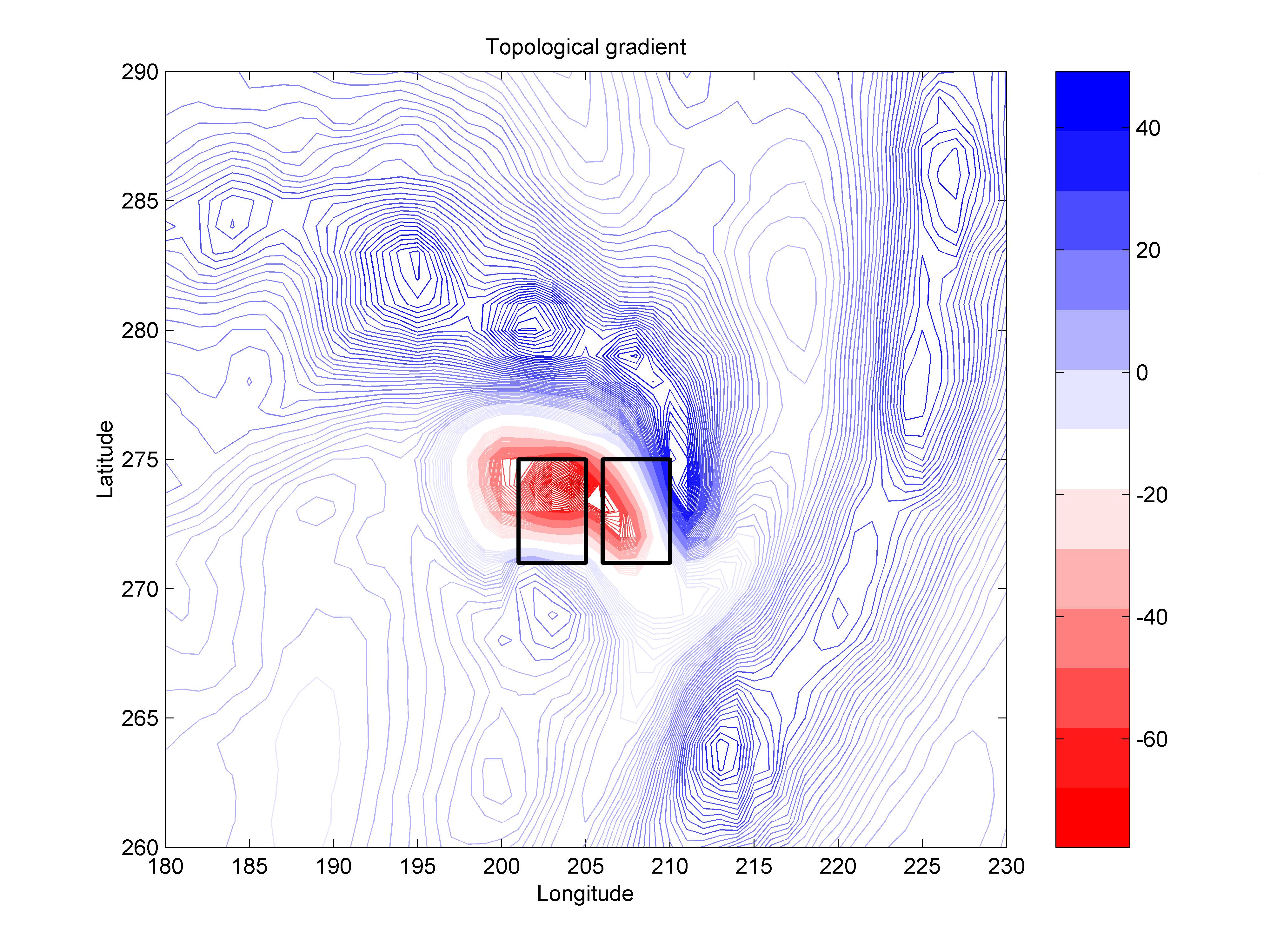}&\includegraphics[width=50mm]{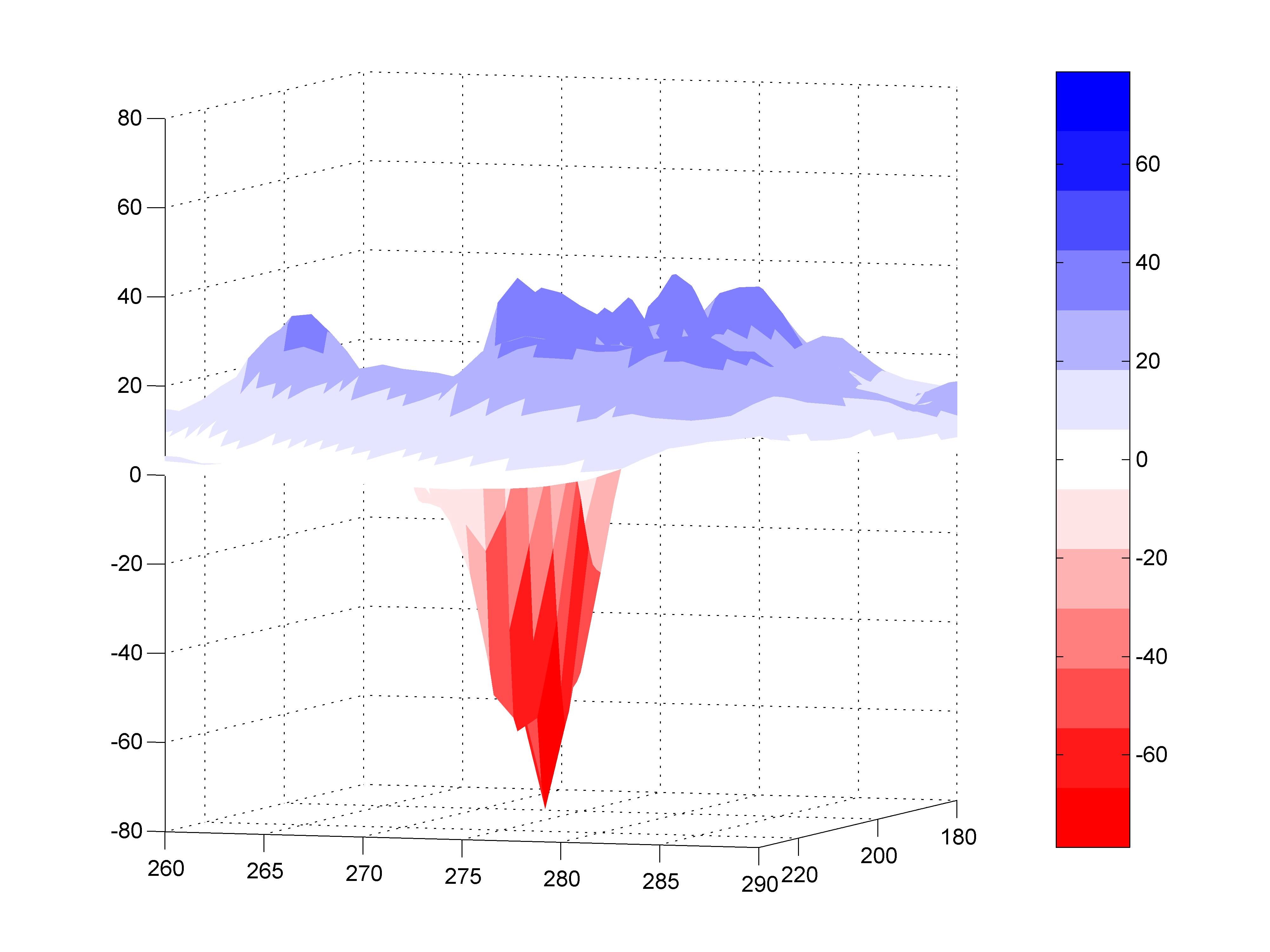} \\
        (j) \( d = 1 \) grid point   & 
        (k) Zoom of (j) & (l) 3D plot of (k)
    \end{tabular}
    \caption{Iso-Values of the topological gradient illustrating the influence of relative distance between two obstacles.}
    \label{iso-gt-2obs-d1}
\end{figure}

\newpage

\begin{figure}[!htb]
    \centering
    \begin{tabular}{cc}
        \includegraphics[width=50mm]{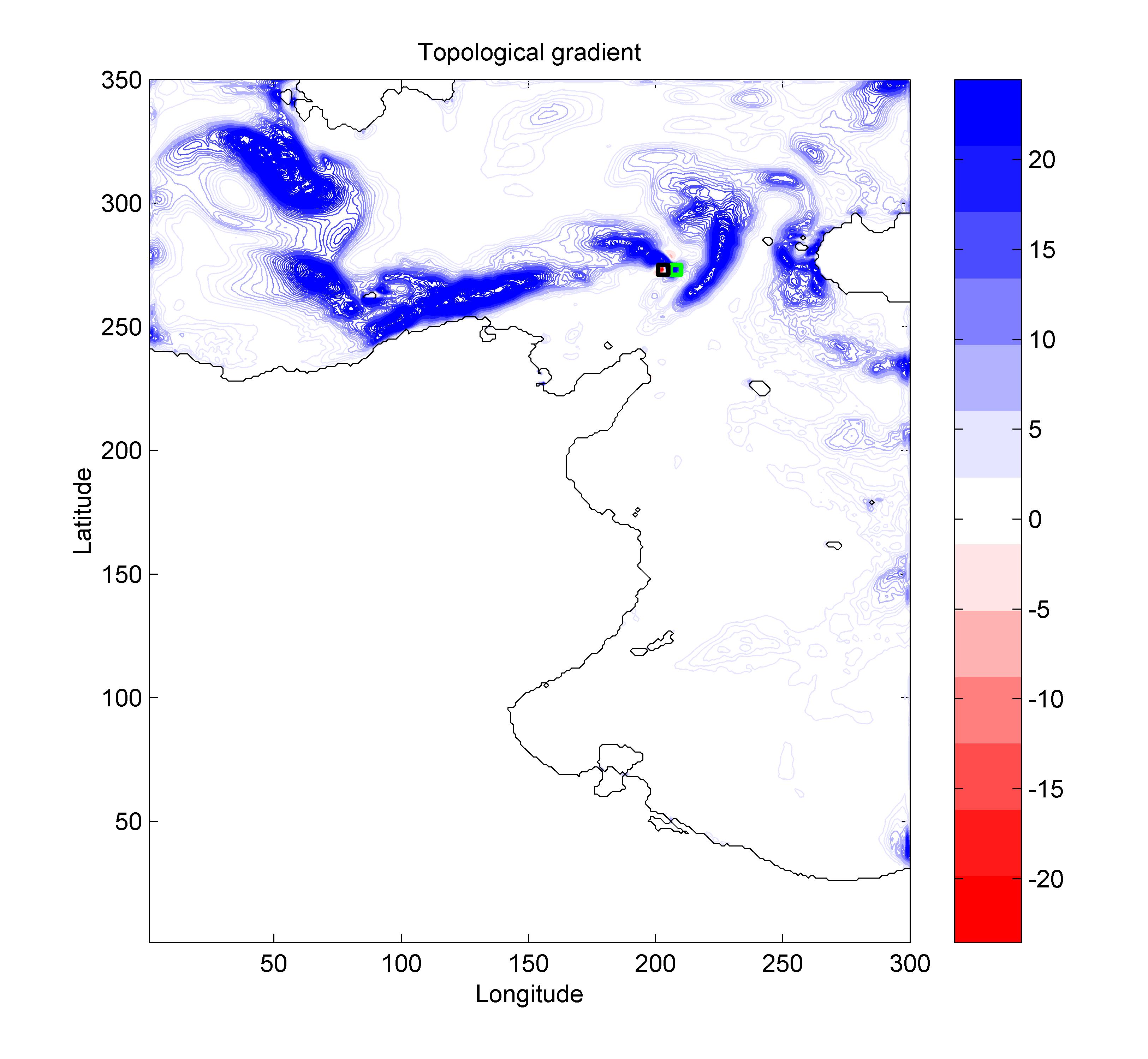} & 
        \includegraphics[width=50mm]{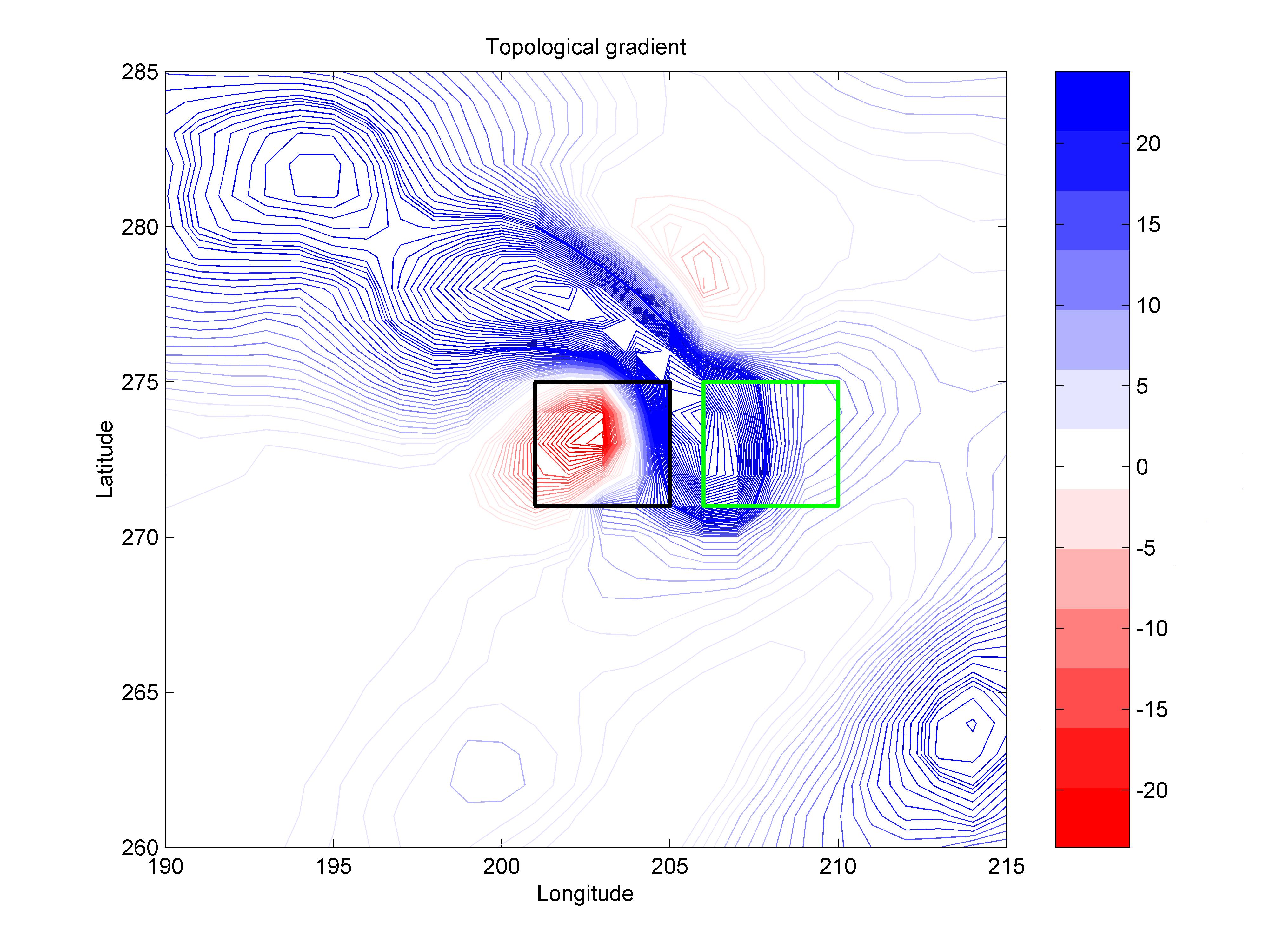}\\
        (a) Test 1  & (b) Zoom of (a)\\
        \includegraphics[width=50mm]{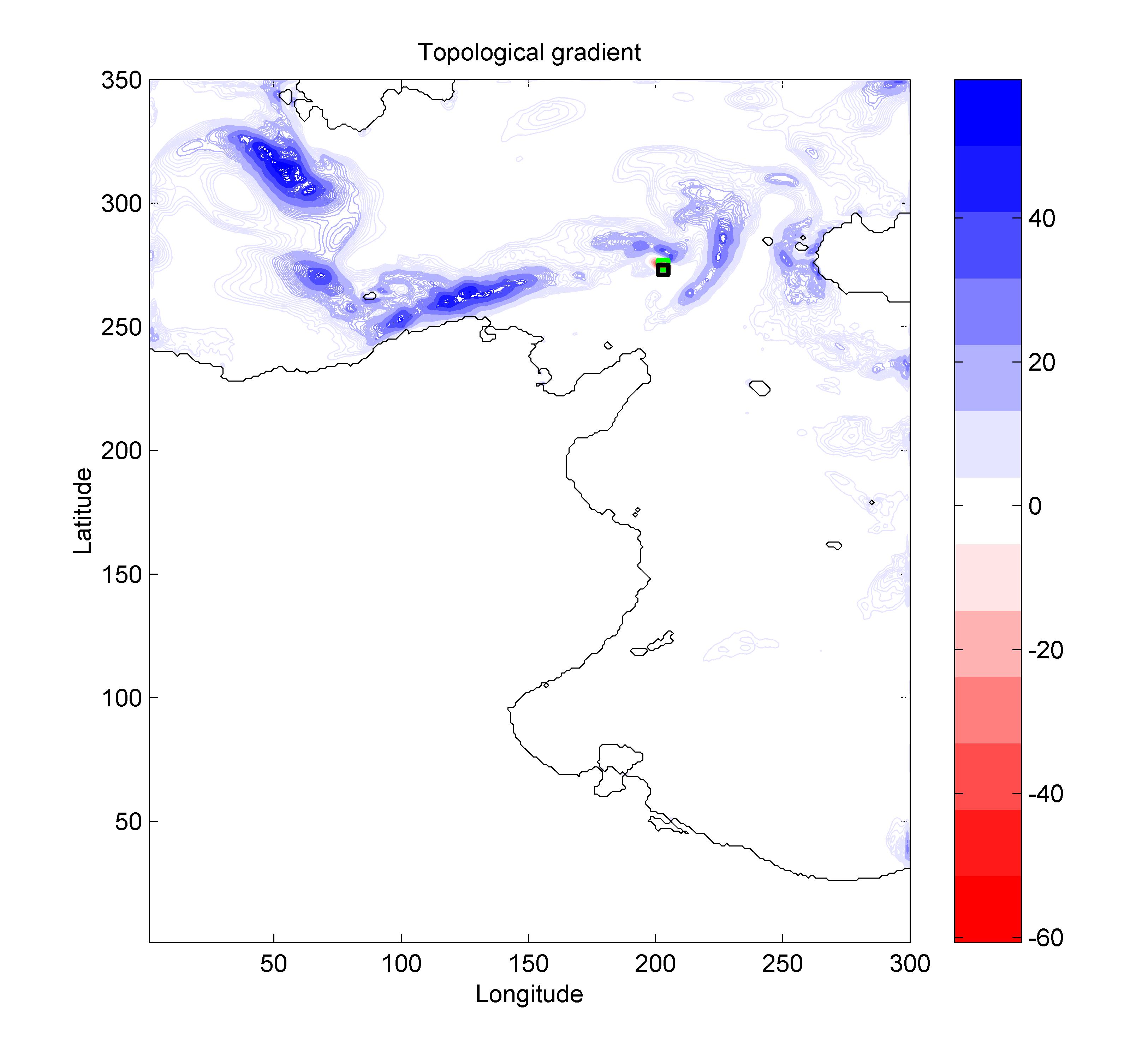}& 
        \includegraphics[width=50mm]{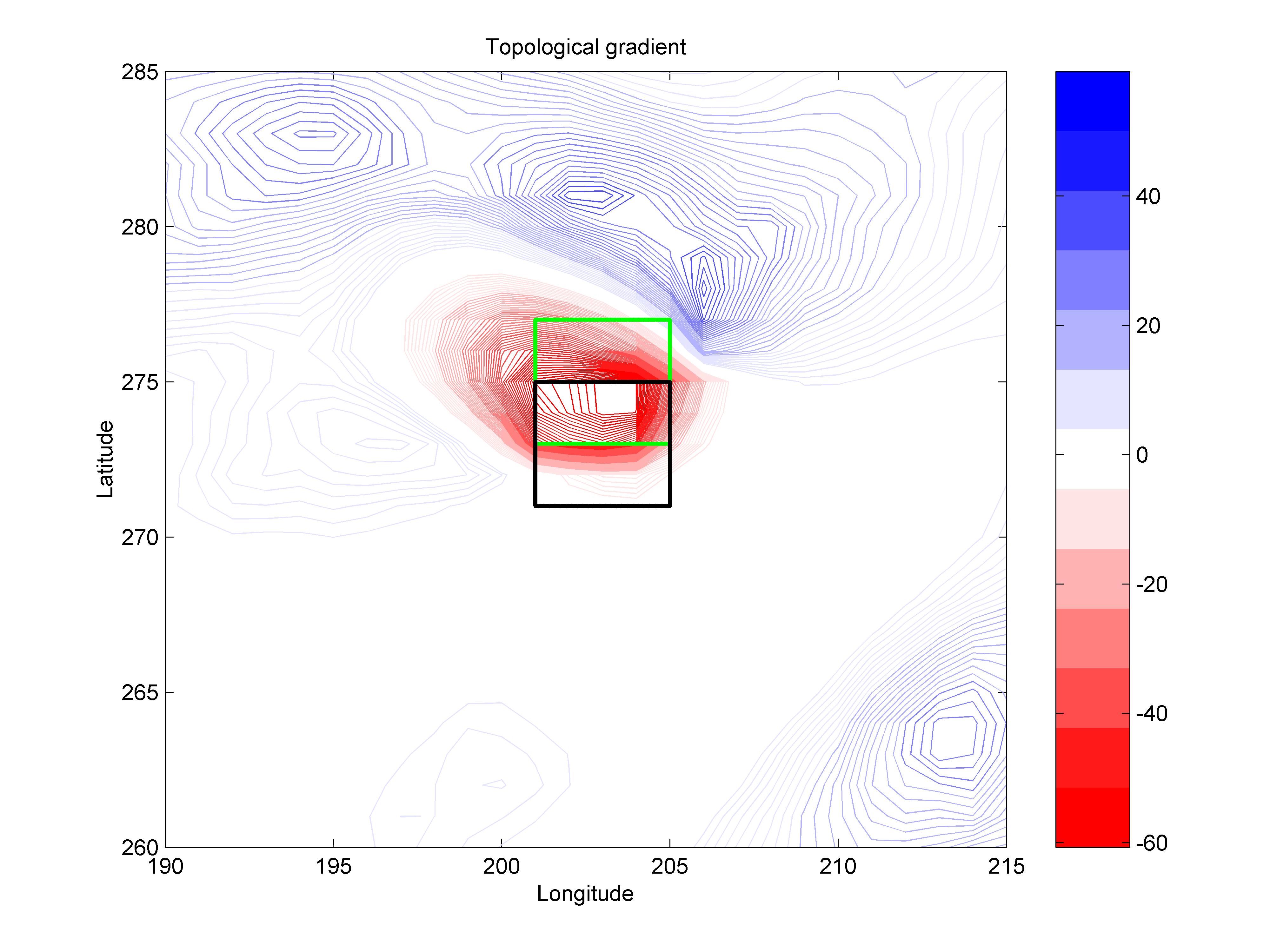}\\
        (c) Test 2  & (d) Zoom of (c) \\
        \includegraphics[width=50mm]{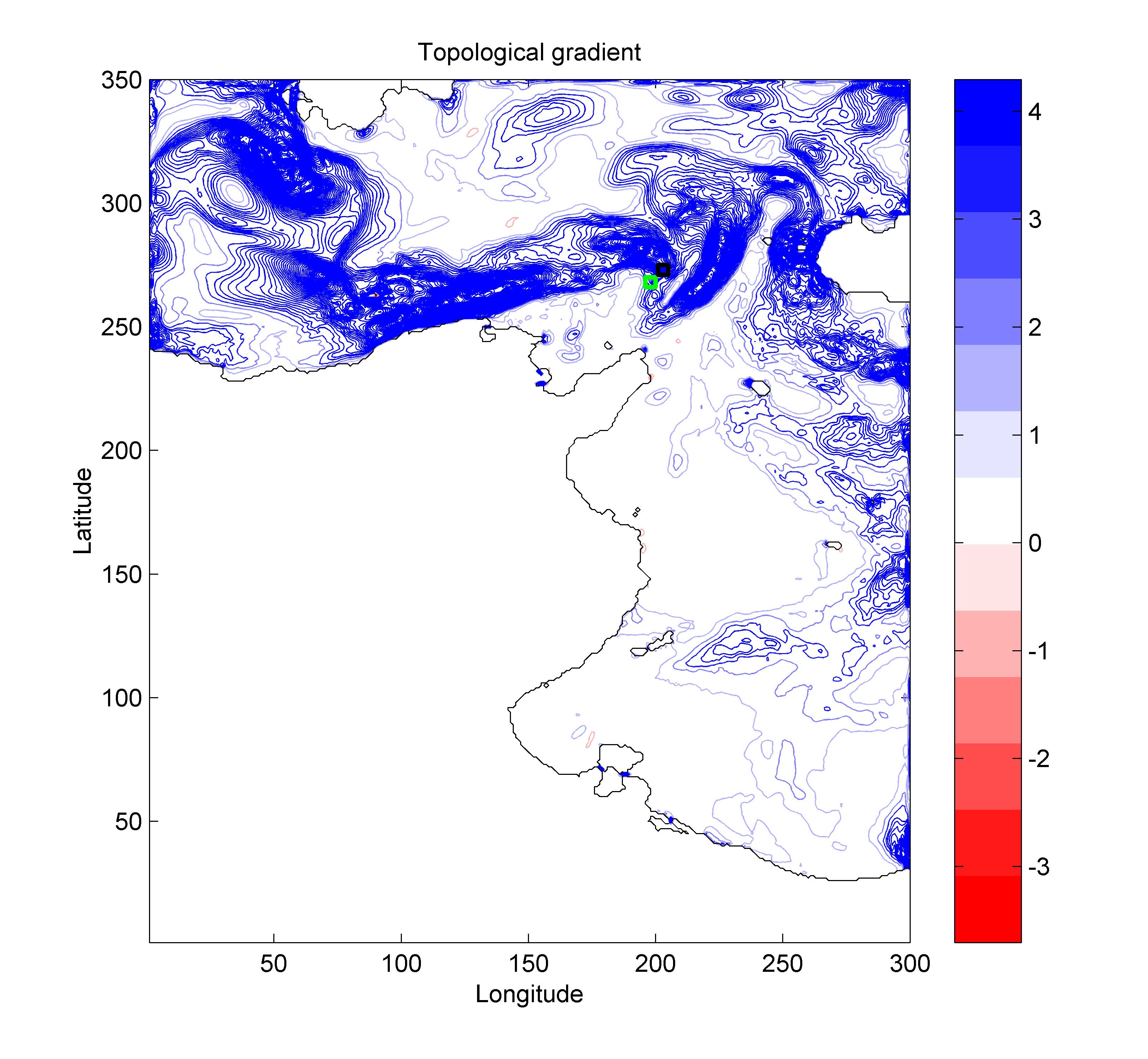}& 
        \includegraphics[width=50mm]{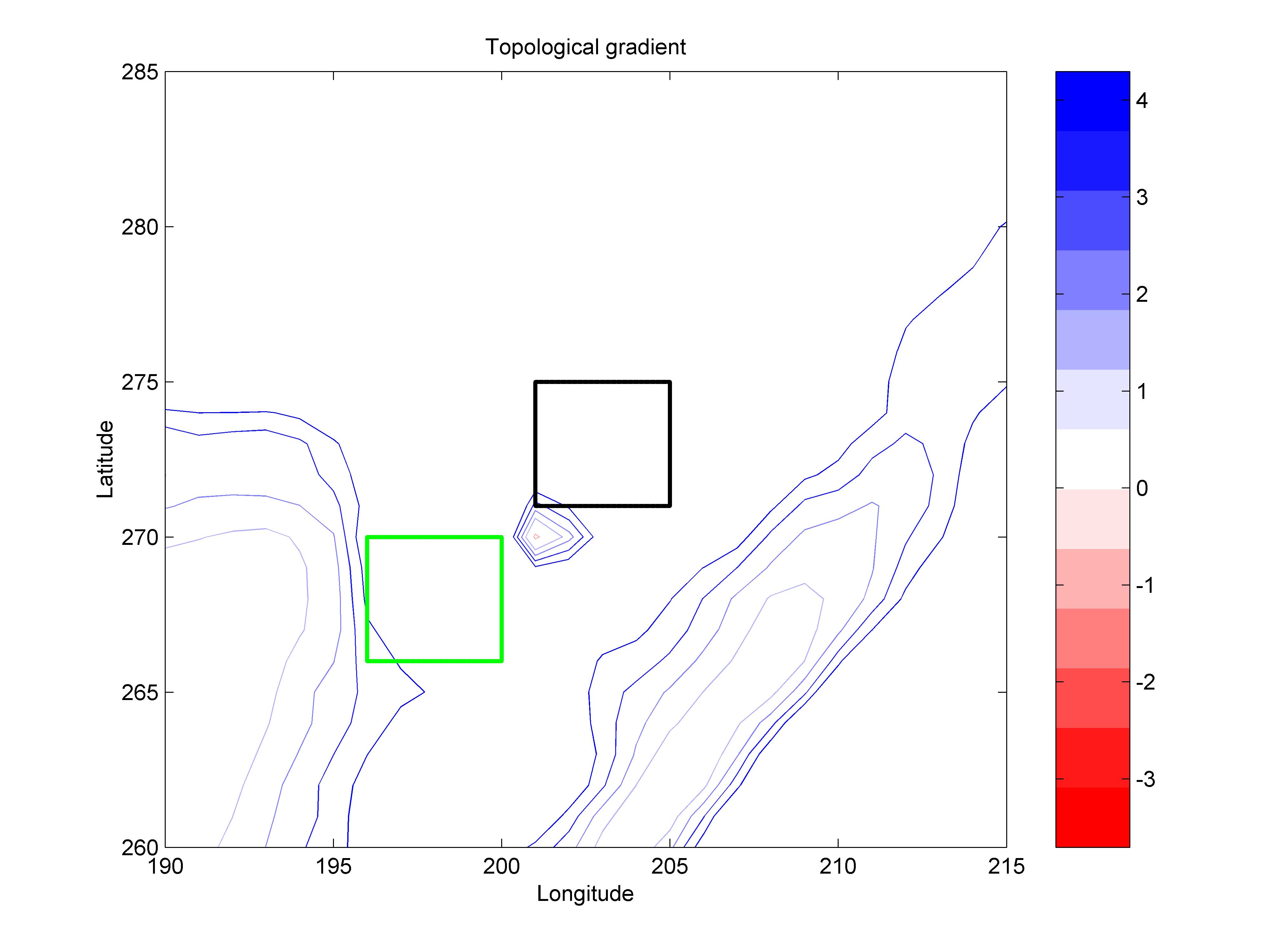}\\
        (e) Test 3  & (f) Zoom of (e) \\
        \includegraphics[width=50mm]{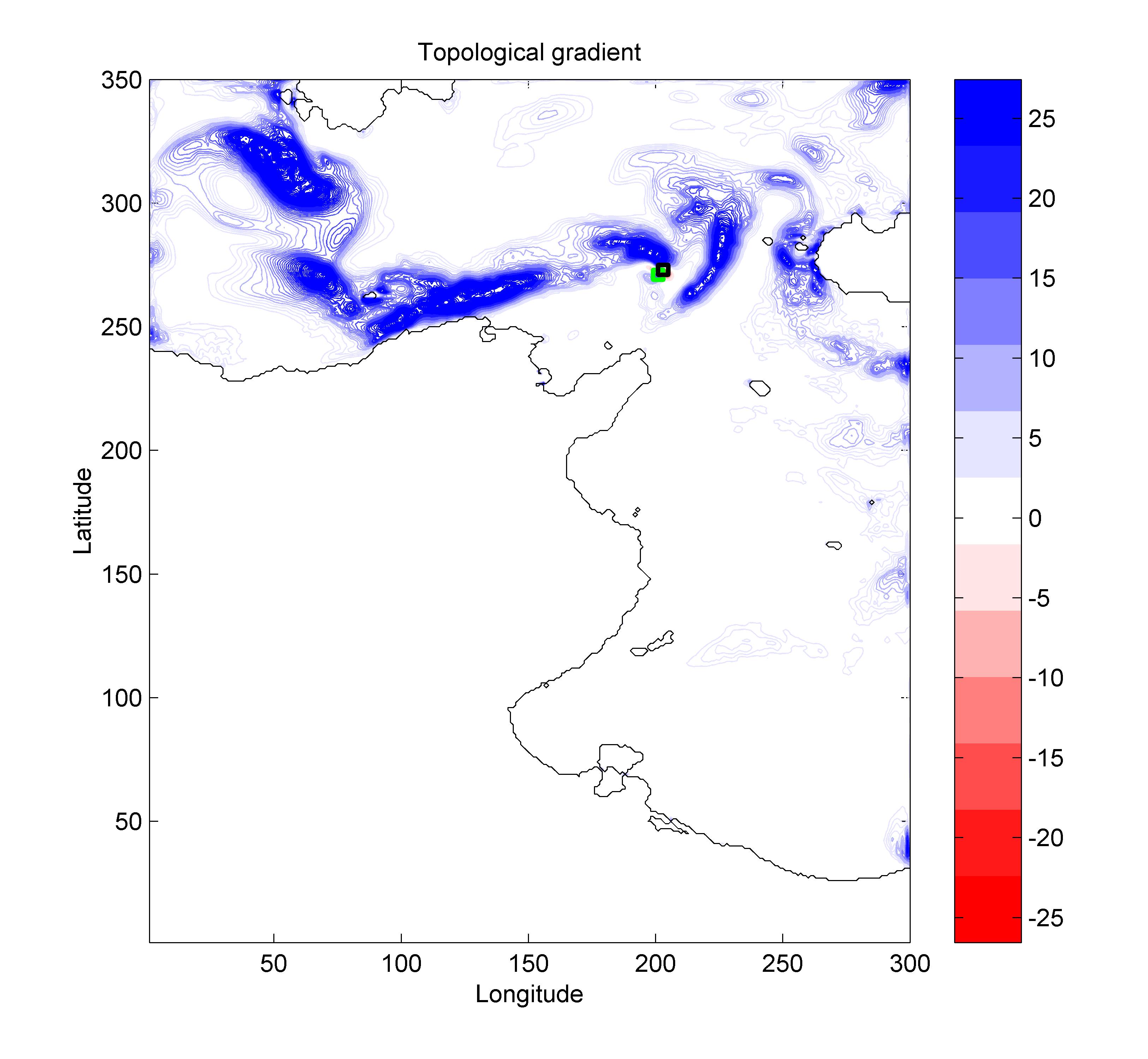}& 
        \includegraphics[width=50mm]{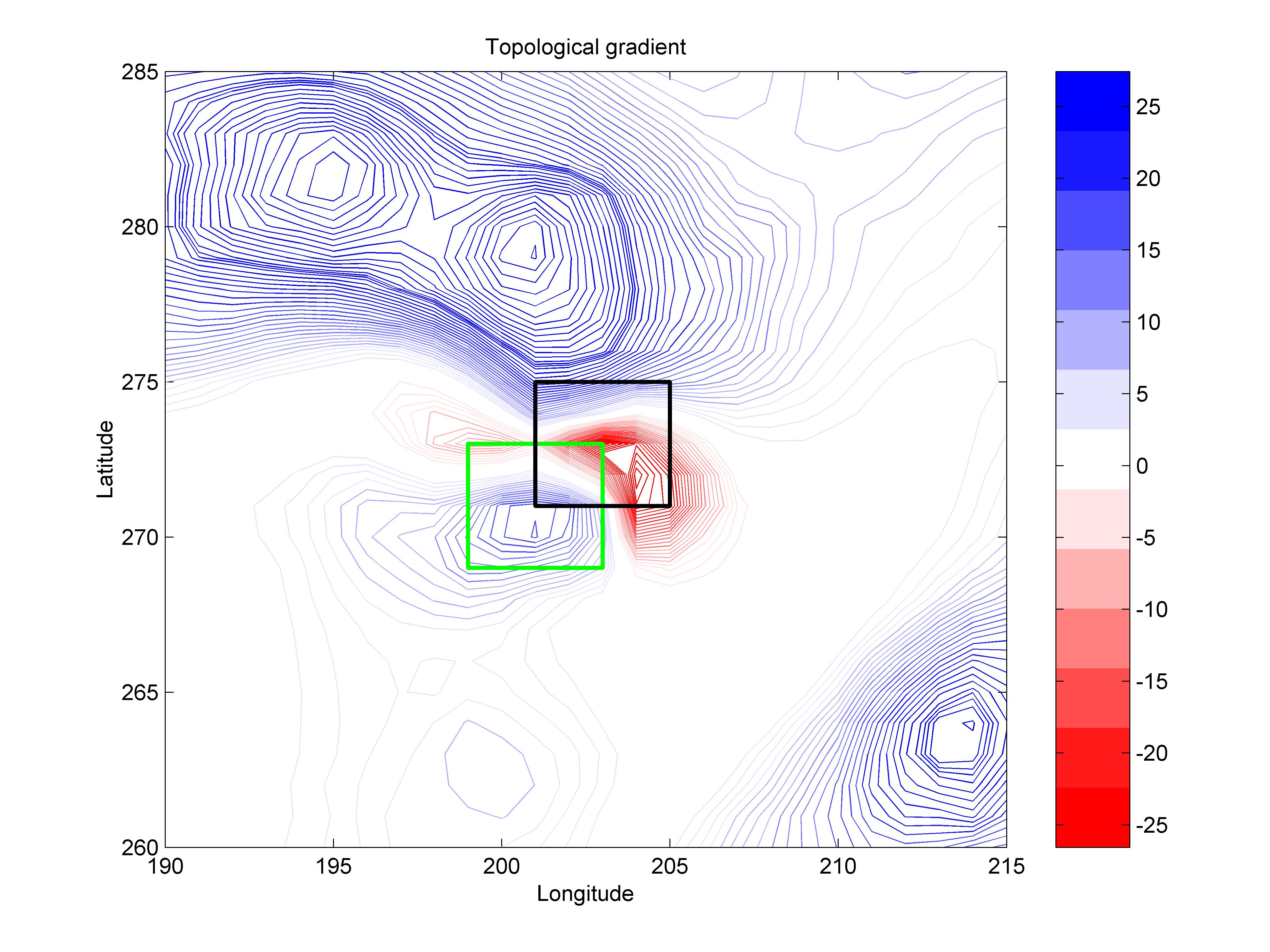}\\
        (g) Test 4  & (h) Zoom of (g) \\
        
    \end{tabular}
    \caption{Iso-Values of the topological gradient illustrating the influence of measurement area locations.}
    \label{multi-sub-meas-era}
\end{figure}

%%%%%%%%%%%%%%%%%%%%%%%%%%%%%%%%%%%%%%%%%%%%%%%%%%%%%%%%%%%%%%%%%%%%%%%%%%%%%%%%%%%%%%
\subsubsection{Example 7: Identification of obstacles from varying measurement area locations}\label{example7}
%%%%%%%%%%%%%%%%%%%%%%%%%%%%%%%%%%%%%%%%%%%%%%%%%%%%%%%%%%%%%%%%%%%%%%%%%%%%%%%%%%%%%%

In this example, we evaluate the performance of the proposed identification algorithm under the realistic constraint of limited measurement availability. Specifically, we examine how the spatial configuration and density of the measurement subdomain \( \Omega_0 \subset \Omega \) influence the detection of an unknown submerged obstacle \( \omega^* \) in the Mediterranean Sea. To this end, six test scenarios are considered, grouped into two categories.

\vspace{0.3cm}
\paragraph{\textbf{Tests 1–4: Influence of the location of the measurement subdomain}}  
In the first four tests, we investigate how the placement of a single measurement subdomain \( \Omega_0 \) affects the obstacle reconstruction. The true obstacle \( \omega^* \), shown as a black square in Figure~\ref{multi-sub-meas-era}, remains fixed across all tests. The measurement areas, marked as green squares in the figure, are configured as follows:

\begin{itemize}
    \item \textbf{Test 1: Very close horizontally.}  
    \( \Omega_0 \) is located at grid points \( (206{:}210) \times (271{:}275) \times (1{:}1) \), corresponding to the first vertical layer of the sigma-coordinate system. The depth ranges from 2.45~m to 3~m, depending on the bathymetry.  
    Horizontal separation from the obstacle: \( d = 1 \) grid point.

    \item \textbf{Test 2: Vertical crossing.}  
    The measurement domain is located at \( (201{:}205) \times (273{:}277) \times (1{:}1) \), with depths between 2~m and 2.9~m.  
    Vertical separation: \( d = -2 \) grid points.

    \item \textbf{Test 3: Very close in both directions.}  
    \( \Omega_0 \) is placed at \( (196{:}200) \times (266{:}270) \times (1{:}1) \), with a depth between 1.5~m and 2~m.  
    Horizontal and vertical separation: \( d = 1 \) grid point.

    \item \textbf{Test 4: Horizontal and vertical crossing.}  
    \( \Omega_0 \) is positioned at \( (199{:}203) \times (269{:}273) \times (1{:}1) \), with depth ranging from 1.5~m to 2.2~m.  
    Horizontal and vertical separation: \( d = -2 \) grid points.
\end{itemize}

\noindent The results in Figure~\ref{multi-sub-meas-era} show that the algorithm performs well when the measurement domain is sufficiently close to the obstacle. In Tests 1, 2, and 4, where the observation zone is either horizontally or both horizontally and vertically near \( \omega^* \), the topological gradient \( D_K \) exhibits a strong localized minimum that aligns accurately with the true obstacle location. This indicates successful identification, evidenced by the overlap between the black square and the zone of highest negative sensitivity.

In contrast, Test 3 demonstrates that when the observation region is vertically or diagonally distant from the obstacle, the detection fails. As illustrated in Figure~\ref{multi-sub-meas-era}(f), the topological gradient lacks a significant minimum near the obstacle, signaling poor reconstruction. This emphasizes the importance of measurement proximity: data collected close to the obstacle are substantially more informative for the inversion process.

While one solution is to relocate the observation zone closer to the expected obstacle, this is often impractical, especially in realistic geophysical settings where the obstacle location is unknown a priori. Furthermore, in cases where the obstacle acts as a radiating source, distant measurements may even be more appropriate. These factors suggest that adaptive measurement strategies, or the deployment of multiple observation regions, may offer more robust alternatives. In the next paragraph, we assess the impact of the number and density of measurement subdomains.

\vspace{0.3cm}
\paragraph{\textbf{Tests 5–6: Influence of the number and density of measurement areas}}  
In these tests, we explore how increasing the number and density of measurement subdomains improves obstacle detection. Here, \( \Omega_0 \) consists of multiple small regions located in the vicinity of \( \omega^* \):

\begin{itemize}
    \item \textbf{Test 5: 8 sub-measurement areas.}  
    The domain includes 8 non-overlapping subdomains, each of size \( 2~\text{km} \times 2~\text{km} \), with depths ranging approximately from 1.5~m to 3~m.  
    Horizontal and vertical separation: \( d = 3 \) grid points.

    \item \textbf{Test 6: 16 sub-measurement areas.}  
    This configuration includes 16 subdomains of the same size and depth, arranged more densely around the obstacle.  
    Horizontal and vertical separation: \( d = 1 \) grid point.
\end{itemize}

\noindent Figure~\ref{multi-sub-meas} shows a clear enhancement in detection quality as both the number and spatial density of measurement regions increase. In Test 6, the topological gradient displays a sharp, localized minimum that aligns precisely with the true obstacle, reflecting high detection accuracy. Test 5 also yields satisfactory results, albeit with a slightly less focused reconstruction.

\begin{figure} [!htbp]
     \centering
    \begin{tabular}{cc}
    \includegraphics[width=50mm]{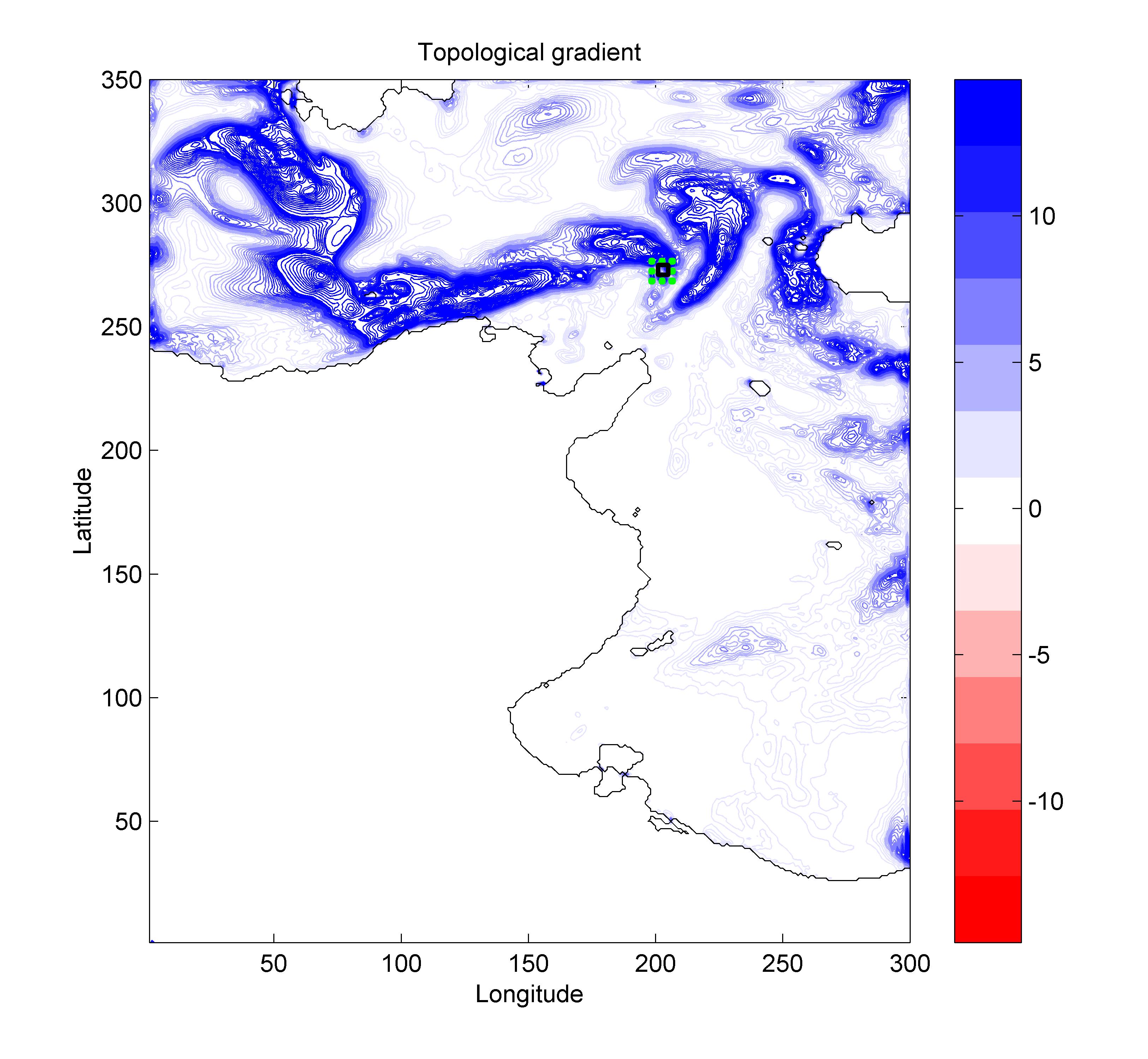}& \includegraphics[width=50mm]{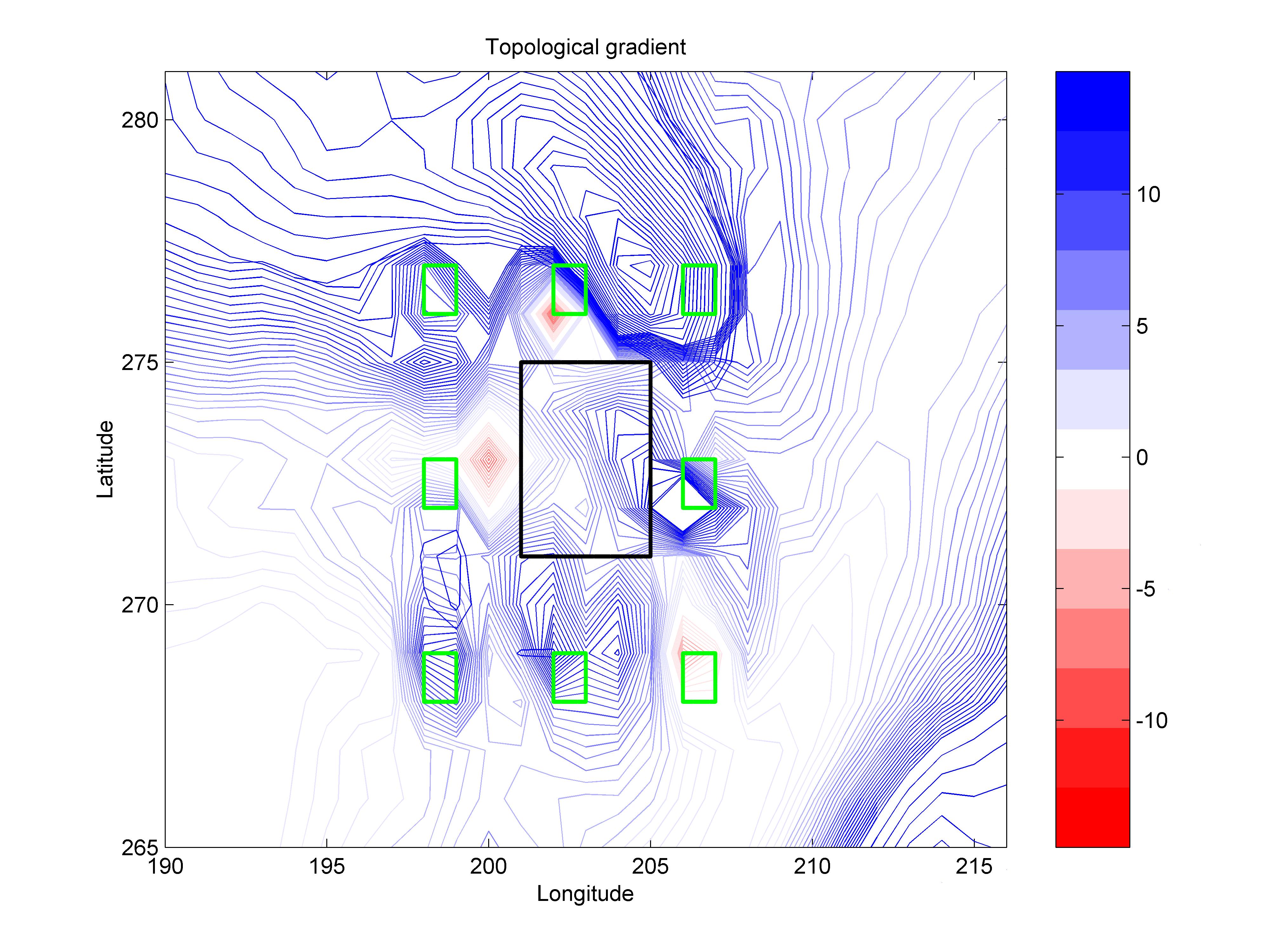}\\
(a) Test 5 & (b) Zoom of (a)\\
    \includegraphics[width=50mm]{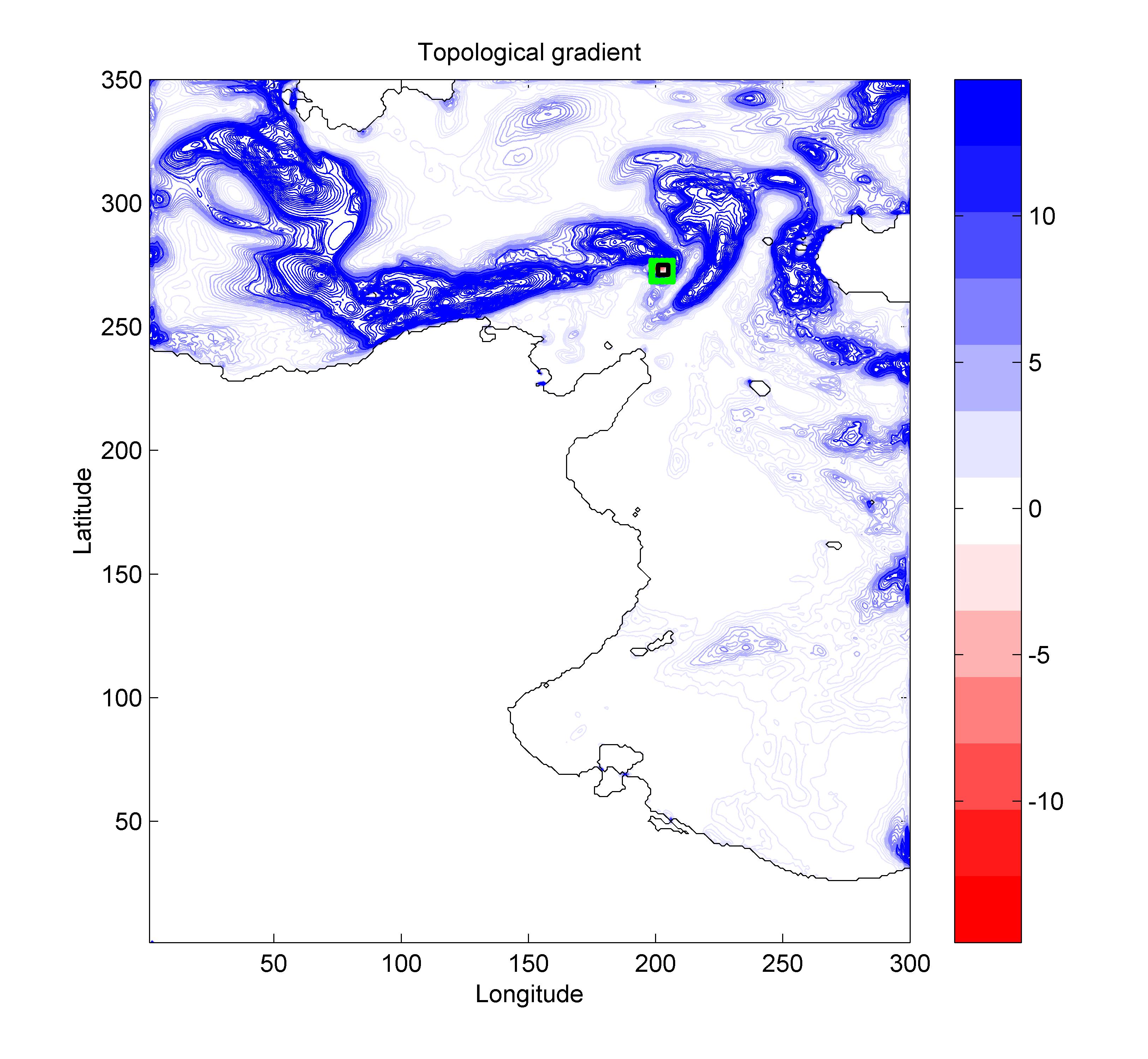}& \includegraphics[width=50mm]{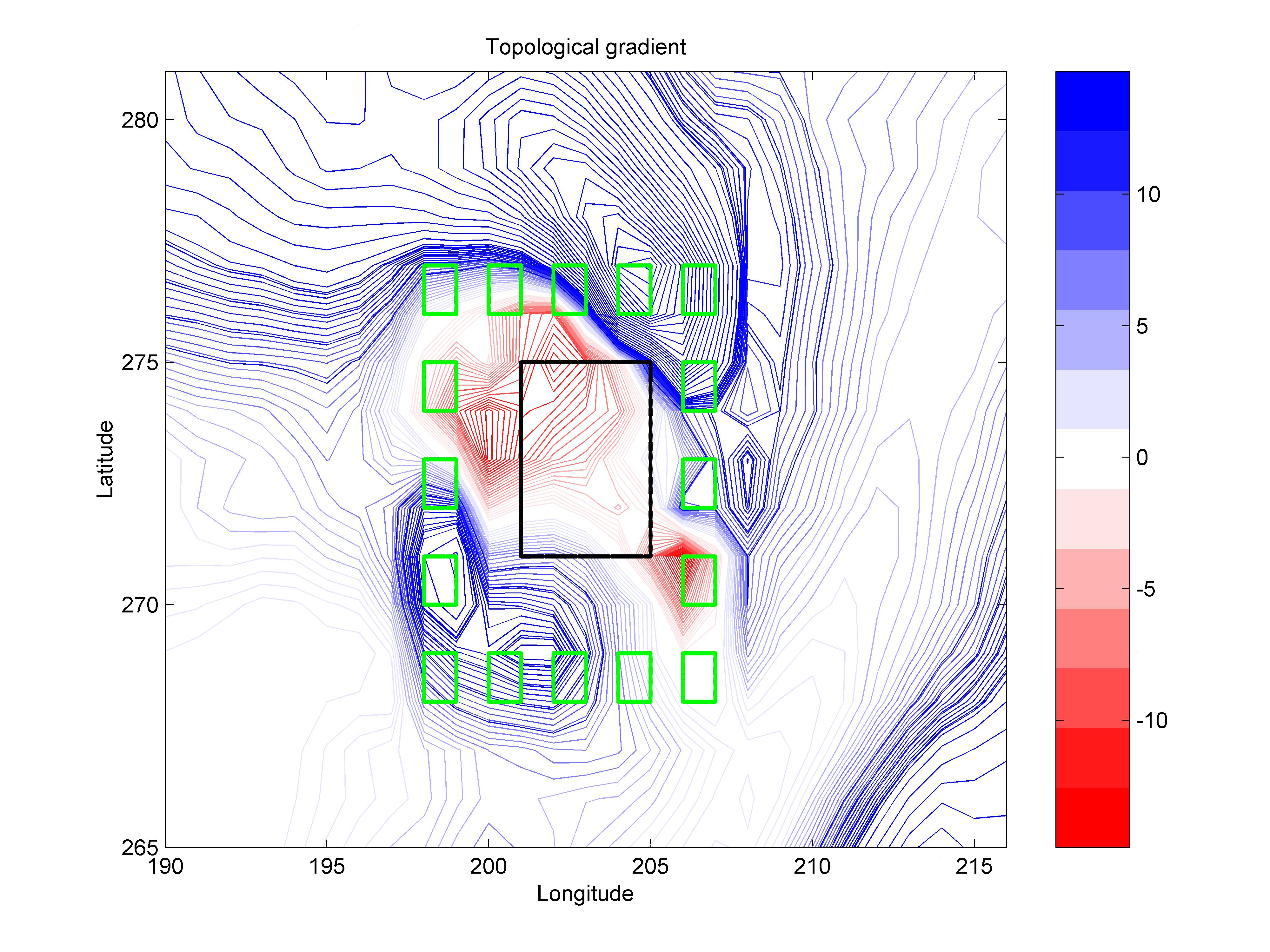}\\
(c) Test 6 & (d) Zoom of (c)
    \end{tabular}
    \caption{Iso-Values of the topological gradient illustrating the influence of the nombre of a mutiple measurement area.}\label{multi-sub-meas}
    \end{figure}

These results confirm that the effectiveness of the topological sensitivity method is strongly influenced by the configuration and distribution of the measurement domains. Increased spatial coverage and density enhance the resolution and reliability of the reconstruction by capturing more detailed flow perturbations induced by the obstacle.

\vspace{1em}

This example underscores the sensitivity of the identification process to the spatial arrangement of the measurement domain. Proximity and density are both critical factors: closer and more numerous subdomains lead to more accurate detection. In realistic scenarios—particularly in environmental and oceanographic applications—this supports the use of adaptive or exploratory measurement strategies, multi-zone observation networks, and the integration of prior information to guide sensor deployment. The combination of spatial diversity and high data density proves essential for robust obstacle identification under partial data conditions.

%%%%%%%%%%%%%%%%%%%%%%%%%%%%%%%%%%%%%%%%%%%%%%%%%%%%%%%%%%%%%%%%%%%%%%%%%%%%%%%%%%%%
\section{Conclusion}\label{sec:conclusion}
%%%%%%%%%%%%%%%%%%%%%%%%%%%%%%%%%%%%%%%%%%%%%%%%%%%%%%%%%%%%%%%%%%%%%%%%%%%%%%%%%%%%

In this work, we addressed the detection of multiple potential objects immersed in a three-dimensional fluid, along with their qualitative locations, using internal observations. The fluid flow is governed by the three-dimensional evolutionary Navier--Stokes equations. We established the uniqueness of the inverse problem and reformulated it as a topology optimization problem through the minimization of a regularized least-squares functional. The existence and stability of the corresponding optimization solution were also proved.  
Our detection approach relies on the concept of topological sensitivity. By deriving a topological asymptotic expansion via simplified mathematical analysis---avoiding the intricate truncation technique---we obtained the topological gradient, i.e., the leading term of the expansion. This gradient enabled the design of a fast, non-iterative reconstruction algorithm requiring no initial guess. The robustness and efficiency of the method were demonstrated through seven numerical experiments in a realistic Mediterranean Sea configuration, simulated with the \textsc{INSTMCOTRHD} ocean model.

%%%%%%%%%%%%%%%%%%%%%%%%%%%%%%%%%%%%%%%%%%%%%%%%%%%%%%
\section{Mathematical justification}\label{sec:jutification}
%%%%%%%%%%%%%%%%%%%%%%%%%%%%%%%%%%%%%%%%%%%%%%%%%%%%%%

%%%%%%%%%%%%%%%%%%%%%%%%%%%%%%%%%%%%%%%%%%%%%%%%%%%%%%%%%%%%%%%%%%%%%%%
\subsection{Proof of Theorem \ref{uniqueness}}\label{sec:uniqueness}
%%%%%%%%%%%%%%%%%%%%%%%%%%%%%%%%%%%%%%%%%%%%%%%%%%%%%%%%%%%%%%%%%%%%%%%

Following the strategy employed in the proof of Theorem 1.1 in \cite{AlvarezIPPIP2005} (see also \cite{doubova2007identification}), we proceed to establish Theorem~\ref{uniqueness}. Before presenting the proof, we recall a key preliminary result: a unique continuation property for the evolutionary Navier--Stokes equations, originally established by Fabre and Lebeau \cite{fabre1996prolongement} (see also \cite[Theorem 2.3]{AlvarezIPPIP2005}).

\begin{lemma}\label{u-c-l}
Let \( T_c > 0 \) and \( \O \subset \mathbb{R}^3 \) be a connected open set. Suppose that \( a \in L_{\text{loc}}^{\infty}(0,T_c;\L_{\text{loc}}^{\infty}(\O)) \) and \( b \in C\left([0, T_c] ; L_{\text{loc}}^r\left(\O, \mathbb{R}^{3 \times 3}\right)\right) \) is a matrix-valued function with \( r > 3 \). If \( (\varphi, h) \in L^2\left(0, T_c; \mathbf{H}_{\text{loc}}^1(\O)\right) \times L_{\text{loc}}^2(0,T_c;L_{\text{loc}}^2(\O)) \) is a solution of  
\begin{equation*}  
\left\{  
\begin{array}{rll}  
\displaystyle \frac {\partial \varphi}{\partial t} - \nu \Delta \varphi + (a \cdot \nabla) \varphi + b \,\varphi + \nabla h &= 0 &\text{ in } \O \times (0, T_c), \\  
\text{div}\, \varphi &= 0 &\text{ in } \O \times (0, T_c),  
\end{array}  
\right.  
\end{equation*}  
with \( \varphi = 0 \) in \( \mathcal{D}_0 \times (0, T_c) \), where \( \mathcal{D}_0 \) is an open subset of \( \O \), then \( \varphi \equiv 0 \) in \( \O \times (0, T_c) \), and \( h \) is constant.  
\end{lemma}

Using these ingredients, we split the proof of Theorem \ref{uniqueness} into two main steps:

\medskip

\noindent
\textbf{Step 1.}  
We first apply the unique continuation property (see Lemma \ref{u-c-l}) to show that
\[
u_1 = u_2 \; \text{ in } (\Omega \setminus \overline{\mathcal{S}}) \times (0, T_c),
\]
where $T_c\in(0,T)$ and \( \mathcal{S} \) denotes the smallest simply connected open set containing \( \omega_1^* \cup \omega_2^* \).

\medskip

\noindent
\textbf{Step 2.}  
We proceed by contradiction.  
Assume \( \omega_1^* \neq \omega_2^* \).  
By applying the unique continuation result (Lemma \ref{u-c-l}) again, we deduce that \( u_2 = 0 \) (respectively, \( u_1 = 0 \)) on \( \partial \Omega \times (0, T_c) \).  
This contradicts the fact that \( u_2 = \phi \neq 0 \) (respectively, \( u_1 = \phi \neq 0 \)) on \( \partial \Omega \times (0, T_c) \).  
Hence, we conclude that \( \omega_1^* = \omega_2^* \).\\

We now proceed to the proof of Theorem \ref{uniqueness}. Let us define the velocity and pressure differences by
$$
u_{1,2} := u_1 - u_2, \qquad \pi_{1,2} := \pi_1 - \pi_2,
$$
where \( (u_\ell,\,\pi_\ell) \) for \( \ell = 1,2 \) denote the solutions to the problem \eqref{u-ell}. Let \( \mathcal{S} \) be the smallest open subset of \( \Omega \) such that \( \omega_1^* \cup \omega_2^* \subset \mathcal{S} \) and such that \( \Omega \setminus \overline{\mathcal{S}} \) remains connected. In the particular case where \( \Omega \setminus \overline{\omega_1^* \cup \omega_2^*} \) is already connected, we have \( \mathcal{S} = \omega_1^* \cup \omega_2^* \). Then, from \eqref{u-ell} and \eqref{obsevtion-condition-1}, it follows that the pair \( (u_{1,2}, \pi_{1,2}) \) satisfies the following system:
\begin{equation}\label{u-ell-diff}
    \left\{
    \begin{array}{rll}
        \displaystyle \frac{\partial u_{1,2}}{\partial t} - \nu \Delta u_{1,2} +(u_{12}\cdot\nabla)u_1+(u_2\cdot\nabla)u_{1,2} + \nabla \pi_{1,2}
        &=0, &\quad \text{in } (\Omega \setminus \overline{\S}) \times (0, T), \\[8pt]
        \text{div}\, u_{1,2} &= 0, &\quad \text{in } (\Omega \setminus \overline{\S}) \times (0, T), \\[8pt]
        u_{1,2} &=0, &\quad \text{on } \partial\Omega \times (0, T),\\[8pt]
        u_{1,2} &=0, &\quad \text{in } \Omega_0 \times (0, T).
    \end{array}
    \right.
    \end{equation}
To apply the unique continuation property stated in Lemma \ref{u-c-l}, we first verify that the system \eqref{u-ell-diff} satisfies its assumptions. Since \( \phi \in  \mathcal{C}^1([0,T];\H^{3/2}(\partial \Omega)) \), classical existence results for the Navier--Stokes equations (see, e.g., \cite[Lemma 25.2, p. 144]{tartar2006introduction}) ensure that there exists \( T_c \in (0, T) \) such that a unique solution \( (u_\ell, \pi_\ell) \) to problem \eqref{u-ell} exists on the interval \( [0, T_c] \), with
\[
u_\ell \in \mathcal{C}\left([0, T_c], \H^2(\Omega \setminus \overline{\mathcal{S}})\right).
\]
Define \( b_{ij} := \partial_j u_1^i \) for \( 1 \leq i,j \leq 3 \). Then \( b_{ij} \in \mathcal{C}\left([0, T_c] ; \H^1(\Omega \setminus \overline{\mathcal{S}})\right) \). By the Sobolev embedding theorem, we have:
\[
u_2 \in \mathcal{C}\left([0, T_c], \H^2(\Omega \setminus \overline{\mathcal{S}})\right) \hookrightarrow \mathcal{C}\left([0, T_c], \L^{\infty}(\Omega \setminus \overline{\mathcal{S}})\right),
\]
and the matrix-valued function \( b := (b_{ij})_{1 \leq i,j \leq 3} \) satisfies:
\[
b \in \mathcal{C}\left([0, T_c], \H^1(\Omega \setminus \overline{\mathcal{S}}; \mathbb{R}^{3 \times 3})\right) \hookrightarrow \mathcal{C}\left([0, T_c], \L^6(\Omega \setminus \overline{\mathcal{S}}; \mathbb{R}^{3 \times 3})\right).
\]
Hence, all the assumptions of Lemma \ref{u-c-l} are fulfilled. We therefore conclude that \( u_{1,2} \equiv 0 \) in \( (\Omega \setminus \overline{\mathcal{S}}) \times (0, T_c) \), i.e., \[u_1 = u_2  \text{   in   } (\Omega \setminus \overline{\mathcal{S}}) \times (0, T_c) \].

We now proceed by contradiction, assuming that \( \omega_1^* \setminus \overline{\omega_2^*} \) is nonempty. In what follows, we suppose that \( \mathcal{S} \setminus \overline{\omega_2^*} \) is connected. If this is not the case, one can simply replace \( \mathcal{S} \setminus \overline{\omega_2^*} \) with one of its connected components in the argument below. We observe that \( u_2 \) satisfies the following system:
\begin{equation}\label{eq-u1-m}
    \frac{\partial u_2}{\partial t} - \nu \Delta u_2 + (u_2 \cdot \nabla) u_2 + \nabla \pi_2 = 0 \quad \text{in } (\mathcal{S} \setminus \overline{\omega_2^*}) \times (0, T_c).
\end{equation}

To clarify the proof, we begin by analyzing the particular case where \( \mathcal{S} = \omega_1^* \cup \omega_2^* \), and then extend the argument to the more general case where \( \mathcal{S} \neq \omega_1^* \cup \omega_2^* \).\\

\paragraph{$\star$ \emph{The case where \( \mathcal{S} = \omega_1^* \cup \omega_2^* \).}} In this setting, we distinguish two subcases depending on whether the set \( \omega_1^* \setminus \overline{\omega_2^*} \) is a Lipschitz domain or not (see Figure \ref{domainHM}).

\begin{figure} [!htbp]
     \centering
    \includegraphics[width=90mm]{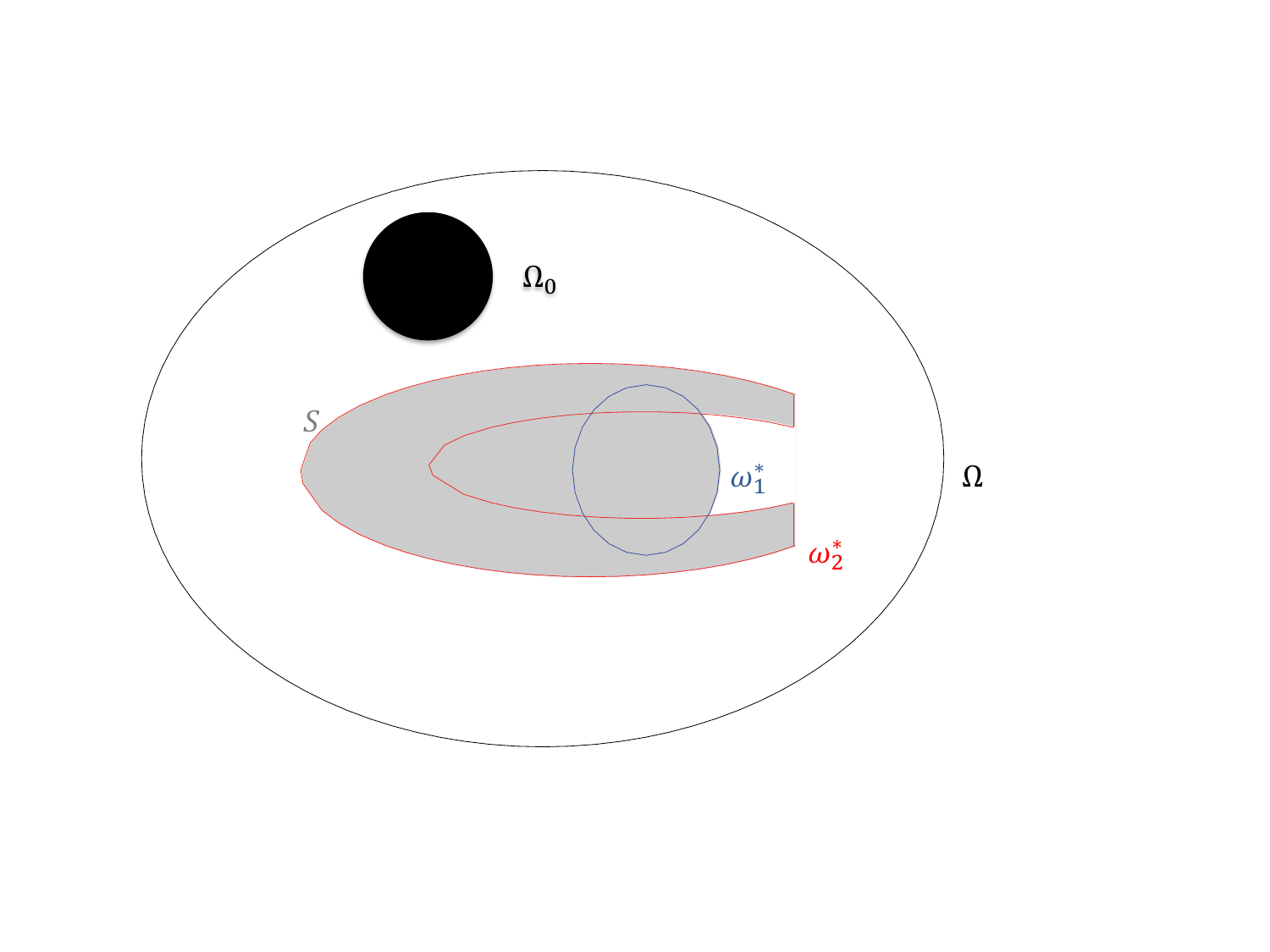}
    \caption{The set $\mathcal{S}$.}\label{domainHM}
    \end{figure}

\noindent$\bullet$ \textbf{Case 1:} Assume first that \( \omega_1^* \setminus \overline{\omega_2^*} \) is a Lipschitz domain. 
Multiplying equation~\eqref{eq-u1-m} by \( u_2 \), integrating by parts over \( \omega_1^* \setminus \overline{\omega_2^*} \), and using that \( u_2 = 0 \) on \( \partial \omega_2^*\times(0,T_c) \), for all  $t \in (0, T_c)$, we obtain
\begin{equation}\label{eq-oo}
\frac{\mathrm{d}}{\mathrm{d}t} \int_{\omega_1^* \setminus \overline{\omega_2^*}} |u_2(x, t)|^2 \, \dx = -\int_{\omega_1^* \setminus \overline{\omega_2^*}} |\nabla u_2(x, t)|^2 \, \dx -\int_{\omega_1^* \setminus \overline{\omega_2^*}} \left[ (u_2(x, t) \cdot \nabla) u_2(x, t) \right] \cdot u_2(x, t)\, \dx.
\end{equation}
On the other hand, integrating by parts in \( \omega_1^* \setminus \overline{\omega_2^*} \) and using the boundary condition \( u_2 = 0 \) on \( \partial \omega_2^* \), we have for all \( t \in (0, T_c) \),
\begin{align*} \int_{\omega_1^* \setminus \overline{\omega_2^*}} &\left[ (u_2(x, t) \cdot \nabla)u_2(x, t) \right] \cdot u_2(x, t)\, \dx=\frac{1}{2}\int_{\omega_1^* \setminus \overline{\omega_2^*}} u_2(x, t)\cdot\nabla \big|u_2(x, t)\big|^2\dx\\ 
&= -\frac{1}{2}\int_{\omega_1^* \setminus \overline{\omega_2^*}} \text{div}\, u_2(x, t)\; \big|u_2(x, t)\big|^2 \dx+ \frac{1}{2}\int_{{\partial\omega_1^* \setminus \overline{\omega_2^*}}} \big(u_2(x, t)\cdot\textbf{n}\big) \big|u_2(x, t)\big|^2 \ds.
\end{align*}
Since \( u_1 = u_2 \) in \( (\Omega \setminus \overline{\mathcal{S}}) \times (0, T_c) \), and \( u_1 = 0 \) on \( \partial \omega_1^*\times (0, T_c) \), it follows that \( u_2 = 0 \) on \( (\partial\omega_1^* \setminus \overline{\omega_2^*})\times (0, T_c) \). Moreover, using the incompressibility condition \( \text{div}\, u_2 = 0 \) in \( \Omega \times (0, T) \), we deduce that the nonlinear convective term vanishes:
\begin{align} \label{just}  \int_{\omega_1^* \setminus \overline{\omega_2^*}} \left[ (u_2(x, t) \cdot \nabla) u_2(x, t) \right] &\cdot u_2(x, t)\, \dx=0\quad\text{for all } t\in(0,T_c).
\end{align}
Inserting this identity into \eqref{eq-oo}, we obtain
\begin{equation}\label{eq-oo002}
\frac{\mathrm{d}}{\mathrm{d}t} \int_{\omega_1^* \setminus \overline{\omega_2^*}} |u_2(x, t)|^2 \, \dx = -\int_{\omega_1^* \setminus \overline{\omega_2^*}} |\nabla u_2(x, t)|^2 \, \dx \qquad \text{for all } t \in (0, T_c).
\end{equation}
This shows that the function
\[
\mathcal{Z}(t) := \int_{\omega_1^* \setminus \overline{\omega_2^*}} |u_2(x, t)|^2 \, \dx
\]
is non-negative and non-increasing on \( [0, T_c] \). Since \( u_2(\cdot,0) = 0 \), it follows that \( \mathcal{Z}(0) = 0 \), and hence \( \mathcal{Z}(t) = 0 \) for all \( t \in [0, T_c] \). Consequently,
\[
u_2 = 0 \quad \text{in } (\omega_1^* \setminus \overline{\omega_2^*}) \times (0, T_c).
\]
By the unique continuation result in Lemma \ref{u-c-l} (with $b\equiv0)$, one can deduce that
\[
u_2 = 0 \quad \text{in } (\Omega \setminus \overline{\omega_2^*}) \times (0, T_c),
\]
which contradicts the boundary condition \( u_2 = \phi \), with \( \phi \neq 0 \) on \( \partial \Omega \times (0, T) \). Therefore, it must be that \( \omega_1^* \setminus \overline{\omega_2^*} = \emptyset \). By a symmetric argument, we similarly obtain \( \omega_2^* \setminus \overline{\omega_1^*} = \emptyset \), and hence conclude that
\[
\omega_1^* = \omega_2^*.
\]

\medskip

\noindent$\bullet$ \textbf{Case 2:} We now address the situation where the set \( \omega_1^* \setminus \overline{\omega_2^*} \) is not necessarily a Lipschitz domain. In this case, the integration by parts leading to equation \eqref{eq-oo002} is not directly justified. To overcome this difficulty and still derive equation \eqref{eq-oo002}, we use a density argument: the space \( \mathcal{D}(\omega_1^* \setminus \overline{\omega_2^*}) \) is dense in \( \mathbf{H}_0^1(\omega_1^* \setminus \overline{\omega_2^*}) \). For this, let us recall the definition of the space \( \mathbf{H}_0^1 \) as given in \cite[Definition 3.3.43]{HenrotBook2006}, where \( \mathbf{H}_0^1(\omega_1^* \setminus \overline{\omega_2^*}) \) is defined as the set of functions in \( \mathbf{H}^1(\omega_1^* \setminus \overline{\omega_2^*}) \) whose extension by zero to \( \Omega \) belongs to \( \mathbf{H}^1(\Omega) \). We then multiply equation \eqref{eq-u1-m} by a test function \( \varphi \in \mathcal{C}(0,T_c;\mathcal{D}(\omega_1^* \setminus \overline{\omega_2^*})) \), and using standard integration by parts (valid for smooth test functions), we obtain for all \( t \in (0, T_c) \),
\begin{align}
    \begin{split}
    \label{eq-oo-1}
&\int_{\omega_1^* \setminus \overline{\omega_2^*}} \frac{\partial u_2}{\partial t}(x,t) \cdot \varphi(x,t) \, \dx + \int_{\omega_1^* \setminus \overline{\omega_2^*}} \nabla u_2(x,t) : \nabla \varphi(x,t) \, \dx\\
&\qquad+\int_{\omega_1^* \setminus \overline{\omega_2^*}} \big[(u_2(x,t)\cdot\nabla)u_2(x,t)\big] \cdot \varphi(x,t) \, \dx=\int_{\omega_1^* \setminus \overline{\omega_2^*}} \pi_2\,\text{div}\,\varphi(x,t) \, \dx.
\end{split}
 \end{align}
Then, it suffices to show that ${u_2}(\cdot,t)_{| \omega_1^* \setminus \overline{\omega_2^*}} \in \mathbf{H}_0^1\left(\omega_1^* \setminus \overline{\omega_2^*}\right)$ for all $t \in (0, T_c)$. To this end, we rewrite equation~\eqref{eq-oo-1} with $\varphi = \varphi_n$, where the sequence $\left(\varphi_n(\cdot, t)\right)_{n \in \mathbb{N}} \subset \mathcal{D}\left(\omega_1^* \setminus \overline{\omega_2^*}\right)$ satisfies
\[
\varphi_n (\cdot,t)\to {u_2(\cdot,t)}_{| \omega_1^* \setminus \overline{\omega_2^*}} \quad \text{in } \mathbf{H}^1\left(\omega_1^* \setminus \overline{\omega_2^*}\right), \quad \text{for all } t \in (0, T_c).
\]
Passing to the limit as $n \to \infty$ yields equation~\eqref{eq-oo}, and we conclude the proof as in the first case.\\

\noindent It remains to justify that ${u_2(\cdot,t)}_{| \omega_1^* \setminus \overline{\omega_2^*}} \in \mathbf{H}_0^1\left(\omega_1^* \setminus \overline{\omega_2^*}\right)$ for all $t \in (0, T_c)$. We define $\widetilde{u_2}$ as the extension of $u_2$ by zero in $\omega_2^* \times (0, T_c)$, that is,
\[
\widetilde{u_2} = 
\begin{cases}
u_2, & \text{in } \Omega \setminus \overline{\omega_2^*} \times (0, T_c), \\
0, & \text{in } \omega_2^* \times (0, T_c).
\end{cases}
\]
Since $u_2 = \mathbf{0}$ on $\partial \omega_2^* \times (0, T_c)$, we have $\widetilde{u_2}(\cdot,t) \in \mathbf{H}^1(\Omega)$ for all $t \in (0, T_c)$. We now consider the restriction $\left.\widetilde{u_2}(\cdot,t)\right|_{\omega_1^*} \in \mathbf{H}^1(\omega_1^*)$ for all $t \in (0, T_c)$, and extend it again by zero to $(\Omega \setminus \overline{\omega_1^*}) \times (0, T_c)$.\\

\noindent By construction, $\widetilde{u_2} = 0$ on $(\partial \omega_1^* \setminus \overline{\omega_2^*}) \times (0, T_c)$, and also on $(\partial \omega_1^* \cap \omega_2^*) \times (0, T_c)$, so the total extension
\[
\widetilde{u_2} = 
\begin{cases}
u_2, & \text{in } (\omega_1^* \setminus \overline{\omega_2^*}) \times (0, T_c), \\
0, & \text{in } \omega_2^* \times (0, T_c), \\
0, & \text{in } (\Omega \setminus \overline{\omega_1^*}) \times (0, T_c),
\end{cases}
\]
belongs to $\mathbf{H}^1(\Omega)$ for all $t \in (0, T_c)$. Thus, the restriction $\widetilde{u_2}(\cdot,t)_{|\omega_1^* \setminus \overline{\omega_2^*}} = {u_2(\cdot,t)}_{| \omega_1^* \setminus \overline{\omega_2^*}}$ belongs to $\mathbf{H}_0^1\left(\omega_1^* \setminus \overline{\omega_2^*}\right)$ for all $t \in (0, T_c)$, which completes the proof. Therefore, we have $\omega_1^* \setminus \overline{\omega_2^*}$. By symmetry, \( \omega_2^* \setminus \overline{\omega_1^*} = \emptyset \), and thus we conclude that \( \omega_1^* = \omega_2^* \).\\

\paragraph{$\star$ \emph{The case where \( \mathcal{S} \neq \omega_1^* \cup \omega_2^* \).}} 
We now highlight the differences and difficulties that arise in the general case where \( \mathcal{S} \) is not necessarily equal to \( \omega_1^* \cup \omega_2^* \). The key idea is to replace \( \omega_1^* \) with \( \mathcal{S} \setminus \overline{\left( \omega_2^* \setminus \overline{\omega_1^*} \right)} \). Indeed, we cannot directly work with \( \omega_1^* \), as we lack information on \( u_2 \) along \( \partial \omega_1^* \setminus \partial \mathcal{S} \); we only know that \( u_1 = u_0 \) in \( \Omega \setminus \overline{\mathcal{S}} \).

If \( \mathcal{S} \setminus \overline{\omega_2^*} \) is a Lipschitz domain, we proceed exactly as in the previous Lipschitz case, replacing \( \omega_1^* \setminus \overline{\omega_2^*} \) with \( \mathcal{S} \setminus \overline{\omega_2^*} \). However, in the general case where \( \mathcal{S} \setminus \overline{\omega_2^*} \) is not necessarily Lipschitz, we can no longer prove that \( u_2|_{\mathcal{S} \setminus \overline{\omega_2^*}} \in \mathbf{H}^1_0\left(\mathcal{S} \setminus \overline{\omega_2^*}\right) \) (for all \( t \in (0, T_c) \)) using the same approach. Specifically, even after extending \( u_2 \) by zero in \( \omega_2^* \), we cannot assert that the extension of \( \left.\widetilde{u}_2\right|_{\mathcal{S} \setminus \overline{\left( \omega_2^* \setminus \overline{\omega_1^*} \right)}} \) by zero in \( \Omega \) belongs to \( \mathbf{H}^1(\Omega) \). To overcome this, we enlarge the domain \( \mathcal{S} \setminus \overline{\omega_2^*} \) inside \( \omega_2^* \) to a smooth (at least Lipschitz) domain \( \widetilde{\omega}_1^* \) (see Figure~\ref{domain00}
). We then replicate the argument used in the Lipschitz case, applying a density argument and replacing \( \omega_1^* \) by \( \widetilde{\omega}_1^* \).
\begin{figure} [!htbp]
     \centering
    \includegraphics[width=90mm]{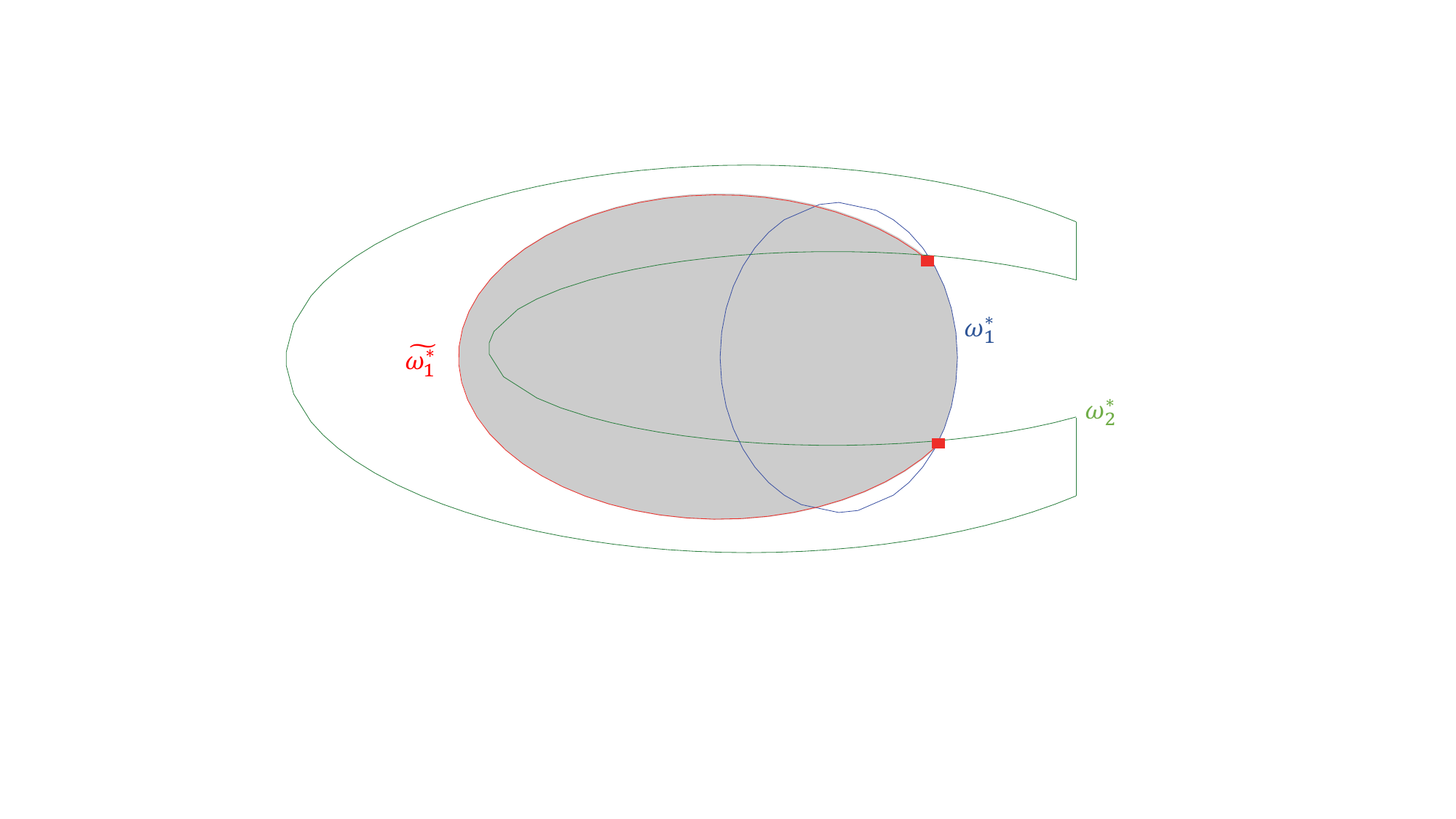}
    \caption{The set $\widetilde{\omega_1^*}$.}\label{domain00}
    \end{figure}

In conclusion, we have shown that \( \omega_1^* \setminus \overline{\omega_2^*} = \emptyset \). By symmetry, the reverse inclusion holds (i.e., \( \omega_2^* \setminus \overline{\omega_1^*} = \emptyset \)), and we deduce
$$\omega_1^* = \omega_2^*$$

%%%%%%%%%%%%%%%%%%%%%%%%%%%%%%%%%%%%%%%%%%%%%%%%%%%%%%%%%%%%
\subsection{Convergence analysis of the penalized problem}\label{Justification-penalizing}
%%%%%%%%%%%%%%%%%%%%%%%%%%%%%%%%%%%%%%%%%%%%%%%%%%%%%%

In this section, we establish the convergence of the solution of the penalized problem \eqref{problem-penalized-0} to that of the perturbed problem \eqref{problem-perturbed} as \( k \to +\infty \). To this end, for each \( k > 0 \), we denote by \( (u_\varepsilon^k,\, \pi_\varepsilon^k) \) the solution to the following boundary value problem
\begin{equation}\label{problem-penalized}
\left\{
\begin{array}{rll}
\displaystyle \frac {\partial u_\varepsilon^k}{\partial t}-\nu
\Delta u_\varepsilon^k+ \N(u_\varepsilon^k)+ \nabla \pi_\varepsilon^k
 +k\chi_{\mathcal{C}_{z,\varepsilon}}\, u_\varepsilon^k&= \mathcal{G} &\mbox{ in }  \,\,\Omega\times (0,\,T) ,\\
\mbox{div}\, u_\varepsilon^k &=  0 & \mbox{ in } \,\, \Omega \times (0,\,T) ,\\
u_\varepsilon^k &= 0 & \mbox{ on }\,\,  \partial\Omega \times (0,\,T) , \\
u_\varepsilon^k(\cdot,0)&= 0&\mbox{ in } \,\,  \Omega.
\end{array}
 \right.
\end{equation}

Under Assumption (A1), and by invoking \cite[Theorem 3.8]{temam1973theory}, it follows that problem \eqref{problem-perturbed-00} admits a unique solution \( u_\varepsilon^k \in L^\infty(0,T;\H^2(\Omega)) \), and hence \( \nabla u_\varepsilon^k \in L^\infty(0,T;\L^2(\Omega)) \). In light of this regularity, we further assume that there exists a constant \( \beta > 0 \) such that
\begin{equation}\label{condition-uvarepsilon}
    \left\|\nabla u_\varepsilon^k\right\|_{L^\infty(0,T;\,\L^2(\Omega))} <\beta<  \frac{\nu}{\rho(\Omega)},
\end{equation}
where \( \rho(\Omega) \) is defined in \eqref{uni-con11}.

\begin{proposition}\label{proposition-bounded}
Let $u_\varepsilon^k$ be the solution of the penalized problem \eqref{problem-penalized}. Then, there exists a constant $C>0$, independent of $k$ and $\varepsilon$, such that the following estimates hold:
    \begin{align} \label{ennq1}\big\|\frac{\partial u_\varepsilon^k}{\partial t}\big\|_{L^{4/3}(0,T;(\H^1_{0,\text{div}}(\Omega))^\prime)}+\big\|u_\varepsilon^k\big\|_{L^\infty(0,T;\L^2(\Omega))}+\big\|u_\varepsilon^k\big\|_{L^2(0,T;\H^1(\Omega))}&\leq C  \big\| \mathcal{G} \big\|_{L^2(0,T;(\H^1_{0,\div}(\Omega))^\prime)},\\
    \label{ennq2}\big\| u_\varepsilon^k \big\|_{L^2(0,T;\L^2(\mathcal{C}_{z,\varepsilon}))}&\leq \frac{C}{\sqrt{k}}\big\| \mathcal{G} \big\|_{L^2(0,T;(\H^1_{0,\div}(\Omega))^\prime)}.
    \end{align}
\end{proposition}

\begin{proof}
By applying the same analysis used in the proof of \eqref{estim-R3-1-new}, one can deduce that 
%Thus, the energy equality becomes:
\begin{align*}
    \frac{1}{2}\int_\Omega \big|u_\varepsilon^k(\cdot,t_0)\big|^2\dx+\nu\int_0^{t_0}\int_\Omega\big|\nabla u_\varepsilon^k\big|^2\dx\dt+k\int_0^{t_0}\int_{\mathcal{C}_{z,\varepsilon}}\big|u_\varepsilon^k\big|^2\dx\dt=\int_0^{t_0}\big<\mathcal{G}, u_\varepsilon^k\big>_\Omega\dt,
\end{align*}  
for almost all $t_0\in[0,\,T].$ By applying the Poincaré inequality and taking the supremum over $t_0\in[0,\,T],$ we can conclude that there exists a constant $C>0$, which is independent of both $k$ and $\varepsilon$, such that
\begin{align}
\begin{split}\label{LK}
\big\|u_\varepsilon^k\big\|^2_{L^\infty(0,T;\L^2(\Omega))}+\big\|u_\varepsilon^k\big\|^2_{L^2(0,T;\H^1(\Omega))}&+k\big\|u_\varepsilon^k\big\|^2_{L^2(0,T;\L^2(\mathcal{C}_{z,\varepsilon}))}\\
&\leq C\big\|\mathcal{G}\big\|_{L^2(0,T;(\H^1_{0,\div}(\Omega))^\prime)}\big\|u_\varepsilon^k\big\|_{L^2(0,T;\H^1_{0,\div}(\Omega))}\\
&\leq C\big\|\mathcal{G}\big\|_{L^2(0,T;(\H^1_{0,\div}(\Omega))^\prime)}\big\|u_\varepsilon^k\big\|_{L^2(0,T;\H^1(\Omega))}.
\end{split}
\end{align} 
From the above inequality \eqref{LK}, we derive
\begin{align*}
\big\|u_\varepsilon^k\big\|^2_{L^\infty(0,T;\L^2(\Omega))}+\big\|u_\varepsilon^k\big\|^2_{L^2(0,T;\H^1(\Omega))}\leq C\big\|\mathcal{G}\big\|_{L^2(0,T;(\H^1_{0,\div}(\Omega))^\prime)}\big\|u_\varepsilon^k\big\|_{L^2(0,T;\H^1(\Omega))}
\end{align*}
which, combined with the Young inequality, provides
\begin{align}\label{LKK}
\big\|u_\varepsilon^k\big\|_{L^\infty(0,T;\L^2(\Omega))}+\big\|u_\varepsilon^k\big\|_{L^2(0,T;\H^1(\Omega))}\leq C\big\|\mathcal{G}\big\|_{L^2(0,T;(\H^1_{0,\div}(\Omega))^\prime)}.
\end{align}
From \eqref{LK}, we also obtain
\begin{equation}
k\big\|u_\varepsilon^k\big\|^2_{L^2(0,T;\L^2(\mathcal{C}_{z,\varepsilon}))}\leq C\big\|\mathcal{G}\big\|_{L^2(0,T;(\H^1_{0,\div}(\Omega))^\prime)}\big\|u_\varepsilon^k\big\|_{L^2(0,T;\H^1(\Omega))}.
\end{equation}
Substituting the bound for $\big\|u_\varepsilon^k\big\|_{L^2(0,T;H^1(\Omega))}$ from \eqref{LKK}, we find
\begin{equation*}
\big\|u_\varepsilon^k\big\|^2_{L^2(0,T;\L^2(\mathcal{C}_{z,\varepsilon}))}\leq \frac{C^2}{k}\big\|\mathcal{G}\big\|^2_{L^2(0,T;(\H^1_{0,\div}(\Omega))^\prime)}.
\end{equation*}
Finally, by following the same line-by-line analysis as in the proof of estimate (3.8) in \cite{pokorny2022navier}, we can show that $\big\|\frac{\partial u_\varepsilon^k}{\partial t}\big\|_{L^{4/3}(0,T;(\H^1_{0,\text{div}}(\Omega))^\prime)}\leq C  \big\| \mathcal{G} \big\|_{L^2(0,T;(\H^1_{0,\div}(\Omega))^\prime)}$. This completes the proof of the proposition.
\end{proof}

Next, we proceed to prove the weak convergence of the sequence $\{u_\varepsilon^k\}_{k>0}$.

\begin{proposition} \label{con-pro}The sequence $u_\varepsilon^k$ converges weakly to $u_\varepsilon$ in  $L^2(0,T;\H^1_{0,\div}(\Omega))$ as $k\longrightarrow\infty,$ where \( u_\varepsilon \) is the solution of problem \eqref{problem-perturbed}.
\end{proposition}

\begin{proof}According to Proposition \ref{proposition-bounded} (Eq. \eqref{ennq1}), there exists $\overline{u}\in L^2(0,T;\H^1(\Omega))$ and a subsequence, still denoted by $\{u_\varepsilon^k\}_{k}$, such that
\begin{align}
  \label{enq1}  &u_\varepsilon^k\rightharpoonup \overline{u} \quad\text{in}\quad L^2(0,T;\H^1(\Omega))\;\;\text{as}\;\; k\longrightarrow\infty,\\
 \label{enq2}    &u_\varepsilon^k\longrightarrow \overline{u} \quad\text{in}\quad L^2(0,T;\L^2(\Omega))\;\;\text{as}\;\; k\longrightarrow\infty,\\
\label{enq3}      &\nabla u_\varepsilon^k\rightharpoonup \nabla\overline{u} \quad\text{in}\quad L^2(0,T;\L^2(\Omega))\;\;\text{as}\;\; k\longrightarrow\infty,\\
   \label{enq4}    &\frac{\partial u_\varepsilon^k}{\partial t}\rightharpoonup \frac{\partial \overline{u}}{\partial t} \quad\text{in}\quad L^{4/3}(0,T;(\H^1_{0,\text{div}}(\Omega))^\prime)\;\;\text{as}\;\; k\longrightarrow\infty.
\end{align}
Using \eqref{enq2} and \eqref{ennq2}, one can deduce that $\overline{u}=0$ in $\mathcal{C}_{z,\varepsilon}\times(0,T).$ Moreover, by applying trace theorem, we can see that
\begin{equation}
\overline{u}=0\;\;\text{on}\;\;\partial\mathcal{C}_{z,\varepsilon}\times(0,T).
\end{equation}
On the other hand, from the weak formulation of the problem \eqref{problem-penalized}, we have that for all $w\in L^2(0,T;\H^1_{0,\div}(\Omega)),$
\begin{align}
\begin{split}\label{FVuk}
    \int_0^T\int_\Omega k\chi_{\mathcal{C}_{z,\varepsilon}}u_\varepsilon^k\cdot w \,\dx\dt&=-\int_0^T\big<\frac{\partial u_\varepsilon^k}{\partial t},\,w \big>_\Omega\dt-\nu\int_0^T\int_\Omega\nabla u_\varepsilon^k:\nabla w \,\dx\dt\\
    &\qquad\qquad-\int_0^T\int_\Omega\big(u_\varepsilon^k\cdot\nabla\big) u_\varepsilon^k\cdot w \,\dx\dt +\int_0^T\big<\mathcal{G},\,w \big>_\Omega\dt.
    \end{split}
\end{align}
Using the convergence results \eqref{enq1}, \eqref{enq3}, and \eqref{enq4}, we have
\begin{align}
\begin{split}
      \int_0^T\int_\Omega k\chi_{\mathcal{C}_{z,\varepsilon}}u_\varepsilon^k\cdot w \,\dx\dt&\longrightarrow-\int_0^T\big<\frac{\partial \overline{u}}{\partial t},\,w \big>_\Omega\dt-\nu\int_0^T\int_\Omega\nabla \overline{u}:\nabla w \,\dx\dt\\
    &\qquad\qquad-\int_0^T\int_\Omega\big(\overline{u}\cdot\nabla \big)\overline{u}\cdot w \,\dx\dt +\int_0^T\big<\mathcal{G},\,w \big>_\Omega\dt\end{split}
\end{align}
as $k\longrightarrow\infty.$ Consequently, the sequence $\{k\chi_{\mathcal{C}_{z,\varepsilon}}u_\varepsilon^k\}_k$ converges weakly to a function $\varphi\in L^2(0,T;(\H^1_{0,\div}(\Omega))^\prime)$ such that $\text{supp}(\varphi)\subset\mathcal{C}_{z,\varepsilon}\times(0,\,T).$ Then, taking the limit $k\longrightarrow\infty$ in \eqref{FVuk}, we have that for all $w\in L^2(0,T;\H^1_{0,\div}(\Omega))$
\begin{align}
\begin{split}\label{FVUbar}
    \int_0^T\big<\frac{\partial \overline{u}}{\partial t},\,w \big>_\Omega\dt+\nu\int_0^T\int_\Omega\nabla \overline{u}:\nabla w \,\dx\dt
    &+\int_0^T\int_{\Omega}\big(\overline{u}\cdot\nabla\big) \overline{u}\cdot w \,\dx\dt\\
    &\qquad+   \int_0^T\big<\varphi,\,w \big>_\Omega\dt=\int_0^T\big<\mathcal{G},\,w \big>_\Omega\dt.
    \end{split}
\end{align}
Since $u_\varepsilon^k=0$ on $\partial\Omega\times(0,T)$, we have $\overline{u}=0$ on $\partial\Omega\times(0,T)$ due the continuity of the trace operator. Thanks to \cite[Proposition I.1.3]{temam1973theory} (see also Theorem 3.2.1 in \cite{pokorny2022navier}), there exists $\overline{\pi}(t)\in L^2(\Omega)$ for all $t\in(0,T)$ such that

\begin{align}
\begin{split}\label{problemhh}
    \int_\Omega\overline{u}(\cdot,T)\cdot \zeta \,\dx- \int_\Omega\overline{u}(\cdot,0)\cdot \zeta\,\dx&+\nu\int_0^T\int_\Omega\nabla \overline{u}:\nabla \zeta \,\dx\dt
    +\int_0^T\int_{\Omega}\big(\overline{u}\cdot\nabla\big) \overline{u}\cdot \zeta \,\dx\dt\\
    &\;+\int_\Omega \overline{\pi}(t)\,\text{div}\, \zeta\,\dx+   \int_0^T\big<\varphi,\,\zeta \big>_\Omega\dt=\int_0^T\big<\mathcal{G},\,\zeta \big>_\Omega\dt,
    \end{split}
\end{align}
for all $\zeta\in \H^1_0(\Omega)$. In addition, by repeating the same argument as that in the proof of \eqref{initial-u0}, we can prove that $\overline{u}(\cdot,0)=0$ in $\Omega.$ From this and using that 
$$
\int_0^T\int_\Omega\frac{\partial \overline{u}}{\partial t}\cdot\zeta\dx\dt= \int_\Omega\overline{u}(\cdot,T)\cdot \zeta \,\dx- \int_\Omega\overline{u}(\cdot,0)\cdot \zeta\,\dx
$$
and using the fact that $\text{supp}( \varphi)\subset\mathcal{C}_{z,\varepsilon} \times(0,T)$, it follows that for all $\zeta\in \H^1_{0,\div}(\Omega)$, with $\zeta=0$ in $\mathcal{C}_{z,\varepsilon}\times(0,T),$ we have
\begin{align*}
\int_0^T\int_{\Omega\backslash\overline{\mathcal{C}_{z,\varepsilon}}}\frac{\partial \overline{u}}{\partial t}\cdot\zeta \,\dx\dt&+\nu\int_0^T\int_{\Omega\backslash\overline{\mathcal{C}_{z,\varepsilon}}}\nabla \overline{u}:\nabla \zeta \,\dx\dt\\
&\qquad+\int_0^T\int_{\Omega\backslash\overline{\mathcal{C}_{z,\varepsilon}}}\big(\overline{u}\cdot\nabla\big) \overline{u}\cdot \zeta \,\dx\dt =\int_0^T\int_{\Omega\backslash\overline{\mathcal{C}_{z,\varepsilon}}}\mathcal{G}\cdot\zeta \,\dx\dt,
\end{align*}
which implies that $(\overline{u}_{|_{\Omega\backslash\overline{\mathcal{C}_{z,\varepsilon}}}},\overline{\pi}_{|_{\Omega\backslash\overline{\mathcal{C}_{z,\varepsilon}}}})$  is a weak solution fo the perturbed problem \eqref{problem-perturbed}.  Since the problem \eqref{problem-perturbed} has a unique solution, this implies that $(\overline{u},\overline{\pi})=(u_\varepsilon,\,\pi_\varepsilon)$ in $(\Omega\backslash\overline{\mathcal{C}_{z,\varepsilon}})\times(0,\,T)$. Consequently, $(\overline{u},\overline{\pi})=(u_\varepsilon,\,\pi_\varepsilon)$ in $\Omega\times(0,\,T)$ and $u_\varepsilon^k$ converges weakly to $u_\varepsilon$ in  $L^2(0,T;\H^1_{0,\div}(\Omega))$. Recall that the pair \( (u_\varepsilon, \pi_\varepsilon) \) is the unique solution of \eqref{problem-perturbed}.
\end{proof}

Now we can establish the following strong convergence result.

\begin{theorem}\label{cv-penalized}
Let $k>0,$ $u_\varepsilon$ be the solution of the perturbed problem \eqref{problem-perturbed}, and $u_\varepsilon^k$ be the solution of the penalized problem \eqref{problem-penalized}. Then, we have:
 \begin{equation}\label{cvv-final}
        \big\| u_\varepsilon^k- u_\varepsilon \big\|_{L^2(0,T;\H^1(\Omega))}\longrightarrow0\quad\text{as}\quad k\longrightarrow\infty.
    \end{equation}
Moreover, there exists a constant $C>0$, independent of $k$, such that
\begin{equation}
    \big\| u_\varepsilon^k-u_\varepsilon \big\|_{L^2(0,T;\L^2(\mathcal{C}_{z,\varepsilon}))}\leq \frac{C}{\sqrt{k}}.
    \end{equation}
\end{theorem}

\begin{proof}
Let \( w^k = u^k_\varepsilon - \overline{u} \). From the variational formulations \eqref{FVuk} and \eqref{FVUbar}, it follows that for all \( w \in \H^1_{0,\mathrm{div}}(\Omega) \), we have
 \begin{align}
\begin{split}\label{FV-diff}
    \int_0^T\int_\Omega\frac{\partial w^k}{\partial t}\cdot w \,\dx\dt&+\nu\int_0^T\int_\Omega\nabla w^k:\nabla w \,\dx\dt+\int_0^T\int_\Omega\big[(u^k_\varepsilon\cdot\nabla) u^k_\varepsilon-(\overline{u}\cdot\nabla) \overline{u}\big]\cdot w \,\dx\dt\\
    &\;\;-   \int_0^T\big<\varphi,\,w \big>_\Omega\dt+k\int_0^T\int_\Omega\chi_{\mathcal{C}_{z,\varepsilon}}\,u_\varepsilon^k\cdot w \,\dx\dt=0.
    \end{split}
\end{align}
By choosing \( w = w^k \) as a test function in \eqref{FV-diff}, noting that \( w^k(\cdot,0) = 0 \) in \( \Omega \) and \( \overline{u} = 0 \) in \( \mathcal{C}_{z,\varepsilon} \times (0,T) \), and applying Lemma \ref{lem-NS-prelim2}, we deduce that
 \begin{align}
\begin{split}\label{FV-diffnew1}
    \frac{1}{2}\int_\Omega\big|w^k(\cdot,T)\big|^2\dx&+\nu\int_0^T\int_\Omega\big|\nabla w^k\big|^2\dx\dt+k\int_0^T\int_\Omega\chi_{\mathcal{C}_{z,\varepsilon}}|w^k|^2 \,\dx\dt\\
    &\qquad=  \int_0^T\big<\varphi,\,w^k \big>_\Omega\dt-\int_0^T\int_\Omega (w^k\cdot\nabla) u^k_\varepsilon\cdot w^k \,\dx\dt.
    \end{split}
\end{align}
By an adaptation of the same technique used in the proof of estimate \eqref{estim-R3-3} and using assumption \eqref{condition-uvarepsilon}, one can deduce that there exists a constant \( \beta > 0 \) such that
\begin{align}\label{estim-R3-3-new}
        \big| \int_0^T \int_\Omega (w^k \cdot \nabla) u_\varepsilon^k \cdot w^k \, \dx\dt \big|  
        \leq \beta\, \rho(\Omega)  \left\|\nabla w^k\right\|^2_{L^2(0,T;\L^2(\Omega))}
    \end{align}
By inserting \eqref{estim-R3-3-new} into \eqref{FV-diffnew1} and invoking the weak convergence result of Proposition \ref{con-pro}, we obtain
\begin{align*}
   (\nu-\beta \, \rho(\Omega))\big\|\nabla w^k\big\|^2_{L^2(0,T;\L^2(\Omega))}+k\big\|w^k\big\|^2_{L^2(0,T;\L^2(\mathcal{C}_{z,\varepsilon}))}\leq  \int_0^T\big<\varphi,\,w^k \big>_\Omega\dt\longrightarrow0\text{ as } k\to\infty
\end{align*}
so we have proved that $\big\|\nabla w^k\big\|_{L^2(0,T;\L^2(\Omega))}\longrightarrow0$ and $\big\|w^k\big\|_{L^2(0,T;\L^2(\mathcal{C}_{z,\varepsilon}))}=O({k}^{\frac{-1}{2}}).$ Finally, the proof of the convergence result in \eqref{cvv-final} is completed by invoking the Poincar\'e inequality.    
\end{proof}

\paragraph{\bf{Data availability}} No data have been used for this article.\\

\paragraph{\bf{ Declarations}}\, \\

\paragraph{\bf{Conflict of interest}} The author certifies that he has no affiliations with or involvement in any organization or
entity with any financial interest or non-financial interest in the subject matter or materials discussed in this
manuscript.\\

\paragraph{\bf{Acknowledgements}} The authors thank Professor Ali Harzallah (National Institute of Marine Sciences and Technologies) for his valuable discussions on the numerical implementation of the POM model used in this work.

%%%%%%%%%%%%%%%%%%%%%%%%%
\bibliographystyle{plain}
\bibliography{sample}
%%%%%%%%%%%%%%%%%%%%%%%%%

\end{document}